\documentclass[letterpaper, 12pt, oneside]{book}
\usepackage[margin = 1in, includehead, footskip=0.25in]{geometry}

\usepackage{setspace}
\usepackage{caption}
\usepackage{subcaption}

\usepackage{tocbibind}

\usepackage{url}
\usepackage{amsmath}
\usepackage{amsthm}
\usepackage{amssymb}
\usepackage{multirow}
\usepackage{booktabs}
\usepackage{tabularx}
\usepackage{longtable}
\usepackage{lscape}
\usepackage{graphicx}
\usepackage{import}
\usepackage[multiple]{footmisc}
\usepackage{multicol}
\usepackage{epigraph}

\usepackage{amssymb,amsmath,amsfonts,amsthm,mathtools,latexsym,ifthen,bezier}%,epic,eepic,picture
\usepackage{graphicx}
\usepackage{amsxtra}
\usepackage{xspace}
\usepackage{dirtytalk}
\usepackage{url}
\usepackage[english]{babel} %ngerman
\usepackage{enumitem}
\usepackage{setspace}
\usepackage[english]{babel}
\usepackage{latexsym}
\usepackage{enumitem}
\usepackage{amsthm}
\usepackage{amssymb}
\DeclareMathAlphabet{\mathpzc}{OT1}{pzc}{m}{it}
\usepackage[latin1]{inputenc}
\usepackage{afterpage}

\newtheorem{theorem}{Theorem}[section]

\newtheorem{lemma}[theorem]{Lemma}
\newtheorem{proposition}[theorem]{Proposition}
\newtheorem{corollary}[theorem]{Corollary}
\newtheorem{observation}[theorem]{Observation}
\newtheorem{fact}[theorem]{Fact}
\newtheorem{claim}[theorem]{Claim}

\theoremstyle{definition}
\newtheorem{definition}[theorem]{Definition}
\newtheorem{example}[theorem]{Example}
\theoremstyle{remark}
\newtheorem{remark}{Remark}

\newtheorem{question}{Question}

\def\hook{\upharpoonright}
\def\forces{\Vdash}
\def\barN{\bN}
\def\barG{\overline{G}}

\newcommand{\hooks}{\mathrel{\angle}}

\newcommand{\kla}[1]{ {\langle #1 \rangle} }

\newcommand{\st}{\;|\;}
\newcommand{\ran}{ {\rm ran} }

\newcommand{\sub}{\subseteq}

\newfont{\ssi}{cmssi12 at 12pt}

\newcommand{\rest}{{\restriction}}

\newcommand{\verl}{{{}^\frown}}%{{\frown\atop }}

\newenvironment{ea*}{\begin{eqnarray*}}{\end{eqnarray*}}
\setbox0=\hbox{$\longrightarrow$}

\newcommand{\ba}{{\bar{a}}}

\newcommand{\bs}{{\bar{s}}}
\newcommand{\bu}{{\bar{u}}}
\newcommand{\bA}{{\bar{A}}}

\newcommand{\bG}{{\bar{G}}}

\newcommand{\bS}{{\bar{S}}}
\newcommand{\bT}{{\bar{T}}}

\newcommand{\bd}{{\bar{d}}}
\newcommand{\bj}{{\bar{j}}}
\newcommand{\bk}{{\bar{k}}}

\newcommand{\btau}{{\bar{\tau}}}

\newcommand{\btheta}{{\bar{\theta}}}

\newcommand{\bsigma}{{\bar{\sigma}}}

\newcommand{\bp}{{\bar{p}}}

\newcommand{\bN}{{\bar{N}}}

\newcommand{\tsigma}{{\tilde{\sigma}}}

\usepackage{bbm}
\usepackage[rgb,dvipsnames]{xcolor}
\usepackage{tikz} % for diagrams and drawing
\usetikzlibrary{decorations.text}
\usetikzlibrary{arrows,arrows.meta,hobby}
\usetikzlibrary{shapes.misc, positioning}

\newcommand{\seq}[2]{{\langle#1\;|\;}\linebreak[0]{#2\rangle}}

\renewcommand{\phi}{\varphi}

\newcommand{\ZFC}{\ensuremath{\mathsf{ZFC}}\xspace}

\newcommand{\V}{\ensuremath{\mathrm{V}}}

\def\<#1>{\langle#1\rangle}

\renewcommand{\P}{{\mathord{\mathbb P}}}

\newcommand{\MA}{\ensuremath{\mathsf{MA}}}
\newcommand{\MM}{\ensuremath{\mathsf{MM}}\xspace}

\newcommand{\PFA}{\ensuremath{\mathsf{PFA}}\xspace}

\newcommand{\SCFA}{\ensuremath{\mathsf{SCFA}}\xspace}

\newcommand{\MP}{\ensuremath{\mathsf{MP}}}
\newcommand{\ColNothing}{\mathrm{Col}}
\newcommand{\Col}[1]{\ColNothing(#1)}

\newcommand{\MPColNothing}[1]{\MP_{\Col{\dot{\kappa}}}}

\newcommand{\GCH}{\ensuremath{\mathsf{GCH}}\xspace}
\newcommand{\CH}{\ensuremath{\mathsf{CH}}\xspace}

\def\hook{\upharpoonright}
\def\forces{\Vdash}
\def\barN{\bN}
\def\barG{\overline{G}}

\def\Me{\mathcal M}
\def\Null{\mathcal N}
\def\ZFC{\mathsf{ZFC}}
\def\GBC{\mathsf{GBC}}
\def\PA{\mathsf{PA}}
\def\MM{\mathsf{MM}}
\def\PFA{\mathsf{PFA}}
\def\MA{\mathsf{MA}}
\def\SCFA{\mathsf{SCFA}}
\def\DCFA{\mathsf{DCFA}}

\def\Kb{\mathcal K}
\def\baire{\omega^\omega}
\def\bbb{(\omega^\omega)^{\omega^\omega}}
\def\cantor{2^\omega}

\def\mfc{\mathfrak{c}}
\def\mfb{\mathfrak b}
\def \mfd{\mathfrak{d}}
\def\GCH {\mathsf{GCH}}
\def\CH {\mathsf{CH}}

\newcommand{\suc}{\mathop{\mathrm{suc}}}

\begin{document}

% front matter\
\frontmatter

% dissertation guide requires that title page is part of page count, but has no page number
% http://tex.stackexchange.com/questions/66562/how-to-set-page-counter-by-skipping-first-page
\begin{titlepage}

\begin{center}

~\vspace{2in}

\textsc{Alternative Cicho\'n Diagrams and Forcing Axioms Compatible with $\mathsf{CH}$} \\[0.5in]
by \\[0.5in]
\textsc{Corey Bacal Switzer}

\vspace{\fill}
A dissertation submitted to the Graduate Faculty in Mathematics in partial fulfillment of the requirements for the degree of Doctor of Philosophy, The City University of New York. \\[0.25in]
2020

\end{center}

\end{titlepage}

\setcounter{page}{2}

% front matter
\phantom{}\vspace{\fill}
\begin{center}
\copyright~2020\\
\textsc{Corey Bacal Switzer}\\
All Rights Reserved\\
\end{center}

\begin{center}
  
\textsc{Alternative Cicho\'n Diagrams and Forcing Axioms Compatible with $\CH$} \\[0.2in]
by \\[0.2in]
\textsc{Corey Bacal Switzer}

\vspace{0.5in}

This manuscript has been read and accepted by the Graduate Faculty in Mathematics in satisfaction of the dissertation requirement for the degree of Doctor of Philosophy.
\end{center}

\vspace{0.5in}

\begin{tabular}{p{1.75in}p{0.5in}p{3.5in}}
~                                   & & \textbf{Professors Gunter Fuchs and Joel David Hamkins}\\
~                                   & & \\
\hrulefill                          & &\hrulefill \\
Date                                & & Co-Chairs of Examining Committee\\
~                                   & & \\
~                                   & & \textbf{Professor Ara Basmajian}\\
~                                   & & \\
\hrulefill                          & &\hrulefill \\
Date                                & & Executive Officer\\
\end{tabular}

\vspace{0.5in}

\begin{tabular}{l}
\textbf{Professor Gunter Fuchs} \\
\textbf{Professor Joel David Hamkins} \\
\textbf{Professor Arthur Apter} \\
Supervisory Committee \\
\end{tabular}

\vspace{\fill}
\begin{center}
\textsc{The City University of New York}
\end{center}

\begin{center}
Abstract \\
\textsc{Alternative Cicho\'n Diagrams and Forcing Axioms Compatible with $\CH$} \\
by \\
\textsc{Corey Bacal Switzer} \\[0.25in]
\end{center}

\vspace{0.25in}

\noindent Advisors: Professors Gunter Fuchs and Joel David Hamkins

\vspace{0.25in}

\noindent
This dissertation surveys several topics in the general areas of iterated forcing, infinite combinatorics and set theory of the reals. There are four largely independent chapters, the first two of which consider alternative versions of the Cicho\'n diagram and the latter two consider forcing axioms compatible with $\CH$. In the first chapter, I begin by introducing the notion of a {\em reduction concept}, generalizing various notions of reduction in the literature and show that for each such reduction there is a Cicho\'n diagram for effective cardinal characteristics relativized to that reduction. As an application I investigate in detail the Cicho\'n diagram for degrees of constructibility relative to a fixed inner model $W \models \ZFC$. 

In the second chapter, I study the space of functions $f:\baire \to \baire$ and introduce 18 new higher cardinal characteristics associated with this space. I prove that these can be organized into two diagrams of 6 and 12 cardinals respecitvely analogous to the Cicho\'n diagram on $\omega$. I then investigate their relation to cardinal invariants on $\omega$ and introduce several new forcing notions for proving consistent separations between the cardinals.

The third chapter concerns Jensen's subcomplete and subproper forcing. I generalize these notions to the (seemingly) larger classes of $\infty$-subcomplete and $\infty$-subproper. I show that both classes are (apparently) much more nicely behaved structurally than their non-$\infty$-counterparts and iteration theorems are proved for both classes using Miyamoto's nice iterations. Several preservation theorems are then presented. This includes the preservation of Souslin trees, the Sacks property, the Laver property, the property of being $\baire$-bounding and the property of not adding branches to a given $\omega_1$-tree along nice iterations of $\infty$-subproper forcing notions. As an application of these methods I produce many new models of the subcomplete forcing axiom, proving that it is consistent with a wide variety of behaviors on the reals and at the level of $\omega_1$. 

The final chapter contrasts the flexibility of $\SCFA$ with Shelah's dee-complete forcing and its associated axiom $\DCFA$. Extending a well known result of Shelah, I show that if a tree of height $\omega_1$ with no branch can be embedded into an $\omega_1$-tree, possibly with branches, then it can be specialized without adding reals. As a consequence I show that $\DCFA$ implies there are no Kurepa trees, even if $\CH$ fails.

\chapter*{Dedication}

\begin{center}

\vspace{2in}

For Paki, who taught me to ask questions

and

Mahala, who showed me how to find answers

\end{center}

\chapter*{Acknowledgments}

It's hard for me to put into words how grateful I am to all of the people who helped me enormously in this undertaking. First and foremost I would like to thank my advisors Joel and Gunter. Thank you both so much. 

Gunter, you are an inspiring mentor. Thank you for all of our meetings at the coffee shop, for all that you taught me, for your patience, and for teaching me to be more careful in checking myself (someday I hope to get this right). 

Joel, your enthusiasm is infectious. Thank you for all you have shown me about how to think about mathematics. It is something I am grateful to be able to carry forward. Thank you also for my visit to Oxford, it was an amazing experience.

Next I want to thank Roman Kossak. Roman, even though you weren't my advisor I am so grateful for all of our time together and all that you taught me. Thank you for introducing me to models of PA and inviting me to give my first ever talk.

Vika, thank you for allowing me to tag along and play organizer with you, for sharing with me your perspective on math, and for inviting me to help organize the conference, even if it didn't happen.

I would like to thank all of the amazing faculty at CUNY that I had the pleasure to interact with during these last four years, particularly the logic group. Thanks to all of the logicians that allowed me to bother you with questions continuously every Friday during and between the seminars. 

Thanks to Arthur Apter for agreeing to be on my defense committee and Alf Dolich for agreeing to be on my oral exam committee. 

Thanks to Alice Medvedev for all of her invaluable advice.

At CUNY I want to also thank all of the staff that I interacted with over the years. I especially want to thank Debbie Silverman at the Graduate Center and Norma Moy at Hunter. 

Next I would like to thank all the (non-CUNY) logicians and mathematicians who met with me and patiently answered all my annoying questions. In particular thanks so much to Ali Enayat, Andrew Brooke-Taylor, Asaf Karagila, Brent Cody, Chris Lambie-Hanson, Giorgio Venturi, Grigor Sargsyan, Henry Towsner, Hiroshi Sakai, J\"{o}rg Brendle, Sean Cox, Saka\'{e} Fuchino, Simon Thomas, and Vera Fischer. Thanks especially to Boban Veli\v{c}kovi\'c for teaching me forcing, supervising my M2 thesis in Paris and all of the conversations we have had since then.

I want to thank Mirna D\v{z}amonja for introducing me to set theory, a gift I couldn't ever hope to repay. 

Thanks to all of the great friends I have made throughout my PhD. Thanks to all of the set theory and MOPA students, Alex, Kameryn, Kaethe, Miha, Eoin, Ryan, Ben and Athar. Double thanks to Kameryn for pushing me to go to Brazil and traveling with me to Brazil and Japan. Thank you to Iv\'{a}n and Micha\l{}. Thanks to James for the music.

Thanks to Sam, Oliver, May, Bo, Jesse and Alan. 

Thank you to Alfie the dog for not being dead yet and being so dumb.

Thank you to my family. To my parents Jeff and Karen for all their love and support and putting up with me living upstairs from them. To my sister Shauna for the chats, the travels and being smarter than I could ever be. To my in-laws Adam, Sylvaine and Noah for including me in your family. To Ron, for teaching me what a shillelagh is, and not beating me with one. To Daniel, Jay, Julia, Matt, Amanda, Alessandra, James, Ian, Brooke, Connor, Chase, Gail, Jeff, Ginny, Michael, Susan, Jesse and Becky. Merci \`{a} la famille Franc\`{e}s-Combes pour m'avoir accueillit \`{a} bras ouverts en France quand j'ai commenc\'{e} ces \'{e}tudes.  Thank you to my grandmother Annie for everything you do and did. 

This academic year we lost my grandfather, Joe ``Paki" Bacal. Words cannot express how much he meant to me. Even though he never would have understood this thesis, he would have read it anyway. I like to think he would have found in it shadows of what he taught me. 

Thank you Mahala. You are my best friend and I love you. Everyday you amaze me, frustrate me and inspire me. Without you I never would have even gone to grad school yet alone finished it.

I never expected to finish my PhD during a global pandemic. The events of the last few months have been difficult and at times, overwhelming. More so than ever I feel so grateful for everything, especially to have had the opportunity to study something I love so much in an as amazing environment as CUNY. On that note I want to finish these acknowledgments by thanking all of the students I have had the privilege to teach at Hunter College. I didn't go to grad school to teach, but it ended up being an endless source of joy and frustration, especially when the pandemic raged. 

Thank you to you all, you've taught me so much.

% tables of contents
\tableofcontents
%\listoftables %%I don't have any, so why have a blank list?
\listoffigures

% content
\mainmatter

%start chapter numbering at 0, like a sane person
\setcounter{chapter}{-1}
\chapter{Introduction}
\chaptermark{Introduction}

Paul Cohen's discovery of forcing, \cite{Cohen1964, Cohen1963, Cohen1966}, revolutionized set theory. The technique not only provided a flexible method for producing new models of $\ZFC$ but also opened up uncountably many possibilities for consistent models of the real line, telling vastly different stories about its topological and measure theoretic properties. Similarly, infinite combinatorics, especially concerning trees and their relatives, were soon seen to be equally malleable. These early results were extended by the discovery of iterated forcing, first seen in \cite{Easton1970}, and then, in the context of forcing axioms in \cite{SolovayTennenbaum, SolovayMartin}.

In this thesis I explore several topics roughly related to forcing constructions and the continuum. While each chapter is essentially independent, there are certain thematic threads that tie them together. Specifically, in each chapter I look at some aspect of set theory which is usually studied in the context of the failure of $\CH$, but I modify it to be compatible with the continuum taking many values, including $\aleph_1$. Here is a brief outline of the structure of the thesis.

%The first chapter contains a smattering of preliminaries that will be used throughout the rest of the thesis. In particular I survey some important points in the theory of iterated forcing; adding different types of reals; combinatorial, measure theoretic and topological aspects of the real line; and forcing axioms.

In the first chapter I introduce the notion of a {\em reduction concept}, generalizing the idea of Turing reduction and prove that for any \say{reasonable} reduction concept there is a corresponding Cicho\'n diagram. This extends known results of Rupprecht \cite{rupprecht} and others \cite{BBNN, Kihara17, GreenbergTuretskyKuyper} who have looked at various cases of \say{effective cardinal characteristics}. As an application of this general setup, I study the Cicho\'n diagram for degrees of constructibility relative to a fixed inner model $W \models \ZFC$. I show that this diagram is complete in the sense that any two-valued cut in the diagram is consistent, however in most interesting cases the diagram splits into more than one set in contrast to the cardinal case. Most of the individual results here are known, however collectively the new perspective sheds light on the relation between the forcing notions used in cardinal characteristics and their effective counterparts. The main result of this chapter is that there is a proper forcing $\mathbb P \in W$ so that every separation in the diagram can be made simultaneously in a way that is preserved by any further forcing. This leads to a new axiom, $\mathsf{CD}(\leq_W)$, which essentially states that the Cicho\'n diagram for $\leq_W$ is as complicated as possible. I show that $\mathsf{CD}(\leq_W)$ is compatible with $\CH$, but a stronger version implies all of the cardinals in the Cicho\'n diagram are greater than $\aleph_1$. The results of this chapter appear in print in \cite{Switz18}. 

In the second chapter I consider a different generalization of cardinal characteristics of the continuum. Much work recently has considered generalizations of well known characteristics to the space $\kappa^\kappa$. Here, I look instead at the space $\bbb$ of functions from $\baire$ to $\baire$. Eighteen new cardinal characteristics for this space are introduced and provable inequalities and consistent separations for these cardinals are investigated. I show that several diagrams similar to the Cicho\'n diagram exist for these spaces. I also show that various constellations for the cardinal characteristics on $\omega$ play a role in the values of these ones and I also introduce three new forcing notions for proving separations between the new cardinals. The results of this chapter appear in print in \cite{SwitzHighDimCC}.

In the third and fourth chapters I switch gears and turn my attention from cardinal characteristics to forcing axioms compatible with $\CH$. The existence of such axioms has long been of interest in set theory see, for example, \cite{MR3037610, MR3821630, AbrahamTodPP}, however, Jensen's recent work in subcomplete forcing (\cite{JensenSPSC}) represents a breakthrough. The subcomplete forcing axiom is a strong axiom which is compatible with $\CH$ and even $\diamondsuit$. In the third chapter I investigate the role of the continuum in this axiom and show that many consequences of $\CH$ and $\diamondsuit$ can be preserved while forcing $\CH$ to fail in a model of $\SCFA$. This involves proving new iteration and preservation theorems for subcomplete and subproper forcing. The type of iteration I use is the nice support iteration of Miyamoto \cite{miyamoto}. One of the unexpected advantages of this approach is it allows a (seemingly) more general class of forcing notions beyond subcomplete and subproper to be iterable. I dub these $\infty$-subcomplete and $\infty$-subproper forcing and consider the structural properties of these classes as well. The work in this chapter also appears in print as part of the larger work \cite{FuSw}.

In the fourth chapter, to contrast my work on subcomplete forcing I look at the axiom $\DCFA$, first considered alongside the assumptions of $\CH$ and $2^{\aleph_1} = \aleph_2$ in \cite{PIPShelah} and less restrictively by Jensen in \cite{JensenCH}. While $\DCFA$ is also compatible with $\CH$, in contrast to $\SCFA$ this axiom seems to effect the universe at the level of the continuum and $\omega_1$. I show that it implies that there are no Kurepa trees, a result sketched by Shelah in \cite{PIPShelah}. There the statement assumes the additional assumptions of $\CH$ and $2^{\aleph_1} = \aleph_2$, though they are not used. To prove this theorem I generalize Shelah's idea, introducing a forcing notion which can specialize certain wide Aronszajn trees of height $\omega_1$ and can be iterated without adding reals. I explore a few other applications of this forcing. The work in this chapter appears in \cite{Switz20widetrees}.

Since each chapter is essentially independent I provide preliminaries at the beginning of each chapter. In some cases, a definition is listed in two chapters for the convenience of the reader. However every effort has been made to uniformize notation. Also, the work in each chapter has led to ongoing research and I briefly outline at the end of each chapter open questions and current investigations along the lines of the content there.

\section{Notation and Some Basic Definitions}

Let me end this introduction by fixing some notation and recalling some basic definitions that will be used in every chapter.  Overall, most notation is standard, and all undefined terms can be found in the well known monographs \cite{KenST} and \cite{JechST}. Also, I use the monograph \cite{BarJu95} as the standard reference for cardinal characteristics of the continuum and occasionally refer to the survey article \cite{BlassHB} as well. Throughout this thesis I use the convention that if $\mathbb P$ is a forcing notion and $p, q \in \mathbb P$ with $q \leq p$ then $q$ is stronger than $p$. One notational convention which varies slightly from the norm is that for the most part I will let letters like $x, y, z, ...$ stand for reals (elements of $\cantor$, $\baire$, etc) and letters like $f, g, h, ...$ stand for functions between uncountable Polish spaces. This will be relevant in particular in chapter $2$ where I will frequently refer to both.

%In what follows I use letters like $x, y, z$ to denote elements of Baire space and letters like $f, g, h$ to denote functions from $\omega^\omega$ to $\omega^\omega$ (or, between uncountable Polish spaces more generally). 

Let $\mathcal I$ be a non-trivial ideal whose dual filter is non-principle. A set is $\mathcal I$-positive if it's not in $\mathcal I$ and is $\mathcal I$-measure one if its complement is in $\mathcal I$. For every such ideal $\mathcal I$ on a set $X$ we naturally associate four cardinal characteristics.
\begin{enumerate}
\item
{\em The additivity number}: $add (\mathcal I)$, the least size of a family of sets in $\mathcal I$ whose union is not in $\mathcal I$.

\item
{\em The uniformity number}: $non(\mathcal I)$, the least size of an $\mathcal I$-positive set.

\item
{\em The covering number}: $cov (\mathcal I)$, the least size of a family of sets in $\mathcal I$ needed to cover $X$.

\item
{\em The cofinality number}: $cof(\mathcal I)$, the least size of a family of sets in $\mathcal I$ which is cofinal in $\mathcal I$ with respect to inclusion.
\end{enumerate}

Given any set $X$ and a relation $R$ on $X$, we say that an element $x \in X$ is an $R$-{\em bound} for a set $A \subseteq X$ if for every $a \in A$ we have that $a \mathrel{R} x$. A set is $R$-{\em bounded} if it has an $R$-bound. It's $R$-{\em unbounded} otherwise. A set $D \subseteq X$ is $R$-{\em dominating} if for every $y \in X$ there is a $d \in D$ so that $y \mathrel{R} d$. For any such $X$ and $R$ I write $\mfb (R)$ for the least size of an $R$-unbounded set and $\mfd(R)$ for the least size of an $R$-dominating set. If $\mathbb{Q} = (Q, \leq_\mathbb Q)$ is a partially ordered set then I also write $\mfb (\mathbb Q)$ and $\mfd (\mathbb Q)$ for $\mfb (\leq_\mathbb Q)$ and $\mfd(\leq_\mathbb Q)$ respectively.

I let $\mu$ denote the Lebesgue measure on $\omega^\omega$ (or any other oft-encountered Polish space under consideration). The symbols $\Null$, $\Me$, $\Kb$ denote the null ideal, the meager ideal and the ideal generated by $\sigma$-compact subsets of $\baire$ respectively. If $x, y \in \baire$ then $x \leq^* y$ if and only if for all but finitely many $n$ we have $x(n) \leq y(n)$ and $\mfb = \mfb (\leq^*)$, $\mfd = \mfd (\leq^*)$. Recall that $A \in \Kb$ if and only if $A$ is $\leq^*$-bounded, see the proof of \cite[Theorem 2.8]{BlassHB}. The relevant properties that all three of these ideals share is that they are non-trivial $\sigma$-ideals containing all countable subsets of $\baire$ and have a Borel base: every element of each ideal is covered by a Borel set in that ideal. In the case $\Null$ and $\Me$ the fact that the underlying set is $\baire$, as opposed to any other perfect Polish space is unimportant in this thesis, however, it obviously matters for $\Kb$ since many Polish spaces are themselves $\sigma$-compact and hence $\Kb$ on such a space is trivial. 

Implicit in several of these chapters is the classical Cicho\'n diagram, see \cite[Chapter 2]{BarJu95}. This diagram relates the cardinal characteristics for $\Null$, $\Me$ and $\mfb$ and $\mfd$ (which are the cardinal invariants associated with $\Kb$). It is produced below for reference, note that $\mathfrak{x} \to \mathfrak{y}$ means that $\mathfrak{x}$ is $\ZFC$-provably less than or equal to $\mathfrak{y}$. 

\begin{figure}[h]\label{Figure.Cichon}
\centering
  \begin{tikzpicture}[scale=1.5,xscale=2]
     % place and draw the nodes
     \draw %(.75, 1) node (BsubN) {$\mathcal B(\subseteq_\mathcal N)$}
	 (2,1) node (addm) {$add(\Me)$}
           (4,1) node (nonn) {$non (\Null)$}
           %(2.25,1) node (Din) {$\mathcal D(\in_\mathcal M)$}
           (3,3) node (cofm) {$cof (\Me)$}
           (2,2) node (bbb) {$\mfb$}
           (3,2) node (ddd) {$\mfd$}
	(1, 3) node (covn) {$cov (\Null)$}
           %(3.75, 1) node (BinN)  {$cov (\Null)$}
           %(3.75, 3) node (DsubN) {}
	(1, 1) node (addn) {$add (\Null)$}
	(4, 3) node (cofn) {$cof (\Null)$}
	(0, 1) node (aleph1) {$\aleph_1$}
          (2, 3)  node (nonm) {$non (\Me)$}
          (5, 3) node (ccc) {$\mfc$}
          (3, 1) node (covm) {$cov (\Me)$}
	    
           ;
     % draw the arrows
     \draw[->,>=stealth]
	      (aleph1) edge (addn)
            (addn) edge (addm)
            (addm) edge (covm)
            (covm) edge (nonn)
            (addn) edge (covn)
            %(Din) edge (Dleq)
            (addm) edge (bbb)
            (bbb) edge (ddd)
	     (covm) edge (ddd)
            (ddd) edge (cofm)
	     %(Din) edge (BinN)
            (nonm) edge (cofm)
           (cofm) edge (cofn)
	    (cofn) edge (ccc)
           (nonn) edge (cofn)
           (covn) edge (nonm)
           (bbb) edge (nonm)
           %(Dslal) edge (DsubN)
           %(Bin) edge (Bneq)
           %(Bneq) edge (Bin)
           %(DsubN) edge (all)
          %(Din) edge (Dneq)
           %(Dneq) edge (BinN)
          %(Dneq) edge (Dleq)
            
            ;
  \end{tikzpicture}
\caption{The Cicho\'n Diagram}
\end{figure}
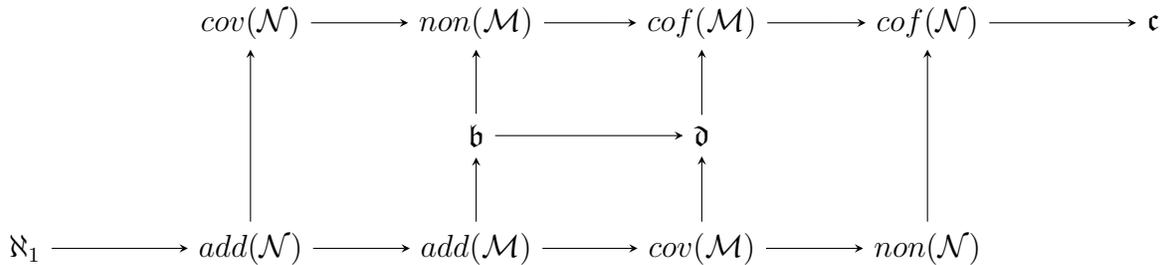

Finally we recall basic terminology of forcing axioms. If $\Gamma$ is a definable class of forcing notions and $\kappa$ is a cardinal then {\em Martin's Axiom For} $\Gamma$, sometimes also called {\em the forcing axiom for} $\Gamma$, which is denoted $\MA_\kappa(\Gamma)$, is the statement that for any $\mathbb P \in \Gamma$, and any $\kappa$ length sequence of dense subsets $\langle D_\alpha \; | \; \alpha < \kappa\rangle$ there is a filter $G \subseteq \mathbb P$ so that for any $\alpha < \kappa$ $G \cap D_\alpha \neq \emptyset$. If $\Gamma$ is the class of c.c.c. forcing notions then $\MA$ denotes $\forall \kappa < 2^{\aleph_0} \; \MA_\kappa (\Gamma)$ holds. If $\Gamma$ is the class of proper forcing notions then we write $\PFA$ for $\MA_{\aleph_1}(\Gamma)$. If $\Gamma$ is the class of stationary set preserving forcing notions then we write $\MM$ for $\MA_{\aleph_1}(\Gamma)$.

%Recall that a Polish space $X$ is a topological space which is separable and admits a complete metric. Of particular interest in this thesis is the Baire space, $\baire$ of infinite sequences of natural numbers, and the Cantor space $\cantor$ of counably infinite binary sequences. In both cases the underlying topology is the product topology on the discrete topology on $\omega$ and $\{0, 1\}$ respectively. Three combinatorial relations defined on Baire space will be relevant to almost every chapter.

\chapter{The Cicho\'n Diagram for Degrees of Relative Constructibility}
\chaptermark{The Cicho\'n Diagram for Degrees}
In this chapter I introduce the notion of a {\em reduction concept} and tie it to cardinal characteristics. While the notion of a reduction concept is rarely written down explicitly in published work (though see \cite{Kihara17, GreenbergTuretskyKuyper}), it has been implicit in the literature since the beginning of the 20th century. Turing reductions, polytime reductions, arithmetic reductions and degrees of constructibility are all examples of reduction concepts. Each comes with its own notion of degree. In this chapter I show that each such degree theory can formulate a variety of \say{highness properties} analogous to some common cardinal characteristics of of the continuum and the implications between these highness properties resemble those of the standard Cicho\'n diagram for cardinals. The case of Turing degrees was worked out in \cite{BBNN}, piggybacking off work from \cite{rupprecht}, so my main contribution here is generalizing the result to the general case. Similar ideas have been explored in Section 5 of \cite{GreenbergTuretskyKuyper} and in \cite{Kihara17}\footnote{Thanks to the anonymous referee of \cite{Switz18} for pointing this out to me.}. In contrast with those papers though I work more on the level building analogues of the Cicho\'n diagram than in considering relations between various types of reducibilities in higher computability theory and Tukey reductions. In particular, Theorem \ref{propdi}, which in slightly less general contexts is essentially folklore, shows that for any ``reasonable" reduction concept, a corresponding Cicho\'n diagram exists. In the second half of the chapter I consider the special case of degrees of constructibility relative to some fixed $W \models \ZFC$. In this case I show, amongst other things, that there is a proper forcing $\mathbb P \in W$ so that in $W^\mathbb P$ the Cicho\'n diagram for $\leq_W$ is fully separated in the sense that every node is non-empty and all consistent non-implications are simultaneously realized. This set up is expressed as the axiom $\mathsf{CD}(\leq_W)$ and some consequences of it are investigated. 

\section{Preliminaries}
Before beginning in earnest I list a few definitions and facts which will be used throughout this chapter. The first definition will in fact be essential throughout this thesis.

\begin{definition}[Combinatorial relations]
Let $x$ and $y$ be elements of $\baire$. Then
\begin{enumerate}
\item
$x \leq^* y$ if and only if for all but finitely many $k$ we have $x(k) \leq y (k)$. In this case we say that $y$ {\em eventually dominates} $x$.
\item
 $x \neq^* y$ if and only if for all but finitely many $k$ we have $x(k) \neq y(k)$. In this case say that $y$ is {\em eventually different from} $x$. Note that the negation of $\neq^*$ is {\em infinitely often equal}, not eventual equality. 
 \item
 Let $z \in \omega^\omega$ and recall that a $z$-{\em slalom} is a function $s: \omega \to [\omega]^{< \omega}$ such that for all $n \in \omega$ the set $|s(n)| \leq z(n)$. In the case where $z$ is the identity function call $s$ simply a {\em slalom}. I denote the space of all slaloms as $\mathcal S$. This space is can be treated as homeomorphic to Baire space in the obvious way, see \cite{PawRec95} for the details of a particularly useful coding of this correspondence. 

For a slalom $s$, I write $x \in^* s$ if and only if for all but finitely many $k$ we have $x(k) \in s(k)$. In this case say that $x$ {\em is eventually captured by} $s$.
 \end{enumerate}
\end{definition}

The cardinal characteristics associated with these relations have nice descriptions in terms of the ideals $\Null$, $\Me$ and $\Kb$. First, for $\in^*$ there is a relation with $\Null$.
\begin{fact}[Bartoszy\'nski, see Theorem 2.3.9 of \cite{BarJu95}]

The following equalities are provable in $\ZFC$.

\begin{enumerate}

\item
$\mathfrak{b} (\in^*) = add(\Null)$
\item
$\mathfrak{d}(\in^*) = cof (\Null)$
\end{enumerate}
\label{infact}
\end{fact}
Next, for $\neq^*$ there is a relation with $\Me$.
\begin{fact}[Bartoszy\'nski, see Thereoms 2.4.1 and 2.4.7 of \cite{BarJu95}]
The following equalities are provable in $\ZFC$.
\begin{enumerate}
\item
$\mfb (\neq^*) = non (\Me)$
\item
$\mfd (\neq^*) = cov (\Me)$
\end{enumerate}
\label{neqfact}
\end{fact}
Finally, for $\leq^*$, there is a relation with $\Kb$.
\begin{fact}[See Theorem 2.8 of \cite{BlassHB}]
The following equalities are provable in $\ZFC$.
\begin{enumerate}
\item
$\mfb = add (\mathcal K) = non (\Kb)$ 
\item
$\mfd = cov (\Kb) = cof (\Kb)$.
\end{enumerate}
\end{fact}

When attempting to control these relations while forcing, the following three properties of forcing notions will be useful.
\begin{definition}(\cite[Definition 6.3.37]{BarJu95})
Let $\mathbb P$ be a forcing notion. We say that $\mathbb P$ has the {\em Sacks property} if given any $p \in \mathbb P$ and any name $\dot{x}$, so that $p \forces \dot{x}:\omega \to V$ is a function then there is a $q\leq p$ and a function $y:\omega \to V$ in $V$ such that for all $n$, $q \forces \dot{x}(\check{n}) \in \check{y}(\check{n})$ and $|y(n) | \leq n$. Slightly less formally this means that every new real (in fact $\omega$ sequence) is caught in an old slalom\footnote{In fact the function bounding $n \mapsto |y(n)|$ can be any function from $V$ tending to $\infty$, see \cite[p. 86]{BlassHB}.}. In particular, all reals added by $\mathbb P$ can be captured by a slalom from the ground model.
\label{sacks}
\end{definition}

\begin{definition}(\cite[Definition 6.3.1]{BarJu95})
Let $\mathbb P$ be a forcing notion. We say that $\mathbb P$ is $\omega^\omega$-{\em bounding} if for each $p \in \mathbb P$ and each $\mathbb P$-name $\dot{x}$, if $p \forces \dot{x} : \check{\omega} \to \check{\omega}$ there is a $y\in \baire \cap V$ and a $q \leq p$ so that $q \forces \dot{x} \leq^* \check{y}$. In other words, every new real is $\leq^*$-dominated by some old real. Note that this implies that the reals of $V$ are dominating in $V^\mathbb P$.
\label{ooboundingdef}
\end{definition}

\begin{definition}(\cite[Definition 6.3.27]{BarJu95} )
Let $\mathbb P$ be a forcing notion. We say that $\mathbb P$ has the {\em Laver Property} if given any $p \in \mathbb P$ and any name $\dot{x}$, so that $p \forces \dot{x}:\omega \to \omega$ is a function which is $\leq^*$-bounded by a ground model real then there is a $q\leq p$ and a function $y:\omega \to V$ in $V$ such that for all $n$, $q \forces \dot{x}(\check{n}) \in \check{y}(\check{n})$ and $|y(n) | \leq n$. In words, this says that every new real which is bounded by an old real is caught in an old slalom \footnote{Again, $n \mapsto n$ can be replaced with any ground model function tending to infinity.}. Note that the Laver property plus $\baire$-bounding is equivalent to the Sacks property.
\label{Laver}
\end{definition}

All three of these properties are preserved by countable support iterations of proper forcing notions. See \cite[Chapter 6]{BarJu95}.

\section{The Cicho\'n Diagram of a Reduction Concept}
Let us think of cardinal characteristics of the continuum in terms of small and large sets relative to some relation giving this notion of smallness and largeness. For example, recall that a family of reals $A$ is ($\leq^*)$ -{\em unbounded} if for all $x \in \omega^\omega$ there is some $y \in A$ such that $y \nleq^* x$. The smallest cardinality of an unbounded family is called the {\em bounding number}, denoted $\mathfrak b = \mathfrak b(\leq^*)$. Dually, a family of reals $A \subseteq \omega^\omega$ is ($\leq^*$) -{\em dominating} if for all $y \in \omega^\omega$ there is a $x \in A$ such that $y \leq^* x$. The least size of a dominating family is called the {\em dominating number}, denoted $\mathfrak{d} = \mathfrak{d}(\leq^*)$. Intuitively one thinks of bounded families as being \say{small} and dominating families as being \say{big}. Thus, heuristically one might think of $\mathfrak{b}$ as the least size of a set that's not \say{small} and $\mathfrak{d}$ as the least size of a set that's \say{big}. To obtain an analogy in the computable world, the authors of \cite{BBNN} define $\mathcal B (\leq^*)$ as the set of oracles computing a function $x$ such that $y \leq^* x$ for each computable function $y$ and $\mathcal D (\leq^*)$ as the set of oracles computing a function $x$ such that $x \nleq^* y$ for all computable $y$. In other words $\mathcal B(\leq^*)$ is the set of oracles which can compute a witness to the fact that the computable functions are \say{small} and $\mathcal D(\leq^*)$ is the set of oracles which can compute a witness to the fact that the computable functions are not \say{big}. Moreover, these sets turn out to correspond to \say{highness} properties of Turing degrees that are well studied in computability theory. Specifically, by a theorem of Martin (cf \cite[pp. 3]{BBNN}), $\mathcal B(\leq^*)$ is the set of high degrees and, by definition, $\mathcal D(\leq^*)$ is the set of hyperimmune degrees. Similar ideas hold for the relations $\neq^*$ and $\in^*$ (as discussed in more detail below).

My key observation is that this formalism has nothing to do with {\em Turing} computability per se. This motivates the following general definition.

\begin{definition}
A {\em reduction concept} is a triple $(X, \sqsubseteq, x_0)$ where $X$ is a nonempty set, $x_0 \in X$ is some distinguished element and $\sqsubseteq$ is a partial pre-order on $X$. We also say that the pair $(\sqsubseteq , x_0)$ is a {\em reduction concept on $X$}. If $(X, \sqsubseteq, x_0)$ is a reduction concept, then for $x, y \in X$ say that $x$ is $\sqsubseteq$-{\em reducible to} $y$ if $x \sqsubseteq y$ and say that $x$ is $\sqsubseteq$-{\em basic} if it is $\sqsubseteq$-reducible to $x_0$. 

Let $(\sqsubseteq, x_0)$ be a reduction concept on $X$ and $R \subseteq X \times X$ be a binary relation. Let $\sqsubseteq \upharpoonright x_0 = \{y\in X \; | \; y \sqsubseteq x_0\}$ be the basic reals. Then define the {\em bounding set} for $R$ as 

\[ \mathcal B_\sqsubseteq (R) = \{x \in X \; | \; \exists y \sqsubseteq x \; \forall z \in \sqsubseteq \upharpoonright x_0 \; [z\mathrel{R}y]\}\]

\noindent and the {\em non-dominating set} for $R$ as 

\[ \mathcal D_\sqsubseteq (R) = \{x \in X \; | \; \exists y \sqsubseteq x \; \forall z \in \sqsubseteq \upharpoonright x_0 \; [\neg y\mathrel{R}z]\}.\] 
\end{definition}
Roughly, if we think of $\sqsubseteq$ is some sort of relative computability relation, then being computable means computable from $x_0$ and $\mathcal B_\sqsubseteq (R)$ is the set of elements of $x \in X$ which compute an $R$-bound on the computable elements of $X$ and $\mathcal D_\sqsubseteq(R)$ is the set of $x \in X$ which compute an element which is not $R$-dominated by any computable element. If $R$ is a relation giving a notion of \say{small} and \say{big} sets as described above one can think of $\mathcal B_\sqsubseteq (R)$ as the set of elements computing a witness to the fact that the $\sqsubseteq$-basic sets are small and $\mathcal D_\sqsubseteq (R)$ as the set of elements computing a witness to the fact that the $\sqsubseteq$-basic elements are not big.

\begin{example}[\cite{BBNN}]
Let $x_0 \in \omega^\omega$ be some computable real, say the constant function at $0$. Then the pair $(\leq_T, x_0)$ forms a reduction concept on the reals. The basic reals are the computable reals. For any binary relation $R$ on the reals $\mathcal B_{\leq_T}(R)$ is the set of Turing degrees computing an element of $X$ which $R$-bounds all the computable sets. Similarly $\mathcal D_{\leq_T} (R)$ is the set of Turing degrees computing an element of $X$ which is not $R$-dominated by any computable set.
\end{example}

The next example will be the central focus of the second half of this chapter.
\begin{example}
Let $x_0 \in \omega^\omega$ be constructible. Then the pair $(\leq_L, x_0)$ is a reduction concept on $\omega^\omega$ where $x \leq_L y$ if $x \in L[y]$. The basic reals are the constructible reals. More generally, fix some inner model $W \subseteq V$ and let $\leq_W$ be constructibility relative to $W$. Then if $x_0 \in (\omega^\omega)^W$ is any given real in $W$ the pair $(\leq_W, x_0)$ forms a reduction concept on Baire space and the basic reals are those of $W$. Since this is the main case let me be explicit about what the bounding and non-dominating sets are. Let $R$ be a relation on the reals of $V$. The set $\mathcal B_{\leq_W} (R)$ consists of all reals $x$ in $V$ such that in $W[x]$ there is an $R$-bound on the reals of $W$. Similarly the set $\mathcal D_{\leq_W} (R)$ consists of all reals $x$ in $V$ such that in $W[x]$ there is a real which is not $R$-bounded by any real in $W$. For example, $\mathcal B_{\leq_W}(\leq^*)$ is the set of dominating reals over $W$ in $V$ and $\mathcal D_{\leq_W}(\leq^*)$ is the set of unbounded reals over $W$ in $V$.
\end{example}
I will come back to this example in the next section. First, let me give some more examples of reduction concepts, though I will not treat them in detail in this thesis.

\begin{example}
Recall that for $x, y \in \mathcal P(\mathbb N)$, the relation $\leq_A$ is defined by $x \leq_A y$ if and only if $x$ is definable in the standard model of arithmetic with an extra predicate for $y$. The pair $(\leq_A, \emptyset)$ forms a reduction concept on $\mathcal P(\mathbb N)$. In this case the basic reals are the sets which are $\emptyset$-definable in the standard model of arithmetic. More generally this could be done with any model of $\PA$.
\end{example}

\begin{example}
Recall that the relation of many-one polytime reduction, $\leq_m^p$ is defined by $x \leq_m^p y$ if and only if there is a function $z$ which is computable in polynomial time such that $n \in x$ if and only if $z(n) \in y$. The pair $(\leq_M^p, \emptyset)$ is a reduction concept on $\mathcal P(\mathbb N)$.
\end{example}

\begin{example}
Let $\kappa > \omega$ be an uncountable cardinal. Recently there has been much work in the descriptive set theory of \say{generalized} Baire and Cantor spaces, $\kappa^\kappa$ and $2^\kappa$, including various generalizations of cardinal characteristics of the continuum, see for instance \cite{brendle16}. The same can be done in my framework for degrees of constructibility. For instance notions of eventual domination, etc all make sense in the general context of $\kappa^\kappa$ and corresponding bounding and non-dominating sets can be constructed over the basic elements, $(\kappa^\kappa)^L$.
\end{example}

The framework described above is flexible enough that $(X, \sqsubseteq, x_0)$ need not be some actual notion of computability on the reals nor have an explicit relation to cardinal characteristics of the continuum. For instance one might consider a class of models of a fixed theory in a fixed language with embeddability. In this case, depending on the relations $R$ one studied, one would arrive at a diagram corresponding to when models with certain properties embed into one another. There are many possibilities, each giving a potentially interesting \say{Cicho\'n diagram} of inclusions between the various bounding and non-dominating sets for an appropriate collection of relations. In future work I hope to explore all of these more fully. 

Presently however, let me restrict my attention to the types of cases described in the preceding examples. Even in this general framework I can now prove a collection of implications giving a version of the Cicho\'n diagram. %To see how these examples can lead to \say{Cicho\'n diagrams} let me define some relations.

%\begin{definition}[Combinatorial relations]
%I consider the reals as elements of Baire space, $\omega^\omega$. Let $f, g$ be reals. Then
%\begin{enumerate}
%\item
 %$f \neq^* g$ if there is some $k$ such that for all $l > k$ $f(l) \neq g(l)$. In this case say that $g$ is {\em eventually not equal to} $f$. Note that the negation of $\neq^*$ is {\em infinitely often equal}, not eventual equality. 
 %\item
 %Let $h \in \omega^\omega$ and recall that an $h$-{\em slalom} is a function $s: \omega \to [\omega]^{< \omega}$ such that for all $n \in \omega$ the set $|s(n)| \leq h(n)$. In the case where $h$ is the identity function call $s$ simply a {\em slalom}. For a slalom $s$, I write $f \in^* s$ if there is some $k$ such that for all $l > k$ $f(l) \in s(l)$. In this case say that $f$ {\em is eventually captured by} $s$. 
 %\end{enumerate}
%\end{definition}

\begin{theorem}
Let $(\sqsubseteq, x_0)$ be a reduction concept on $\omega^\omega$ extending $\leq_T$ such that if $x, y \sqsubseteq z$ then $x\circ y \sqsubseteq z$. Interpreting arrows as inclusions, the implications in Figure \ref{arbcichon} all hold.

\begin{figure}[h]\label{Figure.Cichon-basic}
\centering
  \begin{tikzpicture}[scale=1.5,xscale=2]
     % place and draw the nodes
     \draw (0,0) node (empty) {$\emptyset$}
           (1,0) node (Bin*) {$\mathcal B_\sqsubseteq(\in^*)$}
           (1,1) node (Bleq*) {$\mathcal B_\sqsubseteq(\leq^*)$}
           (1,2) node (Bneq*) {$\mathcal B_\sqsubseteq(\neq^*)$}
           (2,0) node (Dneq*) {$\mathcal D_\sqsubseteq(\neq^*)$}
           (2,1) node (Dleq*) {$\mathcal D_\sqsubseteq(\leq^*)$}
           (2,2) node (Din*) {$\mathcal D_\sqsubseteq(\in^*)$}
           (3,2) node (all) {$\omega^\omega\setminus \{x \; | \; x \sqsubseteq x_0\}$}
           ;
     % draw the arrows
     \draw[->,>=stealth]
            (empty) edge (Bin*)
            (Bin*) edge (Bleq*)
            (Bleq*) edge (Bneq*)
            (Bleq*) edge (Dleq*)
            (Dneq*) edge (Dleq*)
            (Bin*) edge (Dneq*)
            (Bneq*) edge (Din*)
            (Dleq*) edge (Din*)
            (Din*) edge (all)
            
            ;
  \end{tikzpicture}
  \caption{A Cicho\'n diagram for an arbitrary reduction concept on Baire space}
\label{arbcichon}
\end{figure}
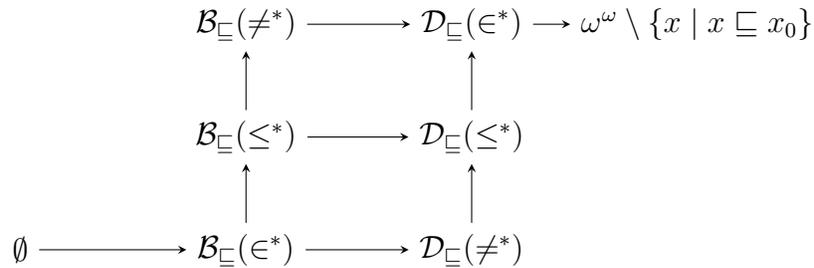
\label{propdi}
\end{theorem}

\begin{proof}
Note that slaloms can be computably coded by reals so, since the relation $\sqsubseteq$ extends Turing computability the $\in^*$ can be seen as a relation on the reals. I drop the $\sqsubseteq$ subscript for readability. Also, I'll write \say{basic} for $\sqsubseteq$-basic and if $y \sqsubseteq x$ then I'll say that \say{$x$ builds $y$}. The requirement that $\sqsubseteq$ be closed downwards under compositions will be used implicitly throughout the argument where I will show that a function can build two other functions hence it can build their composition.

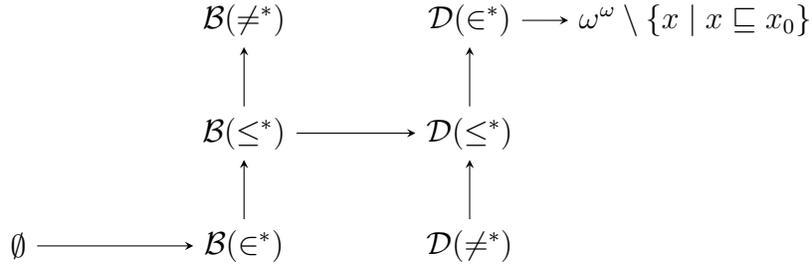
\begin{figure}[h]\label{Figure.Cichon-basic}
\centering
  \begin{tikzpicture}[scale=1.5,xscale=2]
     % place and draw the nodes
     \draw (0,0) node (empty) {$\emptyset$}
           (1,0) node (Bin*) {$\mathcal B(\in^*)$}
           (1,1) node (Bleq*) {$\mathcal B(\leq^*)$}
           (1,2) node (Bneq*) {$\mathcal B(\neq^*)$}
           (2,0) node (Dneq*) {$\mathcal D(\neq^*)$}
           (2,1) node (Dleq*) {$\mathcal D(\leq^*)$}
           (2,2) node (Din*) {$\mathcal D(\in^*)$}
           (3,2) node (all) {$\omega^\omega \setminus \{x \; | \; x \sqsubseteq x_0\}$}
           ;
     % draw the arrows
     \draw[->,>=stealth]
            (empty) edge (Bin*)
            (Bin*) edge (Bleq*)
            (Bleq*) edge (Bneq*)
            (Bleq*) edge (Dleq*)
            (Dneq*) edge (Dleq*)
            (Dleq*) edge (Din*)
            (Din*) edge (all)
            ;
  \end{tikzpicture}
  \caption{The Easy Cases}
\end{figure}

%%begin here

Let's begin with the easy cases, as shown in Figure 3. First I'll show that $\mathcal B(\in^*) \subseteq \mathcal B (\leq^*)$ or that every $x$ building a slalom eventually capturing all the basic reals builds a real which eventually dominates all basic reals. This is proved as follows. Suppose $x \in \mathcal B(\in^*)$ and let $s \sqsubseteq x$ be a slalom witnessing this. Then, define $z(n) = {\rm max} \, (s (n)) +1$. Notice that $z \leq_T s$ so $z \sqsubseteq s$ and hence $z \sqsubseteq x$. Moreover, since $s$ eventually captures all basic reals, $z$ must eventually dominate them all so $x \in \mathcal B (\leq^*)$.

Now let's show that $\mathcal B(\leq^*) \subseteq \mathcal B(\neq^*)$ or that if $x$ can build a real eventually dominating all basic reals then it can build a real eventually different real from all basic reals. Notice however that once stated like this the proof simply the observation that if $x$ eventually dominates $y$ and $y+1$ (which is basic if $y$ is since $\leq_T$ is extended by $\sqsubseteq$), then, in particular $x$ is eventually different from $y$. 

Next let's show $\mathcal B(\leq^*) \subseteq \mathcal D(\leq^*)$ or that if $x$ builds a real which eventually dominates all basic reals then it builds a real which is not dominated by any basic real. But now stated like this it's obvious.

Next I show that $\mathcal D(\neq^*) \subseteq \mathcal D(\leq^*)$ or that if there is a real which is equal to every basic real infinitely often then there is a real which is never dominated by any basic real. This is obvious though since if $x$ were dominated by some basic real $y$ then it could not be infinitely-often-equal to the basic real $y+1$.

Now I show that $\mathcal D(\leq^*) \subseteq \mathcal D(\in^*)$ or that if $x$ builds a real which is not dominated by any basic real then it builds a real that is never eventually captured by any basic slalom. Suppose $x \in \mathcal D (\leq^*)$ and let $y \sqsubseteq x$ witness this. Then, if $s$ is a basic slalom, let $z$ be defined by $z(n)$ is one plus the sum of the elements in $s (n)$. Note that $z \leq_T s$ so $z$ is basic. Thus there are infinitely many $n$ such that $y(n) \geq z(n)$ so $y$ cannot be eventually captured by $s$.

The last easy inclusion, that all the reals in every node are not themselves basic is completely straightforward. For instance, if $x \in \mathcal D(\in^*)$ is a real which is not eventually captured by any basic slalom then of course $x$ is not basic since if it were, the slalom $n \mapsto \{x(n)\}$ would be as well.

Now I move on to the more difficult inclusions, starting with $\mathcal B(\in^*) \subseteq \mathcal D(\neq^*)$. Substantively this states that if a real $x$ builds a slalom eventually capturing all basic functions then $x$ also builds a real which is infinitely-often-equal to all basic functions. In fact I will show a more general claim that implies this. The following lemma and proof is essentially a reinterpretation of Theorem 1.5 from \cite{Bar1987}.

\begin{lemma}
For any real $x$ the following are equivalent. 
\begin{enumerate}
\item
There is a real $y \sqsubseteq x$ such that for all basic $z \in \omega^\omega$, there exist infinitely many $n \in \omega$ such that $y(n) = z(n)$
\item
There is a basic $z \in \omega^\omega$ and a $z$-slalom $s \sqsubseteq x$ such that for all basic $y \in \omega^\omega$ there are infinitely many $n \in \omega$ such that $y(n) \in s(n)$.
\end{enumerate}
Moreover, given an infinitely-often-equal real as in $1)$, one can build from it a $z$-slalom as in $2)$ and given a $z$-slalom $s$ as in $2)$ one can build an infinitely-often-equal real as in $1)$. Thus, $x \in \mathcal D(\neq^*)$ if and only if there is a basic $z \in \omega^\omega$ and a $z$-slalom which captures each of the basic reals infinitely often.
\label{infslalom}
\end{lemma}
Before proving Lemma \ref{infslalom}, notice that it implies the inclusion $\mathcal B(\in^*) \subseteq \mathcal D(\neq^*)$ since any slalom which captures every basic real cofinitely often must in particular capture each basic real infinitely often so if $x \in \mathcal B(\in^*)$ builds such a slalom, by the lemma $x$ must be able to build an infinitely-often-equal real as well.

\begin{proof}[Proof of Lemma \ref{infslalom}]
The forward direction is obvious: suppose that $y$ is an infinitely-often-equal real. Then clearly the $1$-slalom $s:\omega \to [\omega]^1$ such that $s(n) = \{y(n)\}$ is $\leq_T$-computable from $y$ and hence $\sqsubseteq$-reducible to $y$, thus giving the desired $z$-slalom.

For the backward direction fix a basic real $z$ such that there exists a $z$-slalom as in the statement of 2. I need to find a real $y$ which is infinitely often equal to every basic real. In a basic fashion, fix a family of finite, nonempty, pairwise disjoint subsets of $\omega$ enumerated $\{J_{n, k} \; | \; n < \omega \; \& \; k \leq z(n)\}$ which collectively cover $\omega$. Since $z$ is assumed to be basic there is no problem building such a partition, for example one could use singletons. Label $J_n = \bigcup_{k \leq z(n)} J_{n, k}$. Then for each basic $v \in \omega^\omega$ let $v':\omega \to \omega^{< \omega}$ be the function defined by $v'(n) = f\upharpoonright J_n$. More generally let $\mathcal J = \{v: \omega \to \omega^{<\omega} \; | \; {\rm dom} ( v(n) ) = J_n\}$. Notice that the basic elements of $\mathcal J$ are exactly $\{v' \;|\; v \in \omega^\omega \; \& \; v \sqsubseteq 0\}$ since from any $v'$ we can build $v$ and vice versa (by the the fact that the $J_n$'s are basic). But now since the $v'$'s are basic and each one codes a real one can by applying 2 plus some simple coding to find a $z$-slalom, $s: \omega \to (\omega^{< \omega})^{< \omega}$ such that for every $n \in \omega$ $|s(n)| \leq z(n)$ and $s(n)$ is a set of finite partial functions from $J_n$ to $\omega$ and for every basic $v' \in \mathcal J$ there are infinitely many $n\in \omega$ such that $v'(n) \in s(n)$. 

Let me denote $s(n) = \{w^n_1,...,w^n_{z(n)}\}$. Now set $y_n = \bigcup_{k \leq z(n)} w^n_k \upharpoonright J_{n, k}$ and let $y = \bigcup_{n < \omega} y_n$. Notice that this gives an element of $\omega^\omega$ since the $J_{n, k}$'s were disjoint and collectively covered $\omega$. I claim that $y$ is as needed. Clearly $y$ is reducible to the $J_{n, k}$'s, which are basic, and the $w^n_k$'s, which are reducible to $s$ so $y$ is reducible to $s$. It remains to see that it is an infinitely-often-equal real. To see this, let $v \in \omega^\omega$ be basic and fix some $n$ such that $v'(n) \in s(n)$ (recall that there are infinitely many such $n$). Notice that since $v'(n) \in s(n)$ there must be some $k \leq z(n)$ such that $v\upharpoonright J_m = w^n_k$. Now let $x_n \in J_{n, k}$ (recall that this set is assumed to be non-empty). We have that $v(x_n) = w^n_k(x_n) = g(x_n)$. But there are infinitely many such $n$ and hence infinitely many such $x_n$ so this completes the proof.
\end{proof}

A similar proof produces the last inclusion, $\mathcal B(\neq^*) \subseteq \mathcal D(\in^*)$. In words this inclusion states that any real which can build a real which is eventually different from all basic reals can build a real which is not eventually captured by any given slalom. I will prove the following more general lemma, whose statement and proof is inspired by \cite{Bar1987}, Theorem 2.2. Given a $z$-slalom $s$ and a function $x$ let me say that $x$ is {\em eventually never captured by} $s$ if there is some $k$ such that for all $l > k$ we have $x(l) \notin s (l)$.
\begin{lemma}
For any real $x$, the following are equivalent.
\begin{enumerate}
\item
The real $x$ is eventually different from all basic reals.
\item
The real $x$ is such that for all basic reals $z$ and all basic $z$-slaloms $s$ for all but finitely many $n \in \omega$ $x(n) \notin s (n)$.
\end{enumerate}
Therefore $x \in \mathcal B(\neq^*)$ if and only if $x$ builds a real which is eventually never captured by any basic $z$-slalom for any basic $z$.
\label{neqinf}
\end{lemma}
Let me note before I prove Lemma \ref{neqinf} that it proves the inclusion $\mathcal B(\neq^*) \subseteq \mathcal D(\in^*)$ and hence Theorem \ref{propdi}. To see why, suppose that $x \in \mathcal B(\neq^*)$ and, without loss of generality suppose that $x$ itself is a real which is eventually different from all basic reals. Then by the lemma $x$ is eventually never captured by any basic slalom so, in particular for infinitely many $n$ $x(n) \notin s (n)$ for all basic $s$, which means $x \in \mathcal D(\in^*)$.

\begin{proof}[Proof of Lemma \ref{neqinf}]
Fix some $x \in \omega^\omega$. The backward direction of this lemma is easy: if $x$ is eventually never captured by any basic $z$-slalom for any basic $z$ then in particular it is eventually never captured by the slalom sending $n \mapsto \{y(n)\}$ for each basic $y$ and hence it is eventually different from each basic $y$.

For the forward direction, assume $x$ is eventually different from all basic functions. Fix a basic $z$ and, like in the proof of Lemma \ref{infslalom}, in a basic fashion partition $\omega$ into finite, disjoint, non-empty sets $\{J_{n, k} \; | \; k \leq z(n)\}$. Let $J_n = \bigcup_{k \leq z(n)} J_{n, k}$. Let $x':\omega \to \omega^{< \omega}$ be the function defined by $x'(n) = x\upharpoonright J_n$. Then if $s$ is any basic $z$-slalom, let $s '$ be such that on input $n$ gives $z(n)$ many finite partial functions $w^n_1,...,w^n_{z(n)}$ with domain $J_n$ where for all $k \leq z(n)$ and $l \in J_n$ $w^n_k(l)$ is the $k^{\rm th}$ greatest number in the set $s (l)$ if such exists and $0$ (say) otherwise. Suppose now towards a contradiction that there is a basic $z$-slalom $s$ such that  $x(n) \in s (n)$ for infinitely many $n$. For each $n$ let $s'(n) = \{w^n_1,...,w^n_{h(n)}\}$. Then define $y_n = \bigcup_{k \leq z(n)} w^n_k \upharpoonright J_{n, k}$ and let $y = \bigcup_{n < \omega} y_n$. Clearly $y$ can built using $s$, the function $z$ and the $J_{n, k}$'s each of which is basic so $y$ is basic. Thus there is a $k$ such that for all $n > k$ we have that $x(n) \neq y(n)$. But, since there are infinitely many $n$ such that $x(n) \in s (n)$, there are infinitely many $n > k$ such that $x(n) \in s (n)$ and therefore it follows that similarly we must have that there are infinitely many $n > k$ such that $x' (n)$ agrees with some $w^n_j$ on some element of their shared domain for some $j \leq z(n)$. But this means $x(k) = y(k)$ for some $k \in J_{n, j}$ for infinitely many $n$'s and $j$'s which is a contradiction.
\end{proof}

\noindent Since this was the final inclusion to prove, Theorem \ref{propdi} is now proved as well.
\end{proof}

Thus, even in this broad context one can construct diagrams for a wide variety of reduction concepts and a correspondence starts to form with the Cicho\'n diagram. This extends the proof given in the case of Turing degrees in \cite{BBNN} and gives a good framework for investigations into various computability reduction concepts. What it does not show, however, is that any of these nodes are non-empty or that the inclusions are strict. Indeed this is not necessarily the case. For instance $\mathcal B_{\leq_T}(\in^*) = \mathcal B_{\leq_T}(\leq^*)$ (see \cite{BBNN}). This is because, by a theorem of Rupprecht, the set $\mathcal B_{\leq_T} (\in^*)$ is simply the high reals, which as I mentioned above is also $\mathcal B (\leq^*)$. The analogue of this fact in the case of the classical Cicho\'n diagram is false since ${\rm add}(\mathcal N)$, the analogue of $\mathcal B_{\leq_T} (\in^*)$, can consistently be less than $\mathfrak{b}$, the analogue of $\mathcal B(\leq^*)$. The authors of \cite{BBNN} take this as evidence that the $\leq_T$-Cicho\'n diagram provides \say{only an analogy, not a full duality} \cite[p. 3]{BBNN} with the classical Cicho\'n diagram. Theorem \ref{propdi} proves the existence of a wide variety of such diagrams, therefore raising the question in each case of how strong the analogy between the reduction diagram and the classical diagram is, and whether we ever get a full duality. This depends on the strength of the reduction since, while the $\leq_T$ diagram gives only an analogy, I will show in the next section that in the $\leq_W$ diagram the inclusions proved in Theorem \ref{propdi} constitute the only ones true in every model of $\ZFC$, thereby suggesting something closer to a true duality.

\section{The Cicho\'n Diagram for $\leq_W$}

From now on fix an inner model $W \models \ZFC$. I work in the language of set theory with an extra predicate for $W$ and the theory ZFC($W$), that is ZFC with replacement and comprehension holding for formulas containing $W$. I view $W=L$ as a central case but it turns out that the analysis works out the same for arbitrary $W$.

Before presenting the full $\leq_W$-Cicho\'n diagram, let me state clearly what the bounding and dominating sets are for the combinatorial relations defined in the last section for $\leq_W$.
\begin{enumerate}
\item
	$\mathcal B(\in^*)$ is the set of reals $x$ such that there is a slalom $s \in W[x]$ that eventually captures all reals in $W$.
\item
	 $\mathcal B(\leq^*)$ is the set of reals $x$ such that there is a real $y \in W[x]$ that eventually dominates all reals in $W$. These are sometimes called {\em dominating reals} (for $W$).
\item
	$\mathcal B(\neq^*)$ is the set of reals $x$ such that there is a real $y \in W[x]$ that is eventually different from all reals in $W$. These are sometimes called {\em eventually different reals} (for $W$).
\item
	$\mathcal D(\in^*)$ is the set of reals $x$ such that there is a real $y \in W[x]$ that is not eventually captured by any slalom in $W$.
\item
	$\mathcal D(\leq^*)$ is the set of reals $x$ such that there is a real $y \in W[x]$ that is not eventually dominated by any real in $W$. These are sometimes called {\em unbounded reals} (for $W$).
\item
	$\mathcal D(\neq^*)$ is the set of reals $x$ such that there is a real $y \in W[x]$ that is equal infinitely often to every real in $W$. These are sometimes called {\em infinitely-often-equal reals} (for $W$).
	\end{enumerate}

In this section I will study how a variety of known forcing notions over $W$ can create separations in the $\leq_W$-Cicho\'n diagram as described in the previous section. Of course ZFC($W$) cannot prove any separations since if $V=W$ or, more generally $V$ and $W$ have the same reals, every node in the $\leq_W$-diagram will be empty. However, using simple forcing notions I will show that one can produce a wide variety of possible constellations for the $\leq_W$-diagram. The main theorem of this section is the following.

\begin{theorem}
The Cicho\'n diagram for $\leq_W$ as described in the previous section is complete for ZFC($W$)-provable implications. In other words if $A$ and $B$ are two nodes in the diagram and $A \subseteq B$ does not follow from the transitive closure of the arrows in the $\leq_W$-diagram then there is a forcing extension of $W$ where $A \nsubseteq B$.
\label{cichoncomplete}
\end{theorem}
That these implications all hold follows from the main theorem of the previous section since $\leq_W$ extends $\leq_T$.
\begin{figure}[h]\label{Figure.Cichon-basic}
\centering
  \begin{tikzpicture}[scale=1.5,xscale=2]
     % place and draw the nodes
     \draw (0,0) node (empty) {$\emptyset$}
           (1,0) node (Bin*) {$\mathcal B_{\leq_W}(\in^*)$}
           (1,1) node (Bleq*) {$\mathcal B_{\leq_W}(\leq^*)$}
           (1,2) node (Bneq*) {$\mathcal B_{\leq_W}(\neq^*)$}
           (2,0) node (Dneq*) {$\mathcal D_{\leq_W}(\neq^*)$}
           (2,1) node (Dleq*) {$\mathcal D_{\leq_W}(\leq^*)$}
           (2,2) node (Din*) {$\mathcal D_{\leq_W}(\in^*)$}
           (3,2) node (all) {$\omega^\omega\setminus(\omega^\omega)^W$}
           ;
     % draw the arrows
     \draw[->,>=stealth]
            (empty) edge (Bin*)
            (Bin*) edge (Bleq*)
            (Bleq*) edge (Bneq*)
            (Bin*) edge (Dneq*)
            (Bleq*) edge (Dleq*)
            (Bneq*) edge (Din*)
            (Dneq*) edge (Dleq*)
            (Dleq*) edge (Din*)
            (Din*) edge (all)
            ;
  \end{tikzpicture}
\end{figure}

Let me note one word on the relation between my diagram and the standard Cicho\'n diagram as commonly studied, for example in \cite{BarJu95}. Here I have focused on the so-called combinatorial nodes as discussed by \cite{BBNN}. As noted  in the previous section, I view my diagram in correspondence with the classical one via the mapping sending unbounded or dominating families with respect to a certain relation to the sets of reals $x$ such that in $W[x]$ the reals of $W$ are not unbounded or dominating. I have included this fragment of the Cicho\'n diagram to make this analogy clear visually. %The connection between the combinatorial view and the measure theoretic/category theoretic view is described by Facts \ref{infact} and \ref{neqfact}.

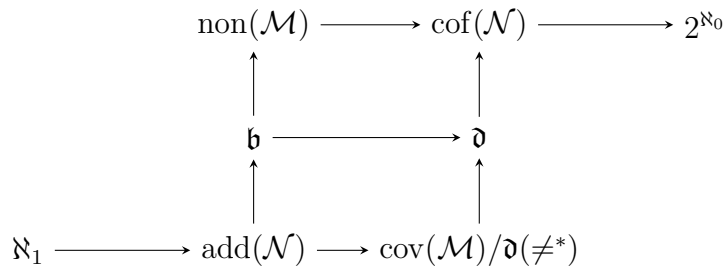
\begin{figure}[h]\label{Figure.Cichon-basic}
\centering
  \begin{tikzpicture}[scale=1.5,xscale=2]
     % place and draw the nodes
     \draw (0,0) node (empty) {$\aleph_1$}
           (1,0) node (Bin*) {${\rm add}(\mathcal N)$}
           (1,1) node (Bleq*) {$\mathfrak b$}
           (1,2) node (Bneq*) {${\rm non}(\mathcal M)$}
           (2,0) node (Dneq*) {${\rm cov}(\mathcal M)/\mathfrak{d}(\neq^*)$}
           (2,1) node (Dleq*) {$\mathfrak{d}$}
           (2,2) node (Din*) {${\rm cof}(\mathcal N)$}
           (3,2) node (all) {$2^{\aleph_0}$}
           ;
     % draw the arrows
     \draw[->,>=stealth]
            (empty) edge (Bin*)
            (Bin*) edge (Bleq*)
            (Bleq*) edge (Bneq*)
            (Bleq*) edge (Dleq*)
            (Dneq*) edge (Dleq*)
            (Bin*) edge (Dneq*)
            (Bneq*) edge (Din*)
            (Dleq*) edge (Din*)
            (Din*) edge (all)
            
            ;
  \end{tikzpicture}
  \caption{The Combinatorial Nodes of the Standard Cicho\'n Diagram}
\end{figure}
The details of these correspondences for $\leq_T$ can be found in \cite{BBNN} and similar ideas hold in the present case, with one exception: ${\rm cov}(\mathcal M) / \mathfrak{d}(\neq^*)$. As noted in Fact \ref{neqfact} these cardinals are the same, however Zapletal has shown in \cite{dimtheoryandforcing} that their degree theoretic analogues are in fact different, thus solving a well known problem of Fremlin. I will mention Zapletal's theorem again at the end of this chapter in connection with extensions of the current work. %In a planned sequel \cite{Switz18b} I will discuss this topic more as well as treat the mising nodes, namely those corresponding to invariants of measure and category. Note however, that the analogy holds between the combinatorial characterizations of the cardinal invariants, not their usual definitions. For example, ${\rm add} (\mathcal N)$ is known to be equal to $\mathfrak{b} (\in^*)$ and it is this cardinal that corresponds to $\mathcal B(\in^*)$. 

%\section{Effects of Adding One Real to the Cicho\'n Diagram for $\leq_W$}

%Finally, before studying the affects of different forcing notions, let me recall the definitions of three properties of forcing notions that will be key.

\subsection{Sacks Forcing}
The first forcing I will look at is Sacks forcing, $\mathbb S$. Recall that conditions in $\mathbb S$ are perfect trees $T \subseteq 2^{< \omega}$ ordered by inclusion. If $G$ is $\mathbb S$-generic then the unique branch in the intersection of all members of $G$ is called a {\em Sacks real}. I denote such a real $s$.

\begin{theorem}
In the Sacks extension all nodes of $\leq_W$-Cicho\'n diagram other than $\omega^\omega \setminus (\omega^\omega)^W$ are empty.
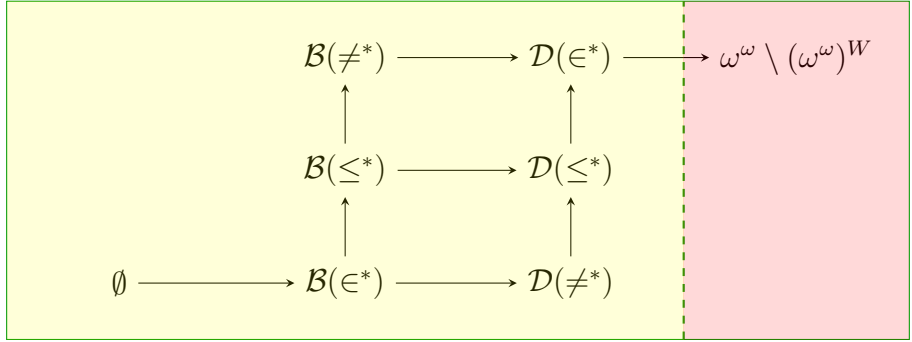
\begin{figure}[h]\label{Figure.Cichon-basic}
\centering
  \begin{tikzpicture}[scale=1.5,xscale=2]
     % place and draw the nodes
     \draw (0,0) node (empty) {$\emptyset$}
           (1,0) node (Bin*) {$\mathcal B(\in^*)$}
           (1,1) node (Bleq*) {$\mathcal B(\leq^*)$}
           (1,2) node (Bneq*) {$\mathcal B(\neq^*)$}
           (2,0) node (Dneq*) {$\mathcal D(\neq^*)$}
           (2,1) node (Dleq*) {$\mathcal D(\leq^*)$}
           (2,2) node (Din*) {$\mathcal D(\in^*)$}
           (3,2) node (all) {$\omega^\omega\setminus(\omega^\omega)^W$}
           ;
     % draw the arrows
     \draw[->,>=stealth]
            (empty) edge (Bin*)
            (Bin*) edge (Bleq*)
            (Bleq*) edge (Bneq*)
            (Bin*) edge (Dneq*)
            (Bleq*) edge (Dleq*)
            (Bneq*) edge (Din*)
            (Dneq*) edge (Dleq*)
            (Dleq*) edge (Din*)
            (Din*) edge (all)
            ;
    % draw separating lines
       \draw[thick,dashed,OliveGreen] (2.5,-.5) -- (2.5,2.5);
    % draw bounding rectangle
       \draw[OliveGreen] (-.5,-.5) rectangle (3.5,2.5);
    % draw fill
       \draw[draw=none,fill=yellow,fill opacity=.15] (-.5,-.5) rectangle (2.5,2.5);
       \draw[draw=none,fill=red,fill opacity=.15] (2.5,-.5) rectangle (3.5,2.5);
  \end{tikzpicture}
  \caption{After Sacks forcing}
\label{aftersacks}
\end{figure}
\end{theorem}

\begin{proof}
Recall that Sacks forcing has the Sacks property, Definition \ref{sacks}, see Lemma 7.3.2 of \cite{BarJu95}. As a result, all reals added by $\mathbb S$ and hence all reals in $V$ that are not in $W$ can be captured by a slalom from the ground model i.e. $W$. Thus, $W[s]$ thinks that $\mathcal D (\in^*)$ is empty but $s \in \omega^\omega \setminus (\omega^\omega)^W$ and hence the only non-empty set in the $\leq_W$-Cicho\'n diagram is the latter.

%For the second part of the theorem, let $\mathbb Q$ be the countably closed forcing to collapse the continuum to $\aleph_1$. Since this forcing is countably closed it adds no new reals so in $V^\mathbb Q$ the $\leq_W$-diagram remains unchanged and $CH$ is true. Thus $\mathbb Q$ is the $\mathbb P_1$ required. For $\mathbb P_2$, recall that $\mathbb S$ is proper and assuming $CH$ in the ground model, forming an $\aleph_2$ length iteration of $\mathbb S$ with countable support will produce a model of $2^{\aleph_0} = \aleph_2$ in which the Sacks property still holds. Thus if $\mathbb P$ is such an iteration, then forcing with $\mathbb P_2 = \mathbb Q * \dot{\mathbb P}$ will give the required model.
\end{proof}
%Notice that this proof actually gives a more general result: instead of forcing over $W$ one can force over $V$. In this case the same proof shows that if $W \subseteq V$ and $s$ is $V$-generic for $\mathbb S^W$, then all the nodes in the $\leq_W$-diagram in $V[s]$ agree with $V$ except the top right. Similar phenomena will hold in the proceeding cases.

\subsection{Cohen Forcing} %%start here
Let $\mathbb C = Add(\omega , 1)$ be the forcing to add one Cohen real. The main theorem of this section is:
\begin{theorem}
Let $c$ be a Cohen real generic over $W$. Then in $W[c]$ the following hold:
\begin{enumerate}
\item   
   $\emptyset = \mathcal B(\in^*) = \mathcal B(\leq^*) = \mathcal B(\neq^*)$ 
\item   
   $\mathcal D(\neq^*) = \mathcal D(\leq^*) = \mathcal D(\in^*) = \{x \; | \; \exists c \in W[x] \; {\rm Cohen \; over \; } W\} = \omega^\omega \setminus (\omega^\omega)^W$
   \end{enumerate}

\end{theorem}

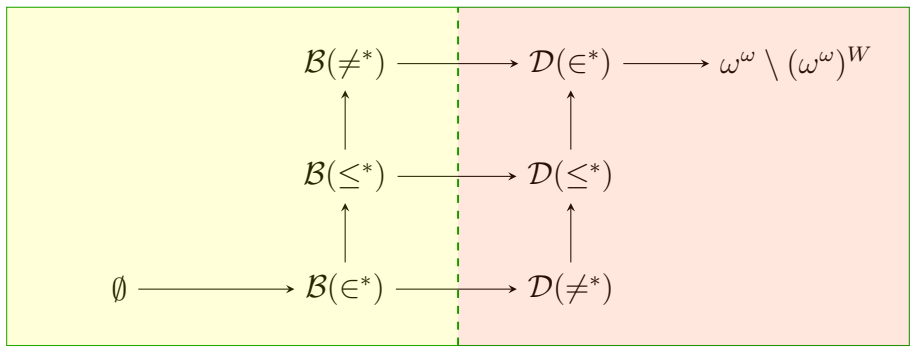
\begin{figure}[h]\label{Figure.Cichon-basic}
\centering
  \begin{tikzpicture}[scale=1.5,xscale=2]
     % place and draw the nodes
     \draw (0,0) node (empty) {$\emptyset$}
           (1,0) node (Bin*) {$\mathcal B(\in^*)$}
           (1,1) node (Bleq*) {$\mathcal B(\leq^*)$}
           (1,2) node (Bneq*) {$\mathcal B(\neq^*)$}
           (2,0) node (Dneq*) {$\mathcal D(\neq^*)$}
           (2,1) node (Dleq*) {$\mathcal D(\leq^*)$}
           (2,2) node (Din*) {$\mathcal D(\in^*)$}
           (3,2) node (all) {$\omega^\omega\setminus(\omega^\omega)^W$}
           ;
     % draw the arrows
     \draw[->,>=stealth]
            (empty) edge (Bin*)
            (Bin*) edge (Bleq*)
            (Bleq*) edge (Bneq*)
            (Bin*) edge (Dneq*)
            (Bleq*) edge (Dleq*)
            (Bneq*) edge (Din*)
            (Dneq*) edge (Dleq*)
            (Dleq*) edge (Din*)
            (Din*) edge (all)
            ;
    % draw separating lines
       \draw[thick,dashed,OliveGreen] (1.5,-.5) -- (1.5,2.5);
    % draw bounding rectangle
       \draw[OliveGreen] (-.5,-.5) rectangle (3.5,2.5);
    % draw fill
       \draw[draw=none,fill=yellow,fill opacity=.15] (-.5,-.5) rectangle (1.5,2.5);
       \draw[draw=none,fill=Orange,fill opacity=.15] (1.5,-.5) rectangle (3.5,2.5);
  \end{tikzpicture}
  \caption{After Cohen forcing}
\end{figure}

\begin{proof}
There are two parts to this proof. First I need to show that all of the elements on the left are empty. Since $\mathcal B(\in^*) \subseteq \mathcal B(\leq^*) \subseteq \mathcal B(\neq^*)$ it suffices to show that Cohen forcing adds no reals which are eventually different from all ground model reals. This is a standard argument but I repeat it here for completeness, see also \cite[p. 83]{BlassHB}. Let $\{p_i\}_{i \in \omega}$ enumerate the conditions of $\mathbb C$ and suppose that $\Vdash_\mathbb C \dot{x}: \omega \to \omega$. Then, for each $i$, pick a $q_i \leq p_i$ which decides the value of $\dot{x}(i)$ in other words let $q_i \Vdash \dot{x}(\check{i}) = \check{j_i}$ for some $j_i$. Now, in the ground model, set $y(i) = j_i$. Finally, suppose for contradiction that there was a $k \in \omega$ and a $p \in \mathbb C$ such that $p \Vdash \forall l > \check{k} \; \check{y}(l) \neq \dot{x}(l)$. But then one can find an $i> k$ and a $q_i \leq p$ such that $q_i \Vdash \check{y(i)} = \dot{x}(i)$, which is a contradiction.

So Cohen forcing leaves the left side of the diagram trivialized. The right side however changes since it's dense for $c$ to equal every real in $W$ infinitely often so $c \in \mathcal D(\neq^*)$. The second part of the proof is to show that every real added by Cohen forcing adds an element to $\mathcal D(\neq^*)$. Since $\mathcal D(\neq^*) \subseteq \mathcal D(\leq^*) \subseteq \mathcal D(\in^*) \subseteq \omega^\omega \setminus (\omega^\omega)^W$ it suffices to show that, $\omega^\omega \setminus (\omega^\omega)^W \subseteq \mathcal D(\neq^*)$. Let $x \in W[c] \setminus W$ be a new real and consider now the model $W[x]$. By the intermediate model theorem it must be the case that $W[x]$ is a generic extension of $W$ and that $W[c]$ is a generic extension of $W[x]$ so the forcing to add $x$ is a non trivial factor of Cohen forcing so it must in fact be isomorphic to it by Theorem 3.3.1 of \cite{BarJu95}. Thus in $W[x]$ there is a real $d$ which is Cohen generic over $W$, and $d$ is infinitely often equal to every real in $W$ so $x \in \mathcal D(\neq^*)$.
\end{proof}

\subsection{Random Real Forcing}
I denote random real forcing by $\mathbb B$. The diagram for random real forcing is as described in the theorem below and can be proved in a very similar way to that of Cohen forcing using the standard facts found in  \cite[Chapter 3]{BarJu95}. 
\begin{theorem}
Let $r$ be a random real over $W$. Then in $W[r]$ the $\leq_W$-Cicho\'n diagram is determined by the separations $\mathcal B(\in^*) = \mathcal B(\leq^*) = \mathcal D(\neq^*) = \mathcal D(\leq^*) = \emptyset$ and $\mathcal B (\neq^*) = \mathcal D(\in^*) = \omega^\omega \setminus (\omega^\omega)^W$.

\begin{figure}[h]\label{Figure.Cichon-basic}
\centering
  \begin{tikzpicture}[scale=1.5,xscale=2]
     % place and draw the nodes
     \draw (0,0) node (empty) {$\emptyset$}
           (1,0) node (Bin*) {$\mathcal B(\in^*)$}
           (1,1) node (Bleq*) {$\mathcal B(\leq^*)$}
           (1,2) node (Bneq*) {$\mathcal B(\neq^*)$}
           (2,0) node (Dneq*) {$\mathcal D(\neq^*)$}
           (2,1) node (Dleq*) {$\mathcal D(\leq^*)$}
           (2,2) node (Din*) {$\mathcal D(\in^*)$}
           (3,2) node (all) {$\omega^\omega\setminus(\omega^\omega)^W$}
           ;
     % draw the arrows
     \draw[->,>=stealth]
            (empty) edge (Bin*)
            (Bin*) edge (Bleq*)
            (Bleq*) edge (Bneq*)
            (Bin*) edge (Dneq*)
            (Bleq*) edge (Dleq*)
            (Bneq*) edge (Din*)
            (Dneq*) edge (Dleq*)
            (Dleq*) edge (Din*)
            (Din*) edge (all)
            ;
    % draw separating lines
       \draw[thick,dashed,OliveGreen] (-.5,1.5) -- (3.5,1.5);
    % draw bounding rectangle
       \draw[OliveGreen] (-.5,-.5) rectangle (3.5,2.5);
    % draw fill
       \draw[draw=none,fill=yellow,fill opacity=.15] (-.5,-.5) rectangle (3.5,1.5);
       \draw[draw=none,fill=Orange,fill opacity=.15] (-.5,1.5) rectangle (3.5,2.5);
  \end{tikzpicture}
  \caption{After Random Real forcing}
\end{figure}
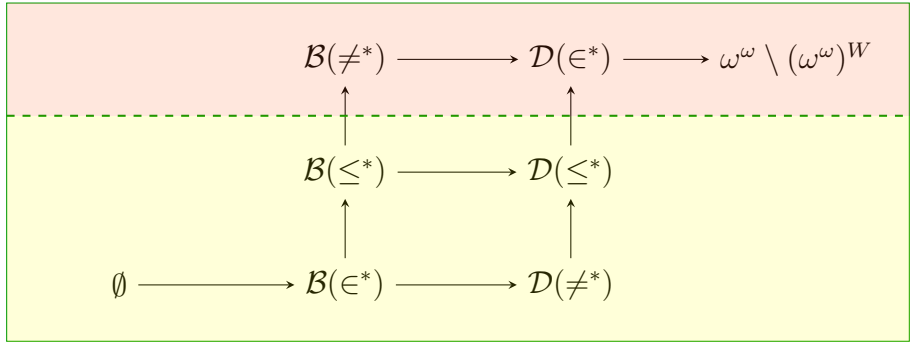
\label{randomthm}
\end{theorem}

The proof of this theorem follows from the following list of facts that are well known and can be found in \cite{BarJu95}, Chapter 3. 
\begin{fact}
The random real forcing $\mathbb B$ 
\begin{enumerate}
\item
Adds no unbounded reals, 
\item
Adds an eventually different real and 
\item
If $x \in W[r] \cap \omega^\omega \setminus W \cap \omega^\omega$ then there is a real which is random over $W$ in $W[x]$.
\end{enumerate}
\label{factB}
\end{fact}

\begin{proof}[Proof of Theorem \ref{randomthm}]
Since by 1 of Fact \ref{factB}, $\mathbb B$ adds no unbounded reals $\mathcal D(\leq^*)$ is empty. Now, suppose $x \in W[r] \setminus W$, then there is a $y \leq_W x$ which is also random over $W$ by 3 of Fact \ref{factB}. Thus by 2 of Fact \ref{factB} we get that $x \in \mathcal B(\neq^*)$. Therefore $\omega^\omega \setminus (\omega^\omega)^W  \subseteq \mathcal B(\neq^*)$ and the result follows.
\end{proof}

\subsection{Laver Forcing}
Let me now turn to Laver forcing, $\mathbb L$. Recall that conditions in Laver forcing are trees $T\subseteq \omega^{ < \omega}$ with a distinguished {\em stem}, that is, a linearly ordered initial segment, after which there is infinite branching at each node. The order is inclusion. The union of the stems of the trees in a generic for $\mathbb L$ form a real, called a Laver real. Let $l$ denote such a real over $W$. Recall that $l$ is dominating. The main theorem of this section is
\begin{theorem}
Let $l$ be a Laver real over $W$. Then in $W[l]$ we have that $\emptyset = \mathcal B(\in^*) = \mathcal D(\neq^*)$ and all other nodes are equal to the set of all new reals.
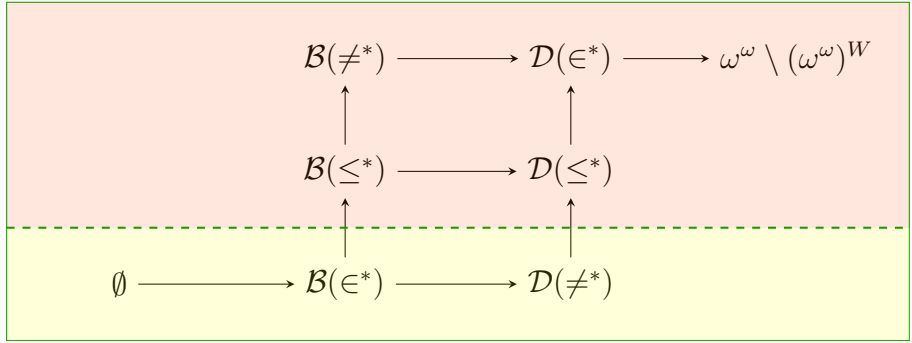
\begin{figure}[h]\label{Figure.Cichon-basic}
\centering
  \begin{tikzpicture}[scale=1.5,xscale=2]
     % place and draw the nodes
     \draw (0,0) node (empty) {$\emptyset$}
           (1,0) node (Bin*) {$\mathcal B(\in^*)$}
           (1,1) node (Bleq*) {$\mathcal B(\leq^*)$}
           (1,2) node (Bneq*) {$\mathcal B(\neq^*)$}
           (2,0) node (Dneq*) {$\mathcal D(\neq^*)$}
           (2,1) node (Dleq*) {$\mathcal D(\leq^*)$}
           (2,2) node (Din*) {$\mathcal D(\in^*)$}
           (3,2) node (all) {$\omega^\omega\setminus(\omega^\omega)^W$}
           ;
     % draw the arrows
     \draw[->,>=stealth]
            (empty) edge (Bin*)
            (Bin*) edge (Bleq*)
            (Bleq*) edge (Bneq*)
            (Bin*) edge (Dneq*)
            (Bleq*) edge (Dleq*)
            (Bneq*) edge (Din*)
            (Dneq*) edge (Dleq*)
            (Dleq*) edge (Din*)
            (Din*) edge (all)
            ;
    % draw separating lines
       \draw[thick,dashed,OliveGreen] (-.5,.5) -- (3.5,.5);
    % draw bounding rectangle
       \draw[OliveGreen] (-.5,-.5) rectangle (3.5,2.5);
    % draw fill
       \draw[draw=none,fill=yellow,fill opacity=.15] (-.5,-.5) rectangle (3.5,.5);
       \draw[draw=none,fill=Orange,fill opacity=.15] (-.5,.5) rectangle (3.5,2.5);
  \end{tikzpicture}
  \caption{After Laver forcing}
\end{figure}
\label{laver}
\end{theorem}

As before this theorem follows from well known facts about $\mathbb L$. In particular the Laver property, Definition \ref{Laver}, which holds of $\mathbb L$ \cite[Theorem 7.3.29]{BarJu95}, implies that there are no infinitely often equal reals in $W[l]$\footnote{I would like to thank Professor Martin Goldstern who explained this fact to me on Mathoverflow, https://mathoverflow.net/questions/287977/does-laver-forcing-add-an-infinitely-often-equal-real .}. Thus it suffices to note that $l$ is dominating and, by \cite[Theorem 7]{Gro87}, that Laver reals satisfy the following minimality property: if $x$ is a real such that $x \in W[l] \setminus W$ then $l \in W[x]$. Therefore every new real constructs a dominating real, hence the equality between $\mathcal B(\leq^*)$ and $\omega^\omega \setminus (\omega^\omega)^W$.
%\begin{proof}
%Clearly $l$ is dominating. Also by \cite[Theorem 7]{Gro87}) Laver reals satisfy the following minimality property: if $x$ is a real such that $x \in W[l] \setminus W$ then $l \in W[x]$. Finally $\mathbb{L}$ does not add any reals in $\mathcal D(\neq^*)$\footnote{I would like to thank Professor Martin Goldstern who explained this fact to me on Mathoverflow, https://mathoverflow.net/questions/287977/does-laver-forcing-add-an-infinitely-often-equal-real .}.This is true because, by Theorem 7.3.29 of \cite{BarJu95} $\mathbb L$ that it satisfies what is known as {\em the Laver property} (see \cite[Definition 6.3.27]{BarJu95}) which states that any real in the extension $W[l]$ which is bounded by a ground model real can eventually be captured in an $h$-slalom for any real $h \in W$ whose lim sup is infinity. The Laver property in turn implies there are no infinitely often equal reals added. As a result, $\mathcal D(\neq^*)$ is empy but every new real adds the dominating real $l$.
%\end{proof}

\subsection{Rational Perfect Tree Forcing}
Next I look at is Miller's rational perfect tree forcing, $\mathbb{PT}$. Recall that $\mathbb{PT}$ is the set of perfect trees $T \subseteq \omega^{< \omega}$ so that for all $s \in T$ there is a $t \supseteq s$ with $\omega$-many immediate successors. The order is inclusion and the unique branch through the trees in the generic is called a {\em Miller real}. Let us denote such a real by $m$.
\begin{theorem}
Let $m$ be a Miller real over $W$. Then the $\leq_W$ diagram in $W[m]$ is determined by $\emptyset = \mathcal B(\neq^*) = \mathcal D(\neq^*)$ and all other nodes are equal to the set of all new reals.
\begin{figure}[h]\label{Figure.Cichon-basic}
\centering
  \begin{tikzpicture}[scale=1.5,xscale=2]
     % place and draw the nodes
     \draw (0,0) node (empty) {$\emptyset$}
           (1,0) node (Bin*) {$\mathcal B(\in^*)$}
           (1,1) node (Bleq*) {$\mathcal B(\leq^*)$}
           (1,2) node (Bneq*) {$\mathcal B(\neq^*)$}
           (2,0) node (Dneq*) {$\mathcal D(\neq^*)$}
           (2,1) node (Dleq*) {$\mathcal D(\leq^*)$}
           (2,2) node (Din*) {$\mathcal D(\in^*)$}
           (3,2) node (all) {$\omega^\omega\setminus(\omega^\omega)^W$}
           ;
     % draw the arrows
     \draw[->,>=stealth]
            (empty) edge (Bin*)
            (Bin*) edge (Bleq*)
            (Bleq*) edge (Bneq*)
            (Bin*) edge (Dneq*)
            (Bleq*) edge (Dleq*)
            (Bneq*) edge (Din*)
            (Dneq*) edge (Dleq*)
            (Dleq*) edge (Din*)
            (Din*) edge (all)
            ;
    % draw separating lines
       \draw[thick,dashed,OliveGreen] (1.5,.5) -- (3.5,.5);
       \draw[thick,dashed,OliveGreen] (1.5,.5) -- (1.5,2.5);
    % draw bounding rectangle
       \draw[OliveGreen] (-.5,-.5) rectangle (3.5,2.5);
    % draw fill
       \draw[draw=none,fill=yellow,fill opacity=.15] (-.5,-.5) rectangle (3.5,.5);
       \draw[draw=none,fill=yellow,fill opacity=.15] (-.5,.5) rectangle (1.5,2.5);
       \draw[draw=none,fill=Orange,fill opacity=.15] (1.5,.5) rectangle (3.5,2.5);
  \end{tikzpicture}
  \caption{After rational perfect tree forcing}
\end{figure}
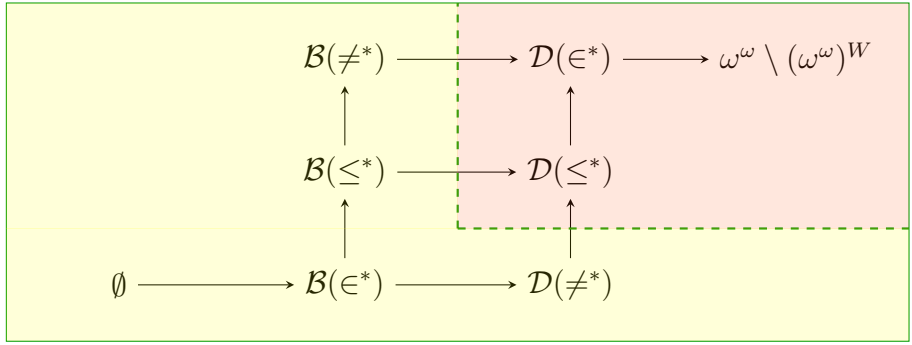
\label{miller}
\end{theorem}

This is proved in the same way as for Laver forcing. It suffices to note that $\mathbb{PT}$ adds no eventually different real, see \cite[Theorem 7.3.46, Part 1]{BarJu95}, $\mathbb{PT}$ adds no infinitely often equal real as it enjoys the Laver property (\cite[Theorem 7.3.45]{BarJu95}) and $m$ is of minimal degree, see \cite[Theorem 3]{Gro87}.

%\begin{proof}
%There are three things to show: that $\mathbb{PT}$ adds no eventually different reals, that $\mathbb{PT}$ adds no infinitely often equal reals, that every subforcing of $\mathbb{PT}$ adds an unbounded real. All of these are standard facts about $\mathbb{PT}$. The fact that $\mathbb{PT}$ adds no eventually different real follows immediately from \cite[Theorem 7.3.46, Part 1]{BarJu95}. The proof that $\mathbb{PT}$ adds no infinitely often equal real is the same as for Laver forcing as $\mathbb{PT}$ also enjoys the Laver property (\cite[Theorem 7.3.45]{BarJu95}). That $m$ is unbounded is clear from the definition of the forcing and so, to finish the theorem it suffices to show that $m$ is of minimal degree. This follows directly from \cite[Theorem 3]{Gro87}.
%\end{proof}

\subsection{Hechler Forcing}
Let $\mathbb D$ be Hechler forcing and let $d$ be the associated dominating real. Recall that conditions of $\mathbb D$ are pairs $(p, \mathcal F)$ where $p$ is a finite partial function from $\omega$ to $\omega$ and $\mathcal F$ is a finite family of elements of $\omega^\omega$. The order is given by $(q, \mathcal G) \leq_\mathbb D (p, \mathcal F)$ if and only if $q \supseteq p$, $\mathcal G \supseteq \mathcal F$ and for all $n \in {\rm dom}(q) \setminus {\rm dom}(p)$ and all $x \in \mathcal F$, $q(n) > x(n)$. Note that since $d$ is dominating, $d \in \mathcal B(\leq^*)$. 
\begin{theorem}
After Hechler forcing over $W$ the $\leq_W$-diagram has
\begin{enumerate}
\item
 $\emptyset = \mathcal B(\in^*)$, 
 \item
 $\mathcal B(\leq^*) = \mathcal B(\neq^*)$ and
 \item
 $\mathcal D(\neq^*) = \mathcal D(\leq^*) = \mathcal D(\in^*) = \omega^\omega \setminus (\omega^\omega)^W$. 
 \end{enumerate}

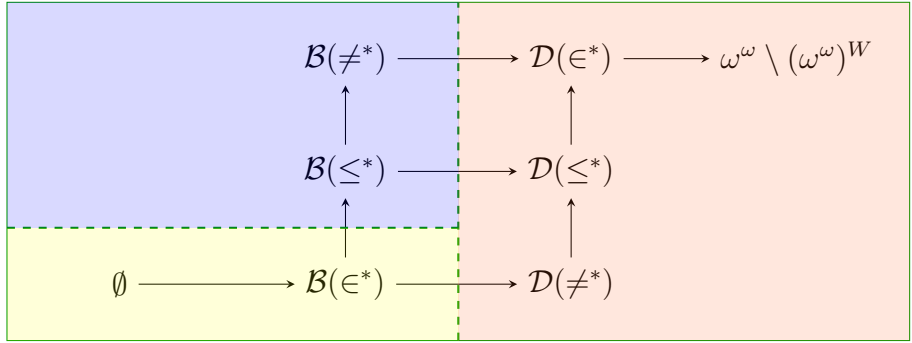
\begin{figure}[h]\label{Figure.Cichon-basic}
\centering
  \begin{tikzpicture}[scale=1.5,xscale=2]
     % place and draw the nodes
     \draw (0,0) node (empty) {$\emptyset$}
           (1,0) node (Bin*) {$\mathcal B(\in^*)$}
           (1,1) node (Bleq*) {$\mathcal B(\leq^*)$}
           (1,2) node (Bneq*) {$\mathcal B(\neq^*)$}
           (2,0) node (Dneq*) {$\mathcal D(\neq^*)$}
           (2,1) node (Dleq*) {$\mathcal D(\leq^*)$}
           (2,2) node (Din*) {$\mathcal D(\in^*)$}
           (3,2) node (all) {$\omega^\omega\setminus(\omega^\omega)^W$}
           ;
     % draw the arrows
     \draw[->,>=stealth]
            (empty) edge (Bin*)
            (Bin*) edge (Bleq*)
            (Bleq*) edge (Bneq*)
            (Bin*) edge (Dneq*)
            (Bleq*) edge (Dleq*)
            (Bneq*) edge (Din*)
            (Dneq*) edge (Dleq*)
            (Dleq*) edge (Din*)
            (Din*) edge (all)
            ;
    % draw separating lines
       \draw[thick,dashed,OliveGreen] (1.5,-.5) -- (1.5,2.5);
       \draw[thick,dashed,OliveGreen] (-.5,.5) -- (1.5,.5);
    % draw bounding rectangle
       \draw[OliveGreen] (-.5,-.5) rectangle (3.5,2.5);
    % draw fill
       \draw[draw=none,fill=yellow,fill opacity=.15] (-.5,-.5) rectangle (1.5,.5);
       \draw[draw=none,fill=blue,fill opacity=.15] (-.5,.5) rectangle (1.5,2.5);
       \draw[draw=none,fill=Orange,fill opacity=.15] (1.5,-.5) rectangle (3.5,2.5);
  \end{tikzpicture}
  \caption{After Hechler forcing}
\end{figure}
\label{hechler}
\end{theorem}

The proof of this theorem is broken up into several lemmas. First I show that $\mathbb D$ adds no slaloms eventually capturing all the ground model reals. This is well known but included here for completeness.

\begin{lemma}
After Hechler forcing over $W$ the set $\mathcal B(\in^*)$ is empty.
\label{L[d], B(in)}
\end{lemma}

\begin{proof}
Let me begin with a simple observation about Hechler forcing: if $s$ is a sentence in the forcing language and $p$ is the stem of a condition (the first coordinate) then it cannot be that there are there are finite families of functions $\mathcal F$ and $\mathcal G$ such that $(p, \mathcal F) \Vdash s$ and $(p, \mathcal G) \Vdash \neg s$. To see why, simply notice that $(p, \mathcal F \cup \mathcal G)$ is a condition extending them both. Now, using the weak homogeneity of Hechler forcing, suppose that $\Vdash_\mathbb D$\say{$\dot{s}$ is a slalom eventually capturing all elements of $(\omega^\omega)^W$}. Now fix an enumeration of $\omega^{< \omega} = \{p_0, p_1, p_2,...\}$ and consider the following function $x:\omega \to \omega$ such that $x(n) = {\rm sup} \; \{k \; | \; \exists i < n \; \exists \mathcal F \; (p_i, \mathcal F) \Vdash \check{k} \in \dot{s}(n)\} + 1$. Note that $x$ is definable in $W$.

\begin{claim}
The function $x$ is total and well defined.
\end{claim}

\begin{proof}
To see this, notice that since the maximal condition forces that $\dot{s}$ names a slalom, all conditions force that for all $n$, $\dot{s}(n)$ has size at most $n$. In particular, no condition can force more than $n$ check names to be in $s(n)$. Moreover, by the simple observation I began with, there cannot be more than $n$ check names forced to be in $\dot{s}$ by any set of conditions sharing the same stem. Thus, since there are only finitely many stems being considered, each of which can only be paired to force at most $n$ check names, there are at most $n^2$ numbers in the set $\{k \; | \; \exists i < n \; \exists \mathcal F \; (p_i, \mathcal F) \Vdash \check{k} \in \dot{s}(n)\}$ so $x$ is well defined and always finite.
\end{proof}

Now work in $W[d]$. It remains to show that $x$ is not eventually captured by the slalom $s = \dot{s}_d$. Suppose not and let $k, j \in \omega$ such that $(p_j, \mathcal F) \Vdash \forall l > \check{k} \; \check{x(k)} \in \dot{s}(\check{k})$. Let now let $l > k, k$ be such that $(p_l, \mathcal G) \leq (p_j, \mathcal F)$. Then, $(q, \mathcal G) \Vdash \check{x(l)} \in \dot{s}(l)$ but this implies $x(l) \geq x(l) +1$, which is a contradiction.
\end{proof} 
 
Continuing, recall the following theorem of Brendle and L\"owe. I have adapted it to our specific situation and terminology:
\begin{theorem}(\cite[Corollary 13]{BL11})
If $d$ is Hechler generic over $W$ and $x \in W[d] \cap \omega^\omega$ is eventually different from every $y \in W \cap \baire$, then $x$ eventually dominates every $y \in W \cap \baire$.
\end{theorem}
Therefore all the reals in $W[d]$ in $\mathcal B(\neq^*)$ are automatically in $\mathcal B(\leq^*)$. As an immediate corollary the following is true.

\begin{corollary}
In the extension of $W$ by a Hechler real, $\mathcal B(\neq^*) = \mathcal B( \leq^*)$.
\end{corollary}

Thus, we know what happens on the left side of the diagram. For the right side of the diagram, the following fact is well known and easily verified:

\begin{fact}
Let $d$ be $\mathbb D$-generic over $W$. Then $d \; {\rm mod} \; 2$ i.e. the parity of $d$ is a Cohen generic over $W$.
\end{fact}

Therefore Hechler forcing adds Cohen reals. Indeed, since $\mathbb C * \dot{\mathbb C}$ is forcing equivalent to $\mathbb C$, by the intermediate model theorem $\mathbb D$ can be decomposed in to $\mathbb C * \mathbb Q$ where $\mathbb Q$ is some quotient forcing. But then $\mathbb D \cong \mathbb C * \dot{\mathbb Q} \cong \mathbb C * \dot{\mathbb C} * \dot{\mathbb Q} \cong \mathbb C * \dot{\mathbb D}$. So Hechler forcing is the same as Cohen forcing, followed by Hechler forcing. The classification of subforcings of Hechler forcing is a very interesting, but somewhat delicate topic due in part to subtle differences in a variety of different \say{Hechler Forcings}. Palumbo \cite{Pal13} has solved this problem completely assuming there is a proper class of Woodin cardinals. I do not know of a full solution in the general case.

Summarizing, we have seen so far that there are at least three different sets present in the diagram for a Hechler real: $\emptyset$ for $\mathcal B(\in^*)$, the dominating reals, for $\mathcal B(\leq^*) = \mathcal B(\neq^*)$ and the reals that add Cohen real, which by the previous section, are all included in $\mathcal D(\neq^*)$ and thus the entire right column. To finish the analysis, I use the following fact, due to Palumbo.

\begin{fact}(\cite[Theorem 8.1]{Pal13})
Let $d$ be $\mathbb D$-generic over $W$ and let $M$ be an intermediate model i.e. $W \subseteq M \subseteq W[d]$. Then if $M \neq W$, there is a real $x \in M$ which is Cohen-generic over $W$.
\label{palumbo1}
\end{fact} 

\noindent Using Fact \ref{palumbo1} I can now show the following:

\begin{corollary}
In the extension of $W$ by a Hechler real, all the new reals construct a real which is equal to the reals in $W$ infinitely often i.e. $\omega^\omega \setminus (\omega^\omega)^W = \mathcal D(\neq^*)$.
\end{corollary}

\begin{proof}
Let $x \in W[d] \setminus W$ be a real. Then by Fact \ref{palumbo1} there is $\mathbb C$-generic real over $W$ in $W[x]$ so $x \in \mathcal D(\neq^*)$.
\end{proof}

Notice that this completely determines the diagram for a Hechler real. Since $\mathcal D(\neq^*)$ is all new reals, every node in the diagram is a subset of it. Thus all nodes on the right side are equal, $\mathcal B(\in^*)$ is empty and $\mathcal B(\leq^*) = \mathcal B(\neq^*)$ form a proper subset of the $\mathcal D$'s. This finishes the proof of Theorem \ref{hechler}

\subsection{Eventually Different Forcing}
Let $\mathbb E$ be {\em eventually different forcing}, which is defined like $\mathbb D$ except that stems of extensions need simply be eventually different from the reals in the second component, not dominating. I will show that:

\begin{theorem}
Assume that every set of reals in $L(\mathbb R)$ has the Baire property (this is implied by sufficiently large cardinals). Let $e$ be an $\mathbb E$-generic real over $W$. Then in $W[e]$ the following hold: 
\begin{enumerate}
\item
$\mathcal B(\in^*) = \mathcal B(\leq^*) = \emptyset$, 
\item
$\mathcal B(\neq^*) \subsetneq \mathcal D(\neq^*) = \mathcal D(\leq^*) = \mathcal D(\in^*) = \omega^\omega \setminus (\omega^\omega)^W$.
\end{enumerate}
Thus in particular the full diagram for eventually different forcing is as shown in Figure \ref{eforcing}.

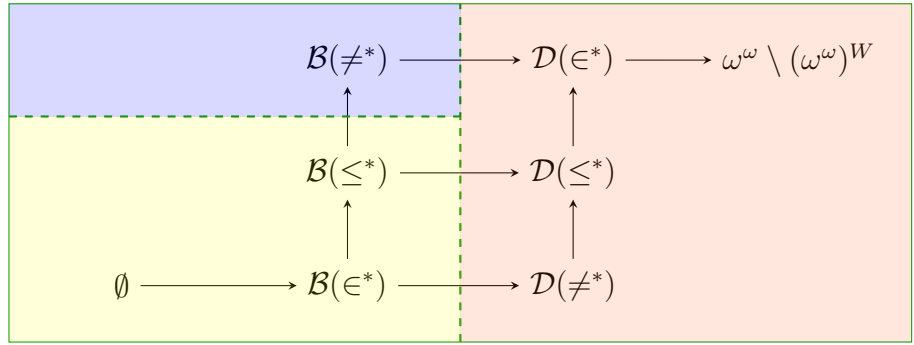
\begin{figure}[h]\label{Figure.Cichon-basic}
\centering
  \begin{tikzpicture}[scale=1.5,xscale=2]
     % place and draw the nodes
     \draw (0,0) node (empty) {$\emptyset$}
           (1,0) node (Bin*) {$\mathcal B(\in^*)$}
           (1,1) node (Bleq*) {$\mathcal B(\leq^*)$}
           (1,2) node (Bneq*) {$\mathcal B(\neq^*)$}
           (2,0) node (Dneq*) {$\mathcal D(\neq^*)$}
           (2,1) node (Dleq*) {$\mathcal D(\leq^*)$}
           (2,2) node (Din*) {$\mathcal D(\in^*)$}
           (3,2) node (all) {$\omega^\omega\setminus(\omega^\omega)^W$}
           ;
     % draw the arrows
     \draw[->,>=stealth]
            (empty) edge (Bin*)
            (Bin*) edge (Bleq*)
            (Bleq*) edge (Bneq*)
            (Bin*) edge (Dneq*)
            (Bleq*) edge (Dleq*)
            (Bneq*) edge (Din*)
            (Dneq*) edge (Dleq*)
            (Dleq*) edge (Din*)
            (Din*) edge (all)
            ;
    % draw separating lines
       \draw[thick,dashed,OliveGreen] (1.5,-.5) -- (1.5,2.5);
       \draw[thick,dashed,OliveGreen] (-.5,1.5) -- (1.5,1.5);
    % draw bounding rectangle
       \draw[OliveGreen] (-.5,-.5) rectangle (3.5,2.5);
    % draw fill
       \draw[draw=none,fill=yellow,fill opacity=.15] (-.5,-.5) rectangle (1.5,1.5);
       \draw[draw=none,fill=blue,fill opacity=.15] (-.5,1.5) rectangle (1.5,2.5);
       \draw[draw=none,fill=Orange,fill opacity=.15] (1.5,-.5) rectangle (3.5,2.5);
  \end{tikzpicture}
  \caption{After Eventually Different forcing}
\end{figure}
\label{eforcing}
\end{theorem}

To prove this I will use a series of lemmas similar to those used in the case of Hechler forcing. First, a straightforward modification of Lemma \ref{L[d], B(in)} shows that there are no dominating reals in $W[e]$:
\begin{lemma}
$W[e] \models \mathcal B(\in^*) = \mathcal B(\leq^*) = \emptyset$
\end{lemma}

To complete the analysis of the $\leq_W$-diagram after forcing with $\mathbb E$, I need the analogy of Palumbo's Fact \ref{palumbo1} for $\mathbb E$. Unfortunately, his argument uses a tree version of $\mathbb D$ that, as far as I can tell, is not available for $\mathbb E$. As such, I only know how to prove Palumbo's result for $\mathbb E$ assuming sufficient large cardinals. I conjecture that it should hold in $\ZFC$.

\begin{lemma}
Assume that every set of reals in $L(\mathbb R)$ has the property of Baire. Then in every nontrivial intermediate model between $W$ and $W[e]$ there is a real $c$ which is $\mathbb C$-generic over $W$.
\end{lemma}

A proof of this is sketched in \cite[pg 38]{Pal13} for $\mathbb D$ but the reader will notice that it goes through equally well for $\mathbb E$. Indeed the centerpiece of the argument involves a fact, due to Shelah and Gitik \cite[Proposition 4.3]{GS93} that given any sufficiently well-defined $\sigma$-centered forcing $\mathbb P$, if certain filters of $\mathbb P$ in $L(\mathbb R)$ have the property of Baire, then $\mathbb P$ will add a Cohen real. It is not hard to see from the combination of the Gitik-Shelah and the Palumbo arguments that \say{sufficiently well defined} includes all subforcings of $\mathbb E$. Thus, assuming all sets of reals have the property of Baire the result goes through.

Using this lemma, by the same argument given for $\mathbb D$, we have the proof of Theorem \ref{eforcing}.

The use of large cardinals here is unfortunate and I hope it can be improved on. Let me note however that even without large cardinals I have shown that there is a model realizing the cut determined by $\mathcal B(\in^*) = \mathcal B(\neq^*) = \emptyset$.

\subsection{Localization Forcing}
In this section I study {\em Localization forcing}, the forcing to add a generic slalom capturing all ground model reals.
\begin{definition}[Localization Forcing (cf \cite{BL11})]
The localization forcing $\mathbb{LOC}$ is defined as the set of pairs $(s, \mathcal F)$ such that $s \in ([\omega]^{<\omega})^{< \omega}$ is a finite sequence with $|s(n)| \leq n$ for all $n < |s|$ and $\mathcal F$ is a a finite family of functions in Baire space with $|\mathcal F| \leq |s|$. The order is $(t, \mathcal G) \leq_{\mathbb{LOC}} (s, \mathcal F)$ if and only if $t \supseteq s$, $\mathcal G \supseteq \mathcal F$ and $x(n) \in t(n)$ for all $x \in \mathcal F$ and all $n \in |t| \setminus |s|$.
\end{definition}
Intuitively we think of the first component as a finite approximation to a slalom we are trying to build and as such I will often refer to the length of the sequence as its \say{domain} and write ${\rm dom}(s)$. The second component is the set of functions we are promising to capture from that stage onwards.

Unfortunately I do not have a full characterization of the diagram in the case of $\mathbb{LOC}$. The following theorem summarizes the state of knowledge.
\begin{theorem}
Let $s$ be a slalom which is $\mathbb{LOC}$-generic over $W$. Then in $W[s]$ all the nodes in the diagram are nonempty (with the exception of $\emptyset$) and we have that $\mathcal B(\in^*)$ is a proper subset of $\mathcal B(\leq^*)$ and $\mathcal D(\neq^*)$. Also $\mathcal B(\leq^*) \subsetneq \mathcal B(\neq^*)$ and $\mathcal D (\leq^*) \subsetneq \mathcal D(\in^*)$. In particular, Figure \ref{locpic} is a partial diagram for $\mathbb{LOC}$. 
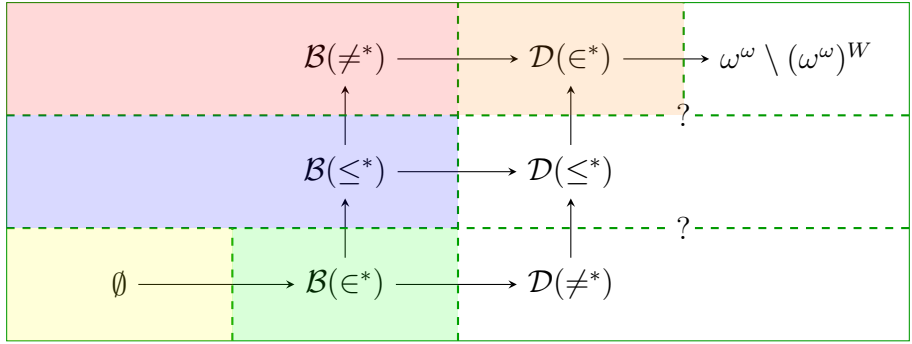
\begin{figure}[h]\label{Figure.Cichon-basic}
\centering
  \begin{tikzpicture}[scale=1.5,xscale=2]
     % place and draw the nodes
     \draw (0,0) node (empty) {$\emptyset$}
           (1,0) node (Bin*) {$\mathcal B(\in^*)$}
           (1,1) node (Bleq*) {$\mathcal B(\leq^*)$}
           (1,2) node (Bneq*) {$\mathcal B(\neq^*)$}
           (2,0) node (Dneq*) {$\mathcal D(\neq^*)$}
           (2,1) node (Dleq*) {$\mathcal D(\leq^*)$}
           (2,2) node (Din*) {$\mathcal D(\in^*)$}
           (3,2) node (all) {$\omega^\omega\setminus(\omega^\omega)^W$}
           (2.5, .5) node[fill=white] (?) {$?$}
           (2.5, 1.5) node[fill=white] (?') {$?$}
           ;
     % draw the arrows
     \draw[->,>=stealth]
            (empty) edge (Bin*)
            (Bin*) edge (Bleq*)
            (Bleq*) edge (Bneq*)
            (Bin*) edge (Dneq*)
            (Bleq*) edge (Dleq*)
            (Bneq*) edge (Din*)
            (Dneq*) edge (Dleq*)
            (Dleq*) edge (Din*)
            (Din*) edge (all)
            ;
    % draw separating lines
        \draw[thick,dashed,OliveGreen] (-.5,.5) -- (2.45,.5);
        \draw[thick,dashed,OliveGreen] (2.55,.5) -- (3.5,.5);
       \draw[thick,dashed,OliveGreen] (1.5,-.5) -- (1.5,2.5);
       \draw[thick,dashed,OliveGreen] (-.5,1.5) -- (.55,1.5);
       \draw[thick,dashed,OliveGreen] (.55,1.5) -- (2.43,1.5);
       \draw[thick,dashed,OliveGreen] (2.55,1.5) -- (3.5,1.5);
       \draw[thick,dashed,OliveGreen] (.5,-.5) -- (.5,.5);
       \draw[thick,dashed,OliveGreen] (2.5,1.65) -- (2.5,2.5);
    % draw bounding rectangle
       \draw[OliveGreen] (-.5,-.5) rectangle (3.5,2.5);
    % draw fill
       \draw[draw=none,fill=yellow,fill opacity=.15] (-.5,-.5) rectangle (.5,.5);
       \draw[draw=none,fill=green,fill opacity=.15] (.5,-.5) rectangle (1.5,.5);
       \draw[draw=none,fill=blue,fill opacity=.15] (-.5,.5) rectangle (1.5,1.5);
        \draw[draw=none,fill=red,fill opacity=.15] (-.5,1.5) rectangle (1.5,2.5);
       \draw[draw=none,fill=orange,fill opacity=.15] (1.5, 1.5) rectangle (2.5,2.5);
  \end{tikzpicture}
  \caption{Partial diagram after Localization forcing}
  \label{locpic}
\end{figure}
\label{locthm}
\end{theorem}

Proving this theorem amounts to showing that $\mathbb{LOC}$ adds $\mathbb B$, $\mathbb D$ and $\mathbb E$ generics. I start with $\mathbb D$. Notice first that $\mathbb{LOC}$ adds a dominating real. Indeed if $s$ is a generic slalom in $W^{\mathbb{LOC}}$ then $d(n) :={\rm max} \; s (n)$ has this property. This is actually a Hechler real:
\begin{lemma}
Let $s \in W^\mathbb{LOC}$ be a generic slalom eventually capturing all ground model reals. Then, $d(n):={\rm max}\; s (n)$ is $\mathbb D$-generic over $W$.
\end{lemma}

To prove this I will need a simplified version of $\mathbb D$: in the first component of a condition I will assume that the domain is a finite initial segment of $\omega$ and instead of having the second component of a condition of $\mathbb D$ be a finite family of functions, it will be a single function. Then $(q, y) \leq_\mathbb D (p, x)$ if and only if $q$ extends $p$, for all $n \in {\rm dom}(q) \setminus {\rm dom}(p)$, $q(n) \geq x(n)$ and for all $n \in \omega$, and $y(n) \geq x(n)$. It's not hard to see that this version of $\mathbb D$ is forcing equivalent to the original one I defined.

\begin{proof}
Recall that a {\em projection} $\pi:\mathbb P \to \mathbb Q$ between two posets is an order preserving map which sends the maximal element of $\mathbb P$ to the maximal element of $\mathbb Q$ and for all $p \in \mathbb P$ and all $q \leq \pi(p)$ there is some $\overline{p} \leq p$ such that $\pi(\overline{p}) \leq q$. If a projection exists between $\mathbb P$ and $\mathbb Q$ then the image $\pi '' G$ of a $\mathbb P$-generic filter generates a $\mathbb Q$-generic filter. Therefore to prove the lemma it suffices to show that the map $\pi:\mathbb{LOC} \to \mathbb D$ such that $\pi(s, \mathcal F) = (n \mapsto {\rm max} \; s(n), \Sigma \mathcal F)$ where $\Sigma \mathcal F$ is the pointwise sum, is a projection. To see why, note that if $(s, \mathcal F) \in \mathbb{LOC}$ and letting, for all $n \in {\rm dom} (s)$, $p(n) = {\rm max} \; s(n)$ and $x = \Sigma \mathcal F$, then the pair $(p, x)$ is a $\mathbb D$ condition and the union of all conditions such defined from elements of the $\mathbb{LOC}$ generic defining $s$ is the $d$ from the statement of the lemma. 

It is routine to check that $\pi(1_{\mathbb{LOC}}) =  1_\mathbb D$ and that the map $\pi$ is order preserving. The difficulty is in verifying the third condition of projections. To this end, let $(s, \mathcal F) \in \mathbb{LOC}$ and let $(p, x)=\pi(s, \mathcal F)$. Let $(p', x') \leq (p, x)$ and let $D \subseteq \mathbb D$ be a set of conditions which is dense below $(p', x')$. It suffices to find a strengthening $(t, \mathcal G)$ of $(s, \mathcal F)$, such that $(n \mapsto {\rm max} \; t(n), \Sigma G) \in D$. To do this, let $(q, z) \in D$ strengthen $(p', x')$ so that $|{\rm dom}(q)| > |{\rm dom}(s)| + 2$.

%To do this, first, find a function $y:\omega \to \omega$ such that for all $n \notin {\rm dom}(p ')$ $y(n) > n + x ' (n)$ and for $n \in {\rm dom} (p ' )$ let $y(n) = x' (n)$. Then, $(p, y)$ strengthens $(p, x)$ and is compatible with $(p', x')$. Let $(q, z) \in D$ strengthen $(p, y)$. 

Now, we can build our new $\mathbb{LOC}$ condition. Define $H:\omega \to \omega$ by $H(n) = z(n) - x(n)$. Notice that since $x'(n)$ was assumed to be bigger than $x(n)$ for all $n$ and $z(n) \geq x'(n)$ since it is a strengthening it follows that $H$ is in fact always nonnegative. Moreover, $x + H = \Sigma \mathcal F + H = z$. It remains to show that there is a $t \supseteq s$ such that ${\rm dom}(t) = {\rm dom}(q)$, for all $n \in {\rm dom}(t)$, ${\rm max} \; t(n) = q(n)$ and for all $n \in {\rm dom}(t) \setminus {\rm dom}(s)$ and all $v\in \mathcal F$, $v(n) \in t(n)$. Once this has been done $(t, \mathcal F \cup \{H\})$ will be the desired condition. I claim that this is all possible. I will describe a $t$ extending $s$ be defined on the domain of $q$ (by construction, the domain of $q$ contains that of $s$). Since $|{\rm dom}(q)| > |{\rm dom}(s)| + 2$, the domain of $t$ will be large enough to accommodate the side condition $\mathcal F \cup \{H\}$. Let $|\mathcal F| = k$ and enumerate $\mathcal F = \{v_0,...,v_{k-1}\}$. Note that $k < n$ for all $n \in {\rm dom}(q) \setminus {\rm dom}(s)$. Now, for each $n \in {\rm dom}(q) \setminus {\rm dom}(s)$, let me define $t(n)$. Notice first that one must put in all $k$ numbers $\{v_0(n),...,v_{k-1}(n)\}$ and we also want ${\rm max} \; t(n) = q(n)$ so add this in too. Since $n > k$, one may add up to $n -k-1$ additional numbers $\{j_0,...,j_{n-k-2}\}$ such that each one is less than $q(n)$ and different from all numbers in the set $\{v_0(n),...,v_{k-1}(n), q(n)\}$. Let $t(n)$ be this set plus any of the additional numbers that fit. Note that the definition of $\mathbb{LOC}$ allows $|t(n)| \leq n$ so we do not need to meet this bound everywhere. What matters is that, since by construction $q(n) \geq x'(n)$ for all $n \notin {\rm dom}(p)$ and $x'(n) \geq \Sigma_{i < k} f_i(n)$ on this domain we can arrange always that $q(n)$ is the maximum of $t(n)$, which is what we needed. 
\end{proof}

Now, I show that $\mathbb{LOC}$ adds an $\mathbb E$-generic real. This fact was first told to me (without proof) in private communication with J. Brendle. I thank him for pointing it out to me.

\begin{lemma}
The forcing $\mathbb{LOC}$ adds an $\mathbb E$-generic real.
\end{lemma}

\begin{proof}
Given a condition $(s, \mathcal F) \in \mathbb{LOC}$ define a stem for an $\mathbb E$-condition as $p_s : {\rm dom} (s) \to \omega$ by letting for all $n \in {\rm dom}(s)$ $p_s(n)$ be equal to the $k^{\rm th}$ natural number $m$ not in the set $s(n)$ where the pointwise sum $\Sigma s(n) \equiv k \; {\rm mod} \; n$. I claim that the map $\pi:\mathbb{LOC} \to \mathbb{E}$ defined by $\pi(s, \mathcal F) = (p_s, \mathcal F)$ is a projection. Clearly the maximal condition is sent to the maximal condition and this map is order preserving. Let $(s, \mathcal F) \in \mathbb{LOC}$, and let $(q, \mathcal G) \leq_\mathbb E (p_s, \mathcal F)$. We need to show that there is a strengthening of $(q, \mathcal G)$ in the image of $\pi$. To this end, note that we can assume with out loss that $|\mathcal G| < {\rm dom}(q)$ since otherwise we can strengthen to make this true. Now, define a partial slalom as follows: $s_q:{\rm dom}(q) \to [\omega]^{<\omega}$. For $n \in {\rm dom}(p)$ let $s_q(n) = s(n)$. For $n \notin {\rm dom}(p)$ let $q(n) = m$ and suppose that $m$ is the $k^{\rm th}$ not in $\{x(n) \; |\; x \in \mathcal F\}$ and suppose that this set has size $l < n$ (the $<$ follows from the fact that $(p, \mathcal F)$ is in the image of $\pi$). Then, pick $n-l$ numbers $m_{l}, m_{l+1},...,m_{n-1}$ all greater than every $x(n)$ for $x \in \mathcal F$ and not equal to $m$ so that $\Sigma_{x \in \mathcal F} x(n) + \Sigma_{i=l}^{n-1} m_i \equiv k \; {\rm mod}\; n$. This can be accomplished, for instance, as follows: if $\Sigma_{x \in \mathcal F} x(n) \equiv j \; {\rm mod} \; n$ then let $m_l \equiv k - j \; {\rm mod} \; n$ greater than all the $f(n)$'s and let all other $m_i$'s be multiples of $n$. Finally let $s_q(n) = \{x(n) \; | \; x \in \mathcal F\} \cup \{m_{l},...,m_{n-1}\}$. Then $(s_q, \mathcal G) \leq (s, \mathcal F)$ and $\pi(s_q, \mathcal G) = (q, \mathcal G)$ as needed.
\end{proof}

Finally,
\begin{lemma}
Any forcing adding a slalom eventually capturing all ground model reals adds a random real. In particular $\mathbb{LOC}$ adds a random real.
\end{lemma}

\begin{proof}
By Corollary 3.2 of \cite{ShJdKM} adding a slalom eventually capturing all ground model reals is equivalent to adding a Borel null set which covers all Borel null sets coded in the ground model. Let $N \subseteq \omega^\omega$ be such a null set and let $y \notin N$. Then $y$ is not in any ground model null set so $y$ is a random real. 
\end{proof}

Combining all of these results then proves Theorem \ref{locthm} since both $\mathbb D$ and $\mathbb E$ add Cohen reals realizing the split down the middle in Figure \ref{locpic} and $\mathbb B$ adds a bounded real not caught in any old slalom so $\mathcal D(\leq^*)$ is strictly contained in $\mathcal D(\in^*)$.

As an aside notice that there seem to be other eventually different reals added by $\mathbb{LOC}$:
\begin{observation}
Let $s \in W^\mathbb{LOC}$ be a generic slalom eventually capturing all ground model reals. Let $a(n)$ be defined as the least $k \notin s (n)$. Then $a$ is a real which is eventually different from all ground model reals but is not an $\mathbb E$-generic real.
\label{evdiffloc}
\end{observation}

\begin{proof}
First notice that the $a$ described in the theorem is in fact eventually different from all ground model reals since every real eventually is captured by $s$ and after that point $a$ is different from it. Moreover, notice that $a$ is not only not dominating over the ground model reals but actually not even unbounded since, given any real $f \in W$ growing faster than the identity ($n \mapsto n+2$ even), the least $k$ not in $s (n)$ must be less than $f(n)$ since $|s(n)| = n$. From this it follows that $a$ is not an $\mathbb E$-generic real since it is not unbounded. 
\end{proof}

This lemma is somewhat surprising and indeed I do not know exactly what the forcing adding the real $a$ is or if it is a previously studied notion. In particular, I don't know if this real is random over $W$, though I conjecture that it is.

\subsection{Cuts in the Diagram and the Analogy with Cardinal Characteristics}
Let me finish this section by noting that it follows from what I have shown that the ZFC($W$)-provable subset implications implied by Theorem \ref{propdi} are the only ones. In other words, Theorem \ref{cichoncomplete} is proved. Indeed a simple inspection of the diagrams above show that every implication shown in Figure 1 is consistently strict and no other implications are true in every $V$ extending $W$. This shows also that the analogue discussed in the previous section holds in a robust way with the traditional Cicho\'n diagram. In fact, we can actually show that a stronger fact is true.

\begin{theorem}
All cuts consistent with the diagram are consistent with ZFC($W$) in the following sense: Given any collection $N$ of (not $\emptyset$)-nodes in the diagram which are closed upwards under $\subseteq$ there is a proper forcing $\mathbb P$ in $W$ so that forcing with $\mathbb P$ over $W$ results in all and only the nodes in $N$ being nonempty. See Figure \ref{allcuts} for a pictorial representation
\end{theorem}

Note that this is slightly weaker than the sense of cuts I have been considering above since I'm making no distinction between various non-empty nodes after forcing.
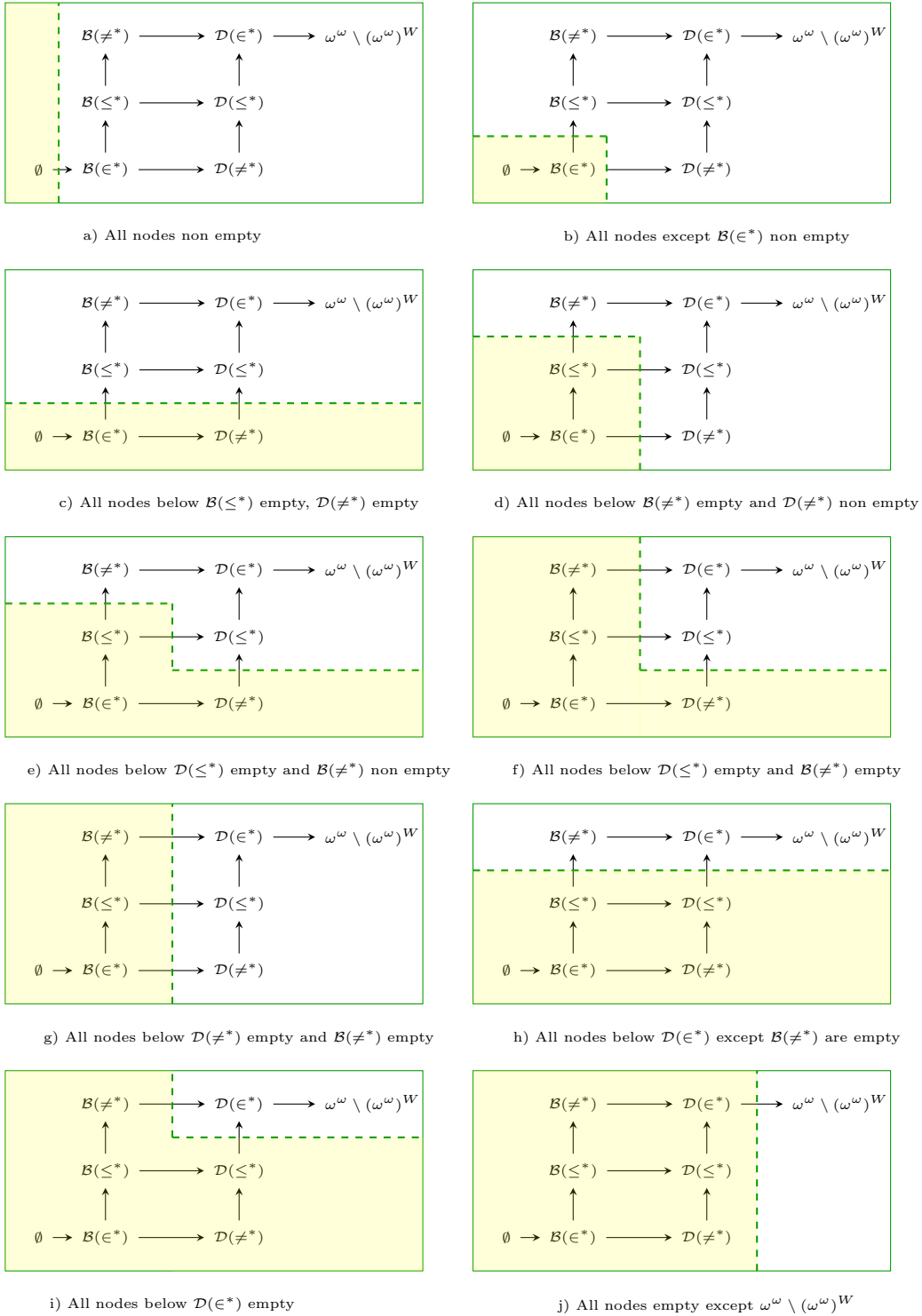
\begin{figure}\label{Figure.Cichon-basic}
\centering
  \begin{tikzpicture}
  \begin{scope}[shift={(0, 16)}]
  
     %1 place and draw the nodes
     \draw (0,0) node (empty) {\tiny{$\emptyset$}}
           (1,0) node (Bin*) {\tiny{$\mathcal B(\in^*)$}}
           (1,1) node (Bleq*) {\tiny{$\mathcal B(\leq^*)$}}
           (1,2) node (Bneq*) {\tiny{$\mathcal B(\neq^*)$}}
           (3,0) node (Dneq*) {\tiny{$\mathcal D(\neq^*)$}}
           (3,1) node (Dleq*) {\tiny{$\mathcal D(\leq^*)$}}
           (3,2) node (Din*) {\tiny{$\mathcal D(\in^*)$}}
           (5,2) node (all) {\tiny{$\omega^\omega\setminus(\omega^\omega)^W$}}
           (2, -1) node (label) {\tiny{a) All nodes non empty}}
           ;
     % draw the arrows
     \draw[->,>=stealth]
            (empty) edge (Bin*)
            (Bin*) edge (Bleq*)
            (Bleq*) edge (Bneq*)
            (Bin*) edge (Dneq*)
            (Bleq*) edge (Dleq*)
            (Bneq*) edge (Din*)
            (Dneq*) edge (Dleq*)
            (Dleq*) edge (Din*)
            (Din*) edge (all)
            ;
    % draw separating lines
       \draw[thick,dashed,OliveGreen] (.3,-.5) -- (.3,2.5);
    % draw bounding rectangle
       \draw[OliveGreen] (-.5,-.5) rectangle (5.75,2.5);
    % draw fill
       \draw[draw=none,fill=yellow,fill opacity=.15] (-.5,-.5) rectangle (.3,2.5);

       \end{scope}
       
        \begin{scope}[shift={(7, 0)}]
  
     %2 place and draw the nodes
     \draw (0,0) node (empty) {\tiny{$\emptyset$}}
           (1,0) node (Bin*) {\tiny{$\mathcal B(\in^*)$}}
           (1,1) node (Bleq*) {\tiny{$\mathcal B(\leq^*)$}}
           (1,2) node (Bneq*) {\tiny{$\mathcal B(\neq^*)$}}
           (3,0) node (Dneq*) {\tiny{$\mathcal D(\neq^*)$}}
           (3,1) node (Dleq*) {\tiny{$\mathcal D(\leq^*)$}}
           (3,2) node (Din*) {\tiny{$\mathcal D(\in^*)$}}
           (5,2) node (all) {\tiny{$\omega^\omega\setminus(\omega^\omega)^W$}}
           (3, -1) node (label) {\tiny{j) All nodes empty except $\omega^\omega \setminus (\omega^\omega)^W$}}
           ;
     % draw the arrows
     \draw[->,>=stealth]
            (empty) edge (Bin*)
            (Bin*) edge (Bleq*)
            (Bleq*) edge (Bneq*)
            (Bin*) edge (Dneq*)
            (Bleq*) edge (Dleq*)
            (Bneq*) edge (Din*)
            (Dneq*) edge (Dleq*)
            (Dleq*) edge (Din*)
            (Din*) edge (all)
            ;
    % draw separating lines
       \draw[thick,dashed,OliveGreen] (3.75,-.5) -- (3.75,2.5);
    % draw bounding rectangle
       \draw[OliveGreen] (-.5,-.5) rectangle (5.75,2.5);
    % draw fill
       \draw[draw=none,fill=yellow,fill opacity=.15] (-.5,-.5) rectangle (3.75,2.5);

       \end{scope}
       
       \begin{scope}[shift={(0, 0)}]
  
     %3 place and draw the nodes
     \draw (0,0) node (empty) {\tiny{$\emptyset$}}
           (1,0) node (Bin*) {\tiny{$\mathcal B(\in^*)$}}
           (1,1) node (Bleq*) {\tiny{$\mathcal B(\leq^*)$}}
           (1,2) node (Bneq*) {\tiny{$\mathcal B(\neq^*)$}}
           (3,0) node (Dneq*) {\tiny{$\mathcal D(\neq^*)$}}
           (3,1) node (Dleq*) {\tiny{$\mathcal D(\leq^*)$}}
           (3,2) node (Din*) {\tiny{$\mathcal D(\in^*)$}}
           (5,2) node (all) {\tiny{$\omega^\omega\setminus(\omega^\omega)^W$}}
           (2, -1) node (label) {\tiny{i) All nodes below $\mathcal D(\in^*)$ empty}}
           ;
     % draw the arrows
     \draw[->,>=stealth]
            (empty) edge (Bin*)
            (Bin*) edge (Bleq*)
            (Bleq*) edge (Bneq*)
            (Bin*) edge (Dneq*)
            (Bleq*) edge (Dleq*)
            (Bneq*) edge (Din*)
            (Dneq*) edge (Dleq*)
            (Dleq*) edge (Din*)
            (Din*) edge (all)
            ;
    % draw separating lines
       \draw[thick,dashed,OliveGreen] (2., 1.5) -- (2,2.5);
       \draw[thick,dashed,OliveGreen] (2., 1.5) -- (5.75,1.5);
    % draw bounding rectangle
       \draw[OliveGreen] (-.5,-.5) rectangle (5.75,2.5);
    % draw fill
       \draw[draw=none,fill=yellow,fill opacity=.15] (-.5,-.5) rectangle (2,2.5);
       \draw[draw=none,fill=yellow,fill opacity=.15] (2,-.5) rectangle (5.75,1.5);

       \end{scope}
       
       \begin{scope}[shift={(7, 4)}]
  
     %4 place and draw the nodes
     \draw (0,0) node (empty) {\tiny{$\emptyset$}}
           (1,0) node (Bin*) {\tiny{$\mathcal B(\in^*)$}}
           (1,1) node (Bleq*) {\tiny{$\mathcal B(\leq^*)$}}
           (1,2) node (Bneq*) {\tiny{$\mathcal B(\neq^*)$}}
           (3,0) node (Dneq*) {\tiny{$\mathcal D(\neq^*)$}}
           (3,1) node (Dleq*) {\tiny{$\mathcal D(\leq^*)$}}
           (3,2) node (Din*) {\tiny{$\mathcal D(\in^*)$}}
           (5,2) node (all) {\tiny{$\omega^\omega\setminus(\omega^\omega)^W$}}
           (3, -1) node (label) {\tiny{h) All nodes below $\mathcal D(\in^*)$ except $\mathcal B(\neq^*)$ are empty}}
           ;
     % draw the arrows
     \draw[->,>=stealth]
            (empty) edge (Bin*)
            (Bin*) edge (Bleq*)
            (Bleq*) edge (Bneq*)
            (Bin*) edge (Dneq*)
            (Bleq*) edge (Dleq*)
            (Bneq*) edge (Din*)
            (Dneq*) edge (Dleq*)
            (Dleq*) edge (Din*)
            (Din*) edge (all)
            ;
    % draw separating lines
       \draw[thick,dashed,OliveGreen] (-.5,1.5) -- (5.75,1.5);
    % draw bounding rectangle
       \draw[OliveGreen] (-.5,-.5) rectangle (5.75,2.5);
    % draw fill
       \draw[draw=none,fill=yellow,fill opacity=.15] (-.5,-.5) rectangle (5.75,1.5);

       \end{scope}
       
              \begin{scope}[shift={(7, 8)}]
  
     %5 place and draw the nodes
     \draw (0,0) node (empty) {\tiny{$\emptyset$}}
           (1,0) node (Bin*) {\tiny{$\mathcal B(\in^*)$}}
           (1,1) node (Bleq*) {\tiny{$\mathcal B(\leq^*)$}}
           (1,2) node (Bneq*) {\tiny{$\mathcal B(\neq^*)$}}
           (3,0) node (Dneq*) {\tiny{$\mathcal D(\neq^*)$}}
           (3,1) node (Dleq*) {\tiny{$\mathcal D(\leq^*)$}}
           (3,2) node (Din*) {\tiny{$\mathcal D(\in^*)$}}
           (5,2) node (all) {\tiny{$\omega^\omega\setminus(\omega^\omega)^W$}}
           (3, -1) node (label) {\tiny{f) All nodes below $\mathcal D(\leq^*)$ empty and $\mathcal B(\neq^*)$ empty}}
           ;
     % draw the arrows
     \draw[->,>=stealth]
            (empty) edge (Bin*)
            (Bin*) edge (Bleq*)
            (Bleq*) edge (Bneq*)
            (Bin*) edge (Dneq*)
            (Bleq*) edge (Dleq*)
            (Bneq*) edge (Din*)
            (Dneq*) edge (Dleq*)
            (Dleq*) edge (Din*)
            (Din*) edge (all)
            ;
    % draw separating lines
       \draw[thick,dashed,OliveGreen] (2., .5) -- (2,2.5);
       \draw[thick,dashed,OliveGreen] (2., .5) -- (5.75,.5);
    % draw bounding rectangle
       \draw[OliveGreen] (-.5,-.5) rectangle (5.75,2.5);
    % draw fill
       \draw[draw=none,fill=yellow,fill opacity=.15] (-.5,-.5) rectangle (2,2.5);
       \draw[draw=none,fill=yellow,fill opacity=.15] (2,-.5) rectangle (5.75,.5);

       \end{scope}
       
              \begin{scope}[shift={(0, 8)}]
  
     %6 place and draw the nodes
     \draw (0,0) node (empty) {\tiny{$\emptyset$}}
           (1,0) node (Bin*) {\tiny{$\mathcal B(\in^*)$}}
           (1,1) node (Bleq*) {\tiny{$\mathcal B(\leq^*)$}}
           (1,2) node (Bneq*) {\tiny{$\mathcal B(\neq^*)$}}
           (3,0) node (Dneq*) {\tiny{$\mathcal D(\neq^*)$}}
           (3,1) node (Dleq*) {\tiny{$\mathcal D(\leq^*)$}}
           (3,2) node (Din*) {\tiny{$\mathcal D(\in^*)$}}
           (5,2) node (all) {\tiny{$\omega^\omega\setminus(\omega^\omega)^W$}}
           (3, -1) node (label) {\tiny{e) All nodes below $\mathcal D(\leq^*)$ empty and $\mathcal B(\neq^*)$ non empty}}
           ;
     % draw the arrows
     \draw[->,>=stealth]
            (empty) edge (Bin*)
            (Bin*) edge (Bleq*)
            (Bleq*) edge (Bneq*)
            (Bin*) edge (Dneq*)
            (Bleq*) edge (Dleq*)
            (Bneq*) edge (Din*)
            (Dneq*) edge (Dleq*)
            (Dleq*) edge (Din*)
            (Din*) edge (all)
            ;
    % draw separating lines
       \draw[thick,dashed,OliveGreen] (2., 1.5) -- (2,.5);
       \draw[thick,dashed,OliveGreen] (-.5, 1.5) -- (2,1.5);
       \draw[thick,dashed,OliveGreen] (2., .5) -- (5.75,.5);
    % draw bounding rectangle
       \draw[OliveGreen] (-.5,-.5) rectangle (5.75,2.5);
    % draw fill
       \draw[draw=none,fill=yellow,fill opacity=.15] (-.5,-.5) rectangle (2.0,1.5);
       \draw[draw=none,fill=yellow,fill opacity=.15] (2.0,-.5) rectangle (5.75,.5);

       \end{scope}
       
              \begin{scope}[shift={(0, 4)}]
  
     %7 place and draw the nodes
     \draw (0,0) node (empty) {\tiny{$\emptyset$}}
           (1,0) node (Bin*) {\tiny{$\mathcal B(\in^*)$}}
           (1,1) node (Bleq*) {\tiny{$\mathcal B(\leq^*)$}}
           (1,2) node (Bneq*) {\tiny{$\mathcal B(\neq^*)$}}
           (3,0) node (Dneq*) {\tiny{$\mathcal D(\neq^*)$}}
           (3,1) node (Dleq*) {\tiny{$\mathcal D(\leq^*)$}}
           (3,2) node (Din*) {\tiny{$\mathcal D(\in^*)$}}
           (5,2) node (all) {\tiny{$\omega^\omega\setminus(\omega^\omega)^W$}}
            (3, -1) node (label) {\tiny{g) All nodes below $\mathcal D(\neq^*)$ empty and $\mathcal B(\neq^*)$ empty}}
           ;
     % draw the arrows
     \draw[->,>=stealth]
            (empty) edge (Bin*)
            (Bin*) edge (Bleq*)
            (Bleq*) edge (Bneq*)
            (Bin*) edge (Dneq*)
            (Bleq*) edge (Dleq*)
            (Bneq*) edge (Din*)
            (Dneq*) edge (Dleq*)
            (Dleq*) edge (Din*)
            (Din*) edge (all)
            ;
    % draw separating lines
       \draw[thick,dashed,OliveGreen] (2,-.5) -- (2,2.5);
    % draw bounding rectangle
       \draw[OliveGreen] (-.5,-.5) rectangle (5.75,2.5);
    % draw fill
       \draw[draw=none,fill=yellow,fill opacity=.15] (-.5,-.5) rectangle (2,2.5);

       \end{scope}
       
              \begin{scope}[shift={(7, 12)}]
  
     %8 place and draw the nodes
     \draw (0,0) node (empty) {\tiny{$\emptyset$}}
           (1,0) node (Bin*) {\tiny{$\mathcal B(\in^*)$}}
           (1,1) node (Bleq*) {\tiny{$\mathcal B(\leq^*)$}}
           (1,2) node (Bneq*) {\tiny{$\mathcal B(\neq^*)$}}
           (3,0) node (Dneq*) {\tiny{$\mathcal D(\neq^*)$}}
           (3,1) node (Dleq*) {\tiny{$\mathcal D(\leq^*)$}}
           (3,2) node (Din*) {\tiny{$\mathcal D(\in^*)$}}
           (5,2) node (all) {\tiny{$\omega^\omega\setminus(\omega^\omega)^W$}}
            (3, -1) node (label) {\tiny{d) All nodes below $\mathcal B(\neq^*)$ empty and $\mathcal D(\neq^*)$ non empty}}
           ;
     % draw the arrows
     \draw[->,>=stealth]
            (empty) edge (Bin*)
            (Bin*) edge (Bleq*)
            (Bleq*) edge (Bneq*)
            (Bin*) edge (Dneq*)
            (Bleq*) edge (Dleq*)
            (Bneq*) edge (Din*)
            (Dneq*) edge (Dleq*)
            (Dleq*) edge (Din*)
            (Din*) edge (all)
            ;
    % draw separating lines
       \draw[thick,dashed,OliveGreen] (2,-.5) -- (2,1.5);
       \draw[thick,dashed,OliveGreen] (-.5,1.5) -- (2,1.5);
    % draw bounding rectangle
       \draw[OliveGreen] (-.5,-.5) rectangle (5.75,2.5);
    % draw fill
       \draw[draw=none,fill=yellow,fill opacity=.15] (-.5,-.5) rectangle (2,1.5);

       \end{scope}
       
              \begin{scope}[shift={(0, 12)}]
  
     %9 place and draw the nodes
     \draw (0,0) node (empty) {\tiny{$\emptyset$}}
           (1,0) node (Bin*) {\tiny{$\mathcal B(\in^*)$}}
           (1,1) node (Bleq*) {\tiny{$\mathcal B(\leq^*)$}}
           (1,2) node (Bneq*) {\tiny{$\mathcal B(\neq^*)$}}
           (3,0) node (Dneq*) {\tiny{$\mathcal D(\neq^*)$}}
           (3,1) node (Dleq*) {\tiny{$\mathcal D(\leq^*)$}}
           (3,2) node (Din*) {\tiny{$\mathcal D(\in^*)$}}
           (5,2) node (all) {\tiny{$\omega^\omega\setminus(\omega^\omega)^W$}}
            (3, -1) node (label) {\tiny{c) All nodes below $\mathcal B(\leq^*)$ empty, $\mathcal D(\neq^*)$ empty}}
           ;
     % draw the arrows
     \draw[->,>=stealth]
            (empty) edge (Bin*)
            (Bin*) edge (Bleq*)
            (Bleq*) edge (Bneq*)
            (Bin*) edge (Dneq*)
            (Bleq*) edge (Dleq*)
            (Bneq*) edge (Din*)
            (Dneq*) edge (Dleq*)
            (Dleq*) edge (Din*)
            (Din*) edge (all)
            ;
    % draw separating lines
       \draw[thick,dashed,OliveGreen] (-.5,.5) -- (5.75, .5);
    % draw bounding rectangle
       \draw[OliveGreen] (-.5,-.5) rectangle (5.75,2.5);
    % draw fill
       \draw[draw=none,fill=yellow,fill opacity=.15] (-.5,-.5) rectangle (5.75,.5);

       \end{scope}
       
              \begin{scope}[shift={(7, 16)}]
  
     %10 place and draw the nodes
     \draw (0,0) node (empty) {\tiny{$\emptyset$}}
           (1,0) node (Bin*) {\tiny{$\mathcal B(\in^*)$}}
           (1,1) node (Bleq*) {\tiny{$\mathcal B(\leq^*)$}}
           (1,2) node (Bneq*) {\tiny{$\mathcal B(\neq^*)$}}
           (3,0) node (Dneq*) {\tiny{$\mathcal D(\neq^*)$}}
           (3,1) node (Dleq*) {\tiny{$\mathcal D(\leq^*)$}}
           (3,2) node (Din*) {\tiny{$\mathcal D(\in^*)$}}
           (5,2) node (all) {\tiny{$\omega^\omega\setminus(\omega^\omega)^W$}}
            (3, -1) node (label) {\tiny{b) All nodes except $\mathcal B(\in^*)$ non empty}}
           ;
     % draw the arrows
     \draw[->,>=stealth]
            (empty) edge (Bin*)
            (Bin*) edge (Bleq*)
            (Bleq*) edge (Bneq*)
            (Bin*) edge (Dneq*)
            (Bleq*) edge (Dleq*)
            (Bneq*) edge (Din*)
            (Dneq*) edge (Dleq*)
            (Dleq*) edge (Din*)
            (Din*) edge (all)
            ;
    % draw separating lines
       \draw[thick,dashed,OliveGreen] (-.5,.5) -- (1.5,.5);
       \draw[thick,dashed,OliveGreen] (1.5,.5) -- (1.5,-.5);
    % draw bounding rectangle
       \draw[OliveGreen] (-.5,-.5) rectangle (5.75,2.5);
    % draw fill
       \draw[draw=none,fill=yellow,fill opacity=.15] (-.5,-.5) rectangle (1.5,.5);

       \end{scope}
      
  \end{tikzpicture}
  \caption{All Possible Cuts in the $\leq_W$ Cicho\'n Diagram. Each one can be achieved by a proper forcing over $W$. White means that the node is not empty while yellow means that it is. No distinction is made between different non-empty nodes. Note that the trivial cut where all nodes remain empty is not shown.}
  \label{allcuts}
\end{figure}

\begin{proof}
There are two cuts I have yet to explicitly show. These correspond to e) and i) in Figure \ref{allcuts} below. However for completeness let me go through all cuts one at a time. 

\noindent a) All nodes are non empty: This is accomplished by $\mathbb{LOC}$.
 
\noindent b) All nodes except $\mathcal B(\in^*)$ are non empty: This is accomplished by $\mathbb D$.
 
\noindent c) All nodes below $\mathcal B(\leq^*)$ are empty and $\mathcal D(\neq^*)$ is empty: This is accomplished by $\mathbb L$.
 
\noindent d) All nodes below $\mathcal B(\neq^*)$ are empty and $\mathcal D(\neq^*)$ is non empty: This is accomplished by $\mathbb E$.
 
 \noindent e) All nodes below $\mathcal D(\leq^*)$ are empty and $\mathcal B(\neq^*)$ is non empty: This is the first case where we still have to prove something. Let $\mathbb P = \mathbb B * \dot{\mathbb{PT}}$. I claim that in $W^\mathbb P$ this cut is realized. We have seen that forcing with $\mathbb B$ adds an eventually different real and, by further forcing with $\mathbb{PT}$ over $W^{\mathbb B}$ will add a real which is unbounded by $W^{\mathbb B} \cap \omega^\omega$ and hence $W \cap \omega^\omega$. It remains therefore to see that in $W^\mathbb P$ there are no dominating or infinitely often equal reals over $W$. To show that there are no dominating reals, note that in general $\mathbb{PT}$ adds no dominating real, so in $W^\mathbb P$ there is no real which is dominating over $W^\mathbb B$. But, since $\mathbb B$ is $\omega^\omega$-bounding, it follows that there is no real dominating over $W$ in $W^\mathbb P$. To show there are no infinitely often equal reals, let us first note the following fact.
\begin{fact}[Corollary 2.5.2 of \cite{BarJu95}]
Suppose $M$ is a transitive model of a sufficiently large fragment of $\ZFC$. Then $M \cap 2^\omega \in \mathcal N$ if and only if there is a sequence $\langle F_n \subseteq 2^n \; | \; n  < \omega \rangle$ such that $\Sigma^\infty_{n=0} |F_n| 2^{-n} < \infty$ and for every $x \in M \cap 2^\omega$ there are infinitely many $n$ so that $x \upharpoonright n \in F_n$.
\end{fact}

As a corollary of this Fact, notice that adding an infinitely often equal real on $\omega^\omega$ makes the ground model reals measure $0$. To see why, suppose $y \in \omega^\omega$ is infinitely often equal over an inner model $M$ and let $\langle \tau_k \; | \; k < \omega \rangle$ be an enumeration in $M$ of the elements of $2^{< \omega}$. Then for every $x \in 2^\omega \cap M$ let $\hat{x} : \omega \to \omega$ be defined by $\hat{x} (n) =k$ if $x \upharpoonright n = k$. Clearly if $x \in M$ the $\hat{x} \in M$ so there are infinitely many $n$ such that $\hat{x} (n) = y(n)$. But then, pulling back, let $y ' : \omega \to 2^{< \omega}$ be defined by $y ' (n) = s_k$ if $g(n) = k$ and $s_k \in 2^n$ and is trivial otherwise. Then we have that for every $x \in M \cap 2^\omega$ if $\hat{x} (n) = y(n)$ then $x \upharpoonright n = y ' (n)$ so the sequence $\langle \{y ' (n) \} \; | \; n < \omega \rangle$ witnesses that $2^\omega \cap M$ is measure $0$ by the Fact. 

From this it follows immediately that $\mathbb P$ does not add infinitely often equal reals since both $\mathbb B$ (\cite[Lemma 6.3.12]{BarJu95}) and $\mathbb{PT}$ (\cite[Theorem 7.3.47]{BarJu95}) preserve outer measure.

\noindent f) All nodes below $\mathcal D(\leq^*)$ are empty and $\mathcal B(\neq^*)$ is empty: This is accomplished by $\mathbb{PT}$.

\noindent g) All nodes below $\mathcal D(\neq^*)$ are empty and $\mathcal B(\neq^*)$ is empty: This is accomplished by $\mathbb C$.

\noindent h) All nodes below $\mathcal D(\in^*)$ except $\mathcal B(\neq^*)$ are empty: This is accomplished by $\mathbb B$.

\noindent i) All nodes below $\mathcal D(\in^*)$ are empty: This is the second cut where we still have something to prove. To achieve this one we force with the infinitely often equal forcing $\mathbb{EE}$ as defined in \cite[Definition 7.4.11]{BarJu95}. This forcing is $\omega^\omega$-bounding so it doesn't add reals to $\mathcal D(\leq^*)$, does not make the ground model reals meager (both of these facts are proved as part of \cite[Lemma 7.4.14]{BarJu95}) so it doesn't add reals to $\mathcal B(\neq^*)$ and generically adds a real which is infinitely often equal to all ground model elements of the product space $\Pi_{n < \omega} 2^n$. Let's see that $\mathbb{EE}$ adds a real to $\mathcal D(\in^*)$. Recall that this means there is a real which is not eventually captured by any ground model slalom. Let $y:\omega \to 2^{<\omega}$ be the infinitely often equal real added by the generic and fix an enumeration $\langle \tau_n \; | \; n < \omega \rangle$ (in $W$) of $2^{<\omega}$. Let $\hat{y}:\omega \to \omega$ be the function defined by $\hat{y}(n) = k$ if $y(n) = \tau_k$. I claim that this $\hat{y}$ is as needed. To see why, let $s \in W$ be a slalom. We can associate (in $W$) a function $x_s:\omega \to 2^{< \omega}$ by letting $x_s (n)$ be $\tau_k$ where $k$ is the least so that $k \notin s(n)$ and $\tau_k \in 2^n$. Note that such a $k$ exists since $|s (n)| = n$. Since $x_s \in W$ there are infinitely many $n$ so that $x_s (n) = y(n)$. Therefore there are infinitely many $n$ so that $\hat{y}(n) \notin s (n)$, as needed.

\noindent j) All nodes except $\omega^\omega \setminus (\omega^\omega)^W$ are empty: This is accomplished by $\mathbb S$.

\noindent k) All nodes are empty: This one is not pictured in Figure \ref{allcuts} since it is trivial. Let $\mathbb P$ be any forcing not adding reals, such as trivial forcing.

\end{proof}

To finish this section, let me observe one more analogue with the standard Cicho\'n diagram. Traditionally in the study of cardinal invariants of the continuum one sandwiches the nodes in Cicho\'n's diagram on one side by $\aleph_1$, the smallest possible value of any node, and on the other side by $2^{\aleph_0}$, the largest possible value of any node. One then views, for a given model $M$ of ZFC, the values of the other nodes on the diagram for $M$ as a measure of how much these two cardinals vary in $M$ with regards to substantive, mathematical applications. My diagram also naturally sandwiches itself between two invariants: the empty set, the smallest possible value of any node, and the entirety of the new reals, $\omega^\omega \setminus (\omega^\omega)^W$, the largest possible value of any node. As such, I view my diagram studied in this paper as measuring, similar to the case of the Cicho\'n diagram, the difference between the reals of the inner model $W$ and the reals of $V$. A natural question to ask, therefore, is how strong this \say{measurement} analogy is between these two diagrams. For example, in the generic extension of $W$ by more than $\aleph_1$ many Cohen reals, all nodes on the right side of the Cicho\'n diagram equal to $2^{\aleph_0}$ and all nodes on the left equal to $\aleph_1$, paralleling the situation I described for the model $W[c]$. However, in similar models studied for Hechler and eventually different forcing, the nodes in the Cicho\'n diagram still split into two cardinals, $\aleph_1$ and $\aleph_2$, whereas the diagram discussed in this paper automatically splits in three different sets of reals, as discussed. It appears that this may be necessary due to a result of Khomskii and Laguzzi in  stating that there is a canonical forcing in a certain sense to add infinitely-often-equal reals and this forcing does not add dominating reals, suggesting that perhaps there is no way that both $\mathcal B(\leq^*)$ and $\mathcal D(\neq^*)$ can be nonempty and equal.

\section{Achieving a Full Separation in the $\leq_W$-Cicho\'n Diagram and the axiom $\mathsf{CD}(\leq_W)$}
In this section building off the work done in the last section I build a model where there is complete separation between all elements in the diagram.
\begin{theorem}($\GBC$)
Given any transitive inner model $W$ of ZFC, there is a proper forcing notion $\mathbb P$, such that in $W^\mathbb P$ all the (non-$\emptyset$) nodes in the $\leq_W$-Cicho\'n diagram are distinct and every possible separation is simultaneously realized. 
\label{ConFS}
\end{theorem}

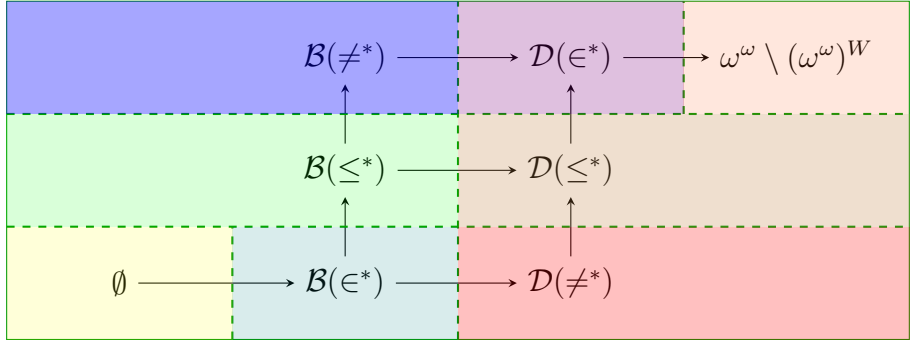
\begin{figure}[h]\label{Figure.Cichon-basic}
\centering
  \begin{tikzpicture}[scale=1.5,xscale=2]
     % place and draw the nodes
     \draw (0,0) node (empty) {$\emptyset$}
           (1,0) node (Bin*) {$\mathcal B(\in^*)$}
           (1,1) node (Bleq*) {$\mathcal B(\leq^*)$}
           (1,2) node (Bneq*) {$\mathcal B(\neq^*)$}
           (2,0) node (Dneq*) {$\mathcal D(\neq^*)$}
           (2,1) node (Dleq*) {$\mathcal D(\leq^*)$}
           (2,2) node (Din*) {$\mathcal D(\in^*)$}
           (3,2) node (all) {$\omega^\omega\setminus(\omega^\omega)^W$}
           ;
     % draw the arrows
     \draw[->,>=stealth]
            (empty) edge (Bin*)
            (Bin*) edge (Bleq*)
            (Bleq*) edge (Bneq*)
            (Bin*) edge (Dneq*)
            (Bleq*) edge (Dleq*)
            (Bneq*) edge (Din*)
            (Dneq*) edge (Dleq*)
            (Dleq*) edge (Din*)
            (Din*) edge (all)
            ;
    % draw separating lines
       \draw[thick,dashed,OliveGreen] (-.5,.5) -- (3.5,.5);
       \draw[thick,dashed,OliveGreen] (1.5,-.5) -- (1.5,2.5);
       \draw[thick,dashed,OliveGreen] (-.5,1.5) -- (3.5,1.5);
       \draw[thick,dashed,OliveGreen] (.5,-.5) -- (.5,.5);
       \draw[thick,dashed,OliveGreen] (2.5,1.5) -- (2.5,2.5);
    % draw bounding rectangle
       \draw[OliveGreen] (-.5,-.5) rectangle (3.5,2.5);
    % draw fill
       \draw[draw=none,fill=yellow, fill opacity=.15] (-.5,-.5) rectangle (.5,.5);
       \draw[draw=none,fill=teal,fill opacity=.15] (.5,-.5) rectangle (1.5,.5);
       \draw[draw=none,fill=red,fill opacity=.25] (1.5,-.5) rectangle (3.5,.5);
       \draw[draw=none,fill=green,fill opacity=.15] (-.5,.5) rectangle (1.5,1.5);
       \draw[draw=none,fill=brown,fill opacity=.25] (1.5,.5) rectangle (3.5,1.5);
       \draw[draw=none,fill=blue,fill opacity=.35] (-.5,1.5) rectangle (1.5,2.5);
       \draw[draw=none,fill=violet,fill opacity=.25] (1.5,1.5) rectangle (2.5,2.5);
       \draw[draw=none,fill=Orange,fill opacity=.15] (2.5,1.5) rectangle (3.5,2.5);
  \end{tikzpicture}
  \caption{Full Separation of the $\leq_W$-diagram}
\end{figure}

\noindent In what follows I call the axiom \say{All consistent separations of the $\leq_W$-diagram are distinct} $\mathsf{CD}(\leq_W)$ or \say{full Cicho\'n Diagram for $\leq_W$}. Thus the above theorem states that $\mathsf{CD}(\leq_W)$ can be forced over $W$ by a proper forcing. For different inner models $W$ the sentence $\mathsf{CD}(\leq_W)$ may vary but they can all be forced the same way.

Before proving this theorem I need a simple technical result about Sacks and Laver forcing.

\begin{lemma}
The product forcing $\mathbb S \times \mathbb L$ satisfies Axiom A and hence is proper.
\label{sackslaver}
\end{lemma}

\begin{proof}
Theorem 1 of \cite{Gro87} gives a general framework for showing that certain arboreal forcing notions satisfy Axiom A (including Sacks and Laver forcings) and here I adapt the proof to the case of a product of two arboreal forcing notions. Recall that if $p, q \in \mathbb S$ and $n \in \omega$ then we let $q \leq^\mathbb S_n p$ if and only if $q \subseteq p$ and every $n^{\rm th}$ splitting node of $q$ is an $n^{\rm th}$ splitting node of $p$ i.e. if $\tau \in q$ is a splitting node with $n$ splitting predecessors in $q$ then the same is true of $\tau$ in $p$. Also, given a canonical enumeration of $\omega^{<\omega}$ in which $s$ appears before $\tau$ if $s \subseteq \tau$ and $s^\frown k$ appears before $s^\frown(k+1)$ then for $p \in \mathbb L$ one gets an enumeration of the elements of $p$ above the stem, $s_1^p,...,s_k^p,...$ and if $p, q \in \mathbb L$ and $n \in \omega$ then let $q \leq^\mathbb L_n p$ if and only if $q \subseteq p$ and $s_i^p = s_i^q$ for all $i =0,...,n$. Clearly if for every $n \in \omega$ and $(p_s, p_l) , (q_s, q_l) \in \mathbb S \times \mathbb L$ we let $(q_s, q_l) \leq_n (p_s, p_l)$ if and only if $q_s \leq_n^\mathbb S p_s$ and $q_l \leq_n^\mathbb L s_l$ then this satisfies the first requirement of Axiom A forcings. Thus, it remains to show that for every $\mathbb S \times \mathbb L$-name $\dot{a}$ and condition $(p_s, p_l) \in \mathbb S \times \mathbb L$ if $(p_s, p_l) \Vdash \dot{a} \in \check{V}$ then for every $n$ there is a $(q_s, q_l)$ and a countable set $A \in V$ such that $(q_s, q_l) \Vdash \dot{a} \in A$.

Fix such a name $\dot{a}$ and condition $p = (p_s, p_l)$. Let $D \subseteq \mathbb S \times \mathbb L$ be the set of all $(q_s, q_l) \leq p$ such that there is some $a(q) \in V$ with $(q_s, q_l) \Vdash \check{a(q)} = \dot{a}$. This set is dense below $p$ since $p$ forces $\dot{a}$ to be an element of $V$. Let $H_D \subseteq p$ be the set of all pairs $(s, \tau) \in p$ such that there is a $(s ', \tau') \subseteq (s, \tau)$ with $s '$ $n$-splitting in $p_s$ and $\tau'$ $n$-splitting in $p_l$ and there is some $r_{s, \tau} = (r_s, r_l) \leq p$ in $D$ whose stem (i.e. the pair of the stems from the two components) is $(s ,  \tau )$. Finally let $Min(H_D)$ be the set of $(s, \tau) \in H_D$ which are minimal with respect to inclusion. Note that $Min(H_D)$ is an antichain since no two elements can be comparable and both minimal. Let $r = (r_s, r_l) = \bigcup \{r_{s, \tau} \; | \; (s, \tau) \in Min(H_D)\}$. A routine check shows that the set $r$ is a condition in $\mathbb S \times \mathbb L$ and $r \leq_n p$.

Now let $A = \{a(r_{\tau, s}) \; | \; (s, \tau) \in Min(H_D)\}$. This set is countable thus to finish the lemma it suffices to show that $r \Vdash \dot{a} \in \check{A}$. To see this, suppose that $t \leq r$ and $t \Vdash \dot{a} = \check{a}$ for some $a$. By extending $t$ if necessary one may assume that the stem of $t$ is in $H_D$. But then some initial segment of the stem is in $Min(H_D)$ so $a \in A$, as needed.
\end{proof}

Now I prove Theorem \ref{ConFS}.
\begin{proof}[Proof of Theorem \ref{ConFS}]
This essentially follows from the theorems of the previous section. Given a definable forcing notion $\mathbb Q$ let me write $\mathbb Q^W$ for the version of that forcing notion as computed in $W$. Let $\mathbb P = \mathbb S^W \times \mathbb L^W \times \mathbb{LOC}^W$. Then in $W^\mathbb P$ not every new real is in an element of the diagram since Sacks reals were added. Moreover, by our arguments above the combination of $\mathbb{LOC}$ and $\mathbb{L}$ will add reals to every node of the diagram but, none of them will be equal and moreover every possible non-separation is realized as one observes by my previous arguments. 

It remains to see that $\mathbb P$ is proper. This follows from Lemma \ref{sackslaver} plus the fact that $\mathbb{LOC}$ is $\sigma$-linked and hence indestructibly ccc.
\end{proof}

Let me finish this paper by briefly studying the axiom $\mathsf{CD}(\leq_W)$. First, let me show that there are other ways to obtain it. Indeed there is another, less finegrained approach to forcing $\mathsf{CD}(\leq_W)$. To describe this, let me make the following simple observation. Recall that the {\em Maximality Principle} $\mathsf{MP}$ of \cite{Hamkins03} states that any statement which is forceably necessary or can be forced to be true in such as a way that it cannot become later forced to be false, is already true. If $\Gamma$ is a class of forcing notions then the maximality principle for $\Gamma$, $\mathsf{MP}_\Gamma$, states the same but only with respect to forcing notions in $\Gamma$.
\begin{proposition}
The axiom $\mathsf{CD}(\leq_W)$ is forceably necessary, that is once it has been forced to be true it will remain so in any further forcing extension. Thus in particular it is implied by the maximality principle, $\mathsf{MP}$.
\label{FSFN}
\end{proposition}

\begin{proof}
This is more or less immediate from the definition. Since $\mathsf{CD}(\leq_W)$ is defined relative to a fixed inner model and the diagram for $W$ concerns only the models $W[x]$ for $x \in \omega^\omega \cap V$, notice that forcing over $V$ cannot change the theories of the models $W[x]$ for $x \in V$ hence if $\mathsf{CD}(\leq_W)$ is true in $V$ it must remain so in any forcing extension. In other words absoluteness for membership in each of the various classes holds and this guarentees that forcing cannot change the relation $x \in A$ for any node $A$ of the diagram.

Since $\mathsf{CD}(\leq_W)$ is forceably necessary it follows that $MP$ implies $\mathsf{CD}(\leq_W)$.
\end{proof}

Now notice that since all the forcing notions used in Theorem \ref{ConFS} have size at most $2^{\aleph_0}$ it follows that the collapse forcing $Coll(\omega, < (2^{2^{\aleph_0}})^+)$ will add a generic making $\mathsf{CD}(\leq_W)$ true. Since $\mathsf{CD}(\leq_W)$ is forceably necessary it follows that the full collapse forcing cannot kill the generic once it is added and, as a result one obtains
\begin{corollary}
$W^{Coll(\omega, < (2^{2^{\aleph_0}})^+)} \models CD(\leq_W)$
\end{corollary}
Moreover, note that while the forcing described in Theorem \ref{ConFS} was proper and hence preserved $\omega_1$ the collapse forcing used above is not. Therefore the following is immediate.
\begin{corollary}
The statement \say{the reals of $W$ are countable} is independent of the theory ZFC($W$) + $\mathsf{CD}(\leq_W)$. Consequently $CD(\leq_W)$ does not imply $\mathsf{MP}$ for any sufficiently definable $W$.
\end{corollary}

Since $\mathsf{CD}(\leq_W)$ is forceably necessary and hence cannot be killed once it is forced to be true it follows that any sentence which can be forced to be true from any model must be consistent with $\mathsf{CD}(\leq_W)$. Such examples include $CH$, $2^{\aleph_0} = \kappa$ for any $\kappa$ of uncountable cofinality, Martin's Axiom and its negation, $\diamondsuit$ and its negation, and a wide variety of forcing notions associated with the classical Cicho\'n's diagram. In particular, $\mathsf{CD}(\leq_W)$ is independent of any consistent assignment of cardinals to the nodes in the Cicho\'n diagram (cf \cite{BarJu95} for a variety of examples of such).

Let me finish now by showing the consistency of a strong version of $\mathsf{CD}(\leq_W)$, which was suggested to me by Gunter Fuchs. The idea is to iteratively force with the forcing $\mathbb P$ of Theorem \ref{ConFS} for long enough that a large collection of inner models $W$ simultaneously satisfy $\mathsf{CD}(\leq_W)$.
\begin{theorem}
Assume $V=L$. Then there is an $\aleph_2$-c.c. proper forcing extension in which $2^{\aleph_0} = \aleph_2$ and for every $\aleph_1$-sized set of reals $A$ there is an $\aleph_1$-sized set of reals $B \supseteq A$ so that $\mathsf{CD}(\leq_{L[B]})$ holds.
\label{ConIFS}
\end{theorem}

\begin{proof}
Assume $V=L$ and let $\vec{\mathbb P} = \langle (\mathbb P_\alpha, \dot{\mathbb Q}_\alpha) \;  | \; \alpha < \omega_2 \rangle$ be an $\omega_2$-length countable support iteration of copies of the forcing $\mathbb P$ from Theorem \ref{ConFS} (i.e. $\dot{\mathbb Q}_{\alpha +1}$  evaluates to $( \mathbb P)^{L^{\mathbb P_\alpha}}$). Clearly $\vec{\mathbb P}$ is proper. Moreover, since CH holds in the ground model and the forcing $\mathbb P$ is easily seen to be of size continuum, and does not kill CH it follows that $\vec{\mathbb P}$ has the $\aleph_2$-c.c. and every intermediate stage in the iteration preserves CH: $L^{\mathbb P_\alpha} \models {\rm CH}$ for all $\alpha < \omega_2$. However, since reals are added at every stage the final model satisfies $2^{\aleph_0} = \aleph_2$.

It remains to show that for every $\aleph_1$-sized set of reals $A$ there is a set of reals $B \supseteq A$ of size $\aleph_1$ so that $\mathsf{CD}(\leq_W)$ holds for $W = L[B]$. Let $A$ be a set of reals of size $\aleph_1$. Then, there is some $\alpha$ so that $A \in L[G_\alpha]$ for $G_\alpha$ be $\mathbb P_\alpha$-generic. Note that we can code $G_\alpha$ by a set of reals of size at most $\aleph_1$, say $B$, and without loss we can assume that $A \subseteq B$ for $L[G_\alpha] = L[B]$. Then at stage $\mathbb P_{\alpha+1}$ we added a generic witnessing that $CD(\leq_{L[B]})$ holds. Moreover, by the fact that this statement is forceably necessary, it cannot be killed by the tail end of the iteration so it holds in the final model.
\end{proof}

While it is not entirely clear what consequences we can expect from $\mathsf{CD}(\leq_W)$ for an arbitrary $W$, the stronger version obtained in Theorem \ref{ConIFS} has several low hanging fruits in this regard. Let me pluck a particularly simple one connecting the constructibility diagram to the standard Cicho\'n diagram. 
\begin{lemma}
Assume for every $\aleph_1$-sized set of reals $A$ there is an $\aleph_1$-sized set of reals $B \supseteq A$ so that $\mathsf{CD}(\leq_{L[B]})$ holds. Then all the cardinals in the Cicho\'n diagram have size at least $\aleph_2$.
\end{lemma}

\begin{proof}
It suffices to show that ${\rm add}(\mathcal N) \geq \aleph_2$. Towards this goal, recall Bartoszy\'nski's characterization of ${\rm add}(\mathcal N)$ as the least cardinal $\kappa$ so that there is a set of reals $X$ of size $\kappa$ so that no single slalom can capture all the reals in $X$ (Fact \ref{infact}, part 1). The result is then immediate for, given any set of reals $A$ of size $\aleph_1$, we can find a slalom $s$ eventually capturing all reals in $L[B]$ for some $B \supseteq A$ so that $CD(\leq_{L[B]})$ holds so ${\rm add}(\mathcal N) > \aleph_1$.
\end{proof}

%\begin{theorem}
%Assume $\mathsf{CD}(\leq_W)$ where $V$ is a forcing extension of $W$. Then there are Cohen and Hechler reals generic over $W$.
%\label{mainthm3}
%\end{theorem}

%\begin{proof}
%In unpublished work Goldstern has shown that if $\mathbb P$ and $\mathbb Q$ are forcing notions each adding a dominating real then $\mathbb P \times \mathbb Q$ adds a Cohen real [cf \cite[Theorem 7.1]{Pal13}]. Palumbo subsequently strengthened this to say that under the same assumptions $\mathbb P \times \mathbb Q$ add a Hechler real \cite[Theorem 7.2]{Pal13}. Now by $\mathsf{CD}(\leq_W)$ one can find $x \in \mathcal B(\leq^*) \setminus \mathcal B(\in^*)$ and $y \in \mathcal B(\in^*)$ which are both dominating over $W$ and, since $V$ is a forcing extension of $W$, by the intermediate model theorem it follows that $x$ and $y$ were added by forcing. Thus by Palumbo's result there is a Hechler and hence also a Cohen generic over $W$ in $W[x][y]$ and hence in $V$. 
%\end{proof}

\section{Open Questions and Further Work}
I finish this chapter by collecting the open questions that have appeared. First I ask about the Cicho\'n diagram for other reduction concepts. Recall that in the case of $\leq_T$, the sets $\mathcal B(\in^*)$ and $\mathcal B(\leq^*)$ were equal.
\begin{question}
For which reductions $(\sqsubseteq, x_0)$ on the reals is $\mathcal B_\sqsubseteq(\in^*) \subsetneq \mathcal B_\sqsubseteq(\leq^*)$?
\end{question}
The anonymous referee of \cite{Switz18} has pointed out to me that Monin (unpublished) has shown that for $\sqsubseteq$ equal to hyperarithmetic reduction, the equation $\mathcal B_\sqsubseteq(\in^*) = \mathcal B_\sqsubseteq(\leq^*)$ holds (hence it does for $\leq_A$ as well). See \cite[Fact 2.6]{Kihara17}. This question has also been considered in \cite{Kihara17}, see Problem 5.7 and the discussion preceding it. This shows that for many natural reduction concepts the answer to the question above is negative. Taking this into account it seems reasonable to ask if indeed {\em any} ``reasonable" reduction concept (whatever that means) provably does this in $\mathsf{ZFC}$? Note that if $V=L$ all $\leq_W$ relations are trivial.

Next I ask about the ZFC($W$)-provable relations between the nodes of the $\leq_W$-Cicho\'n diagram. While I have shown that there are no other implications it is entirely possible that there are other relations more generally.
\begin{question}
What other ZFC($W$)-provable relations are there between the sets in the $\leq_W$-diagram?
\end{question}

My next collection of questions concerns the subforcings of $\mathbb{LOC}$, a topic that deserves more study.
\begin{question}
What is the forcing adding the eventually different real described in Lemma \ref{evdiffloc}? Does it add a dominating real? Note that it must be ccc, in fact $\sigma$-linked and add eventually different reals which are bounded by nearly all ground model reals.
\end{question}

Similarly, one might ask whether there is a similarly exotic subforcing of $\mathbb{LOC}$ for adding a dominating real.
\begin{question}
Does every subforcing of $\mathbb{LOC}$ adding a dominating real add a $\mathbb D$-generic real?
\end{question}
\begin{question}
Does every subforcing of $\mathbb{LOC}$ add a Cohen real or a random real?
\end{question}

There are also several open questions about the axiom $CD(\leq_W)$.
\begin{question}
What statements are implied by $CD(\leq_W)$? In particular, does it imply that there are $W$-generics for the forcings to add reals we have discussed (Cohen, random, etc)?
\end{question}
\begin{question}
How does $CD(\leq_W)$ relate to standard forcing axioms? In particular does ${\rm MA}_{\aleph_1}$ imply $CD(\leq_{L[A]})$ for all $\aleph_1$-sized sets of reals $A$? Does BPFA?
\end{question}

Finally let me report on some ongoing attempts to generalize the work in this chapter further. After the results of this chapter were announced, J. Brendle pointed out to me that the setup described above can accommodate nodes explicitly corresponding to the cardinal characteristics for measure and category. Recall that for each real $y$, there is a canonical Borel null (meager) set $N_y$ ($M_y$) coded by $y$, see Lemmas 3.2 and 3.4 of \cite{BarHB}. For reals $x$ and $y$, let $x \in_\Null y$ if and only if $x \in N_y$ and $x \subseteq_\Null y$ if and only if $N_x \subseteq N_y$ and the same for $\Me$. These relations give the following bounding and non-dominating sets.
\begin{enumerate}
\item 
$\mathcal B (\subseteq_{\mathcal N})$ is the set of all reals $x$ so that there is a $y \leq_W x$ such that for each $z \in W$ we have that $z \subseteq_\mathcal N y$. In other words, in $W[x]$ the union of all of the null sets coded in $W$ is null.

\item
$\mathcal D (\subseteq_{\mathcal N})$ is the set of all reals $x$ so that there is a $y \leq_W x$ such that for each $z \in W$ we have that $y \nsubseteq_\mathcal N Z$. In other words, there is a null set coded in $W[x]$ which is not a subset of any null set coded in $W$. 

\item
$\mathcal B (\in_{\Null})$ is the set of all reals $x$ so that there is a $y \leq_W x$ so that for each $z \in W$ we have that $z \in_{\mathcal N} y$. In other words, in $W[x]$ the reals of $W$ are measure zero.

\item
$\mathcal D (\in_{\mathcal N})$ is the set of all reals $x$ so that there is a $y \leq_W x$ so that for each $z \in W$ $y \notin_{\mathcal N} z$. In other words in $W[x]$ there is a real $y$ not in any measure zero set coded in $W$. Note that $x \in \mathcal D (\in_{\mathcal N})$ if and only if in $W[x]$ there is a real $y$ which is random over $W$ i.e. $y$ is $\mathbb B^W$ generic. In the case of $\mathcal M$, the corresponding statement is the same with \say{random} replaced by \say{Cohen}.
\end{enumerate}

Working through the definitions and using the discussion from this chapter, it's straightforward to see that the following analogies hold:

\begin{enumerate}
\item
$\mathcal B (\subseteq_{\mathcal N})$ ($\mathcal B (\subseteq_{\mathcal M})$) corresponds to $add(\Null)$ ($add(\Me)$)

\item
$\mathcal D (\subseteq_{\mathcal N})$ ($\mathcal D (\subseteq_{\mathcal M})$) corresponds to $cof(\Null)$ ($cof(\Me)$)

\item
$\mathcal B (\in_{\mathcal N})$ ($\mathcal B (\in_{\mathcal M})$) corresponds to $non(\Null)$ ($non(\Me)$)

\item
$\mathcal D (\in_{\mathcal N})$ ($\mathcal D (\in_{\mathcal M})$) corresponds to $cov(\Null)$ ($cov(\Me)$)

\end{enumerate}

Moreover, using the theory of small sets worked out in \cite{PawRec95} it's not hard to show that the analogous implications hold as well. Integrating these nodes with those from the diagram for the combinatorial nodes gives one in which all of the equivalences hold from Facts \ref{infact} and \ref{neqfact} with the exception of $\mathcal D(\neq^*)$ since as previously discussed Zapletal has shown this one to be false. Thus we get the following diagram with 11 non-trivial nodes. 

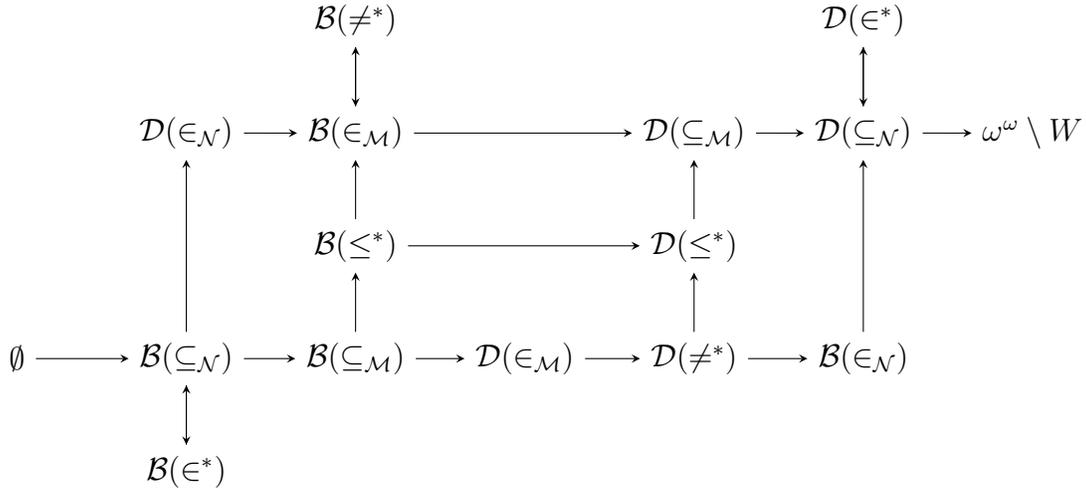
\begin{figure}[h]\label{Figure.Cichon-basic}
\centering
  \begin{tikzpicture}[scale=1.5,xscale=2]
     % place and draw the nodes
     \draw (.75, 1) node (BsubN) {$\mathcal B(\subseteq_\mathcal N)$}
	     (1.5,1) node (Bsub) {$\mathcal B(\subseteq_\mathcal M)$}
           (1.5,3) node (Bin) {$\mathcal B(\in_\mathcal M)$}
           (2.25,1) node (Din) {$\mathcal D(\in_\mathcal M)$}
           (3,3) node (Dsub) {$\mathcal D(\subseteq_\mathcal M)$}
           (1.5,2) node (Bleq) {$\mathcal B(\leq^*)$}
           (3,2) node (Dleq) {$\mathcal D(\leq^*)$}
	     (.75, 3) node (DinN) {$\mathcal D(\in_\mathcal N)$}
           (3.75, 1) node (BinN)  {$\mathcal B(\in_\mathcal N)$}
           (3.75, 3) node (DsubN) {$\mathcal D(\subseteq_\mathcal N)$}
	    (.75, 0) node (Bslal) {$\mathcal B (\in^*)$}
	    (3.75, 4) node (Dslal) {$\mathcal D (\in^*)$}
	    (0, 1) node (empty) {$\emptyset$}
          (1.5, 4)  node (Bneq) {$\mathcal B (\neq^*)$}
          (4.5, 3) node (all) {$\omega^\omega \setminus W$}
          (3, 1) node (Dneq) {$\mathcal D(\neq^*)$}
	    
           ;
     % draw the arrows
     \draw[->,>=stealth]
	      (BsubN) edge (Bsub)
            (Bsub) edge (Bleq)
            (Bleq) edge (Bin)
            (Bsub) edge (Din)
            (Bin) edge (Dsub)
            %(Din) edge (Dleq)
            (Dleq) edge (Dsub)
            (Bleq) edge (Dleq)
	     (BsubN) edge (DinN)
            (DinN) edge (Bin)
	     %(Din) edge (BinN)
            (BinN) edge (DsubN)
           (Dsub) edge (DsubN)
	    (empty) edge (BsubN)
           (Bslal) edge (BsubN)
           (BsubN) edge (Bslal)
           (DsubN) edge (Dslal)
           (Dslal) edge (DsubN)
           (Bin) edge (Bneq)
           (Bneq) edge (Bin)
           (DsubN) edge (all)
          (Din) edge (Dneq)
           (Dneq) edge (BinN)
          (Dneq) edge (Dleq)
            
            ;
  \end{tikzpicture}
\caption{Integrating the Combinatorial Nodes and the Nodes for Measure and Category}
\end{figure}

The following question is still open and seems related to several open problems concerning the forcing Zapletal introduces in \cite{dimtheoryandforcing}.

\begin{question}
Can the larger diagram including the nodes for measure and category be fully separated by proper (or at least $\omega_1$-preserving) forcing in the same way that the diagram with only the combinatorial nodes was in this chapter?
\end{question}
\clearpage

\chapter{Cardinal Characteristics for Sets of Functions}
As we saw in the last chapter, many cardinal characteristics on $\omega^\omega$ arise as follows: fix some relation $R \subseteq \omega \times \omega$ and let $R^* \subseteq \omega^\omega \times \omega^\omega$ be defined by $x\mathrel{R}^*y$ if and only if for all but finitely many $n$ we have $x(n) \mathrel{R} y(n)$. For instance, letting $R$ be the the usual order on $\omega$ gives the eventual domination ordering. Each such $R$ then gives rise to two cardinal characteristics, $\mathfrak{b}(R^*)$, the least size of a set $A \subseteq \omega^\omega$ with no $R^*$-bound and $\mathfrak{d}(R^*)$ the least size of a set $D \subseteq \omega^\omega$ which is $R^*$-dominating. A natural generalization of this is as follows: fix two sets $X$ and $Y$, let $\mathcal I$ be an ideal on $X$ and $R \subseteq Y \times Y$ be a binary relation on $Y$. Let $Y^X$ be the set of functions $f:X \to Y$ and consider the relation $R_\mathcal I \subseteq Y^X \times Y^X$ given by $f \mathrel{R_\mathcal I} g$ if and only if for $\mathcal I$-almost all $x$ we have $f(x) \mathrel{R} g(x)$ i.e. $\{x \in X \; | \; \neg f(x) \mathrel{R} g(x)\} \in \mathcal I$. Again we get two cardinal characteristics, this time on the set $Y^X$: $\mathfrak{b}(R_\mathcal I)$, the least size of a set $A \subseteq Y^X$ which has no $R_\mathcal I$-bound and $\mathfrak{d}(R_\mathcal I)$, the least size of a set $D \subseteq Y^X$ which is dominating with respect to $R_\mathcal I$. Note that letting $X=Y = \omega$ and $\mathcal I$ be the ideal of finite sets we recover the original setting for cardinal characteristics on Baire space and letting $Y = 2$ we recover the same for Cantor space.

Recently, much work has been done on the case of $X = \kappa$ and $Y = \kappa$ or $2$ for arbitrary $\kappa$, thus generalizing cardinal characteristics to larger cardinals, see for example the article \cite{brendle16} or the survey \cite{questionsbaire} for a list of open questions. In this case the interesting ideals are the ideal of sets of size $< \kappa$, the non-stationary ideal, and, if $\kappa$ has a large cardinal property, then potentially some ideal related to this. See \cite{CuSh95}, Theorems 6 and 8 for a particularly striking result relating cardinal invariants modulo different ideals. 

However, this framework is more flexible than just allowing one to study generalized Baire space and Cantor space. Indeed it is easy to imagine numerous new cardinal characteristics. In this chapter I consider a different generalization, based on the function space $(\omega^\omega)^{\omega^\omega}$ of functions $f:\omega^\omega \to \omega^\omega$. Since Baire space comes with ideals that are not easily defined on $\kappa^\kappa$ we get further generalizations of cardinal characteristics. Specifically I will consider the ideals $\mathcal N$, $\mathcal M$ and $\mathcal K$. The result is a \say{higher dimensional} version of several well-known cardinal characteristics. While many different generalizations are possible let me stick with the three relations we have already seen for simplicity: $\leq^*$, $\neq^*$, and $\in^*$. By considering two cardinals for each of these three relations and three ideals I end up with 18 new cardinals characteristics above the continuum. The first main theorem of this chapter is to show that these \say{higher dimensional} cardinals behave, provably under $\ZFC$, similar to their Baire space analogues (the cardinals mentioned below will be defined in detail in the next section).

\begin{theorem}
Interpreting $\to$ as $\leq$ the inequalities shown in Figures \ref{nulldia} and \ref{media} are all provable in $\ZFC$.

\begin{figure}[h]
\centering
  \begin{tikzpicture}[scale=1.5,xscale=2]
     % place and draw the nodes
     \draw %(0,0) node (empty) {$\emptyset$}
           (1,0) node (Bin*) {$\mfb(\in_\Null^*)$}
           (1,1) node (Bleq*) {$\mfb(\leq_\Null^*)$}
           (1,2) node (Bneq*) {$\mfb(\neq_\Null^*)$}
           (2,0) node (Dneq*) {$\mfd(\neq_\Null^*)$}
           (2,1) node (Dleq*) {$\mfd(\leq_\Null^*)$}
           (2,2) node (Din*) {$\mfd(\in_\Null^*)$}
           %(3,2) node (all) {$\omega^\omega\setminus(\omega^\omega)^W$}
           ;
     % draw the arrows
     \draw[->,>=stealth]
            %(empty) edge (Bin*)
            (Bin*) edge (Bleq*)
            (Bleq*) edge (Bneq*)
            (Bin*) edge (Dneq*)
            (Bleq*) edge (Dleq*)
            (Bneq*) edge (Din*)
            (Dneq*) edge (Dleq*)
            (Dleq*) edge (Din*)
            %(Din*) edge (all)
;      
  \end{tikzpicture}
\caption{Higher Dimensional Cardinal Characteristics Mod the Null Ideal}
\label{nulldia}
\end{figure}

\begin{figure}[h]
\centering
  \begin{tikzpicture}[scale=1.5,xscale=2]
     % place and draw the nodes
     \draw %(0,0) node (empty) {$\emptyset$}
           (1,0) node (Bin*) {$\mfb(\in_\Me^*)$}
           (1,1) node (Bleq*) {$\mfb(\leq_\Me^*)$}
           (1,2) node (Bneq*) {$\mfb(\neq_\Me^*)$}
           (2,0) node (Dneq*) {$\mfd(\neq_\Me^*)$}
           (2,1) node (Dleq*) {$\mfd(\leq_\Me^*)$}
           (2,2) node (Din*) {$\mfd(\in_\Me^*)$}
	     (0,0) node (BinK*) {$\mfb(\in_\Kb^*)$}
           (0,1) node (BleqK*) {$\mfb(\leq_\Kb^*)$}
           (0,2) node (BneqK*) {$\mfb(\neq_\Kb^*)$}
           (3,0) node (DneqK*) {$\mfd(\neq_\Kb^*)$}
           (3,1) node (DleqK*) {$\mfd(\leq_\Kb^*)$}
           (3,2) node (DinK*) {$\mfd(\in_\Kb^*)$}
           %(3,2) node (all) {$\omega^\omega\setminus(\omega^\omega)^W$}
           ;
     % draw the arrows
     \draw[->,>=stealth]
            %(empty) edge (Bin*)
            (Bin*) edge (Bleq*)
            (Bleq*) edge (Bneq*)
            (Bin*) edge (Dneq*)
            (Bleq*) edge (Dleq*)
            (Bneq*) edge (Din*)
            (Dneq*) edge (Dleq*)
            (Dleq*) edge (Din*)
            (Dleq*) edge (DleqK*)
            (Dneq*) edge (DneqK*)
            (Din*) edge (DinK*)
            (BleqK*) edge (Bleq*)
            (BneqK*) edge (Bneq*)
            (BinK*) edge (Bin*)
            (BinK*) edge (BleqK*)
            (BleqK*) edge (BneqK*)
            (DneqK*) edge (DleqK*)
	      (DleqK*) edge (DinK*)

            %(Din*) edge (all)
;
 \end{tikzpicture}
\caption{Higher Dimensional Cardinal Characteristics Mod the Meager and $\sigma$-Compact Ideals}
\label{media}
\end{figure}

\end{theorem}

\pagebreak

The rest of this chapter is organized as follows. In the next section I introduce the cardinals $\mathfrak{b}(R_\mathcal I)$ and $\mathfrak{d}(R_\mathcal I)$ and basic relations between them are shown. The second section investigates the relation between these higher dimensional cardinal characteristics and the standard cardinal characteristics on $\omega$. Section 3 contains a number of consistency results and introduces three new forcing notions based on generalizations of Cohen, Hechler, and localization forcing. In section 4 I list a number of open questions, as well as some extensions.

\section{Higher Dimensional Variants of A Fragment of Cicho\'n's Diagram}

In this section I define the cardinals that will be studied for the rest of the chapter. Recall that I write $\mathcal S$ for the space of slaloms.

\begin{definition}
Let $\mathcal I \in \{\mathcal N, \mathcal M, \mathcal K\}$ and $R \in \{\leq^*, \neq^*, \in^*\}$.
\begin{enumerate}
\item
$\mathfrak{b}(R_\mathcal I)$ is the least size of a set $A$ of functions from $\omega^\omega$ to $\omega^\omega$ for which there is no $g:\omega^\omega \to \omega^\omega$ ($g:\omega^\omega \to \mathcal S$ in the case of $R = \in^*$) such that for all $f \in A$ the set $\{x \in \omega^\omega \; | \; \neg f(x) \mathrel{R} g(x)\}$ is in $\mathcal I$.
\item
$\mathfrak{d}(R_\mathcal I)$ is the least size of a set $A$ of functions from $\omega^\omega$ to $\omega^\omega$ ($\baire$ to $\mathcal S$ in the case of $R = \in^*$) so that for all $g:\omega^\omega \to \omega^\omega$ there is an $f \in A$ for which the set $\{x \in \omega^\omega \; | \; \neg g(x) \mathrel{R} f(x)\}$ is in $\mathcal I$.
\end{enumerate}
\end{definition}

By varying $\mathcal I$ and $R$ this definition gives 18 new cardinals. For readability, let me give the details below for the case of the null ideal. Similar statements hold for $\Me$ and $\Kb$. First let's see explicitly what each relation $R_\mathcal I$ is. On the two lists below let $f, g:\baire \to \baire$ and $h:\baire \to \mathcal S$.

\begin{enumerate}
\item
$f \neq^*_\Null g$ if and only if for all but a measure zero set of $x \in \baire$ we have that $f (x) \neq^* g(x)$.

\item
$f \leq^*_\Null g$ if and only if for all but a measure zero set of $x \in \baire$ we have that $f(x) \leq^* g(x)$.

\item
$f \in^*_\Null h$ if and only if for all but a measure zero set of $x \in \baire$ we have that $f(x) \in^* h(x)$.
\end{enumerate}

For the cardinals now we get the following.

\begin{enumerate}

\item
$\mfb(\neq^*_\Null)$ is the least size of a $\neq^*_\Null$-unbounded set $A \subseteq \bbb$ i.e. $A$ is such that for each $f:\baire \to \baire$ there is a $g \in A$ so that the set of $\{x \; | \; \exists^\infty n \, g(x)(n) = f(x)(n)\}$ is not measure zero.

\item
$\mfd(\neq^*_\Null)$ is the least size of a $\neq^*_\Null$-dominating set $A \subseteq \bbb$ i.e. $A$ is such that for every $f:\baire \to \baire$ there is a $g \in A$ so that $\mu (\{x \; | \; f(x)\neq^* g(x)\} )= 1$.

\item
$\mfb(\leq^*_\Null)$ is the least size of a $\leq^*_\Null$-unbounded set $A \subseteq \bbb$ i.e. $A$ is such that for each $f:\baire \to \baire$ there is a $g \in A$ so that the set of $\{x \; | \; \exists^\infty n \, f(x)(n) < g(x)(n)\}$ is not measure zero.

\item
$\mfd(\leq^*_\Null)$ is the least size of a $\leq^*_\Null$-dominating set $A \subseteq \bbb$ i.e. $A$ is such that for every $f:\baire \to \baire$ there is a $g \in A$ so that $\mu (\{x \; | \; f(x)\leq^* g(x)\}) = 1$.

\item
$\mfb(\in^*_\Null)$ is the least size of a $\in^*_\Null$-unbounded set $A \subseteq \bbb$ i.e. $A$ is such that for each $f:\baire \to \mathcal S$ there is a $g \in A$ so that the set of $\{x \; | \; \exists^\infty n \, g(x)(n) \notin f(x)(n)\}$ is not measure zero.

\item
$\mfd(\in^*_\Null)$ is the least size of a $\leq^*_\Null$-dominating set $A \subseteq \bbb$ i.e. $A$ is such that for every $f:\baire \to \baire$ there is a $g \in A$ so that $\mu (\{x \; | \; f(x)\in^* g(x)\} )= 1$.

\end{enumerate}

The first goal is to prove the following theorem, which shows that for each ideal the six associated cardinals fit together as in the case of the corresponding fragment of Cicho\'n's diagram on $\omega$, note not all cardinals in the $\omega$ case have analogues here.

\begin{theorem}[The Higher Dimensional Cicho\'n diagram]
For an ideal $\mathcal I \in \{\Null, \Me, \Kb\}$ and interpreting $\to$ as \say{is $\ZFC$-provably less than or equal to} the following all hold:

\begin{figure}[h]
\centering
  \begin{tikzpicture}[scale=1.5,xscale=2]
     % place and draw the nodes
     \draw %(0,0) node (empty) {$\emptyset$}
           (1,0) node (Bin*) {$\mfb(\in_\mathcal I^*)$}
           (1,1) node (Bleq*) {$\mfb(\leq_\mathcal I^*)$}
           (1,2) node (Bneq*) {$\mfb(\neq_\mathcal I^*)$}
           (2,0) node (Dneq*) {$\mfd(\neq_\mathcal I^*)$}
           (2,1) node (Dleq*) {$\mfd(\leq_\mathcal I^*)$}
           (2,2) node (Din*) {$\mfd(\in_\mathcal I^*)$}
           %(3,2) node (all) {$\omega^\omega\setminus(\omega^\omega)^W$}
           ;
     % draw the arrows
     \draw[->,>=stealth]
            %(empty) edge (Bin*)
            (Bin*) edge (Bleq*)
            (Bleq*) edge (Bneq*)
            (Bin*) edge (Dneq*)
            (Bleq*) edge (Dleq*)
            (Bneq*) edge (Din*)
            (Dneq*) edge (Dleq*)
            (Dleq*) edge (Din*)
            %(Din*) edge (all)
;
            
  \end{tikzpicture}
\end{figure}

\label{higherdimcichon}
\end{theorem}

\begin{proof}
The proof of this theorem mirrors that of Theorem \ref{propdi}. Most of these implications are easy, however two are more substantial. The easy cases, exactly the same as those for Theorem \ref{propdi} are shown below in Figure \ref{easycasesfunctions} and the arguments for these are exactly identical to those outlined in that proof. For instance, if $A$ is $\leq^*_\mathcal I$-bounded, then of course it is not $\leq^*_\mathcal I$-dominating hence $\mfb(\leq^*_\mathcal I) \leq \mfd(\leq^*_\mathcal I)$. Similarly, if $A \subseteq \bbb$ is a set so that there is a function $h:\baire \to \mathcal S$ so that for all $f \in A$ $f \in^*_\mathcal I h$ then $\hat{h} (x) (n) = {\rm max} \, h (x) (n) + 1$ witnesses the $\leq^*_\mathcal I$-bound of $A$. The other easy cases follow the same lines.

\begin{figure}[h]
\centering
  \begin{tikzpicture}[scale=1.5,xscale=2]
     % place and draw the nodes
     \draw %(0,0) node (empty) {$\emptyset$}
           (1,0) node (Bin*) {$\mfb(\in_\mathcal I^*)$}
           (1,1) node (Bleq*) {$\mfb(\leq_\mathcal I^*)$}
           (1,2) node (Bneq*) {$\mfb(\neq_\mathcal I^*)$}
           (2,0) node (Dneq*) {$\mfd(\neq_\mathcal I^*)$}
           (2,1) node (Dleq*) {$\mfd(\leq_\mathcal I^*)$}
           (2,2) node (Din*) {$\mfd(\in_\mathcal I^*)$}
           %(3,2) node (all) {$\omega^\omega\setminus(\omega^\omega)^W$}
           ;
     % draw the arrows
     \draw[->,>=stealth]
            %(empty) edge (Bin*)
            (Bin*) edge (Bleq*)
            (Bleq*) edge (Bneq*)
            %(Bin*) edge (Dneq*)
            (Bleq*) edge (Dleq*)
            %(Bneq*) edge (Din*)
            (Dneq*) edge (Dleq*)
            (Dleq*) edge (Din*)
            %(Din*) edge (all)
;
            
  \end{tikzpicture}
\caption{The Easy Cases of the Higher Cicho\'n Diagram}
\label{easycasesfunctions}
\end{figure}
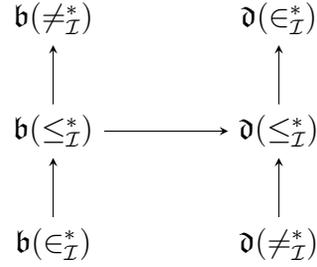

The two more substantial inequalities are $\mfb(\in^*_\mathcal I) \leq \mfd(\neq^*_\mathcal I)$ and $\mfb(\neq^*_\mathcal I) \leq \mfd(\in^*_\mathcal I)$, so I turn my attention to these. For the rest of this section, fix an ideal $\mathcal I \in\{ \Null, \Me, \Kb\}$.

The proofs of the inequalities consist of \say{lifting} the proofs for the Cicho\'n diagram to the higher dimensional case, particularly those in \cite{Bar1987} or for Lemmas \ref{infslalom} and \ref{neqinf}. Like in the proofs of those lemmas, fix finite, disjoint subsets of $\omega$ which collectively cover $\omega$, say $\mathcal J = \{J_{n, k} \; | \; k < n\}$. Let $J_n = \bigcup_{k < n} J_{n, k}$. Let's say that a $\mathcal J$-function is a function $x: \omega \to \omega^{< \omega}$ so that for every $n$ we have that $x(n)$ has domain $J_n$. Similarly a $\mathcal J$-slalom is a function $s:\omega \to [\omega^{< \omega}]^{< \omega}$ so that for each $n$ $|s(n)| \leq n$ and if $w \in s(n)$ then the domain of $w$ is $J_n$. If $x$ is a $\mathcal J$ function and $s$ a $\mathcal J$-slalom then we let $x \in^* s$ if and only if for all but finitely many $n$ $x(n) \in s(n)$. Clearly via some simple coding we can find homeomorphisms/measure isomorphisms between $\omega^\omega$ and the set of $\mathcal J$-functions (with the obvious topology) and $\mathcal S$ and the set of $\mathcal J$-slaloms. It's then routine to verify that $\mathfrak{b}(\in^*_\mathcal I)$ is the same for $\in^*$ defined on slaloms and elements of Baire space or their $\mathcal J$-versions. 

\begin{lemma}
$\mathfrak{b} (\in^*_\mathcal I) \leq \mathfrak{d}(\neq^*_\mathcal I)$
\end{lemma}

\begin{proof}
I use the version of $\mathfrak{b}(\in^*_\mathcal I)$ defined in terms of $\mathcal J$-slaloms as in the paragraph before the statement of the lemma. Let $\kappa < \mathfrak{b} (\in_\mathcal I^*)$. I need to show that $\kappa < \mathfrak{d}(\neq^*_\mathcal I)$. Fix a set $A \subseteq (\omega^\omega)^{\omega^\omega}$ of size $\kappa$. Let's see that $A$ is not $\neq^*_\mathcal I$-dominating. To be clear, a set is $\neq^*_\mathcal I$ dominating if for every function $f:\omega^\omega \to \omega^\omega$ there is a $g \in A$ so that for all $x$ save for a set in $\mathcal I$ $f(x) \neq^* g(x)$. Negating this, we need to find a function $f:\omega^\omega \to \omega^\omega$ so that for all $g \in A$ the set $\{x \; | \; \exists^\infty n \, g(x)(n) = f(x) (n) \}$ is $\mathcal I$-positive. In fact, I will show that under the assumption, such an $f$ can be found so that each such set is $\mathcal I$-measure one.

Given an element of Baire Space, $x:\omega \to \omega$ let $x'$ be the $\mathcal J$-function defined by $x' (n) = x \hook J_n$. Note that since the $J_n$'s cover $\omega$ and are disjoint the function $x \mapsto x'$ is a bijection. Given a function $f:\omega^\omega \to \omega^\omega$ let $f '$ similarly be defined by letting $f'(x) = f(x) '$. Let $A' = \{g' \; | \; g \in A\}$. Since this set has size $\kappa$ it is $\in^*_\mathcal I$-bounded i.e. there is a function $f_A$ with domain the set of $\mathcal J$-functions and range the set of $\mathcal J$-slaloms so that for all $g' \in A'$ $\{x \; | \; g'(x) \notin^* f_A(x)\} \in \mathcal I$. I need to transform $f_A$ into a function $f$ as advertized in the previous paragraph. The crux of the argument is the following claim, which will also be used in Lemma \ref{bd2} below as well.

\begin{claim}
Given a $\mathcal J$-slalom $s$, there is a function $x_s:\omega \to \omega$ so that for all $y:\omega \to \omega$ if $y'(n) \in^* s(n)$ then there are infinitely many $n < \omega$ so that $x_s(n) = y (n)$.
\end{claim}

\begin{proof}[Proof of Claim]
Fix a $\mathcal J$-slalom $s$. For each $n$ let $s(n) = \{w^n_0, ..., w^n_{n-1}\}$. Define $x_s:\omega \to \omega$ by letting for each $n$ and $k < n$ and $l \in J_{n, k}$ $x_s (l) = w^n_k(l)$. Suppose now that $y: \omega \to \omega$ is such that $y' (n) \in s (n)$ for all but finitely many $n < \omega$. Fix some $n$ so that $y' (n) \in s(n)$, say, $y' (n) = w^n_k$. Then for each $l \in J_{n, k}$ $y (l) = w^n_k (l) = x_s(l)$. Since there are cofinitely many such $n$'s there are infinitely many such $l$'s so $x_s$ is as needed.
\end{proof}

Now, returning to the proof of the lemma, let $f:\omega^\omega \to \omega^\omega$ be defined by letting $f(x)$ be the function $x_{f_A(x)}$ in the terminology of the claim. In particular, if $g:\omega^\omega \to \omega^\omega$ then for every $x \in \omega^\omega$ if $g'(x) \in^* f_A(x)$ then there are infinitely many $n$ so that $g(x) (n) = f(x)(n)$. In particular the set $\{x \; | \; g'(x) \in^* f_A (x)\}$ is contained in the set $\{x \; | \; \exists^\infty n \, g(x)(n) = f(x)(n)\}$. For $g \in A$ the former is $\mathcal I$-measure one and so the latter is as well. As a result $f$ is as needed.
\end{proof}

By essentially dualizing the proof above we get as well the following.
\begin{lemma}
$\mathfrak{b}(\neq^*_\mathcal I) \leq \mathfrak{d}(\in^*_\mathcal I)$
\label{bd2}
\end{lemma}

\begin{proof}
Suppose $\kappa < \mfb (\neq^*_\mathcal I)$ and let $A \subseteq (\mathcal S)^{\baire}$ be of size $\kappa$. I need to show that there is an $f \in \bbb$ so that for each $h \in A$ the set of $x$ so that $f(x) \notin^* h(x)$ does not have $\mathcal I$-measure one. For each $h \in A$ let $g_h \in \bbb$ be defined by letting, for each $x \in \baire$ $g_h (x) = x_{h(x)}$ as defined in the claim of the previous lemma. In particular, for each $x$ note that if $f(x) \in^* h(x)$ then $\exists^\infty n \, g_h (x) (n) = f(x)(n)$. Now let $\bar{A} = \{g_h \; | \; h \in A\}$. This set has size at most $\kappa$ so there is a $\neq^*_\mathcal I$-bound by assumption, say $f$. This means that for each $g_h \in \bar{A}$ we have that $\{x \; | \; g_h (x) \neq^* f(x)\}$ is $\mathcal I$-measure one. But now the lemma is proved since for every $x$ so that $g_h (x) \neq^* f(x)$ by the contrapositive of the implication defining $g_h$ we have that $f(x) \notin^* h(x)$.
\end{proof}

Combining the easy cases shown in Figure \ref{easycasesfunctions} with the proofs of the above two lemmas then completes the proof of Theorem \ref{higherdimcichon}. 
\end{proof}

Using the fact that every set in $\Kb$ is meager, we get the following relation between the diagrams for $\Me$ and $\Kb$.

\begin{proposition}
The following inequalities are provable in $\ZFC$:

\begin{figure}[h]
\centering
  \begin{tikzpicture}[scale=1.5,xscale=2]
     % place and draw the nodes
     \draw %(0,0) node (empty) {$\emptyset$}
           (1,0) node (Bin*) {$\mfb(\in_\Me^*)$}
           (1,1) node (Bleq*) {$\mfb(\leq_\Me^*)$}
           (1,2) node (Bneq*) {$\mfb(\neq_\Me^*)$}
           (2,0) node (Dneq*) {$\mfd(\neq_\Me^*)$}
           (2,1) node (Dleq*) {$\mfd(\leq_\Me^*)$}
           (2,2) node (Din*) {$\mfd(\in_\Me^*)$}
	     (0,0) node (BinK*) {$\mfb(\in_\Kb^*)$}
           (0,1) node (BleqK*) {$\mfb(\leq_\Kb^*)$}
           (0,2) node (BneqK*) {$\mfb(\neq_\Kb^*)$}
           (3,0) node (DneqK*) {$\mfd(\neq_\Kb^*)$}
           (3,1) node (DleqK*) {$\mfd(\leq_\Kb^*)$}
           (3,2) node (DinK*) {$\mfd(\in_\Kb^*)$}
           %(3,2) node (all) {$\omega^\omega\setminus(\omega^\omega)^W$}
           ;
     % draw the arrows
     \draw[->,>=stealth]
            %(empty) edge (Bin*)
            (Bin*) edge (Bleq*)
            (Bleq*) edge (Bneq*)
            (Bin*) edge (Dneq*)
            (Bleq*) edge (Dleq*)
            (Bneq*) edge (Din*)
            (Dneq*) edge (Dleq*)
            (Dleq*) edge (Din*)
            (Dleq*) edge (DleqK*)
            (Dneq*) edge (DneqK*)
            (Din*) edge (DinK*)
            (BleqK*) edge (Bleq*)
            (BneqK*) edge (Bneq*)
            (BinK*) edge (Bin*)
            (BinK*) edge (BleqK*)
            (BleqK*) edge (BneqK*)
            (DneqK*) edge (DleqK*)
	      (DleqK*) edge (DinK*)

            %(Din*) edge (all)
;
 \end{tikzpicture}
\end{figure}

\end{proposition}

\begin{proof}
Fix a relation $R$. To see that $\mfb (R_\Kb) \leq \mfb ( R_\Me)$ note that if $A \subseteq \bbb$ is $R_\Kb$-bounded, then it means that there is a function $f:\baire \to \baire$ so that for each $g \in A$ the set $\{ x \; | \; \neg g (x) \mathrel{R} f(x)\}$ is $\leq^*$-bounded by some $z \in \baire$. But this means in particular that it is meager and hence for each $g \in A$ $g \mathrel{R_\Me} f$. 

To see that $\mfd (R_\Me) \leq \mfd (R_\Kb)$, suppose that $A \subseteq \bbb$ is not $R_\Me$-dominating. This means that there is some $f:\baire \to \baire$ so that for each $g \in A$ the set $\{x \; | \; f(x) \mathrel{R} g(x)\}$ is not comeager. It follows that in particular it is not $\Kb$-measure one then (since each such set is comeager) and therefore no $g \in A$ is a $R_\Kb$ bound on $f$, so $A$ is not $R_\Kb$-dominating.
\end{proof}

\section{ Relations between the Higher Dimensional Cardinals and Standard Cardinal Characteristics}

This section concerns the relationship between provable inequalities between the cardinals introduced previously and cardinal characteristics of the continuum. I look first at the relationship between the higher dimensional cardinals and cardinal $\mfc^+$ and then I compare the diagrams for the null and meager ideals. 

I would like to argue that the standard diagonal arguments show that the cardinals defined above are greater than or equal to $\mathfrak{c}^+$, however this is not the case in $\ZFC$ alone. What is true is that this holds under additional assumptions on certain cardinal characteristics on $\omega$. For the statement of the lemma below, recall that $non (\Kb) =  \mathfrak{b}$, see Theorem 2.8 of \cite{BlassHB}.
\begin{lemma}
For each $\mathcal I \in \{\Null, \Me, \Kb\}$ and $R \in \{ \in^*, \leq^*, \neq^*\}$, if $\mathfrak{b}(R) = non (\mathcal I) = \mfc$ then $\mfc^+ \leq \mathfrak{b}(R_\mathcal I)$. In particular, if $add (\Null) = \mfc$ then all 18 cardinals introduced in the previous section are greater than $\mfc$.
\end{lemma}

\begin{proof}
This is essentially a generalization of the standard diagonal arguments used to show that various cardinal characteristics are uncountable. The point is that in that case, the relations (on $\omega$) under consideration are always so that every finite set has an upper bound and the ideal is always the ideal of finite sets. It is exactly because arithmetic of cardinal characteristics is not so simple that the additional hypotheses are needed.

Fix $R$ and $\mathcal I$ and assume $\mathfrak{b}(R) = non (\mathcal I) = \mfc$. Let $f_\alpha : \baire \to \baire$ for each $\alpha < \mfc$. We want to find a $g:\baire \to \baire$ (or $g:\baire \to \mathcal S$ in the case of $R = \in^*$) so that for all $\alpha$ $f_\alpha \mathrel{R_\mathcal I} g$. This is done as follows. First, list the elements of $\baire$ as $\{x_\alpha \; | \; \alpha < \mfc\}$. Next, note that for each $\beta < \mfc$, by the fact that $non (\mathcal I) = \mfc$ we have that $\{x_\alpha \; | \; \alpha < \beta\} \in \mathcal I$ and, by the fact that $\mathfrak{b}(R) = \mfc$ we have that for each $x_\gamma \in \baire$ the set $\{f_\alpha (x_\gamma) \; | \; \alpha < \beta\}$ has an $R$-bound, say $y^\gamma_\beta$. Now define $g$ so that $g(x_\alpha) = y^\alpha_\alpha$. It follows that for all $\alpha$ if $\gamma > \alpha$ then $f_\alpha (x_\gamma) \mathrel{R} g(x_\gamma)$ and since the set $\{x_\gamma \; | \; \gamma > \alpha \}$ is $\mathcal I$-measure one we're done.
\end{proof}

In ongoing joint work with J. Brendle we have since shown that the cardinals of the form $\mfd(R_\mathcal I)$ are provably at least $\mfc^+$ but the cardinals of the form $\mfb(R_\mathcal I)$ can be both consistently equal to and strictly less than $\mfc$, even $\aleph_1$ with the continuum arbitrarily large. Hopefully this will appear in print soon.

Finally in this section let me compare the cardinals for $\Me$ and $\Null$. Every argument given so far has worked equally well for each of them, and the theorem below suggests that this is not an accident.

\begin{theorem}
If $add (\Null) = cof (\Null)$ then for every relation $R \in\{ \in^*, \leq^*, \neq^*\}$ we have that $\mfb (R_\Null) = \mfb (R_\Me)$ and $\mfd (R_\Null) = \mfd (R_\Me)$.
\label{M=N}
\end{theorem}

The proof of this theorem follows immediately from the following two lemmas, the first of which is well known.
\begin{lemma}[Theorem 2.1.8 of \cite{BarJu95}]
If $add(\Null) = cof(\Null)$ then there is a bijection $f:\baire \to \baire$ so that for all $A \subseteq \baire$ $f(A) \in \Null$ if and only if $A \in \Me$ and $f(A) \in \Me$ if and only if $A \in \Null$.
\label{erdosthm}
\end{lemma}

\begin{lemma}
If there is a bijection $f:\baire \to \baire$ as in Lemma \ref{erdosthm} then for every relation $R \in \{ \in^*, \leq^*, \neq^*\}$ we have that $\mfb (R_\Null) = \mfb (R_\Me)$ and $\mfd (R_\Null) = \mfd (R_\Me)$.
\end{lemma}

\begin{proof}
Fix a relation $R \in \{ \in^*, \leq^*, \neq^*\}$ and let $f:\baire \to \baire$ be a bijection as described in Lemma \ref{erdosthm}. First, suppose that $\kappa < \mfb (R_\Null)$ and let $A \subseteq ( \baire )^{\baire}$ be a set of size $\kappa$. I claim that there is a function $g_A:\baire \to \baire$ ($g_A:\baire \to \mathcal S$ in the case $R = \in^*$) so that for all $g \in A$ $g \mathrel{R_\Me} g_A$ and hence $\kappa < \mfb (R_\Me)$. Let $A_f = \{ g \circ f \; | \; g \in A\}$. Since $f$ is a bijection $|A_f| = \kappa$. By the hypothesis, let $\bar{g}:\baire \to \baire$ be an $R_\Null$ bound on $A_f$. I claim that $g_A = \bar{g} \circ f^{-1}$ is as needed. We have that for every $g \in A$ if $f^{-1}(x) = y$ is in the measure one set for which $g(f(y)) R \bar{g} (y)$ is true then the following holds:

\begin{equation*}
g(x) =  g (f(f^{-1}(x)) \mathrel{R} \bar{g} ( f^{-1} (x)) = g_A (x)
\end{equation*}
\linebreak
Therefore $f^{-1} (\{x \; | \; \neg g(x) \mathrel{R} g_A (x)\})$ is contained in $\{x \; | \; \neg g(f(x)) \mathrel{R} \bar{g}(x)\}$, which is null by assumption and so the former is null as well. Hence by the property of $f$ it follows that $\{x \; | \; \neg g(x) \mathrel{R} g_A(x)\}$ is meager so $g_A$ is an $R_\Me$-bound as needed.

This shows that $\mfb (R_\Null) \leq \mfb (R_\Me)$ however an identical argument, flipping the roles of the meager and null sets, shows the reverse inequality so we get that $\mfb (R_\Null) = \mfb (R_\Me)$. 

An essentially dual argument works to show that $\mfd (R_\Null) = \mfd (R_\Me)$. Let me sketch it, though I leave out the details. Assuming that $\kappa < \mfd (R_\Null)$ we fix a set $A \subseteq (\baire)^{\baire}$ of size $\kappa$, define $A_f$ as before and let $\bar{g}$ be a function not dominated by any member of $A_f$. Then essentially the same argument shows that $\bar{g} \circ f^{-1}$ is a function not dominated by any member of $A$ and, again by symmetry we obtain the required equality.
\end{proof}

Again in joint work with J. Brendle we have shown how to separate these. In particular, in the random model with $\mfc = \kappa$ we have shown $\mfb(\in^*_\Null) = \aleph_1 < \mfb(\in^*_\Kb) = \aleph_2 < \mfb (\in^*_\Me) = \kappa^+$.

\section{Consistency Results}
In this section I consider consistent separations between the cardinals. For readability, I focus on the case of $\mathcal I = \Null$, however, it's routine to check that the arguments go through for $\mathcal I = \Me$ and $\mathcal I = \Kb$. Indeed the essential point will be simply that $\Null$ has a Borel base and contains the countable subsets of $\baire$. Also, I will only be considering models of $\CH$ so by Theorem \ref{M=N} any separation between nodes in the $\Null$ diagram will hold equally for the $\Me$ diagram.

From now on assume $\GCH$ holds and fix an enumeration of $\baire$ in order type $\omega_1$, say $\{x_\alpha \; | \; \alpha < \omega_1\}$. Also fix an enumeration of the Borel sets in $\Null$ in order type $\omega_1$, say $\{N_\alpha \; | \; \alpha < \omega_1\}$.  Suppose we have some forcing notion $\mathbb P$ which does not add reals. Note that in this case if $B$ is a Borel set then $\mathbb P$ forces that the name for $B$ is equal to its evaluation in the ground model. Also, since $\mathbb P$ does not add reals, it does not add any Borel sets either. This translates to the following idea, which is used in several proofs. Suppose $\dot{A}$ is a $\mathbb P$ name for a subset of $\baire$. If for some condition $p \in \mathbb P$ we have that $p \forces\mu( \dot{A} ) = \check{0}$, then we can always find a $q \leq p$ and a Borel null set in the ground model $N$ so that $q \forces \dot{A} \subseteq \check{N}$.

The following simple lemma will be used in several proofs.

\begin{lemma}
Let $\vec{N}_0 = \langle N_{0, \alpha} \; \ | \; \alpha < \omega_1\rangle$ and $\vec{N}_1 = \langle N_{1, \alpha} \; \ | \; \alpha < \omega_1\rangle$ be two sequences of null sets of length $\omega_1$. There is an enumeration in order type $\omega_1$, say $\langle(N'_{0, \alpha}, N'_{1, \alpha} ) \; | \; \alpha <\omega_1\rangle$ of the set of all pairs $(N_{0, \beta}, N_{1, \gamma})$ so that for each $\alpha < \omega_1$ we have $x_\alpha \notin N'_{0, \alpha} \cup N'_{1, \alpha}$.
\label{mainlemmafunctions}
\end{lemma}

\begin{proof}
First fix any enumeration of $\vec{N}_0 \times \vec{N}_1$, say $\langle (N''_{0, \alpha}, N''_{1, \alpha}) \; | \; \alpha < \omega_1\rangle$ and define inductively for each $\alpha$ $(N'_{0, \alpha}, N'_{1, \alpha})$ to be the least $\gamma$ so that $(N''_{0, \gamma}, N''_{1, \gamma})$ has not yet been enumerated and $x_\alpha \notin  (N''_{0, \gamma}, N''_{1, \gamma})$. I need to show that every $(N''_{0, \gamma}, N''_{1, \gamma})$ gets enumerated under this procedure. Suppose not and let $\gamma$ be least so that $(N''_{0, \gamma}, N''_{1, \gamma})$ is not enumerated. Since for every $\beta < \gamma$ the pair $(N''_{0, \beta}, N''_{1, \beta})$ was enumerated, there was some countable stage by which this happened and so for cocountably many $\alpha$ it must have been the case that $x_\alpha \in N''_{0, \gamma} \cup N''_{1, \gamma}$. But this is impossible since $N''_{0, \gamma} \cup N''_{1, \gamma}$ is measure zero and hence cannot contain a cocountable set.
\end{proof}

\subsection{Generalizing Cohen Forcing}

The point of this subsection is to prove the following theorem.

\begin{theorem}[$\GCH$]
Let $\kappa$ be a regular cardinal greater than $\aleph_1$. There is a cofinality preserving forcing notion $\mathbb P_\kappa$ so that if $G \subseteq \mathbb P_\kappa$ is $V$-generic then in $V[G]$ we have $\mfc^+ = \aleph_2 = \mathfrak{b} (\neq^*_\Null) < \mfd (\neq^*_\Null) = 2^\mfc = \kappa$.
\label{cohenforcingfornull}
\end{theorem}

The proof will involve an iteration of length $\kappa$ of a certain forcing notion, $\mathbb{C}_\Null$. Let me begin by introducing this forcing notion and studying its properties.
\begin{definition}
The $\Null$-Cohen forcing, denoted $\mathbb{C}_\Null$, is the set of all $p:{\rm dom} (p) \subseteq \baire \to \baire$ so that ${\rm dom}(p)$ and ${\rm graph}(p)$ are both Borel and ${\rm dom} (p)$ is measure zero. We let $p \leq q$ if and only if $p \supseteq q$.
\end{definition}

The following observations are easy but will be useful.
\begin{proposition}
The forcing $\mathbb{C}_\Null$ is $\sigma$-closed and has size $\mfc$, hence it has the $\mfc^+$-c.c. In particular, under $\CH$ all cofinalities and hence cardinalities are preserved.
\end{proposition}

\begin{proof}
First let's see that $\mathbb C_\Null$ is $\sigma$-closed. Given a descending sequence $p_0 \geq p_1 \geq p_2 ...$ let $p = \bigcup_{n < \omega} p_n$. Since the countable union of Borel sets is Borel it follows that $p$ has a Borel graph and since the countable union of null sets is null, it follows that $p$ has null domain. Thus $p$ is a condition so it is a lower bound on the sequence of $p_n$'s. 

To see that $\mathbb C_\Null$ has size $\mfc$ it suffices to note that each condition is a Borel subset of $(\omega^\omega)^2$, of which there are only $\mfc$ many.
\end{proof}

Note that since $\mathbb{C}_\Null$ adds no reals or Borel sets and every condition $p \in \mathbb{C}_\Null$ is a Borel set it follows that $\forces_{\mathbb{C}_\Null} \dot{\mathbb{C}}_\Null = \check{\mathbb{C}}_{\Null}$ and so in particular, the product and iteration of $\mathbb{C}_\Null$ are the same. Now, a straightforward density argument shows that $\mathbb{C}_\Null$ adds a function $g:\omega^\omega \to \omega^\omega$, namely the union of the generic filter. Indeed it's easy to see that if $p$ is any condition and $N$ is any Borel null set then there is a condition $q \leq p$ so that $N \subseteq {\rm dom}(q)$. I need to verify two properties of $\mathbb{C}_\Null$, given as Lemmas \ref{lemma1cohen} and \ref{lemma2cohen} below. The first will imply that in an iterated extension $\mfd(\neq^*_\Null)$ becomes large and the second will imply that $\mfb(\neq^*_\Null)$ remains small in an iterated extension. 

\begin{lemma}
If $G \subseteq \mathbb{C}_\Null$ is generic over $V$ then in $V[G]$ the set of $f \in (\baire)^{\baire} \cap V$ is not dominating with respect to the relation $\neq^*_\Null$.
\label{lemma1cohen}
\end{lemma}

\begin{proof}
In fact a stronger statement is true, namely if $g = \bigcup G$ and $\dot{g}$ is the name for $g$, then for any $f:\omega^\omega \to \omega^\omega$ in the ground model the set $\{x \; | \; f(x) = g(x)\}$ is not measure zero. To see this, suppose that for some condition $p$ and ground model function $f$ we have that $p \forces \mu (\{x \; | \; \check{f}(x) = \dot{g}(x)\}) = \check{0}$. Since every null set is contained in a Borel null set, there is a Borel Null set $N$, necessarily in the ground model since $\mathbb{C}_\Null$ is $\sigma$-closed, and a strengthening $q \leq p$ so that $q \forces \{x \; | \; \check{f}(x) = \dot{g} (x)\} \subseteq \check{N}$. But now let $x \notin N \cup {\rm dom}(q)$ (this is possible since $N \cup {\rm dom}(q) \in \mathcal N$). It is straightforward to verify that $q^* = q \cup \{\langle x, f(x)\rangle \}$ is a condition extending $q$ but clearly $q^* \forces \{x \; | \; \check{f}(x) = \dot{g}(x)\} \nsubseteq \check{N}$, which is a contradiction. It follows in particular that for every $f \in V$ we have that on a non null set of $x$ there are infinitely many $n < \omega$ so that $f(x) (n) = g(x) (n)$. This implies the lemma.
\end{proof}

\begin{lemma}
If $G \subseteq \Pi_I \mathbb{C}_\Null$ is generic over $V$ for the countable support product of $\mathbb C_\Null$ over an index set $I$ of size at most $\aleph_1$ then in $V[G]$ the set of $f \in (\baire)^{\baire} \cap V$ is unbounded with respect to the relation $\neq^*_\Null$.
\label{lemma2cohen}
\end{lemma}

\begin{proof}
I need to show that in $V[G]$ there is no $h:\baire \to \baire$ so that for all $f:\baire \to \baire$ in $V$ the set of $x$ for which $f(x) \neq^* g(x)$ is measure one. Thus suppose for a contradiction that there is a condition $p$ and a name $\dot{h}$ so that $p \forces \dot{h}: \check{\baire} \to \check{\baire}$ is such a function. I need to define in $V$ a function for which this fails. 

Note that (under $\CH$) $\Pi_I \mathbb{C}_\Null$ has size $\aleph_1$. For each condition $p \in  \Pi_I \mathbb{C}_\Null$, let $N^p$ be the union of the domains of the coordinate conditions. Since $ \Pi_I \mathbb{C}_\Null$ has countable support, it follows that $N^p$ is null. Now, using Lemma \ref{mainlemmafunctions} fix an enumeration $\langle (N_{0, \alpha}, p_\alpha) \; | \; \alpha < \omega_1 \rangle$ of all pairs where $N_{0, \alpha}$ ranges over the Borel null sets $N_\alpha$ and $p_\alpha$ is a condition in $\Pi_I \mathbb{C}_\Null$, and $x_\alpha \notin N_{0, \alpha} \cup N^{p_\alpha}$. For each $\alpha$, let $r_\alpha \leq p_\alpha$ decide $\dot{h}(x_\alpha)$. Say that $r_\alpha \forces \dot{h}(\check{x}_\alpha) = \check{y}_\alpha$ for some $y_\alpha$. Let $h^* :\baire \to \baire$ be the function (defined in $V$) so that $h^*(x_\alpha) = y_\alpha$ for all $\alpha$. Suppose that there is some Borel null set $N$ and some condition $p$ which forces that $\{x \; | \; \exists^\infty n \, \dot{h} (x) (n) = \check{h}^* (x) (n) \} \subseteq \check{N}$. Let $\alpha$ be such that $(N, p) = (N_{0, \alpha}, p_\alpha)$. Then $p_\alpha \forces \{x \; | \; \exists^\infty n \, \dot{h} (x) (n) = \check{h}^* (x) (n) \} \subseteq \check{N}_{0, \alpha}$. But $r_\alpha \leq p_\alpha$ forces that $\dot{h} (x_\alpha) = \check{h}^* (x_\alpha)$ and by the choice of enumeration we had that $x_\alpha \notin N_{0, \alpha}$, which is a contradiction.
\end{proof}

I'm now ready to prove Theorem \ref{cohenforcingfornull}. In fact it follows from the following theorem, which is just a more precise statement of what will be shown.
\begin{theorem}
Let $\kappa$ be a regular cardinal greater than $\aleph_1$ and let $\mathbb P_\kappa$ be the countable support product of $\mathbb{C}_\Null$. Then $\mathbb P_\kappa$ preserves cofinalities and cardinals and if $G\subseteq \mathbb P_\kappa$ is $V$-generic then in $V[G]$ $\mfc^+ = \aleph_2 = \mathfrak{b} (\neq^*_\Null) < \mfd (\neq^*_\Null) = 2^\mfc = \kappa$.
\end{theorem}

\begin{proof}
Fix $\kappa > \aleph_1$ regular, let $\mathbb P = \mathbb P_\kappa$ be the countable support product of $\mathbb{C}_\Null$ of length $\kappa$. Clearly $\mathbb P$ is $\sigma$-closed and a straightforward $\Delta$-system argument using $\GCH$ shows that it has the $\aleph_2$-c.c. It follows that all cardinals and cofinalities are preserved.

Also, for each $\alpha$ the $\alpha$-stage forcing $\mathbb P_\alpha$ adds a new function $g_\alpha:\baire \to \baire$ so in the extension $2^\mfc \geq \kappa$. A standard nice name counting argument, again using $\GCH$ shows that in fact $2^\mfc = \kappa$. 

It remains to show that $\aleph_2 = \mathfrak{b} (\neq^*_\Null)$ and $\mfd (\neq^*_\Null) = \kappa$. For the first of these, it suffices to see that $(\baire)^{\baire} \cap V$ is unbounded with respect to $\neq^*_\Null$. To see this, by Lemma \ref{lemma2cohen}, it suffices to note that if $\dot{f}$ is a name for a function in $\bbb$ then $\dot{f}$ is equivalent to a $\Pi_I \mathbb{C}_\Null$ for $I$ an index set of size $\aleph_1$. This latter statement is proved as follows: let, for each $x \in \baire$ $\mathcal A_x$ be an antichain of conditions deciding $\dot{f}(\check{x})$ and note that the cardinality of the supports of the elements of $\bigcup_{x \in \baire} \mathcal A_x$ has size $\aleph_1$ by $\CH$ using the countable support of the product.

Finally for $\mfd (\neq^*_\Null) = \kappa$, suppose that $A \subseteq (\baire)^{\baire}$ of size $< \kappa$. It follows that $A$ must have been added by some initial stage of the iteration, the next stage of which killed the possibility that it was dominating by Lemma \ref{lemma1cohen}.
\end{proof}

Let me reiterate that, defining $\mathbb{C}_\Me$ and $\mathbb{C}_\Kb$ in the obvious way the proofs can be repeated verbatim to obtain similar consistencies for the $\Me$ and $\Kb$ ideals. The same is true in the remaining subsections, though I won't explicitly say this again. An interesting open question though is the following.
\begin{question}
Are the forcing notions $\mathbb C_\Null$, $\mathbb C_\Me$ and $\mathbb C_\Kb$ provably forcing equivalent?
\end{question}

\subsection{Generalizing Hechler Forcing}

In this subsection I consider a generalization of Hechler forcing called $\mathbb D_\Null$ and look at two models obtained by iterating this forcing. First I consider the countable support iteration of $\mathbb D_\Null$ and then I look at a non-linear iteration of $\mathbb D_\Null$ similar to the one used in \cite{CuSh95}. In the latter case I obtain the following consistency result.
\begin{theorem}
Let $\aleph_2\leq \kappa \leq \lambda$ with $\kappa$ and $\lambda$ regular. Then there is a forcing notion $\mathbb P_{\kappa, \lambda}$ which preserves cardinals and cofinalities such that if $G \subseteq \mathbb P_\kappa$ is generic then in $V[G]$ we have that $\mfb (\leq^*_\Null) = \kappa < \mfd (\leq^*) = 2^\mfc = \lambda$.
\label{mainthmhechler}
\end{theorem}

Similar to the last subsection I start by introducing the one step and studying its properties. This forcing is reminiscent of Hechler Forcing. As before, I work with the null ideal for definiteness but it's easy to see that the proofs adapt to the case of the other ideals.
\begin{definition}
The $\Null$-Hechler forcing $\mathbb{D}_\Null$ consists of the set of pairs $(p, \mathcal F)$ where $p \in \mathbb{C}_\Null$ and $\mathcal F$ is a countable set of functions $f:\baire \to \baire$. We let $(p, \mathcal F) \leq (q, \mathcal G)$ in case $p \supseteq q$, $\mathcal F \supseteq \mathcal G$ and for all $x \in {\rm dom}(p) \setminus {\rm dom}(q)$ and all $g \in \mathcal G$ $g(x) \leq^* p(x)$.
\end{definition}

If $d= (p_d, \mathcal F_d) \in \mathbb{D}_\Null$, let me call $p_d$ the {\em stem} of the condition and $\mathcal F_d$ the {\em side part}. The basic properties I will need for $\mathbb{D}_\Null$ are as follows.
\begin{proposition}
$\mathbb{D}_\Null$ is $\sigma$-closed and has the $\mfc^+$-c.c., thus assuming $\CH$, it preserves cofinalities and cardinals. Also, if $G \subseteq \mathbb D_\Null$ is $V$-generic then the union of $G$ is a function $g:\baire \to \baire$ so that for any $f:\baire \to \baire$ in the ground model the set of $x$ so that $g(x)$ does not eventually dominate $f(x)$ is null.
\end{proposition}

\begin{proof}
That $\mathbb{D}_\Null$ is $\sigma$-closed is the same as the proof for $\mathbb{C}_\Null$. To see that it has the $\mfc^+$-c.c. it suffices to note that if two conditions have the same stem then they are compatible.

Now to see that $g$ is total is a simple density argument, noting that if $d$ is some condition and $x \notin {\rm dom}(p_d)$ then there is a $y$ dominating all of the $f(x)$ for $f \in \mathcal F_d$ since this set is countable and hence $(p_d \cup \{\langle x, y\rangle \}, \mathcal F_d)$ extends $d$ as needed. Moreover, if $f:\baire \to \baire$ is a function in the ground model and $d$ is any condition then clearly we can strengthen $d$, say to $d'$ so that $f$ is included in the side part $d'$. This strengthening forces, by the definition of the extension relation, that for all $x \notin {\rm dom}(p_{d'})$ $f(x) \leq^* \dot{g} (x)$. Since the domain of $p_{d'}$ was measure zero this proves the second part.
\end{proof}

\begin{remark}
While it's not used in any proof let me note that, unlike with $\mathbb C_\Null$ it is {\em not} the case that every condition in $\mathbb D_\Null$ can be extended to include any Borel null set in the domain of its stem. This is because given any uncountable Borel set, (at least under $\CH$) one can use a simple diagonal argument to build a function which is not dominated by any Borel function on that set. What is true however, is that the stem of any condition can be extended to include any {\em countable} set.
\end{remark}

Let me now show what happens in the generic extension by a countable support iteration of $\mathbb{D}_\Null$.
\begin{theorem}
Let $\kappa$ be regular and let $\mathbb P_\kappa$ be the countable support iteration of $D_\Null$. If $G \subseteq \mathbb P_\kappa$ is $V$-generic then in $V[G]$ $\mfb (\leq^*_\Null) = \mfd (\neq^*) = \kappa = 2^\mfc$.
\label{mainthmhechler2}
\end{theorem}

\begin{proof}
Let $\kappa > \aleph_2$ be regular and let $\mathbb P_\kappa$ be the countable support iteration of length $\kappa$ of $\mathbb{D}_\kappa$. That cardinals and cofinalities are preserved follows as for $\mathbb{C}_\Null$. Also, every set $A \subseteq \bbb$ of size less than $\kappa$ is added by some initial stage, after which a function bounding $A$ was added so $\mfb (\leq^*_\Null) = \kappa$. Moreover a nice name argument easily gives that $2^\mfc = \kappa$.

It remains to see that $\mfd (\neq_\Null^*) = \kappa$ in this model. For this, I use the fact that countable support iterations always add a generic for $Add (\omega_1, 1)$ at limit stages of cofinality $\omega_1$. Now, given any function $f:\omega_1 \to \omega_1$ we can think of it as a function $\hat{f}$ from $\baire$ to $\baire$ by letting $\hat{f}(x_\alpha) = x_\beta$ just in case $f(\alpha) = \beta$. Suppose $A \subseteq \bbb$ is a set of size less than $\kappa$. It must have been added by some initial stage of the iteration $\mathbb P_\kappa$ and therefore there is a later stage which adds an $Add (\omega_1, 1)$-generic function $g:\omega_1 \to \omega_1$. By density, given any $f:\baire \to \baire$ in $A$ and any Borel null set $N$ we can find an $x \notin N$ so that $\hat{g}(x) = f(x)$ and therefore, for any $f \in A$ the set of $x$ for which $\hat{g}(x) = f(x)$ is not null. Therefore in particular $A$ is not a $\neq^*_\Null$-dominating family. Thus $\mfd (\neq^*_\Null) = \kappa$.
\end{proof} 

I'm now ready to prove Theorem \ref{mainthmhechler}. This uses a version of the iteration discussed in Section 3 of \cite{CuSh95}. Let me recall the basics of what I need. Fix $\kappa \leq \lambda$ regular cardinals greater than or equal to $\aleph_2$ and let $\mathbb{Q} = (Q, \leq_\mathbb Q)$ be a well founded partial order so that $\mfb(\mathbb Q) = \kappa$ and $\mfd (\mathbb Q) = \lambda$. For example, under $\GCH$, $\kappa \times [\lambda]^{< \kappa}$ ordered by $(\alpha, \tau) \leq (\beta, \sigma)$ if and only if $\alpha < \beta$ and $\tau \subseteq \sigma$ is such an order, see Lemma 2 of \cite{CuSh95} for a proof. I need to define a $\sigma$-closed, $\aleph_2$-c.c. forcing notion $\mathbb{D} (\mathbb Q)$ so that forcing with this partial order adds a cofinal embedding of $\mathbb{Q}$ into $(\bbb, \leq^*_\Null)$. If I can do this, then by Lemmas 3 and 5 of \cite{CuSh95} it follows that in the extension by this forcing notion $\mfb (\leq^*_\Null) = \kappa$ and $\mfd (\leq^*_\Null) = \lambda$. For completeness, here are the cited lemmas.

\begin{lemma}[Lemma 3 of \cite{CuSh95}]
If $\mathbb P$ and $\mathbb Q$ are partially ordered sets and $\mathbb P$ embeds cofinally into $\mathbb Q$ then $\mfb(\mathbb P) = \mfb (\mathbb Q)$ and $\mfd (\mathbb P) = \mfd (\mathbb Q)$.
\end{lemma}

\begin{lemma}[Lemma 5 of \cite{CuSh95}]
Suppose $\mathbb P$ is a partial order with $\mfb (\mathbb P) = \beta$ and $\mfd (\mathbb P) = \delta$.
\begin{enumerate}
\item
If $V[G]$ is a generic extension of $V$ so that every set of ordinals of size less than $\beta$ in $V[G]$ is covered by a set of ordinals of size less than $\beta$ in $V$ then $V[G] \models \mfb (\mathbb P) = \beta$.
\item
If $V[G]$ is a generic extension of $V$ so that every set of ordinals of size less than $\delta$ in $V[G]$ is covered by a set of ordinals of size less than $\delta$ in $V$ then $V[G] \models \mfd (\mathbb P) = \delta$.
\end{enumerate}

\end{lemma}

Let me define the forcing notion I need. In what follows, if $a \in \mathbb Q$ let $\mathbb Q \hook a = \{b \in \mathbb Q \; | \; b < a\}$. Similarly if $p$ is a function with domain contained in $\mathbb Q$, let $p \hook  a$ be the restriction of $p$ to $\mathbb Q \hook a$.
\begin{definition}
Let $\mathbb{Q}_{top}$ be $\mathbb{Q}$ with the addition of a top element, $top$, greater than all other elements. For each $a \in \mathbb Q_{top}$ define inductively a forcing notion $\mathbb{D}(\mathbb Q)_a$ to be the set of functions $p$ with ${\rm dom}(p) \subseteq\mathbb Q \hook a$ countable and for each $b \in {\rm dom}(p)$ $p(b)$ is a $\mathbb{D} (\mathbb Q)_b$-name for an element of $\mathbb{D}_\Null$ of the form $(\check{p}, \dot{\mathcal F})$. Let $p \leq_{\mathbb{D} (\mathbb Q)_a} q$ if and only if $p \supseteq q$ and for every $b \in {\rm dom}(q)$ we have that $p \hook b \forces_{\mathbb D (\mathbb Q)_b} p(b) \leq_{\mathbb{D}_\Null} q(b)$. Finally we let $\mathbb{D} (\mathbb Q) = \mathbb D(\mathbb Q)_{top}$.
\end{definition}

\begin{remark}
Below I show that $\mathbb D (\mathbb Q)_a$ is $\sigma$-closed for each $a \in \mathbb Q$. It follows that there is no loss in generality in insisting that for all $b$ the name for the stem of $p(b)$ is a check name, since the latter is always coded by a real and hence the set of conditions like this is dense.
\end{remark}

\begin{lemma}
For every $a \in \mathbb Q$, the partial order $\mathbb{D} (\mathbb Q)_a$ is $\sigma$-closed, and under $\GCH$, has the $\aleph_2$-c.c.. 
\end{lemma}

\begin{proof}
Fix $a \in \mathbb Q$. Suppose that $p_0 \geq p_1 \geq ...\geq p_n \geq ...$ is a decreasing sequence of elements in $\mathbb D(\mathbb Q)_a$. Let $p$ be defined as the function whose domain is the union of the domains of all of the $p_n$'s and so that for each $b \in {\rm dom} (p)$ we let $p(b)$ name a lower bound on the set of $\{p_n (b) \;  |\; n < \omega \, {\rm and} \, b \in {\rm dom} (p_n)\}$. Since $\mathbb D_\Null$ is $\sigma$-closed such a name exists. Clearly $p$ is a lower bound on the sequence so $\mathbb D (\mathbb Q)_a$ is $\sigma$-closed.

To see that $\mathbb D (\mathbb Q)_a$ has the $\aleph_2$-c.c., suppose that $A = \{p_\alpha \; | \; \alpha < \omega_2\}$ is a set of conditions. Applying the $\Delta$-system lemma (by $\GCH$) we can thin out $A$ so a $\Delta$-system so that any two conditions' domains coincide on some countable set $B \subseteq \mathbb Q \hook a$. But now, since each name for the stem in $p(b)$ is a check name, the set of all possible sequences of stems on the coordinates in $B$ has size $\omega_1^\omega = \omega_1$ (using $\CH$). Thus $A$ contains $\omega_2$ many conditions so that on $B$ the stems agree, and each such condition's overlapping domains is $B$. But this means that those conditions are all compatible.
\end{proof}

The next lemma is entirely straightforward to verify.
\begin{lemma}
Suppose $a < b \in \mathbb Q_{top}$ then $\mathbb D (\mathbb Q)_a$ completely embeds into $\mathbb D(\mathbb Q)_b$ and the map $\pi:\mathbb D (\mathbb Q)_b \to \mathbb D(\mathbb Q)_a$ defined by $\pi (p) = p\hook a$ is a projection.
\label{projection}
\end{lemma}

Now we get to the heart of the matter. Let $G \subseteq\mathbb{D} (\mathbb Q)$ be generic over $V$. For each $a \in \mathbb Q$ let $f_G^a : \baire \to \baire$ be the $\mathbb{D}_\Null$-generic function added by the $a^{\rm th}$ coordinate i.e. $f_G^a = \bigcup_{p \in G} \{ p (a)_0\}$

\begin{lemma}
The map $a\mapsto f_G^a$ is a cofinal mapping of $\mathbb{Q}$ into $(\bbb, \leq^*_\Null)$.
\label{cof}
\end{lemma}

\begin{proof}
I need to show that for $a, b \in \mathbb Q$, first of all that $a < b$ if and only if $f_G^a \leq^*_\Null f_G^b$ and second of all that for each $f \in \bbb$ there is an $a \in \mathbb Q$ so that $f \leq^*_\Null f_G^a$. First suppose that $a < b$. By Lemma \ref{projection} for any $p \in \mathbb D(\mathbb Q)$ we can find a  strengthening $q$ so that $q(b)$ forces that $f_G^a \leq^*_\Null f_G^b$ since $f_G^a$ is added at an earlier stage.

Now suppose that $a$ and $b$ are incomparable (the case where $b < a$ is symmetric to the above). Suppose for a contradiction that there is some condition $p \in \mathbb \mathbb D (\mathbb Q)$ so that $p \forces \{x \; | \; \dot{f}_G^a(\check{x}) \nleq^* \dot{f}_G^b (\check{x}) \} \subseteq \check{N}$, so in particular $p$ forces that $f_G^a \leq^*_\Null f_G^b$. By strengthening if necessary we can assume that $a, b \in {\rm dom}(p)$. Now choose $\alpha$ so that $x_\alpha$ is not in $N$, ${\rm dom} (p (a)_0)$ or ${\rm dom}(p(b)_0)$. Since all three are null sets, their union is null so there is such an $x_\alpha$. Now let $q_b \leq_{\mathbb D(\mathbb Q)_b} p \hook b$ be a strengthening so that if $\dot{\mathcal F}_b$ is the name of the side part of $p(b)$ then $q_b$ decides the check name values of all countably many elements of $\{\dot{f} (\check{x}_\alpha) \; | \; \dot{f} \in \dot{\mathcal F}_b\}$, this is possible by the fact that the forcing is $\sigma$-closed. Let $p_b$ be the condition obtained by letting $p_b(x) = p(x)$ if $x \notin {\rm dom}(q)$ and $p_b(x) = q(x)$ otherwise. Now let $q_a \leq p_b \hook a$ strengthen $p_b$ to decide the values $\{\dot{f} (\check{x}_\alpha) \; | \; \dot{f} \in \dot{\mathcal F}_a\}$ for $\dot{\mathcal F}_a$ the name of the side part of $p(a)$. Finally let $q$ be the condition which agrees with $q_a$ on its domain and agrees with $p_b$ otherwise. Note that since $a$ and $b$ are incomparable in $\mathbb Q$ neither $a$ nor $b$ is in the domains of $q_a$ or $q_b$ and so $q (a) = p(a)$ and $q(b) = p(b)$ and $q \leq p$. Now, let $x_b$ be a $\leq^*$-bound on the set $\{\dot{f} (\check{x}_\alpha) \; | \; \dot{f} \in \dot{\mathcal F}_b\}$ and let $x_a$ be such that $x_b + 1 \leq^* x_a$ and $x_a$ is a $\leq^*$-bound on the set $\{\dot{f} (\check{x}_\alpha) \; | \; \dot{f} \in \dot{\mathcal F}_a\}$. Finally let $q^*$ be the strengthening of $q$ so that the stem of $q^*(a)$ includes $(x_\alpha, x_a)$ and the stem of $q^*(b)$ includes $(x_\alpha, x_b)$. Then $q^* \forces \dot{f}_G^b (\check{x}_\alpha) + 1 \leq^* \dot{f}_G^a (\check{x}_\alpha)$, but this contradicts the choice of $p$ and $N$.

Finally to see that the mapping is cofinal, let $f \in \bbb$. By the $\aleph_2$-c.c. there is some name $\dot{f}$ so that $\dot{f}_G = f$ and some $X \subseteq \mathbb Q$ of size $\aleph_1$, so that the conditions needed to decide $\dot{f}$ all have supports included in $X$. Since $X$ has size $\aleph_1$, it is bounded by assumption by some $b \in \mathbb Q$ so we can conclude that $f$ is equivalent to a $\mathbb D(\mathbb Q)_b$ name. But then $f \leq^*_\Null f_G^b$ so we're done. 
\end{proof}

Putting together these lemmas, the rest of the proof of Theorem \ref{mainthmhechler} is relatively straightforward. Let me record the details below.
\begin{proof}[Proof of Theorem \ref{mainthmhechler}]
Fix $G$ as in the lemma above. By Lemma \ref{cof} in $V[G]$ there is a cofinal embedding of $\mathbb Q$ into $(\bbb, \leq^*_\Null)$ and so $\mfb(\leq^*_\Null) = \mfb (\mathbb Q)$ and $\mfd (\leq^*_\Null) = \mfd (\mathbb Q)$. By the fact that the forcing is $\sigma$-closed and has the $\aleph_2$-c.c. it follows that in $V[G]$ $\mfb (\leq^*) = \kappa$ and $\mfd (\leq^*) = \lambda$. Finally, assuming that $|\mathbb{Q}| = \lambda$, like in the example given above of $\mathbb Q = \kappa \times [\lambda]^{< \kappa}$, we can apply a nice name counting argument to also get that $2^\mfc = \lambda$. 
\end{proof}

Let me also observe that the proof of this theorem gives slightly more, in fact it gives a weakened higher dimensional version of Hechler's classical theorem on $\leq^*$, see the remark preceding Theorem 2.5 of \cite{BlassHB}.

\begin{corollary}
Assume $\GCH$ and let $\mathbb Q$ be any well-founded partial order so that $\aleph_2 \leq \mfb (\mathbb Q) \leq \mfd (\mathbb Q)$ with $\mfb (\mathbb Q)$ and $\mfd (\mathbb Q)$ regular. Then it's consistent that $\mathbb Q$ embeds cofinally into $(\bbb, \leq^*_\Null)$.
\end{corollary}

%I do not know if in either the model constructed in Theorem \ref{mainthmhechler} or the model constructed in Theorem \ref{mainthmhechler2} we have that $\mfc^+ = \mfb (\in^*_\Null)$, though I suspect that this is the case. 

J. Brendle has shown that under $\CH$ $\mfb (\in^*_\Null) = \mfb (\leq^*_\Null)$ so in the models constructed in Theorems \ref{mainthmhechler} and \ref{mainthmhechler2} $\mfb(\in^*_\Null)$ increases over the iteration. I do not know whether there is a way to iterate, necessarily while adding reals, to avoid this. However, these cardinals can be different, again shown in joint work with J. Brendle. Whether one can iterate $\mathbb D_\Null$ in some way without increasing $\mfb (\in^*_\Null)$ is open, however in the one step case the following is promising.

%To prove this, one would need to show that at no stage in either iteration is a function $h:\baire \to \mathcal S$ added so that if $f:\baire \to \baire$ is in $V$ then $f \in^*_\Null h$. If this were the case then $\bbb \cap V$ would witness that $\mfb (\in^*_\Null) = \aleph_2$. While I do not know how to prove this, I do know it to be the case for the one step.

\begin{lemma}
If $\dot{h}$ is a $\mathbb{D}_\Null$ name for a function and $d \in \mathbb D_\Null$ is a condition forcing $\dot{h}:\baire \to \mathcal S$ then there is a $d ' \leq d$ and a ground model function $f:\baire \to \baire$ so that $d ' \forces \check{f} \notin^*_\Null \dot{h}$. In particular, in the extension by $\mathbb{D}_\Null$ the set $\bbb \cap V$ remains unbounded with respect to $\in^*_\Null$.
\end{lemma}

\begin{proof}
Suppose not and let $d \in \mathbb{D}_\Null$ be a condition and $\dot{h}$ be a name so that $d \forces$\say{$\dot{h} : \baire \to \mathcal S$ and for all $f \in \bbb \cap V$ $\check{f} \in^*_\Null \dot{h}$}. Pick an enumeration in order type $\omega_1$ of all pairs $(p_\alpha, N'_\alpha)$ so that $p_\alpha$ is the stem of a condition extending $d$ and $N'_\alpha$ is a Borel null set. Using Lemma \ref{mainlemmafunctions} we can assume that $x_\alpha \notin N^{p_\alpha} \cup N'_\alpha$. For each $\alpha$ let $y_\alpha$ be the least $x_\gamma$ not in the set $\{x \; | \; \exists \beta \leq \alpha \exists \mathcal F \, (p_\beta, \mathcal F) \forces \check{x} \in^* \dot{h}(\check{x}_\alpha )\}$. Note that this is well defined since first of all every condition with the same stem is compatible and for any slalom the set of $x$ eventually captured by that slalom is measure zero hence the set above must be measure zero. Let $f \in \bbb$ be the function so that $f(x_\alpha) =y_\alpha$. Now suppose that there is some $d' \leq d$ and a Borel null set $N$ so that $d' \forces \{x \; | \; \check{f}(x) \notin^* \dot{h} (x) \} \subseteq \check{N}$. The stem of $d'$ and the set $N$ appeared enumerated as a pair $(p_\xi, N'_\xi)$ and $d' \forces f(x_\xi) \in^* \dot{h} (x_\xi)$ since $x_\xi \notin N'_\xi = N$. But then by definition of $f$ it must be the case that $f(x_\xi)$ is not in the set $\{x \; | \; \exists \beta \leq \alpha \exists \mathcal F \, (p_\beta, \mathcal F) \forces \check{x} \in^* \dot{h}(\check{x}_\alpha )\}$, which is a contradiction.
\end{proof}

\subsection{Generalizing $\mathbb{LOC}$-Forcing}
In the final subsection here I prove the consistency of having all the cardinals in the diagram arbitrarily large. The relevant forcing is a generalization of the $\mathbb{LOC}$-forcing. 

\begin{definition}
The $\Null$-$\mathbb{LOC}$ forcing, denoted $\mathbb{LOC}_\Null$, is the set of all pairs $(p, \mathcal F)$ so that $p:{\rm dom}(p) \subseteq \baire \to \mathcal S$ is a partial function with a Borel graph and a Borel domain which is measure zero and $\mathcal F \subseteq \bbb$ is countable. We let $(p, \mathcal F) \leq (q, \mathcal G)$ if and only if $p \supseteq q$, $\mathcal F \supseteq \mathcal G$ and for all $x \in {\rm dom}(p) \setminus {\rm dom} (q)$ we have that $g(x) \in^* p(x)$ for every $g \in \mathcal G$.
\end{definition}

Using the same template as with $\mathbb{D}_\Null$ it is straightforward to show the following.
\begin{lemma}
$\mathbb{LOC}_\Null$ is $\sigma$-closed, has the $\mfc^+$-c.c. and adds a function $h:\baire \to \mathcal S$ so that for every $f \in \bbb \cap V$ $f\in^*_\Null h$.
\end{lemma}

As a result of this lemma, using the same ideas as before we get immediately.
\begin{theorem}
Let $\kappa$ be a regular cardinal greater than $\aleph_1$ and let $\mathbb P_\kappa$ be the countable support iteration of $\mathbb{LOC}_\Null$. Then if $G \subseteq \mathbb P_\kappa$ is generic over $V$ in $V[G]$ we have $\mfb (\in^*_\Null) = \kappa = 2^\mfc$.
\end{theorem}

\section{Conclusion and Questions}

The consistency results above barely hint at the possible constellations of the 18 cardinals considered. Restricted to one ideal, there are a number of splits I have yet to show. Nevertheless, I am willing to conjecture that every split in the diagram is consistent.
\begin{question}
Is every two valued split in the diagram for one ideal is consistent? What about simultaneously splitting all 18 cardinals?
\end{question}
Presumably this would involve developing analogues for well known forcing notions on the reals such as Sacks, Laver etc as I have done for Cohen, Hechler and $\mathbb{LOC}$. I leave this project for future research. 

Since writing this chapter, I have worked jointly on this extensively with J. Brendle and we have computed the values of these cardinals in standard models of $\neg \CH$ such as the Cohen model, the random model, the Sacks model etc. We have shown that many interesting things happen to the $\mfb (R_\mathcal I)$ cardinals, yet the $\mfd (R_\mathcal I)$ cardinals all stay $\mfc^+$. The following is an interesting line of research we are currently pursuing.
\begin{question}
Can the techniques of the previous section be woven with forcing on the reals to produce models where the $\mfb (R_\mathcal I)$ cardinals are small and the $\mfd (R_\mathcal I)$ cardinals are larger than $\mfc^+$?
\end{question}

Also, the looming problem in this subject is that no model we have found so far splits any $\mfd (R_\mathcal I)$ cardinal. 
\begin{question}
Is it consistent that any two $\mfd(R_\mathcal I)$ cardinals are different?
\end{question}

Finally let me conclude by noting that, as mentioned in the introduction, the framework introduced is very flexible and many other generalizations are possible. For instance, while I have been working with Baire space, a similar study could easily be carried out for any other uncountable Polish space. One particularly interesting possibility, which I leave for future work, is to consider variations on $\mathfrak{a}$ where $[\omega]^\omega$ is replaced by the $\mathcal I$-positive sets of some Polish space and \say{almost disjoint} means that such sets have intersection in $\mathcal I$. A generalization in this spirit for ideals on $\omega$ has been considered in \cite{HST18}.
\clearpage

\chapter{Iterating Subversion Forcing Notions and the Subcomplete Forcing Axiom}
\chaptermark{Iterating Subversion Forcing Notions}
The purpose of this chapter and the next is to study models of forcing axioms compatible with $\CH$. Traditionally such axioms have been studied only in conjunction with $\CH$, often to show that various statements were compatible with $\CH$ even if the original models of their consistency necessarily involved the failure of $\CH$. See the survey article \cite{AbrahamHB} particularly section 5 for a discussion of this point. In this chapter and the next however, I look at building models of such axioms alongside controlled failures of $\CH$, for instance keeping various cardinal characteristics small. The result is many new models of such axioms with a finer understanding of the possible behavior of combinatorics on $\omega_1$ and the reals under these axioms. In this chapter I focus on Jensen's subcomplete forcing axiom, $\SCFA$ see \cite{JensenSPSC}, and in the next I look at Shelah's Dee Complete forcing axiom, $\DCFA$, see \cite{PIPShelah} and \cite{JensenCH}. Both axioms were originally considered only in conjunction with $\CH$ but here I look at them with this assumption removed, thus giving more flexibility. 

Subproper and subcomplete forcing were originally introduced by Jensen in \cite{JensenSPSC} and many striking applications were given in a series of notes culminating in \cite{Jen14}. The PhD thesis \cite{Minden:dissertation} studied both the foundational properties of subcomplete forcing and its axioms in more detail. The forcing axiom for subcomplete forcing, $\SCFA$ is notable in that it has many of the strong consequences of $\MM$ such as failure of vaious square principles (see \cite{Jen14, FuchsFACHSP, Fuchs17, FuchsRinot})  but is consistent with $\CH$ and even $\diamondsuit$. This follows immediately from the fact that subcomplete forcing never adds reals, nor kills diamond sequences. Mixed with Theorem 3 of Chapter 4 of \cite{Jen14} which states that RCS iterations of subcomplete forcing notions are again subcomplete it follows that, if $\kappa$ is supercompact, one can force $\SCFA$ by a $\kappa$-length RCS iteration using the standard Baumgartner argument. Following Jensen I'll call this \say{the natural model of $\SCFA$}.

Every subcomplete forcing is stationary set preserving hence $\MM$ implies $\SCFA$. In particular, $\SCFA$ is compatible with $\mfc = \aleph_2$. There are therefore at least two essentially different models of $\SCFA$: the natural model where $\diamondsuit$ holds, and any model of $\MM$, where $\CH$ fails. With this fact in mind it is very easy to see that $\SCFA$ does not suffice to prove either any number of standard consequences of forcing axioms, nor their negations: for instance $\SCFA$ cannot prove there are no Souslin trees, since such trees exist under $\diamondsuit$. But this feels cheap. What one would like to know is whether the $\SCFA$ is compatible with Souslin trees even if the $\diamondsuit$ or $\CH$ obstruction is removed. These questions are what I would like to address in this chapter. More specifically I aim to answer the following question.
\begin{question}
What possible behaviors for the reals and combinatorics on $\omega_1$ are consistent with $\SCFA$, regardless of if $\CH$ holds?
\end{question}

I answer this question in a number of ways, ultimately showing the following, amongst other such results.
\begin{theorem}
Assuming the consistency of a supercompact cardinal, the following are consistent with $\SCFA + \neg \CH$.
\begin{enumerate}
\item
There are Souslin trees.
\item
$\mfd = \aleph_1 < \mfc = \aleph_2$ %%add more
\item
$\MA_{\aleph_1}(\sigma{\rm -linked})$ holds but $\MA_{\aleph_1}$ fails.
\end{enumerate}
\label{mainthm}
\end{theorem}

The idea of the proof of Theorem \ref{mainthm} is to \say{weave} other forcing notions into the iteration to force $\SCFA$ in such a way that $\diamondsuit$ is killed but some fixed consequence remains. This involves proving new iteration theorems for subcomplete forcing notions and their relatives, this time using nice iterations in the sense of Miyamoto \cite{miyamoto}. After these theorems were proved, Gunter Fuchs observed that they apply to a much larger class of forcing notions, one that previously was not known to have an iteration theorem. Riffing off the idea of $\varepsilon$-subcompleteness, introduced by Fuchs in \cite{FuchsPSC}, we call these iterations $\infty$-subcomplete. 

The results stated in the above theorem apply equally for the $\infty$-subcomplete forcing axiom, however, somewhat awkwardly, we don't know if it's the case that every $\infty$-subcomplete forcing is forcing equivalent to a subcomplete forcing or even if $\SCFA$ is equivalent to the forcing axiom for $\infty$-subcomplete forcing notions. Still $\infty$-subcomplete forcing notions represent an exciting development in the theory of subversion forcing for a number reasons. First, the definition is simplified from that of subcomplete forcing, yet all previous applications of subcomplete forcing work with $\infty$-subcomplete forcing, hence simplifying the theory as a whole. Next, the class of $\infty$-subcomplete forcing notions appears to be much nicer structurally than that of subcomplete forcing notions because it satisfies certain closure properties we expect from a \say{well-behaved} forcing class, though again this may be cosmetic if it turns out that subcomplete forcing and $\infty$-subcomplete forcing are the same. In particular, we can show that $\infty$-subcomplete and $\infty$-subproper forcing notions are closed under forcing equivalence and factors whereas the corresponding facts for subcomplete and subproper forcing are not known. 

In the next section of this chapter, I introduce $\infty$-subcomplete and $\infty$-subproper forcing notions and prove a number of simple results about their classes. In the following section I review Miyamoto's theory of nice iterations. This sets the stage for the following section where I prove the iteration theorems. The final two sections give applications. I prove some preservation theorems and use them to build models of the statements advertized in Theorem \ref{mainthm}.

\section{$\infty$-Subversion Forcing}
The first goal is to introduce $\infty$-subcomplete and $\infty$-subproper forcing. I follow closely the presentation given in \cite[Chapters 1 and 2]{Minden:dissertation} and refer the reader there for more details. 
\begin{definition}
Given transitive structures $M, N \models \ZFC^-$ we say that $N$ is {\em regular in} $M$ if given any function $f \in M$ so that $f:x \to N$ with $x \in N$ we have that $f``x \in N$. A structure $N$ is {\em full} if it is transitive, $\omega \in N$ and there is an ordinal $\gamma$ so that $N$ is regular in $L_\gamma (N)$ where $L_\gamma (N) \models \ZFC^-$. 
\end{definition}
The idea of fullness is that $N$ is some $H_\theta$ from the point of view of $L_\gamma (N)$. In fact the following holds.
\begin{lemma}[Lemma 1.2.5 of \cite{Minden:dissertation}]
Given $M, N \models \ZFC^-$ transitive structures. $N$ is regular in $M$ if and only if $N = H_\gamma^M$ where $\gamma = {\rm height}(N)$ is a regular cardinal in $M$.
\end{lemma}
I will be using full models to define $\infty$-subcomplete and $\infty$-subproper forcing notions. The relevant fact is the following.
\begin{lemma}[Lemma 1.2.7 of \cite{Minden:dissertation}]
If $M$ is countable and full then $M$ is not pointwise definable.
\end{lemma}

Let $\mathbb P$ be a forcing notion. Following {\cite[Definition 2.1.3]{Minden:dissertation}} I recall the {\em standard setup}.
\begin{definition}
Let $\theta$ be a cardinal. We say that $\mathbb P, N, \sigma, \bN$ and $\theta$ are {\em in the standard setup} if the following all hold.
\begin{enumerate}
\item
$\mathbb P \in H_\theta \subseteq N = L_\tau[A] \models \ZFC^-$ for some $\tau > \theta$ and $A \subseteq \tau$.
\item
$\sigma:\bN \cong X \prec N$ where $X$ is countable and $\bN$ is full
\item
$\sigma(\bar{\theta}, \bar{\mathbb P}, \bs) = \theta, \mathbb P, s$ for any fixed parameter $s \in N$.
\end{enumerate}
\end{definition}

Given an embedding $\sigma:\bN \prec N$ as in the standard setup I employ the following convention: if $a \in {\rm range}(\sigma)$ then I will often write $\bar{a}$ for the preimage of $a$. Often this will not even be explicitly mentioned. The standard setup is used to define subcomplete and subproper forcing. Below, $\delta (\mathbb P)$ is the {\em weight} of $\mathbb P$ i.e. the least size of a dense subset.
\begin{definition}[Subproperness]
A forcing notion $\mathbb P$ is {\em subproper} if all sufficiently large $\theta$ verify the subproperness of $\P$, meaning that if $\mathbb P$, $N$, $\sigma$, $\bN$ and $\theta$ in the standard setup, $s \in N$ any parameter fixed in advance and $\bar{p} \in \bar{\mathbb P}$ then there is a condition $q \leq p$ so that if $G$ is any $\mathbb P$-generic containing $q$ there is an embedding $\sigma ' \in V[G]$ so that the following hold:
\begin{enumerate}
\item
$\sigma ' : \bar{N} \prec N$
\item
$\sigma ' (\bar{\theta}, \bar{\mathbb P}, \bar{s}, \bar{p} ) = \theta, \mathbb P, s, p$
\item
$\bar{G} : = (\sigma '{}^{-1}) `` G$ is a $\bar{\mathbb P}$-generic filter over $\bar{N}$
\item
${\rm Hull}^N (\delta (\mathbb P) \cup {\rm range} (\sigma ')) = {\rm Hull}^N (\delta (\mathbb P) \cup {\rm range}(\sigma))$
\end{enumerate}
\end{definition}

Note that if $\mathbb P$ is subproper then the embedding $\sigma '$ induces an embedding $\hat{\sigma} ':\bN[\bG] \prec N[G]$ by specifying $\hat{\sigma} (\bG) = G$ and extending by elementarity. Also note that subproperness is a weakening of properness, the latter being the case where $\sigma = \sigma '$, see Proposition 2.1.2 of \cite{Minden:dissertation} for a proof of this fact. Such a weakening for a general cass of forcing notions $\Gamma$ gives a new class, which we call {\em sub}-$\Gamma$. We also describe this process as {\em subverting} $\Gamma$. The subversion of $\sigma$-closed posets, which uses the forcing equivalent definition of completeness, is the most interesting example of a subverted forcing class. 
\begin{definition}[Subcompleteness]
A forcing notion $\mathbb P$ is {\em subcomplete} if all sufficiently large $\theta$ verify the subcompleteness of $\P$, meaning that if $\mathbb P$, $N$, $\sigma$, $\bN$ and $\theta$ in the standard setup, $s \in N$ any parameter fixed in advance, and $\bar{p} \in \bar{G} \subseteq \bar{\mathbb P}$ is $\bar{\P}$-generic over $\bar{N}$ then there is a $q \in \P$ so that if $G \ni q$ is $\P$-generic over $V$ then there is an embedding $\sigma ' \in V[G]$ so that the following hold:
\begin{enumerate}
\item
$\sigma ' : \bar{N} \prec N$
\item
$\sigma ' (\bar{\theta}, \bar{\mathbb P}, \bar{s}, \bar{p} ) = \theta, \mathbb P, s, p$ for any fixed parameter $s \in N$
\item
$\sigma ' ``\bar{G} \subseteq G$
\item
${\rm Hull}^N (\delta (\mathbb P) \cup {\rm range} (\sigma ')) = {\rm Hull}^N (\delta (\mathbb P) \cup {\rm range}(\sigma))$
\end{enumerate}

\end{definition}

In both the definition of subproper and subcomplete forcing the most mysterious condition is the third one, which I call the \say{Hulls condition}. This condition is not needed in any application of subcomplete or subproper forcing, but it is needed to prove the iteration theorem for both classes. One of the main contributions of this chapter is to show that in fact the Hulls condition is not necessary for iterating subversion forcing classes. The two more general iterable classes of forcing notions are given below for completeness, however note that these are the same with the Hulls condition removed.

\begin{definition}[$\infty$-Subproperness]
A forcing notion $\mathbb P$ is {\em $\infty$-subproper} if all sufficiently large $\theta$ verify the subproperness of $\P$, meaning that if $\mathbb P$, $N$, $\sigma$, $\bN$ and $\theta$ in the standard setup, $s \in N$ any parameter fixed in advance and $\bar{p} \in \bar{\mathbb P}$ then there is a condition $q \leq p$ so that if $G$ is any $\mathbb P$-generic containing $q$ there is an embedding $\sigma ' \in V[G]$ so that the following hold:
\begin{enumerate}
\item
$\sigma ' : \bar{N} \prec N$
\item
$\sigma ' (\bar{\theta}, \bar{\mathbb P}, \bar{s}, \bar{p} ) = \theta, \mathbb P, s, p$
\item
$\bar{G} : = (\sigma '{}^{-1}) `` G$ is a $\bar{\mathbb P}$-generic filter over $\bar{N}$
\end{enumerate}
\end{definition}

\begin{definition}[$\infty$-Subcompleteness]
A forcing notion $\mathbb P$ is {\em $\infty$-subcomplete} if all sufficiently large $\theta$ verify the subcompleteness of $\P$, meaning that if $\mathbb P$, $N$, $\sigma$, $\bN$ and $\theta$ in the standard setup, $s \in N$ any parameter fixed in advance, and $\bar{p} \in \bar{G} \subseteq \bar{\mathbb P}$ is $\bar{\P}$-generic over $\bar{N}$ then there is a $q \in \P$ so that if $G \ni q$ is $\P$-generic over $V$ then there is an embedding $\sigma ' \in V[G]$ so that the following hold:
\begin{enumerate}
\item
$\sigma ' : \bar{N} \prec N$
\item
$\sigma ' (\bar{\theta}, \bar{\mathbb P}, \bar{s}, \bar{p} ) = \theta, \mathbb P, s, p$ for any fixed parameter $s \in N$
\item
$\sigma ' ``\bar{G} \subseteq G$
\end{enumerate}

\end{definition}

The name $\infty$-subversion refers to the $\varepsilon$-subcomplete forcings defined in \cite{FuchsPSC}. In that article a forcing notion $\mathbb P$ is said to be $\varepsilon$-subcomplete if it satisfies the definition of subcompleteness with $\delta(\mathbb P)$ being replaced by a given ordinal $\varepsilon$. Clearly as $\varepsilon$ increases, being $\varepsilon$-subcomplete becomes easier to satisfy. Since $\infty$-subcomplete forcings are defined without mention of the $\delta(\mathbb P)$ there are the weakest of these classes, and for any $\varepsilon \in ORD$ an $\varepsilon$-subcomplete forcing is $\infty$-subcomplete, hence the name. Let me note that in \cite{FuchsPSC} it's shown that the forcing notions which are $\varepsilon$-subcomplete for some $\varepsilon$ collectively form the closure of the subcomplete forcing notions under closure by forcing equivalence. I do not know if the same is true of the $\infty$-subcomplete forcing notions.

The main utility of working with $\infty$-subversions of forcing notions as opposed to their \say{non-infinity} counterparts is that these are easier to keep track of since they have less conditions to check, while at the same time it does not seem that they have any substantially different properties. Indeed essentially all of the standard facts about subcomplete and subproper forcing also hold for the infinity versions since the hulls condition is not invoked in any of these proofs. Let me list some key ones below as an observation. %Since the proofs are completely  unchanged I omit them.

\begin{observation}
Let $\P$ be $\infty$-subcomplete.
\begin{enumerate}
	\item $\P$ preserves stationary subsets of $\omega_1$.
	\item $\P$ preserves Souslin trees.
	\item $\P$ preserves the principle $\diamondsuit$.
	\item $\P$ does not add reals.
\end{enumerate}
\end{observation}

Also $\infty$-subcomplete and $\infty$-subproper are slightly more \say{natural} in the sense that they are closed under forcing equivalence and subforcing, whereas the corresponding facts for subcomplete and subproper forcings are less clear. Note that both propositions below are nearly identical to the corresponding facts for $\varepsilon$-subcomplete forcings as shown in \cite{FuchsPSC}. I repeat them here for completeness.

\begin{proposition}[Essentially Lemma 2.3 of \cite{FuchsPSC}]
$\infty$-subcomplete and $\infty$-subproper forcings are closed under forcing equivalence.
\end{proposition}

\begin{proof}
I do the case of $\infty$-subcomplete forcing, the other being similar. Let $\mathbb P$ be $\infty$-subcomplete and $\mathbb Q$ be forcing equivalent to $\mathbb P$. Let $\theta$ be large enough to verify the $\infty$-subcompleteness of $\mathbb P$ and assume $\mathcal P(\mathbb Q \cup \mathbb P) \in H_\theta$. Let $\barN, N, \sigma$, etc be as in the definition of $\infty$-subcompleteness. Then there is a condition $p \in \mathbb P$ so that if $p \in G$ is $\mathbb P$-generic over $V$ then in $V[G]$ there is an embedding $\sigma ' : \barN \prec N$ as in the definition of $\infty$-subcompleteness. But then by forcing equivalence there are $\bar{H}$ and $H$ which are $\bar{\mathbb Q}$ and $\mathbb Q$-generic over $\barN$ and $V$ (and hence $N$) respectively and $\sigma '$ lifts to an embedding $\sigma ' _H: \barN[\bar{H}] \prec N[H]$, which by elementarity has the property that $\sigma '_H (\bar{H}) = H$ and is therefore generic over $V$. Moreover this implies that $G \in V[H]$ so $\sigma ' \in V[H]$. Thus there must be some $q' \in \mathbb Q$ forcing this situation so we are done.
\end{proof}

\begin{proposition}[Essentially Theorem 2.7 of \cite{FuchsPSC}]
$\infty$-subcomplete and $\infty$-subproper forcings are closed under factors. In other words, if $\mathbb P$ is a poset, $\dot{\mathbb Q}$ is a $\mathbb P$-name for a poset and $\mathbb P * \dot{\mathbb Q}$ is $\infty$-subcomplete ($\infty$-subproper) then $\mathbb P$ is $\infty$-subcomplete ($\infty$-subproper). 
\end{proposition}

The following proof, which greatly simplifies my original one (itself, essentially copied from that of \cite[Theorem 2.7]{FuchsPSC}), was suggested to me by Gunter Fuchs.

\begin{proof}
Again, I do the case of $\infty$-subcomplete forcings, the other being similar. Let $\theta$ be large enough to verify that $\mathbb P * \dot{\mathbb Q}$ is $\infty$-subcomplete. We claim that it is also large enough to verify that $\mathbb P$ is subcomplete. To see this, let $N = L_\tau[A]$, be a ZFC$^-$ model with $\tau > \theta$ regular, and $H_\theta \subseteq N$. Fix a parameter $s \in N$ and let $\sigma:\bN \prec N$ with $\bN$ countable, transitive and full so that $\sigma (\bar{\P}, \dot{\bar{\mathbb Q}}, \bar{\theta}, \bar{s}) = \mathbb P, \dot{\mathbb Q}, \theta, s$. Let $\bar{G} \subseteq \bar{\P}$ be generic over $\bN$. Let $(p, \dot{q})$ be a condition witnessing the $\infty$-subcompleteness of $\mathbb P * \dot{\mathbb Q}$ and let $(p, \dot{q}) \in G * H$ be $\mathbb P * \dot{\mathbb Q}$-generic over $\V$. Work in $\V[G*H]$ and let $\sigma ':\bar{N} \prec N$ be an embedding so that $\sigma ' (\bar{\P}, \dot{\bar{\mathbb Q}}, \bar{\theta}, \bar{s}) = \mathbb P, \dot{\mathbb Q}, \theta, s$ and $\sigma ' ``\bar{G} \subseteq G$. Fixing an enumeration of $\bN$ in order type $\omega$, we can consider the tree $T_{G}$ of finite initial segments of an elementary embedding $\sigma_0  : \bN \prec N$ with $\sigma_0 (\bar{\P}, \dot{\bar{\mathbb Q}}, \bar{\theta}, \bar{s}) = \mathbb P, \dot{\mathbb Q}, \theta, s$ and so that $\sigma_0 ' `` \bar{G} \subseteq G$. Note that this tree is in fact in $\V[G]$. Moreover, in $\V[G*H]$ it's ill-founded since $\sigma '$ generates an infinite branch. But then by the absoluteness of ill-foundedness, $T_G$ is ill-founded in $\V[G]$. So there is an infinite branch in $\V[G]$ and this branch witnesses that $\mathbb P$ is $\infty$-subcomplete.
\end{proof}

\section{A Subpar Primer on Nice Iterations}

I will prove iteration theorems for $\infty$-subcomplete and $\infty$-subproper forcing notions in the next section. The limit construction used in the iteration theorems is that of nice iterations in the sense of Miyamoto, \cite{miyamoto}. To begin, I collect here first the relevant facts and definitions from \cite{miyamoto}. For a more in depth discussion, including proofs, see that article. For basic notions of projection etc, see the introduction there. For a sequence $x$ I denote its length by $l(x)$.
\begin{definition}[Iterations]
Let $\nu$ be a limit ordinal. A sequence of separative partial preorders%
\footnote{Here, $(\P,\le,1)$ is a partial preorder if $(\P,\le)$ is reflexive and transitive, and if $1$ is a greatest element. There may be several such greatest elements since $\P$ is not required to be antisymmetric. If $p\le q$ and $q\le p$, then I write $p\equiv q$.}
of length $\nu$, $\langle (\P_\alpha, \leq_\alpha, 1_\alpha)\; | \; \alpha < \nu\rangle$ is called a {\em general iteration} if and only if for any $\alpha \leq \beta < \nu$ the following holds
\begin{enumerate}
\item
For any $p \in \P_\beta$, $p \upharpoonright \alpha \in \P_\alpha$ and $1_\beta \upharpoonright \alpha = 1_\alpha$.
\item
For any $p \in \P_\alpha$ and any $q \in \P_\beta$, if $p \leq_\alpha q \upharpoonright \alpha$ then $p^\frown q \upharpoonright [\alpha, \beta) \in \P_\beta$ and $p^\frown q \upharpoonright [\alpha, \beta) \leq_\beta q$
\item
For any $p, q \in \P_\beta$, if $p \leq_\beta q$ then $p \hook \alpha \leq_\alpha q \hook \alpha$ and $p\leq_\beta p \hook \alpha^\frown q\hook[\alpha, \beta)$.
\item
If $\beta$ is a limit ordinal and $p, q \in \P_\beta$, $p \leq_\beta q$ if and only if for all $\alpha < \beta$ $p \hook \alpha \leq_\alpha q \hook \alpha$.
\end{enumerate}
A general iteration $\langle (\P_\alpha, \leq_\alpha, 1_\alpha)\; | \; \alpha < \nu\rangle$ is an \emph{iteration} iff for every limit ordinal $\beta<\nu$ and all $p,q\in \P_\beta$, $p\le_\beta q$ iff for all $\alpha<\beta$, $p\rest\alpha\le_\alpha q\rest\alpha$.
\end{definition}

I will use the following fact (and the notation introduced there).

\begin{fact}[see {\cite[Prop.~1.3]{miyamoto}}]
\label{fact:FactorsAndQuotients}
Let $\langle (\P_\alpha, \leq_\alpha, 1_\alpha)\; | \; \alpha < \nu\rangle$ be a general iteration, and let $\alpha\le\beta<\nu$. Then
\begin{enumerate}[label=(\arabic*)]
  \item Let $G_\beta$ be $\P_\beta$-generic over $\V$. Set $G_\beta\rest\alpha=\{p\rest\alpha\st p\in G_\beta\}$, $G_\beta\rest[\alpha,\beta)=\{p\rest[\alpha,\beta)\st p\in G_\beta \}$, and let $\P_{\alpha,\beta}=\P_\beta/(G_\beta\rest\alpha)=\{p\rest[\alpha,\beta)\st p\in \P_\beta\ \text{and}\ p\rest\alpha\in G_\beta\rest\alpha\}$ be equipped with the ordering $p\le_{\alpha,\beta}q$ iff there is an $r\in G_\beta\rest\alpha$ such that $r\verl p\le_\beta r\verl q$.
      Then $G_\beta\rest\alpha$ is $\P_\alpha$-generic over $\V$ and $G_\beta \rest [\alpha, \beta)$ is $\P_{\alpha,\beta}$-generic over $\V[G_\beta\rest\alpha]$.
  \item If $G_\alpha$ is $\P_\alpha$-generic over $\V$ and $H$ is $\P_{\alpha,\beta}=\P_\beta/G_\alpha$-generic over $\V[G_\alpha]$, then $G_\alpha*H=\{p\in \P_\alpha\st p\rest\alpha\in G_\alpha\ \text{and}\ p\rest[\alpha,\beta)\in H\}$ is $\P_\beta$-generic over $\V$, $(G_\alpha*H)\rest\alpha=G_\alpha$ and $(G_\alpha*H)\rest[\alpha,\beta)=H$.
  \item
  \label{item:CriterionForBeingInGeneric}
  Let $G_\beta$ be $\P_\beta$-generic over $\V$. Then a condition $p\in \P_\beta$ is in $G_\beta$ iff $p\rest\alpha\in G_\beta\rest\alpha$ and $p\rest[\alpha,\beta)\in G_\beta\rest[\alpha,\beta)$.
\end{enumerate}
\label{factPij}
\end{fact}

In what follows I suppress the notation $\leq_\alpha, 1_\alpha$ and associate a partial preorder with its underlying set. The following proposition is useful.
\begin{proposition}[Proposition 1.7 of \cite{miyamoto}]
Let $\langle \P_\alpha \; | \; \alpha < \nu\rangle$ be an iteration and $\beta < \nu$ limit. Then for any $p \in \P_\beta$ and any $\P_\beta$-generic $G_\beta$ we have $p \in G_\beta$ if and only if for all $\alpha < \beta$, $p \hook \alpha \in G_\beta \hook \alpha$.
\label{generic}
\end{proposition}

From now on I always assume that the sequence $\vec{\P} = \langle \P_\alpha \; | \; \alpha < \nu\rangle$ is an iteration. Let me define the following building blocks of nice iterations: nested antichain, $S \hooks T$, mixtures and the property of being $(T, \beta)$-nice. Again, the reader is referred to \cite{miyamoto} for more information about these ideas and their significance.

\begin{definition}[The machinery of Nice Iterations]
\label{definition:NestedAC-Mixture-beta-Nice-hooks}
A {\em nested antichain} in $\vec{\P}$ is a triple $\langle T, \langle T_n \; | \; n < \omega\rangle, \langle {\rm suc}^n_T \; | \; n < \omega\rangle \rangle$ so that
\begin{enumerate}
\item
$T = \bigcup_{n < \omega} T_n$
\item
$T_0$ consists of a unique element of some $\P_\alpha$ for $\alpha < \nu$

\noindent For each $n < \omega$ we have that

\item $T_n \subseteq \bigcup\{\P_\alpha \; | \; \alpha < \nu\}$ and ${\rm suc}^n_T: T_n \to \mathcal P(T_{n+1})$

\item For $a \in T_n$ and $b \in {\rm suc}^n_T (a)$, $l(a) \leq l(b)$ and $b \hook l(a) \leq a$

\item For $a \in T_n$ the set of all $b\hook l(a)$ so that $b \in {\rm suc}^n_T(a)$ forms a maximal antichain below $a$ in $\P_{l(a)}$. In particular any two elements in this set are incompatible and it is non-empty.

\item $T_{n+1} = \bigcup \{ {\rm suc}^n_T (a) \; | \; a \in T_n\}$
\end{enumerate}

Given such a nested antichain $T$ in $\vec{\P}$ with $\{a_0\} = T_0$, for a condition $p \in \P_\beta$ with $\beta < \nu$ we say that $p$ is {\em a mixture of $T$ up to }$\beta$ if for all $i < \beta$ the condition $p \hook i$ forces that
\begin{enumerate}
\item
$p \hook [i, i+1) \equiv a_0 \hook [i, i+1)$ if $i < l(p_0)$ and $a_0 \hook i \in \dot{G}_i$, the canonical name for the $\P_i$-generic.
\item
$p \hook [i, i+1) \equiv s \hook [i, i+1)$ if there is $(r, s)$ so that $r, s \in T$ with $s \in {\rm suc}^n_T(r)$ for some $n$ and $l(r) \leq i < l(s)$ and $s\hook i \in \dot{G}_i$.
\item
$p \hook [i, i+1) \equiv 1_{i+1} \hook [i, i+1)$ if there is a sequence $\langle a_n \; | \; n < \omega \rangle$ so that $a_0 \in T_0$ and for all $n < \omega$ $a_{n+1} \in {\rm suc}^n_T(a_n)$ and $l(a_n) \leq i$ and $a_n \in \dot{G}_i \hook l(a_n)$.
\end{enumerate}

If $\beta$ is a limit ordinal we say that a sequence $p$ of length $\beta$ (not necessarily in $\mathbb P_\beta$) is $(T, \beta)$-{\em nice} if for all $\alpha < \beta$, $p \hook \alpha$ is a mixture of $T$ up to $\alpha$.

Finally, given two nested antichains $S$ and $T$ in $\vec{\P}$ we define $S \hooks T$ (``$S$ hooks $T$'') if for every $n < \omega$ and all $b \in S_n$ there is a $a \in T_{n+1}$ so that $l(a) \leq l(b)$ and $b\hook l(a) \leq a$.
\end{definition}

I will need the following characterization of mixtures.

\begin{fact}[see {\cite[Prop.~2.5]{miyamoto}}]
\label{fact:CharacterizationOfMixtures}
Let $T$ be a nested antichain in an iteration $\seq{\P_\alpha}{\alpha<\nu}$, $\beta<\nu$ and $p\in \P_\beta$. Then $p$ is a mixture of $T$ up to $\beta$ iff the following hold:
\begin{enumerate}[label=(\arabic*)]
    \item
    \label{item:Root}
    Let $T_0=\{a_0\}$ and $\mu=\min(l(a_0),\beta)$. Then $a_0\rest\mu\equiv p\rest\mu$.
    \item
    \label{item:General}
    For any $a\in T$, letting $\mu=\min(l(a),\beta)$, we have that $a\rest\mu\le p\rest\mu$.
    \item
    \label{item:HigherUp}
    If $n<\omega$, $a\in T_n$, $b\in\suc_T^n(a)$ and $l(a)\le\beta$, then, letting $\mu=\min(\beta,l(b))$, we have that $b\rest\mu\equiv b\rest l(a)\verl p\rest[l(a),\mu)$.
    \item
    \label{item:Chain}
    For any $i\le\beta$ and any $q\in \P_i$ with $q\le_i p\rest i$, if $q$ forces with respect to $\P_i$ that there is a sequence $\seq{a_n}{n<\omega}$ such that $a_0\in T_0$, and for all $n<\omega$, $a_{n+1}\in\suc_T^n(a_n)$, $l(a_n)\le i$ and $a_n\in\dot{G}_i\rest l(a_n)$, then $q\verl 1_\beta\rest[i,\beta)\equiv q\verl p\rest[i,\beta)$.
\end{enumerate}
\end{fact}

This previous definition combines Definitions 2.0, 2.4 and 2.10 of \cite{miyamoto}. The following is Definition 3.6 in Miyamoto's article.

\begin{definition}[Nice Iterations]
\label{def:NiceIterations}
An iteration $\langle \P_\alpha \; | \; \alpha < \nu\rangle$ is called {\em nice} if
\begin{enumerate}
\item
For any $i$ such that $i + 1 < \nu$ if $p \in \P_i$ and $\tau$ is a $\P_i$ name such that $p \forces_i$``$\tau \in \P_{i+1}$ and $\tau \hook i \in \dot{G_i}$'' then there is a $q \in \P_{i +1}$ so that $q \hook i = p$ and $p \forces_i\tau\hook[i, i+1) \equiv q \hook [i, i+1)$.
\item
For any limit ordinal $\beta < \nu$ and any sequence $x$ of length $\beta$, $x \in \P_\beta$ if and only if there is a nested antichain $T$ in $\langle \P_\alpha \; | \; \alpha < \beta\rangle$ such that $x$ is $(T, \beta)$-nice.
\end{enumerate}
\end{definition}

I will use the following facts.

\begin{lemma}[Lemma 2.7 of \cite{miyamoto}]
Let $\nu$ be a limit ordinal and $A \subseteq \nu$ be cofinal. Suppose that $T$ is a nested antichain in an iteration $\langle \P_\alpha \; | \; \alpha < \nu\rangle$ and $p$ is a sequence of length $\nu$ such that $p$ is $(T, \nu)$-nice. Then for any $\beta < \nu$ and any $s \in \P_\beta$ strengthening $p \hook \beta$ we get a nested antichain $S$ so that
\begin{enumerate}
\item
If $T_0 = \{a_0\}$ and $S_0 = \{b_0\}$ then $l(b_0) \in A$ and $l(a_0), \beta \leq l(b_0)$.
\item
For any $b \in S$, $l(b) \in A$.
\item
$r = s^\frown p \hook [\beta, \nu)$ is $(S, \nu)$-nice.
\end{enumerate}
\label{2.7}
\end{lemma}

\begin{lemma}[Lemma 2.11 of \cite{miyamoto}]
Let $\langle \P_\alpha \; | \; \alpha < \nu\rangle$ be an iteration with limit ordinal $\nu$ and $A \subseteq \nu$ a cofinal subset of $\nu$. If $(T, U, p, q, r)$ satisfy the following: $T$ and $U$ are nested antichains, $p$ and $q$ are sequences of length $\nu$ with $p$ $(T, \nu)$-nice and $q$ $(U, \nu)$-nice and $r \in T_1$ so that $q \hook l(r) \leq r$ and for all $\alpha \in [l(r), \nu)$, $q \hook \alpha \leq r^\frown p\hook[l(r), \alpha)$; then there is a nested antichain $S$ in $\langle \P_\alpha \; | \; \alpha \leq \nu\rangle$ so that $q$ is $(S, \nu)$-nice, if $\{b_0\} = S_0$ then $l(r) \leq l(b_0)$ and $b_0 \hook l(r) \leq r$, for all $s \in S$, $l(s) \in A$ and $S \hooks T$.
\label{2.11}
\end{lemma}

I also recall the definition of a fusion structure.

\begin{definition}[Fusion Structure]
Let $\vec{\P} = \langle \P_\alpha \; | \; \alpha < \nu\rangle$ be an iteration with limit ordinal $\nu$. Given a nested antichain $T$ in $\vec{\mathbb P} \hook \nu$ we call a structure $\langle q^{(a, n)}, T^{(a, n)} \; | \; a \in T_n, n < \omega \rangle$ a {\em fusion structure} if for all $n < \omega$ and $a \in T_n$ the following hold:
\begin{enumerate}
\item
$T^{(a, n)}$ is a nested antichain in $\langle \P_\alpha \; | \; \alpha < \nu\rangle$.
\item
$q^{(a, n)} \in \P_\nu$ is a mixture of $T^{(a, n)}$ up to $\nu$.
\item
$a \leq q^{(a, n)} \hook l(a)$ and if $\{p_0\} = T_0^{(a, n)}$ then $l(a) = l(p_0)$.
\item
For any $b \in {\rm suc}^n_T(a), T^{(b, n+1)} \hooks T^{(a, n)}$ so $q^{(b, n+1)} \leq q^{(a, n)}$.
\end{enumerate}
If $p \in \P_\nu$ is a mixture of $T$ up to $\nu$ then we call $p$ the {\em fusion} of the fusion structure.
\end{definition}

\begin{proposition}[Proposition 3.5 of \cite{miyamoto}]
Let $\langle \P_\alpha \; | \; \alpha \leq \nu\rangle$ be an iteration with limit ordinal $\nu$. If $p \in \P_\nu$ is a fusion of a fusion structure $\langle p^{(a, n)}, T^{(a, n)} \; | \; a \in T_n, n < \omega \rangle$ then there is a sequence $\langle a_n \; | \; n < \omega\rangle$ such that $p$ forces the following hold:
\begin{enumerate}
\item
$a_0 \in T_0$, and for all $n < \omega$ $a_{n+1} \in {\rm suc}^n_T(a_n)$, $q_n \in \dot{G}_\nu\hook l(a_n)$ and $p^{(a_n, n)} \in \dot{G}_\nu$.
\item
If $\beta = {\rm sup}\{l(a_n) \; | \; n < \omega \}$ then $p^{(a_n, n)} \hook \beta \in \dot{G}_\nu \hook \beta$ and $p^{(a_n, n)} \hook [\beta, \nu) \equiv 1_\nu \hook [\beta, \nu)$.
\end{enumerate}
\label{fusion}
\end{proposition}

\section{Nice Iterations of $\infty$-Subversion Forcing}

In this section I prove first that $\infty$-subcomplete forcing is preserved by nice iterations, then I prove the same for $\infty$-subproper forcing notions. I use the following notational convention: if $\langle \P_\alpha \; | \; \alpha \leq \nu\rangle$ is an iteration then for $i \leq j \leq \nu$ the poset $\P_{i, j}$, which is defined in Fact \ref{factPij}, depends on the $\P_i$-generic chosen so I will identify it with its $\P_i$ name $\P_j / \dot{G}_i$.

The special case $i=0$ and $j=\nu$ of following theorem implies that if every successor stage of a nice iteration is forced to be $\infty$-subcomplete, then so is the iteration.

\begin{theorem}
\label{thm:NiceIterationsOfSCforcingAreSC}
Let $\vec{\mathbb P}= \langle \P_\alpha \; | \; \alpha \leq \nu\rangle$ be a nice iteration so that $\P_0 = \{1_0\}$ and for all $i$ with $i + 1 < \nu$,  $\forces_i \P_{i, i+1}$ is $\infty$-subcomplete. Then for all $j\le\nu$ the following statement $\phi(j)$ holds:

{\bf if} $i\le j$, $p\in \P_i$, $\dot{\sigma}\in\V^{\P_i}$, $\theta$ is a sufficiently large cardinal, $\tau$ is an ordinal, $H_\theta\sub N=L_\tau[A]\models\ZFC^-$, $\bN$ is a transitive model, $\bs,\bar{\vec{\P}},\bar{i},\bar{j}\in\bN$, $\bar{G}_{\bar{i}},\bar{G}_{\bar{i},\bar{j}}\sub\bN$, and $p$ forces with respect to $\P_i$ that the following assumptions hold:
\begin{enumerate}[label=(A\arabic*)]
\item
  \label{item:FirstAssumption}
  $\dot{\sigma}(\check{\bar{\vec{\P}}},\check{\bar{i}},\check{\bar{j}},\check{\btheta},\check{\bG}_{\bar{i}})=\check{\vec{\P}},\check{i},\check{j},\check{\theta},\dot{G}_i$, %and $\dot{\sigma}``\check{\bar{G}}_{\bar{i}}\sub\dot{G}_i$,
  \item
  \label{item:StuffMovedRight}
  the following holds in $\V$:
  $\bar{G}_{\bar{i}}$ is $\bar{\P}_{\bar{i}}$-generic over $\bN$ and
  $\bar{G}_{\bar{i},\bar{j}}$ is $\bar{\P}_{\bar{i},\bar{j}}$-generic over $\bN[\bar{G}_{\bar{i}}]$, where $\bar{\P}_{\bar{i},\bar{j}}=\bar{\P}_j/\bar{G}_{\bar{i}}$,
  \item 
  \label{item:Lastassumption}
  $\dot{\sigma}:\check{\bN}[\check{\bG}_{\bar{i}}]\prec \check{N}[\dot{G}_i]$ is countable, transitive and full.
\end{enumerate}
{\bf then} there is a condition $p^*\in \P_j$ such that %$y^*\le x$,
$p^*\rest i=p$ and whenever $G_j\ni p^*$ is $\P_j$-generic, then in $\V[G_j]$, there is a $\sigma'$ such that, letting $\sigma=\dot{\sigma}^{G_i}$, the following hold:
\begin{enumerate}[label=(\alph*)]
  \item 
  \label{item:FirstConclusion}
  $\sigma'(\bs,\bar{\vec{\P}},\bar{i},\bar{j},\btheta,\bG_{\bar{i}})=\sigma(\bs,\bar{\vec{\P}},\bar{i},\bar{j},\btheta,\bG_{\bar{i}})$,
  \item
  \label{item:Sigma'MovesEverythingCorrectly}
  $(\sigma')``\bar{G}_{\bar{i},\bar{j}}\sub G_{i,j}$,
  \item
  \label{item:LastConclusionSC}
  $\sigma':\bN[\bG_{\bar{i}}]\prec N[G_i]$.
\end{enumerate}
\label{sciteration}
\end{theorem}

Let me stress that my proof is similar to that of \cite[Lemma 4.3]{miyamoto}, adapting for the case of $\infty$-subcomplete forcing notions in place of semiproper forcing notions.

\begin{proof}
The proof is by induction on $j$. So let us assume that $\phi(j')$ holds for every $j'<j$.
%Fixing $j$, we prove the claim by induction on $i$. 
Fix some $i\le j$.
%assuming the theorem is proven for all $i'<i$ (and also for all $i'$, $j'$ with $i'\le j'<j$). 
Since nothing is to be shown when $i=j$, let $i<j$. In particular, the case $j=0$ is trivial.

Fix $p\in \P_i$, $\dot{\sigma}\in\V^{\P_i}$, %$x\in P_j$, $y\le x\rest i$, 
$\theta$, $\tau$, $A$, $N$, $\bN$, $\bs,\bar{\vec{\P}},\bar{i},\bar{j}\in\bN$, $\bar{G}_{\bar{i}},\bar{G}_{\bar{i},\bar{j}}\sub\bN$ so that assumptions \ref{item:FirstAssumption}-\ref{item:Lastassumption} hold.

%Note that if $\sigma_0 : \bN \prec N$ then for any sequence $a \in N$ in the range of $\sigma_0$ the length of $a$, $l(a)$ is such that $\bar{l(a)} = l(\bar{a})$ by elementarity (this being an example of the aforementioned ``bar'' convention).

%Let us also mention the following fact that we will use, see \cite[Fact 1.5]{Fuchs:ParametricSubcompleteness} for a proof of this folkloristic fact, which is also implicit Jensen's work on subcomplete forcing:
%\begin{fact}
%Suppose $\bN, N$ are transitive models of ${\rm ZFC}^-$ with $N = L_\tau[A]$, $\sigma_0:\bN \prec N$ and $\bar{\epsilon}$ an ordinal in $\bN$ with $\epsilon = \sigma_0(\bar{\epsilon})$. Then ${\rm Hull}^N(\epsilon \cup range (\sigma_0)) =\bigcup \{\sigma_0(u) \; | \; u \in \bN \, \& \, \card{u}^{\bN} \leq \bar{\epsilon}\}$.
%\label{ufact}
%\end{fact}

\noindent{Case 1:} $j$ is a limit ordinal.

Let $\{t_n \; | \; n < \omega \}$ enumerate the elements of $\bN$. Throughout this proof I will identify the $t_n$'s with their check names when it causes no confusion. Without loss of generality we may assume that $t_0 = \emptyset$. Also let $\{\bar{p}_n \; | \; n <\omega \}$ enumerate the elements of $\bar{G}_{\bar{i}}*\bar{G}_{\bar{i},\bar{j}}$, where we assume that $\bar{p}_0=1_{\bar{P}_{\bar{j}}}$. It follows then that $p$ forces that $\dot{\sigma}(\bar{p}_0)=\check{1}_j$, so I write $p_0=1_j$.

Note that since in some forcing extension, there is an elementary embedding from $\bN$ to $N = L_\tau [A]$, we may assume that there is a definable well order of $\bN$, call it $\leq_\bN$. %%%What should I say?
As noted in \cite{miyamoto}, by Lemma \ref{2.7}, in $\bN$, $\bar{p}_0$ is a mixture up to $\bar{j}$ of some nested antichain in $\bar{\vec{\P}}\rest{\bar{j}}$ whose root has length $\bar{i}$. Letting $\bar{W}$ be the $\leq_\bN$-least one, we know that $p$ forces that $\dot{\sigma}(\check{\bar{W}})$ is the $L_\tau[A]$-least nested antichain $W$ in $\vec{\P}\rest j$ such that $p_0$ is a mixture of $W$ up to $j$, since we know that $p$ forces that $\bar{p}_0,\bar{i},\bar{j}$ are mapped to $p_0,i,j$ by $\dot{\sigma}$, respectively.

I will define a nested antichain $\langle T, \langle T_n \; | \; n < \omega \rangle, \langle \suc^n_T \; | \; n < \omega\rangle \rangle$, a fusion structure $\seq{\kla{q^{(t, n)}, T^{(a, n)}}}{n < \omega, a\in T_n}$ in $\seq{\P_\alpha}{\alpha \leq j}$ and a sequence $\langle \dot{\sigma}^{(a, n)} \; | \; n < \omega , a \in T_n\rangle$ so that the following conditions hold.

\begin{enumerate}[label=(\arabic*)]
\item
\label{claim:Start}
$T_0 = \{p\}$, $q^{(p,0)}=p_0$, $T^{(p,0)}=W$ and $\dot{\sigma}^{(p,0)} = \dot{\sigma}$.
\end{enumerate}
Further, for any $n<\omega$ and $a\in T_n$:
\begin{enumerate}[label=(\arabic*)]
\setcounter{enumi}{1}
\item %6
\label{item:middlepart-beginning}
%$x^{(a, n)} \leq x$ and 
$q^{(a, n)}\in \P_j$ and $\dot{\sigma}^{(a, n)}$ is a $\P_{l(a)}$-name,
\item %7
\label{item:sigma(a,n)}
$a$ forces the following statements with respect to $\P_{l(a)}$:
\begin{enumerate}[label=(\alph*)]
\item 
\label{subitem:iselementary}
$\dot{\sigma}^{(a, n)} : \bN[\bG_{\bar{i}}]\prec N[\dot{G}_i]$,
\item 
\label{subitem:moveseverythingcorrectly}
$\dot{\sigma}^{(a, n)}(\bar{\theta}, \bar{\vec{\P}}, \bar{s}, \bar{\nu}, \bar{i},\bar{j},\bar{G}_{\bar{i}}) = \theta, \vec{\P}, \tau,\dot{\sigma}(\bs), \nu, i, j, \dot{G}_i$,
%\item ${\mathrm{Hull}}^N(\delta(P_\nu)\cup\ran(\dot{\sigma}^{(a, n)})) = {\mathrm{Hull}}^N(\delta(P_\nu)\cup\ran(\check{\sigma}))$.
\item 
\label{subitem:hasx(a,n)initsrange}
$q^{(a, n)} \leq_j \dot{\sigma}^{(a, n)}(\bar{p}_n)$, and $q^{(a,n)}\in\ran(\dot{\sigma}^{(a,n)})$.
\end{enumerate}
%\item %8
%$a$ decides $\dot{\sigma}^{(a, n)}(t_n)$.
\item %9
\label{item:middlepart-end}
for some $\bar{q}^{(a,n)}\in\bar{\P}_{\bar{j}}$ and $\bar{T}^{(a,n)}\in\bN$, we have that $\bar{q}^{(a, n)} \in \bG_{\bar{i}}*\bar{G}_{\bar{i},\bar{j}}$, $\bar{q}^{(a, n)} \leq \bar{p}_n$ and \\
$a \forces \dot{\sigma}^{(a, n)} (\kla{\bar{q}^{(a, n)}, \bar{T}^{(a, n)}, l(\bar{a})}) =\kla{q^{(a, n)}, T^{(a, n)}, l(a)}$.
\end{enumerate}
\noindent If $b \in {\rm suc}^n_T (a)$ and $m\le n$, then
\begin{enumerate}[label=(\arabic*)]
\setcounter{enumi}{4}
\item %10
\label{item:coherence}
$b \forces_{l(b)} \dot{\sigma}^{(b, n+1)}(t_m) = \dot{\sigma}^{(a, n)}(t_m)$ and $\dot{\sigma}^{(b, n+1)}(\bar{p}_m) = \dot{\sigma}^{(a, n)}(\bar{p}_m)$.

%\ \text{and}\ \dot{\sigma}^{(b, n+1)}(v_m) = \dot{\sigma}^{(a, n)}(v_m),$ where $v_m$ is the $\bN$-least $v$ with $\sigma(t_m) \in \dot{\sigma}^{(a_m, m)}(v)$ and $\card{v}^{\bN} \leq \bar{\delta(P_\nu)}$, where $a_m$ is the predecessor of $a$ in $T_m$. Note that such a $v$ always exists by Fact \ref{ufact}.
\item %11
\label{claim:LastOfFirstList}
$b\rest l(a)\forces_{l(a)}q^{(b,n+1)},T^{(b,n+1)}\in\ran(\dot{\sigma}^{(a,n)})$.
\end{enumerate}
First let's see that constructing such objects is sufficient to prove the existence of a condition $p^*$ as in the statement of the theorem. So suppose that we have constructed a nested antichain $\langle T, \langle T_n \; | \; n < \omega \rangle, \langle \suc^n_T \; | \; n < \omega\rangle \rangle$, a fusion structure $\seq{\kla{q^{(t, n)}, T^{(a, n)}}}{n < \omega, a\in T_n}$ in $\seq{\P_\alpha}{\alpha \leq j}$ and a sequence $\langle \dot{\sigma}^{(a, n)} \; | \; n < \omega , a \in T_n\rangle$, so that \ref{claim:Start} through \ref{claim:LastOfFirstList} above are satisfied. Let $q^*\in \P_j$ be a fusion of the fusion structure, and let $p^*=p\verl q^*\rest[i,j)$. By \ref{claim:Start}, we have that $q^*\rest i\equiv p$, so $p^*\equiv q^*$ and $p^*\rest i=p$, as required.
To see that $p^*$ is as wished, let $G_j$ be $\P_j$-generic over $\V$ with $p^* \in G_j$. We have to show that in $\V[G_j]$, there is a $\sigma'$ so that conclusions \ref{item:FirstConclusion}-\ref{item:LastConclusion} are satisfied.
Since $p^*\equiv q^*$, we have that $q^*\in G_j$. Work in $\V[G_j]$. By Proposition \ref{fusion} there is a sequence $\langle a_n \; | \; n < \omega\rangle \in \V[G_j]$ so that for all $n < \omega$, $a_{n+1} \in {\rm suc}^n_T(a_n)$, $a_n \in G_j \hook l(a_n)$ and $q^{(a_n, n)} \in G_j$. Let $\sigma_n$ be the evaluation of $\dot{\sigma}^{(a_n, n)}$ by $G_j$. Then we define $\sigma': \bN \to N$ to be the map such that $\sigma ' (t_n) = \sigma_n (t_n)$. %Note that $a_n$ decides $\dot{\sigma}^{(a, n)}(t_n)$, and this is exactly the value given in the extension. 
We claim that $\sigma'$ satisfies the conclusions \ref{item:FirstConclusion}-\ref{item:LastConclusionSC}.

Condition \ref{item:FirstConclusion} says that $\sigma'$ moves the parameters $\bar{s},\bar{\vec{\P}},\bar{i},\bar{j},\btheta$ and $\bG_{\bar{i}}$ the same way $\sigma=\dot{\sigma}^{G_i}$ does. But this is true of every $\sigma_n$, hence also of $\sigma'$.

Condition \ref{item:Sigma'MovesEverythingCorrectly} says that $\sigma'``\bG_{\bar{i},\bar{j}}\sub G_{i,j}$. So let $\bp\in\bG_{\bar{i},\bar{j}}$. I have to show that $\sigma'(\bp)\in G_{i,j}$.
Recall $\bG_{\bar{i}}*\bG_{\bar{i},\bar{j}} = \{\bar{p}_n \; | \; n < \omega \}$. Let $n$ be such that $\bp=\bp_n\rest[\bar{i},\bar{j})$. By \ref{item:sigma(a,n)}\ref{subitem:hasx(a,n)initsrange}, we have that $q^{(a_n, n)}\leq_j\sigma_n(\bar{p}_n)$, so since $q^{(a_n, n)}\in G_j$, it follows that $\sigma_n(\bar{p}_n)\in G_j$ as well. By \ref{item:coherence} and the definition of $\sigma'$,  we have that $\sigma_n(\bar{p}_n) = \sigma'(\bar{p}_n)$.
It follows that  $\sigma'(\bp)=\sigma'(\bp_n\rest[\bar{i},\bar{j}))=\sigma'(\bp_n)\rest[i,j)\in G_{i,j}$, as claimed.

Condition \ref{item:LastConclusionSC} says that $\sigma':\bN\to N$ is elementary.
Since any one formula can only use finitely many parameters, and $\sigma'\rest\{t_0,\ldots,t_n\}=\sigma_n\rest\{t_0,\ldots,t_n\}$, this is true by \ref{item:coherence}.

Therefore it remains to show that the construction described above can actually be carried out. This is done by recursion on $n$. The recursion proceeds as follows. At stage $n+1$ of the construction, we
assume that $T_m$, $T^{(a,m)}$, $q^{(a,m)}$ and $\dot{\sigma}^{(a,m)}$ have been defined, for all $m\le n$ and all $a\in T_m$. Also, for $m<n$ and $a\in T_m$, we assume that $\suc_T^m(a)$ has been defined. Our inductive hypothesis is that for all $m\le n$ and all $a\in T_m$, conditions 
\ref{item:middlepart-beginning}-\ref{item:middlepart-end}
hold, and that for all $m<n$, all $a\in T_m$ and all $b\in\suc_T^m(a)$, conditions \ref{item:coherence}-\ref{claim:LastOfFirstList} are satisfied. In order to define $T_{n+1}$, I will specify $\suc_T^n(a)$, for every $a\in T_n$, which implicitly defines $T_{n+1}=\bigcup_{a\in T_n}\suc_T^n(a)$. Simultaneously, I will define, for every such $a$ and every $b\in\suc_T^n(a)$, the objects $T^{(b,n+1)}$, $q^{(b,n+1)}$ and $\dot{\sigma}^{(b,n+1)}$ in such a way that whenever $a\in T_n$ and $b\in\suc_T^n(a)$, \ref{item:coherence}-\ref{claim:LastOfFirstList} are satisfied by $a$ and $b$, and \ref{item:middlepart-beginning}-\ref{item:middlepart-end} are satisfied by $b$ and $n+1$ (instead of $a$ and $n$).

For stage 0 of the construction, notice that \ref{claim:Start} gives the base case where $n =0$ and in this case \ref{item:middlepart-beginning}-\ref{item:middlepart-end} are satisfied, $p$ forces that $q^{(p, 0)}=p_0=\dot{\sigma}(\bar{p}_0)$ and $p$ forces that $W=T^{(p,0)}$ has a preimage under $\dot{\sigma}$, namely $\bar{W}$.

At stage $n+1$ of the construction, work under the assumptions described above. Fixing $a\in T_n$, I have to define $\suc_T^n(a)$.
To this end let $D$ be the set of all conditions $b$ for which there are a nested antichain $S$ in $\vec{\P}\rest j$, and objects $\dot{\sigma}^b$, $u$, $\bar{u}$ and $\bar{S}$ satisfying the following:

\begin{enumerate}[label=(D\arabic*)]
\item %12
\label{item:FirstConditionDefiningPredenseSet}
$b \in \P_{l(b)}$ and $l(b) < j$.
\item %13
$l(a) \leq l(b)$ and $b \hook l(a) \leq a$.
\item %14
$S \hooks T^{(a, n)}$, $\bS\in\bN$, $S\in N$.
\item %15
$u\in \P_j$, $u \leq q^{(a,n)}$ and $u$ is a mixture of $S$ up to $j$.
\item %16
$\bar{u} \in \bG_{\bar{i}}*\bG_{\bar{i},\bar{j}}$, and $\bar{u} \leq \bar{p}_{n+1}$.
%\item %17
%$b$ forces that $\dot{\sigma}^b$ is a function.%, and $b$ decides $\dot{\sigma}^b(t_{n+1})$.
\item %18
\label{item:InitialSegmentOfbForcesStuffIntoRangeOfTheOldEmbedding}
$b\rest l(a)\forces_{l(a)} S, u, l(b) \in\ran (\dot{\sigma}^{(a,n)})$.
%$b \forces_{l(b)} S, u, l(b) \in\ran (\dot{\sigma}^b)$.
\item %19
\label{item:LastConditionDefiningPredenseSet}
$b$ forces the following statements with respect to $P_{l(b)}$: \begin{enumerate}
\item $\dot{\sigma}^b (\bar{\theta},\bar{i},\bar{j},\bar{\vec{\P}},\bG_{\bar{i}}, \bar{s},\bar{u},\bS) = \theta,i,j,\vec{\P},\dot{G}_i,\dot{\sigma}(\bs), u,S$.
\item $\forall m \leq n$ $\dot{\sigma}^b(t_m) = \dot{\sigma}^{(a, n)}(t_m)$ and $\dot{\sigma}^b(\bar{p}_m) = \dot{\sigma}^{(a, n)}(\bar{p}_m)$,
\item $\dot{\sigma}^b: \bN[\bG_{\bar{i}}] \prec N[\dot{G}_i]$,
%\item ${\mathrm{Hull}}^N(\delta(P_\nu) \cup \ran (\dot{\sigma}^b)) = {\mathrm{Hull}}^N(\delta(P_\nu) \cup \ran (\dot{\sigma}^{(a, n))})$,
\end{enumerate}
\end{enumerate}

Note that if $b\in D$ and $b'\le_{l(b)}b$, then $b'\in D$ as well. It follows that $D\rest l(a):=\{b\rest l(a)\st b\in D\}$ is open in $P_{l(a)}$.
Thus, it suffices to show that $D\rest l(a)$ is predense below $a$ in $\P_{l(a)}$. For if we know this, $D\rest l(a)$ is dense below $a$, and we may choose a maximal antichain $A\sub D\rest l(a)$ (with respect to $\P_{l(a)}$), which then is a maximal antichain in $\P_{l(a)}$ below $a$. Thus, for every $c\in A$, we may pick a condition $b(c)\in D$ such that $b(c)\rest l(a)=c$, and define $\suc_T^n(a)=\{b(c)\st c\in A\}$ (in order to satisfy Definition \ref{definition:NestedAC-Mixture-beta-Nice-hooks}, part $(5)$). Now, for every $b\in\suc_T^n(a)$, let $S$, $\dot{\sigma}^b$, $u$ and $\bar{u}$ witness that $b\in D$, i.e., let them be chosen in such a way that \ref{item:FirstConditionDefiningPredenseSet}-\ref{item:LastConditionDefiningPredenseSet} hold. Set $T^{(b,n+1)}=S$, $\dot{\sigma}^{(b, n+1)}=\dot{\sigma}^b$, $q^{(b, n+1)}=u$, $\bar{q}^{(b,n+1)}$ and $\bar{T}^{(b,n+1)}=\bS$. Then $a$, $b$ satisfy \ref{item:coherence}-\ref{claim:LastOfFirstList} at stage $n$, and $b$ satisfies \ref{item:middlepart-beginning}-\ref{item:middlepart-end} at stage $n+1$.

To see that $D\rest l(a)$ is predense below $a$, let
$G_{l(a)}$ be $\P_{l(a)}$-generic over $V$ with $a \in G_{l(a)}$. We have to find a $b\in D$ so that $b \hook l(a) \in G_{l(a)}$. Work in $V[G_{l(a)}]$. Let $\sigma_n = (\dot{\sigma}^{(a, n)})_{G_{\l(a)}}$. Since \ref{item:sigma(a,n)} holds at stage $n$, we have that $\sigma_n:\bN[\bG_{\bar{i}}]\prec N[G_i]$ %${\mathrm{Hull}}^N(\delta(P_\nu) \cup \ran(\sigma)) = {\mathrm{Hull}}^N(\delta(P_\nu) \cup \ran(\sigma_n))$, 
and $\sigma_n(\bar{\theta}, \bar{\vec{\P}},\bs,\bar{\nu},\bar{i},\bar{j},\bG_{\bar{i}})
=\theta,\vec{\P},\sigma(\bs),\nu,i,j,G_i$.
%
% \bar{p}_0,...,\bar{p}_n, \bar{T^{(a, n)}}, l(\bar{a}), \bar{x}^{(a, n)}) = \theta, P_\nu, \nu, p_0,...,p_n, T^{(a, n)}, l(a), x^{(a, n)}$.
We also have objects $\bar{q}^{(a,n)},\bT^{(a,n)},l(\bar{a})$ satisfying condition \ref{item:middlepart-end}, so that $\bar{q}^{(a,n)}\in\bG_{\bar{i}}*\bG_{\bar{i},\bar{j}}$, $\bT^{(a,n)}\in\bN$, $\bar{q}^{(a,n)}\le\bp_n$, $\sigma_n(\bar{q}^{(a,n)},\bT^{(a,n)})=q^{(a,n)},T^{(a,n)}$ and $\sigma_n(l(\bar{a}))=l(a)$.

First I find the requisite $u$ and $\bar{u}$. By elementarity, $\bar{q}^{(a, n)}$ is a mixture of $\bar{T}^{(a, n)}$ up to $\bar{j}$. Recall that $\sigma_n:\kla{L_{\btau}[\bA][\bG_{\bar{i}}],\in,\bA}\prec\kla{L_\tau[A][G_i],\in,A}$, so in particular, $\bsigma:=\sigma_n\rest L_\btau[\bA]:\kla{L_\btau[\bA],\in,\bA}\prec\kla{L_\tau[A],\in,A}$, and $\bsigma(\bar{q}^{(a,n)},\bT^{(a,n)})=q^{(a,n)},T^{(a,n)}$. Clearly, in $L_\tau[A]$, it is true that $q^{(a,n)}$ is a mixture of $T^{(a,n)}$ up to $j$, so it is true in $L_\btau[\bA]$ that $\bar{q}^{(a,n)}$ is a mixture of $\bT^{(a,n)}$ up to $\bar{j}$, and by absoluteness, this it true in $\V$ as well.

Let $\bT_0^{(a,n)}=\{\ba_0\}$. Let's write $\bG_{\bar{j}}=\bG_{\bar{i}}*\bG_{\bar{i},\bar{j}}$, and for $k\le\bar{j}$, let's set $\bG_k=\bG_{\bar{j}}\rest k$. By Fact \ref{fact:CharacterizationOfMixtures}.\ref{item:Root}, $\ba_0\equiv\bar{q}^{(a,n)}\rest l(\ba_0)\in\bG_{l(\ba_0)}$, since $l(\bar{q}^{(a,n)})=\bar{j}>l(\ba_0)$. So $\ba_0\in\bG_{l(\ba_0)}$.
Let $\bar{r}\in\bT_1^{(\ba,n)}$, that is, $\bar{r}\in\suc_{\bT^{(\ba,n)}}^0(\ba_0)$, be such that $\bar{r}\rest l(a_0)\in\bar{G}_{l(\ba_0)}$. There is such a $\bar{r}$ by Definition \ref{definition:NestedAC-Mixture-beta-Nice-hooks}(5). By Fact \ref{fact:CharacterizationOfMixtures}.\ref{item:HigherUp}, again since $l(\bar{r})<\bar{j}=l(\bar{q}^{(a,n)})$, it follows that $\bar{r}\equiv\bar{r}\rest l(\ba_0)\verl\bar{q}^{(a,n)}\rest[l(\ba_0),l(\bar{j}))$.
Since $\bar{r}\rest l(\ba_0)\in\bG_{l(\ba_0)}$, this implies that $\bar{r}\rest[l(\ba_0),l(\bar{r}))\equiv\bar{q}^{(a,n)}\rest[l(\ba_0),l(\bar{r}))$ (in the partial order $\bar{\P}_{l(\ba_0),l(\bar{r})}=\bar{\P}_{l(\bar{r})}/\bG_{l(\ba_0)}$), and $\bar{q}^{(a,n)}\rest[l(\ba_0),l(\bar{r}))\in\bG\rest[l(\ba_0),l(\bar{r}))$.
So we have that $\bar{r}\rest l(\ba_0)\in\bG_{l(\ba_0)}$ and $\bar{r}\rest[l(\ba_0),l(\bar{r}))\in\bG\rest[l(\ba_0),l(\bar{r}))$. By Fact \ref{fact:FactorsAndQuotients}.\ref{item:CriterionForBeingInGeneric}, this implies that $\bar{r}\in\bG_{l(\bar{r})}$.

It follows that $\bar{r}^{\frown}\bar{q}^{(a, n)} \hook [\bar{l(r)}, \bar{j}) \in \bar{G}_{\bar{j}}$, again using Fact \ref{fact:FactorsAndQuotients}.\ref{item:CriterionForBeingInGeneric}.
Let $\bar{u} \in \bar{G}_{\bar{j}}$ strengthen both $\bar{r}^{\frown}\bar{q}^{(a, n)} \hook [\bar{l(r)},\bar{j})$ and $\bar{p}_{n+1}$.
By Lemma \ref{2.11}, applied in $\bar{N}$, there is a nested antichain $\bar{S} \hooks \bar{T}^{(a, n)}$ such that $\bar{u}$ is a mixture of $\bS$ up to $\bar{j}$ and such that letting $\bS_0 = \{\bd_0\}$, we have that  $l(\bar{r}) \leq l(\bd_0)$ and $\bd_0 \rest l(\bar{r}) \leq \bar{r}$. Let $S,d_0,u=\sigma_n(\bS,\bd_0,\bu)$, and let $w\in G_{l(a)}$ force this. Since $a\in G_{l(a)}$, we may choose $w$ so that $w\le a$.

Note that $S,d_0,u$ are in $N$ (and hence in $\V$), since $\bS,\bd_0,\bu\in\bN$.

I am going to apply our inductive hypothesis $\phi(l(d_0))$, noting that $l(d_0)<j$, to $i =l(a) \leq l(d_0)$, the filters $\bG_{l(\bar{a})}$, $\bG_{{l(\bar{a})},l(\bd_0)}$, the models $\bN$, $N$, the condition $w$ (in place of $p$), the name $\dot{\sigma}^{(a,n)}$ (in place of $\dot{\sigma}$ and the parameter $\bs'\in\bN$ which I will specify below).
No matter which $\bs'$ we choose, by the inductive hypothesis, there is a condition $w^*\in\P_{l(d_0)}$ with $w^*\rest l(a)=w$ and a name $\dot{\sigma}'$ such that $w^*$ forces with respect to $\P_{l(a)}$:
%such that if $H\ni w^*$ is generic for $\P_{l(d_0)}$, then in $\V[H]$ there is a $\sigma'$ such that letting $\sigma'_n=(\sigma^{(a,n)})^H$,
\begin{enumerate}[label=(\alph*)]
  \item
  $\dot{\sigma}'(\check{\bs}',\check{\bar{\vec{\P}}},\check{l(\bar{a})},\check{l(\bar{d}_0)},\check{\btheta},\dot{\bG}_{l(\bar{a})})=\dot{\sigma}^{(a,n)}(\check{\bs}',\check{\bar{\vec{\P}}},\check{l(\bar{a})},\check{l(\bar{d}_0)},\check{\btheta},\dot{\bG}_{l(\bar{a})})$,
  \item
  $(\dot{\sigma}')``\dot{\bar{G}}_{l(\bar{a}),l(\bd_0)}\sub \dot{G}_{l(a),l(d_0)}$,
  \item
  $\dot{\sigma}':\check{\bN}[\dot{\bG}_{l(\bar{a})}]\prec \check{N}[\dot{G}_{l(a)}]$.
\end{enumerate}

By choosing $\bs'$ appropriately, and temporarily fixing $H$ as above, we may insure that it is forced that $\dot{\sigma}'$ moves any finite number of members of $\bN$ the same way $\dot{\sigma}^{(a,n)}$ does.
Thus, we may insist that $w^*$ forces that $\dot{\sigma}'(\bu,\bd_0,\bS)=\dot{\sigma}^{(a,n)}(\bu,\bd_0,\bS)$. Recall that $w$ forced that $\dot{\sigma}^{(a,n)}(\bu,\bd_0,\bS)=u,d_0,S$.
Hence, since $w^*\rest l(a)=w$, we get that $w^*$ forces that $\dot{\sigma}'(\bu,\bd_0,\bS)=u,d_0,S$ as well.

In addition, we may insist that $\sigma'$ moves the parameters $\bar{i},\bar{j},\bar{\vec{\P}},\btheta,\bs,\bp_0,\ldots,\bp_n,t_0,\ldots,t_n$ the same way $\dot{\sigma}^{(a,n)}$ does. Note that already $a$ forced with respect to $\P_{l(a)}$ that $\bar{i},\bar{j},\bar{\vec{\P}},\btheta$ are mapped to $i,j,\vec{\P},\theta$ by $\dot{\sigma}^{(a,n)}$.

%Note that $\sigma'$ lifts to an elementary embedding $\sigma':\bN[\bG_{l(\bd_0)}]\prec N[H]$ with $\sigma'(\bG_{l(\bd_0)})=H_{l(d_0)}$. It follows that $\sigma'(\bu\rest l(\bd_0))=u\rest l(d_0)\in H$, since $\bu\in\bG_{\bar{j}}$, so that $\bu\rest l(\bd_0)\in\bG_{l(\bd_0)}$.

%Since $\bu\in\bG$, it follows that $\bu\rest l(d_0)\in\bG\rest l(\bd_0)$, and so, $w^*$ forces that $u\rest l(d_0)=\sigma'(\bu\rest l(\bd_0))\in\dot{G}_{l(d_0)}$. Thus, given any $P_{l(a),l(d_0)}$-generic filter $H$ containing $w^*$, we may find a common extension of $w^*$ and $u\rest l(d_0)$ in $H$. So we may assume that $w^*\le u\rest l(d_0)$.

Now, setting $b=w^*$, $\dot{\sigma}^b=\dot{\sigma}'$, conditions \ref{item:FirstConditionDefiningPredenseSet}-\ref{item:LastConditionDefiningPredenseSet} are satisfied, that is, $b\in D$. Most of these are obvious; let me just remark that $b$ forces that $\dot{\sigma}^b(\bG_{\bar{i}})=\dot{G}_i$ because it forces that $\dot{\sigma}^b(\bar{i})=i$ and $\dot{\sigma}^b(\bG_{l(\bar{a})})=\dot{G}_{l(a)}$. Condition \ref{item:InitialSegmentOfbForcesStuffIntoRangeOfTheOldEmbedding} holds because $b\rest l(a)=w$. 
For the same reason, we have that $b\rest l(a)\in G_{l(a)}$, completing the proof that $D\rest l(a)$ is predense below $a$. This concludes the treatment of case 1.

\noindent\emph{Case 2:} $j$ is a successor ordinal.

Let $j=k+1$. Since we assumed $i<j$, it follows that $i\le k$. Inductively, we know that $\phi(k)$ holds. Note that $\bar{j}$ is of the form $\bk+1$, where $p$ forces with respect to $\P_i$ that $\dot{\sigma}(\bk)=k$, and if we let $\bG_{\bk}=\bG_\bj\rest\bk$, then the assumptions \ref{item:FirstAssumption}-\ref{item:LastAssumption} are satisfied by 
$p\in \P_i$, $\dot{\sigma}\in\V^{P_i}$, %$x\in P_j$, $y\le x\rest i$,
$\theta$, $\tau$, $A$, $N$, $\bN$, $\bs,\bar{\vec{\P}},\bar{i},\bar{k}\in\bN$, $\bar{G}_{\bar{i}},\bar{G}_{\bar{i},\bar{k}}\sub\bN$ and $k$. By $\phi(k)$, we obtain a condition $p^{**}\in\P_k$ with $p^{**}\rest i=y$ and a $\P_k$-name $\dot{\bsigma}$ such that $p^{**}$ forces
\begin{enumerate}[label=(\alph*1)]
  \item
  $\dot{\bsigma}(\check{\bs},\check{\bar{\vec{\P}}},\check{\bar{i}},\check{\bk},\check{\bar{j}},\check{\btheta},\check{\bG}_{\bar{i}})=\sigma(\check{\bs},\check{\bar{\vec{\P}}},\check{\bar{i}},\check{\bk},\check{\bar{j}},\check{\btheta},\check{\bG}_{\bar{i}})$,
  \item
  $\dot{\bsigma}``\check{\bar{G}}_{\bar{i},\bar{k}}\sub\dot{G}_{i,k}$,
  \item
  $\dot{\bsigma}:\check{\bN}[\check{\bG}_{\bar{i}}]\prec \check{N}[\dot{G}_i]$.
\end{enumerate}
It follows then that $p^{**}$ forces that $\dot{\bsigma}``\bG_\bk\sub\dot{G}_k$, and hence that $\dot{\bsigma}$ lifts to an elementary embedding from $\bN[\bG_\bk]\prec N[\dot{G}_k]$ that maps $\bG_\bk$ to $\dot{G}_k$. Let $\dot{\tsigma}$ be a $\P_k$-name such that $p^{**}$ forces that $\dot{\tsigma}$ is that lifted embedding.

Temporarily fix a $\P_k$-generic filter $H$ that contains $p^{**}$. In $\V[H]$, the forcing $\P_{k,k+1}=\P_{k,j}=\P_j/H$ is $\infty$-subcomplete. Letting $\tsigma=\dot{\tsigma}^H$, we have that $\tsigma:\bN[\bG_\bk]\prec N[H]$, and thus, since $\bN[\bG_\bk]$ is full, there is a condition $q$ in $\P_{k,j}$ such that $q$ forces the existence of an elementary embedding $\sigma'$ with
\begin{enumerate}[label=(\alph*2)]
  \item
  $\sigma'(\check{\bs},\check{\bar{\vec{\P}}},\check{\bar{i}},\check{\bk},\check{\bar{j}},\check{\btheta},\check{\bG}_{\bar{i}},\check{\bG}_{\bk})=\tsigma(\check{\bs},\check{\bar{\vec{\P}}},\check{\bar{i}},\check{\bk},\check{\bar{j}},\check{\btheta},\check{\bG}_{\bar{i}},\check{\bG}_{\bk})$,
  \item
  $(\sigma')``\check{\bar{G}}_{\bar{k},\bar{j}}\sub\dot{G}_{k,j}$,
  \item
  $\sigma':\check{\bN}[\check{\bG}_{\bar{k}}]\prec \check{N}[\dot{H}]$.
\end{enumerate}
Since this holds in $\V[H]$ whenever $p^{**}\in H$, there is a $\P_k$-name $\tau$ which is essentially a name for $q$ above - more precisely, $\tau$ is such that $p^{**}$ forces that $\tau\in \P_j$, $\tau\rest k\in\dot{G}_k$ and $\tau\rest[k,j)$ has the properties of $q$, as listed above.
Since the iteration is nice, there is a condition $p^*\in P_j$ such that $p^*\rest k=p^{**}$ and $p^*$ forces that $\tau\rest[k,j)\equiv p^*\rest[k,j)$; see Definition \ref{def:NiceIterations}, part (1). I claim that $p^*$ is as wished.

First, note that $p^*\rest i=(p^*\rest k)\rest i=p^{**}\rest i=p$. Now, let $G_j$ be a $\P_j$-generic filter with $p^*\in G_j$. I have to show that in $\V[G_j]$, there is a $\sigma'$ such that, letting $\sigma=\dot{\sigma}^{G_i}$, the following hold:
\begin{enumerate}[label=(\alph*)]
  \item
  $\sigma'(\bs,\bar{\vec{\P}},\bar{i},\bar{j},\btheta,\bG_{\bar{i}})=\sigma(\bs,\bar{\vec{\P}},\bar{i},\bar{j},\btheta,\bG_{\bar{i}})$,
  \item
  $(\sigma')``\bar{G}_{\bar{i},\bar{j}}\sub G_{i,j}$,
  \item
  $\sigma':\bN[\bG_{\bar{i}}]\prec N[G_i]$.
\end{enumerate}
But this follows, because $\V[G_j]=\V[G_k][G_{k,j}]$, where $p^*\rest[k,j)\in G_{k,j}$, where $p^{**}\in G_k$, writing $H$ for $G_k$, puts us in the situation described above. Moreover, $p^*\rest[k,j)\in G_{k,j}$ and $p^*\rest[k,j)\equiv\dot{q}^{G_j}$, where $\dot{q}$ is a name for the condition $q$ mentioned above. Thus, there is a $\sigma'$ in $\V[G_j]$ such that the conditions (a2)-(c2) listed above hold in $\V[G_j]$. Remembering that $\tsigma$ lifts $\bsigma$ and $\bsigma$ moves the required parameters as prescribed (by (a1)-(c1)), it follows that (a)-(c) are satisfied.

%Let us check this (i.e. statements 11-19) for completeness. Statement 11 holds just in case $u$, $S$ and $l(b)$ are forced by $b$ to be in the range of $\dot{\sigma}^b$, which is indeed the case. Statements 12 and 13 hold since $l(a) \leq l(d_0) < \nu$. Statement 14 says that $S \hooks T^{(a, n)}$, which is true by construction (using Lemma \ref{2.11}). Statement 15 holds by the fact that $\bar{u} \leq \bar{z} \in \bar{G}_\nu \hook \bar{l(z)} \leq \bar{x}^{(a, n)}$ so by elementarity $u \leq x^{(a, n)}$.  Statement 16 holds by the first sentence of this paragraph. Statement 17 holds since $S$ $u$ and $l(d_0)$ are forced by $b$ to be in the range of $\dot{\sigma}^b$, noting that $l(b) = l(d_0)$. Statement 18 holds by the fact that that we have fixed the necessary parameters inductively (see the centered text just above this paragraph). Finally statement 19 holds since $\bar{u} \leq \bar{p}_{n+1}$ and so $\sigma_n(\bar{u}) \leq \sigma_n(\bar{p}_{n+1})$ by elementarity and $b$ forced $\sigma_n(\bar{u}, \bar{p}_{n+1}) = \dot{\sigma}^b(\bar{u}, \bar{p}_{n+1})$.
%
%Thus there is a $b \in V[G_{l(a)}] \cap D$ so $b \hook l(a) \in G_{l(a)}$ so we are done.
\end{proof}

Next I prove a similar theorem for $\infty$-subproper forcings. After having proved Theorem \ref{thm:NiceIterationsOfSPforcingAreSP} I learned that Miyamoto (unpublished) had also proved this result earlier. What I am calling $\infty$-subproper here, Miyamoto calls \say{preproper}.

\begin{theorem}
\label{thm:NiceIterationsOfSPforcingAreSP}
Let $\vec{\P} = \langle \P_\alpha \; | \; \alpha \leq \nu\rangle$ be a nice iteration so that $\P_0 = \{1_0\}$ and for all $i$ with $i + 1 < \nu$,  $\forces_i \P_{i, i+1}$ is $\infty$-subproper. Then for all $j\le\nu$ the following statement $\phi(j)$ holds:

if $i\le j$, $p\in \P_i$, $\dot{\sigma}, \dot{\bG}_{\bar{i}} \in\V^{\P_i}$, $q\in P_j$,
$\theta$ is a sufficiently large cardinal, $\tau$ is an ordinal, $H_\theta\sub N=L_\tau[A]\models\ZFC^-$
$\bN$ is a countable, full, transitive model which elementarily embeds into $N$ so that $\bs,\bar{\vec{\P}},\bar{i},\bar{j} \in \bN$ and $p$ forces with respect to $\P_i$ that the following assumptions hold:
\begin{enumerate}[label=(A\arabic*)]
%  \item 
%  $H_\theta^{\V[\dot{G}_i]}\sub L_{\check{\tau}}[\dot{A}]$,
%  \item $\dot{N}=L_{\check{\tau}}^{\dot{A}}$,
  \item
  \label{item:FirstAssumption}
  $\dot{\sigma}(\check{\bs}, \check{\bar{\vec{\P}}},\check{\bar{i}},\check{\bar{j}},\check{\btheta}, \check{\bar{q}})=\check{s}, \check{\vec{\P}},\check{i},\check{j},\check{q}$
  \item
  \label{item:Generics}
   $\dot{\bG}_{\bar{i}}$ is the pointwise image of the generic under $\dot{\sigma}$ and is $\bar{\P}_i$-generic over $\check{\bN}$
  \item 
  \label{item:LastAssumptionSP}
  $\dot{\sigma}:\check{\bN}[\dot{\bG}_{\bar{i}}]\prec \check{N}[\dot{G}_i]$ is countable, transitive and full.
\end{enumerate}
{\bf then} there is a condition $p^*\in P_j$ such that $p^*\hook [i, j)\le q \hook [i, j)$,
$p^*\rest i=p$ and whenever $G_j\ni p^*$ is $\P_j$-generic, then in $\V[G_j]$, there is a $\sigma'$ such that, letting $\sigma=\dot{\sigma}^{G_i}$, the following hold:
\begin{enumerate}[label=(\alph*)]
  \item 
  \label{item:FirstConclusion}
  $\sigma'(\bs,\bar{\vec{\P}},\bar{i},\bar{j},\btheta,\bG_{\bar{i}}, \bar{q})=\sigma(\bs,\bar{\vec{\P}},\bar{i},\bar{j},\btheta,\bG_{\bar{i}}, q),$
  \item
  \label{item:Sigma'MovesEverythingCorrectly}
  $(\sigma')^{-1} G_{i,j} : = \bG_{i, j}$ is $\P_{i, j}$-generic over $\bN[\bG_{\bar{i}}]$,
  \item
  \label{item:LastConclusionSP}
  $\sigma':\bN[\bG_{\bar{i}}]\prec N[G_i]$.
\end{enumerate}
\label{spiteration}
\end{theorem}

Many parts of this proof are verbatim the same as in Theorem \ref{sciteration} but I repeat them for the convenience of the reader.

\begin{proof}
Like in the previous proof, the idea is to induct on $j$. So let me assume that $\phi(j')$ holds for every $j'<j$.
%Fixing $j$, we prove the claim by induction on $i$. 
Fix some $i\le j$.
%assuming the theorem is proven for all $i'<i$ (and also for all $i'$, $j'$ with $i'\le j'<j$). 
Since nothing is to be shown when $i=j$, let $i<j$. In particular, the case $j=0$ is trivial.

Let me fix $p\in \P_i$, $\dot{\sigma}, \dot{\bG}_{\bar{i}}\in\V^{\P_i}$, $q\in P_j$, and without loss suppose $p\le q\rest i$. Also me fix $\theta$, $\tau$, $A$, $N$, $\bN$, $\bs,\bar{\vec{\P}},\bar{i},\bar{j}\in\bN$, $\bar{G}_{\bar{i}},\bar{G}_{\bar{i},\bar{j}}\sub\bN$ so that assumptions \ref{item:FirstAssumption}-\ref{item:LastAssumptionSP} hold.

%When it causes no confusion, we will again employ our \say{bar} convention.

\noindent{Case 1:} $j$ is a limit ordinal.

Let $\{t_n \; | \; n < \omega \}$ enumerate the $\bar{\P}_i$-names in $\bN$. Without loss of generality we may assume that $t_0$ is the check name for $\emptyset$. Also let $\{\bar{D}_n \; | \; n <\omega \}$ enumerate the names in $\bN$ for the dense open subsets of $\bar{\P}_{i, j}$. Without loss assume that $\bar{q}$ is forced to be in $\bar{D}_0$.

As before we may assume that there is a definable well ordering of the universe, $\leq_\bN$, of $\bN$.
As noted in \cite{miyamoto}, by Lemma \ref{2.7}, in $\bN$, $\bar{p}_0$ is a mixture up to $\bar{j}$ of some nested antichain in $\bar{\vec{\P}}\rest{\bar{j}}$ whose root has length $\bar{i}$. Letting $\bar{W}$ be the $\leq_\bN$-least one, we know that $p$ forces that $\dot{\sigma}(\check{\bar{W}})$ is the $L_\tau[A]$-least nested antichain $W$ in $\vec{\P}\rest j$ such that $p_0$ is a mixture of $W$ up to $j$, since we know that $p$ forces that $\bar{p}_0,\bar{i},\bar{j}$ are mapped to $p_0,i,j$ by $\dot{\sigma}$, respectively.

I will define a nested antichain $\langle T, \langle T_n \; | \; n < \omega \rangle, \langle \suc^n_T \; | \; n < \omega\rangle \rangle$, a fusion structure $\seq{\kla{q^{(a, n)}, T^{(a, n)}}}{n < \omega, a\in T_n}$ in $\seq{\P_\alpha}{\alpha \leq j}$ and a sequence $\langle \dot{\sigma}^{(a, n)} \; | \; n < \omega, a \in T_n\rangle$ so that the following conditions hold.

\begin{enumerate}[label=(\arabic*)]
\item
\label{claim:Start}
$T_0 = \{p\}$, $q^{(p,0)}=x$, $T^{(p,0)}=W$ and $\dot{\sigma}^{(p,0)} = \dot{\sigma}$.
\end{enumerate}
Further, for any $n<\omega$ and $a\in T_n$:
\begin{enumerate}[label=(\arabic*)]
\setcounter{enumi}{1}
\item %6
\label{item:middlepart-beginning}
$q^{(a, n)} \leq q$ and 
$q^{(a, n)}\in \P_j$ and $\dot{\sigma}^{(a, n)}$ is a $\P_{l(a)}$-name,
\item %7
\label{item:sigma(a,n)}
$a$ forces the following statements with respect to $\P_{l(a)}$:
\begin{enumerate}[label=(\alph*)]
\item 
\label{subitem:iselementary}
$\dot{\sigma}^{(a, n)} : \check{\bN}[\dot{\bG}_{\bar{i}}]\prec \check{N}[\dot{G}_i]$,
\item 
\label{subitem:moveseverythingcorrectly}
$\dot{\sigma}^{(a, n)}(\check{\bar{\theta}}, \check{\bar{\vec{\P}}}, \bar{s}, \check{\bar{i}},\check{\bar{j}},\dot{\bar{G}}_{\bar{i}}, \check{\bar{q}}) = \check{\theta}, \check{\vec{\P}}, \check{s}, \check{i}, \check{j}, \dot{G}_i, \check{q}$
%\item ${\mathrm{Hull}}^N(\delta(P_\nu)\cup\ran(\dot{\sigma}^{(a, n)})) = {\mathrm{Hull}}^N(\delta(P_\nu)\cup\ran(\check{\sigma}))$.
\item 
\label{subitem:hasx(a,n)initsrange}
 $q^{(a,n)}\in\ran(\dot{\sigma}^{(a,n)})$ and its preimage is in $\bar{D}_n$.
\end{enumerate}
%\item %8
%$a$ decides $\dot{\sigma}^{(a, n)}(t_n)$.
\item %9
\label{item:middlepart-end}
for some $\bar{q}^{(a,n)}\in\bar{\P}_{\bar{j}}$ and $\bar{T}^{(a,n)}\in\bN$, we have that $\bar{q}^{(a, n)} \in \bG_{\bar{i}}*\bar{G}_{\bar{i},\bar{j}}$, $\bar{q}^{(a, n)} \in \bar{D}_n$ and \\
$a \forces \dot{\sigma}^{(a, n)} (\kla{\bar{q}^{(a, n)}, \bar{T}^{(a, n)}, l(\bar{a})}) =\kla{q^{(a, n)}, T^{(a, n)}, l(a)}$.
\end{enumerate}
\noindent If $b \in {\rm suc}^n_T (a)$ and $m\le n$, then
\begin{enumerate}[label=(\arabic*)]
\setcounter{enumi}{4}
\item %10
\label{item:coherence}
$b \forces_{l(b)} \dot{\sigma}^{(b, n+1)}(t_m) = \dot{\sigma}^{(a, n)}(t_m)$ and $\dot{\sigma}^{(b, n+1)}(\bar{D}_m) = \dot{\sigma}^{(a, n)}(\bar{D}_m)$.

%\ \text{and}\ \dot{\sigma}^{(b, n+1)}(v_m) = \dot{\sigma}^{(a, n)}(v_m),$ where $v_m$ is the $\bN$-least $v$ with $\sigma(t_m) \in \dot{\sigma}^{(a_m, m)}(v)$ and $\card{v}^{\bN} \leq \bar{\delta(P_\nu)}$, where $a_m$ is the predecessor of $a$ in $T_m$. Note that such a $v$ always exists by Fact \ref{ufact}.
\item %11
\label{claim:LastOfFirstList}
$b\rest l(a)\forces_{l(a)}q^{(b,n+1)},T^{(b,n+1)}\in\ran(\dot{\sigma}^{(a,n)})$.
\end{enumerate}
First let's see that constructing such objects is sufficient to prove the existence of a condition $p^*$ as in the statement of the theorem. Suppose that we have constructed sequences satisfying \ref{claim:Start} through \ref{claim:LastOfFirstList} above. Let $q^*\in \P_j$ be a fusion of the fusion structure, and let $p^*=p\verl q^*\rest[i,j)$. By \ref{claim:Start}, we have that $q^*\rest i\equiv p$, so $p^*\equiv q^*$ and $p^*\rest i=p$, as required.
To see that $p^*$ is as wished, let $G_j$ be $\P_j$-generic over $\V$ with $p^* \in G_j$. I have to show that in $\V[G_j]$, there is a $\sigma'$ so that conclusions \ref{item:FirstConclusion}-\ref{item:LastConclusionSP} are satisfied.
Since $p^*\equiv q^*$, we have that $q^*\in G_j$. Work in $\V[G_j]$. By Proposition \ref{fusion} there is a sequence $\langle a_n \; | \; n < \omega\rangle \in V[G_j]$ so that for all $n < \omega$, $a_{n+1} \in {\rm suc}^n_T(a_n)$, $a_n \in G_j \hook l(a_n)$ and $q^{(a_n, n)} \in G_j$. Let $\sigma_n$ be the evaluation of $\dot{\sigma}^{(a_n, n)}$ by $G_j$. Then, as in the iteration theorem for subcompleteness, I define $\sigma': \bN \to N$ to be the map such that $\sigma ' (t_n) = \sigma_n (t_n)$. %Note that $a_n$ decides $\dot{\sigma}^{(a, n)}(t_n)$, and this is exactly the value given in the extension. 
I claim that $\sigma'$ satisfies the conclusions \ref{item:FirstConclusion}-\ref{item:LastConclusionSP}. Indeed the verification of this fact exactly mirrors the case of subcompleteness with one difference: we need to ensure that the pointwise preimage of $G_{i, j}$ is $\bar{\P}_{i, j}$-generic over $\bN[\bG_{\bar{i}}]$. However this requirement is taken care of in the construction since the pre-image of $q^{(a, n)}$ is in the evaluation of $\bar{D}_n$ by \ref{item:sigma(a,n)} \ref{subitem:hasx(a,n)initsrange}.

Thus it remains to see that such a construction can be carried out. This is done by induction on $n$, in a manner similar to the previous proof. Like last time, at stage $n+1$ of the construction, assume that $T_m$, $T^{(a,m)}$, $q^{(a,m)}$ and $\dot{\sigma}^{(a,m)}$ have been defined, for all $m\le n$ and all $a\in T_m$. Also, for $m<n$ and $a\in T_m$, we assume that $\suc_T^m(a)$ has been defined. The inductive hypothesis is that for all $m\le n$ and all $a\in T_m$, conditions 
\ref{item:middlepart-beginning}-\ref{item:middlepart-end}
hold, and that for all $m<n$, all $a\in T_m$ and all $b\in\suc_T^m(a)$, conditions \ref{item:coherence}-\ref{claim:LastOfFirstList} are satisfied. In order to define $T_{n+1}$, I will specify $\suc_T^n(a)$, for every $a\in T_n$, which implicitly defines $T_{n+1}=\bigcup_{a\in T_n}\suc_T^n(a)$. Simultaneously, I will define, for every such $a$ and every $b\in\suc_T^n(a)$, the objects $T^{(b,n+1)}$, $q^{(b,n+1)}$ and $\dot{\sigma}^{(b,n+1)}$ in such a way that whenever $a\in T_n$ and $b\in\suc_T^n(a)$, \ref{item:coherence}-\ref{claim:LastOfFirstList} are satisfied by $a$ and $b$, and \ref{item:middlepart-beginning}-\ref{item:middlepart-end} are satisfied by $b$ and $n+1$ (instead of $a$ and $n$).

For stage 0 of the construction, notice that \ref{claim:Start} gives the base case where $n =0$ and in this case \ref{item:middlepart-beginning}-\ref{item:middlepart-end} are satisfied, $p$ forces that $q^{(p, 0)}=q=\dot{\sigma}(\bar{q}) \in \bar{D}_0$ and $p$ forces that $W=T^{(p,0)}$ has a preimage under $\dot{\sigma}$, namely $\bar{W}$.

At stage $n+1$ of the construction, work under the assumptions described above. Fixing $a\in T_n$, we have to define $\suc_T^n(a)$.
To this end let $D$ be the set of all conditions $b$ for which there are a nested antichain $S$ in $\vec{\P}\rest j$, and objects $\dot{\sigma}^b$, $u$, $\bar{u}$ and $\bar{S}$ satisfying the following:

\begin{enumerate}[label=(D\arabic*)]
\item %12
\label{item:FirstConditionDefiningPredenseSet}
$b \in \P_{l(b)}$ and $l(b) < j$.
\item %13
$l(a) \leq l(b)$ and $b \hook l(a) \leq a$.
\item %14
$S \hooks T^{(a, n)}$, $\bS\in\bN$, $S\in N$.
\item %15
$u\in \P_j$, $u \leq x^{(a,n)}$ and $u$ is a mixture of $S$ up to $j$.
\item %16
$\bar{u} \hook \bar{i} \in \bG_{\bar{i}}$, and $\bar{u} \in \bar{D}_{n+1}$ (in $\bN[\bG_{\bar{i}}]$).
%\item %17
%$b$ forces that $\dot{\sigma}^b$ is a function.%, and $b$ decides $\dot{\sigma}^b(t_{n+1})$.
\item %18
\label{item:InitialSegmentOfbForcesStuffIntoRangeOfTheOldEmbedding}
$b\rest l(a)\forces_{l(a)} S, u, l(b) \in\ran (\dot{\sigma}^{(a,n)})$.
%$b \forces_{l(b)} S, u, l(b) \in\ran (\dot{\sigma}^b)$.
\item %19
\label{item:LastConditionDefiningPredenseSet}
$b$ forces the following statements with respect to $\P_{l(b)}$: \begin{enumerate}
\item $\dot{\sigma}^b (\check{\bar{\theta}},\check{\bar{i}},\check{\bar{j}},\check{\bar{\vec{\P}}},\dot{\bG}_{\bar{i}}, \check{\bar{s}},\check{\bar{u}},\check{\bS}, \check{\bar{q}}) = \check{\theta},\check{i},\check{j},\check{\vec{\P}},\dot{G}_i,\dot{\sigma}(\bs), \check{u},\check{S}, \check{q}$.
\item $\forall m \leq n$ $\dot{\sigma}^b(t_m) = \dot{\sigma}^{(a, n)}(t_m)$ and $\dot{\sigma}^b(\bar{D}_m) = \dot{\sigma}^{(a, n)}(\bar{D}_m)$,
\item $\dot{\sigma}^b: \check{\bN}[\dot{\bG}_{\bar{i}}] \prec \check{N}[\dot{G}_i]$,
%\item ${\mathrm{Hull}}^N(\delta(P_\nu) \cup \ran (\dot{\sigma}^b)) = {\mathrm{Hull}}^N(\delta(P_\nu) \cup \ran (\dot{\sigma}^{(a, n))})$,
\end{enumerate}
\end{enumerate}

Note that if $b\in D$ and $b'\le_{l(b)}b$, then $b'\in D$ as well. It follows that $D\rest l(a):=\{b\rest l(a)\st b\in D\}$ is open in $\P_{l(a)}$.
Thus, it suffices to show that $D\rest l(a)$ is predense below $a$ in $\P_{l(a)}$. For if we know this, $D\rest l(a)$ is dense below $a$, and we may choose a maximal antichain $A\sub D\rest l(a)$ (with respect to $\P_{l(a)}$), which then is a maximal antichain in $\P_{l(a)}$ below $a$. Thus, for every $c\in A$, we may pick a condition $b(c)\in D$ such that $b(c)\rest l(a)=c$, and define $\suc_T^n(a)=\{b(c)\st c\in A\}$. Now, for every $b\in\suc_T^n(a)$, let $S$, $\dot{\sigma}^b$, $u$ and $\bar{u}$ witness that $b\in D$, i.e., let them be chosen in such a way that \ref{item:FirstConditionDefiningPredenseSet}-\ref{item:LastConditionDefiningPredenseSet} hold. Set $T^{(b,n+1)}=S$, $\dot{\sigma}^{(b, n+1)}=\dot{\sigma}^b$, $q^{(b, n+1)}=u$, $\bar{q}^{(b,n+1)}$ and $\bar{T}^{(b,n+1)}=\bS$. Then $a$, $b$ satisfy \ref{item:coherence}-\ref{claim:LastOfFirstList} at stage $n$, and $b$ satisfies \ref{item:middlepart-beginning}-\ref{item:middlepart-end} at stage $n+1$.

To see that $D\rest l(a)$ is predense below $a$, let
$G_{l(a)}$ be $\P_{l(a)}$-generic over $V$ with $a \in G_{l(a)}$. I have to find a $b\in D$ so that $b \hook l(a) \in G_{l(a)}$. Work in $\V[G_{l(a)}]$. Let $\sigma_n = (\dot{\sigma}^{(a, n)})_{G_{\l(a)}}$. Since \ref{item:sigma(a,n)} holds at stage $n$, we have that $\sigma_n:\bN[\bG_{\bar{i}}]\prec N[G_i]$ %${\mathrm{Hull}}^N(\delta(P_\nu) \cup \ran(\sigma)) = {\mathrm{Hull}}^N(\delta(P_\nu) \cup \ran(\sigma_n))$, 
and $\sigma_n(\bar{\theta}, \bar{\vec{\P}},\bs,\bar{\nu},\bar{i},\bar{j},\bG_{\bar{i}})
=\theta,\vec{\P},\sigma(\bs),\nu,i,j,G_i$.

We also have objects $\bar{q}^{(a,n)},\bT^{(a,n)},l(\bar{a})$ satisfying condition \ref{item:middlepart-end}, so that $\bar{q}^{(a,n)}\in\bG_{\bar{i}}*\bG_{\bar{i},\bar{j}}$, $\bT^{(a,n)}\in\bN$, $\bar{q}^{(a,n)}\le\bp_n$, $\sigma_n(\bar{q}^{(a,n)},\bT^{(a,n)})=q^{(a,n)},T^{(a,n)}$ and $\sigma_n(l(\bar{a}))=l(a)$.

Let's first find $u$ and $\bar{u}$ again. By elementarity, $\bar{q}^{(a, n)}$ is a mixture of $\bar{T}^{(a, n)}$ up to $\bar{j}$. We have that $\sigma_n:\kla{L_{\btau}[\bA][\bG_{\bar{i}}],\in,\bA}\prec\kla{L_\tau[A][G_i],\in,A}$, so in particular, $\bsigma:=\sigma_n\rest L_\btau[\bA]:\kla{L_\btau[\bA],\in,\bA}\prec\kla{L_\tau[A],\in,A}$, and $\bsigma(\bar{q}^{(a,n)},\bT^{(a,n)})=q^{(a,n)},T^{(a,n)}$. Clearly, in $L_\tau[A]$, it is true that $q^{(a,n)}$ is a mixture of $T^{(a,n)}$ up to $j$, so it is true in $L_\btau[\bA]$ that $\bar{q}^{(a,n)}$ is a mixture of $\bT^{(a,n)}$ up to $\bar{j}$, and by absoluteness, this it true in $\V$ as well.

Let $\bT_0^{(a,n)}=\{\ba_0\}$. Since we're working in $\V[G_{l(a)}]$, we have access to all $\bG_{k}$ for $k \leq l(a)$ by considering the pointwise pre image of $\sigma^{(a, n)}$. By induction these are all generic over $\bN[G_{\bar{i}}]$. Moreover, by Fact \ref{fact:CharacterizationOfMixtures}.\ref{item:Root}, $\ba_0\equiv\bar{q}^{(a,n)}\rest l(\ba_0)\in\bG_{l(\ba_0)}$, since $l(\bar{q}^{(a,n)})=\bar{j}>l(\ba_0)$. So $\ba_0\in\bG_{l(\ba_0)}$.
Let $\bar{r}\in\bT_1^{(\ba,n)}$, that is, $\bar{r}\in\suc_{\bT^{(\ba,n)}}^0(\ba_0)$, be such that $\bar{r}\rest l(a_0)\in\bar{G}_{l(\ba_0)}$. There is such a $\bar{r}$ by Definition \ref{definition:NestedAC-Mixture-beta-Nice-hooks}(5). By Fact \ref{fact:CharacterizationOfMixtures}.\ref{item:HigherUp}, again since $l(\bar{r})<\bar{j}=l(\bar{q}^{(a,n)})$, it follows that $\bar{r}\equiv\bar{r}\rest l(\ba_0)\verl\bar{q}^{(a,n)}\rest[l(\ba_0),l(\bar{j}))$.
Since $\bar{r}\rest l(\ba_0)\in\bG_{l(\ba_0)}$, this implies that $\bar{r}\rest[l(\ba_0),l(\bar{z}))\equiv\bar{q}^{(a,n)}\rest[l(\ba_0),l(\bar{r}))$ (in the partial order $\bar{P}_{l(\ba_0),l(\bar{r})}=\bar{P}_{l(\bar{r})}/\bG_{l(\ba_0)}$), and $\bar{q}^{(a,n)}\rest[l(\ba_0),l(\bar{r}))\in\bG\rest[l(\ba_0),l(\bar{r}))$.
So we have that $\bar{r}\rest l(\ba_0)\in\bG_{l(\ba_0)}$ and $\bar{r}\rest[l(\ba_0),l(\bar{r}))\in\bG\rest[l(\ba_0),l(\bar{r}))$. By Fact \ref{fact:FactorsAndQuotients}.\ref{item:CriterionForBeingInGeneric}, this implies that $\bar{r}\in\bG_{l(\bar{r})}$.

Now let $\bar{u}$ strengthen $\bar{r}^{\frown}\bar{r}^{(a, n)} \hook [\bar{l(r)}, \bar{j})$ so that $\bar{u} \in \bar{D}_{n+1}$.
By Lemma \ref{2.11}, applied in $\bN$, there is a nested antichain $\bar{S} \hooks \bar{T}^{(a, n)}$ such that $\bar{u}$ is a mixture of $\bS$ up to $\bar{j}$ and such that letting $\bS_0 = \{\bd_0\}$, we have that  $l(\bar{r}) \leq l(\bd_0)$ and $\bd_0 \rest l(\bar{r}) \leq \bar{r}$. Let $S,d_0,u=\sigma_n(\bS,\bd_0,\bu)$, and let $w\in G_{l(a)}$ force this. Since $a\in G_{l(a)}$, we may choose $w$ so that $w\le a$.

Note that $S,d_0,u$ are in $N$ (and hence in $\V$), since $\bS,\bd_0,\bu\in\bN$.

I will apply our inductive hypothesis $\phi(l(d_0))$, noting that $l(d_0)<j$, to $i =l(a) \leq l(d_0)$, the filters $\bG_{l(\bar{a})}$, $\bG_{{l(\bar{a})},l(\bd_0)}$, the models $\bN$, $N$, the condition $w$ (in place of $p$), the name $\dot{\sigma}^{(a,n)}$ (in place of $\dot{\sigma}$ and the parameter $\bs'\in\bN$ which I will specify below.
No matter which $\bs'$ chosen, by the inductive hypothesis, there is a condition $w^*\in\P_{l(d_0)}$ with $w^*\rest l(a)=w$ and a name $\dot{\sigma}'$ such that $w^*$ forces with respect to $\P_{l(a)}$:
%such that if $H\ni w^*$ is generic for $\P_{l(d_0)}$, then in $\V[H]$ there is a $\sigma'$ such that letting $\sigma'_n=(\sigma^{(a,n)})^H$,
\begin{enumerate}[label=(\alph*)]
  \item
  $\dot{\sigma}'(\check{\bs}',\check{\bar{\vec{\P}}},\check{l(\bar{a})},\check{l(\bar{d}_0)},\check{\btheta},\dot{\bG}_{l(\bar{a})})=\dot{\sigma}^{(a,n)}((\check{\bs}',\check{\bar{\vec{\P}}},\check{l(\bar{a})},\check{l(\bar{d}_0)},\check{\btheta},\dot{\bG}_{l(\bar{a})})$,
  \item
  $(\dot{\sigma}'{}^{-1})`` \dot{G}_{l(a),l(d_0)}$ is generic over $\check{\bN}[\dot{\bG}_{l(\bar{a})}]$, and
  \item
  $\dot{\sigma}':\check{\bN}[\dot{\bG}_{l(\bar{a})}]\prec \check{N}[\dot{G}_{l(a)}]$.
\end{enumerate}

By choosing $\bs'$ appropriately, and temporarily fixing $H$ as above, we may insure that it is forced that $\dot{\sigma}'$ moves any finite number of members of $\bN$ the same way $\dot{\sigma}^{(a,n)}$ does.
Thus, we may insist that $w^*$ forces that $\dot{\sigma}'(\bu,\bd_0,\bS)=\dot{\sigma}^{(a,n)}(\bu,\bd_0,\bS)$. Recall that $w$ forced that $\dot{\sigma}^{(a,n)}(\bu,\bd_0,\bS)=u,d_0,S$.
Hence, since $w^*\rest l(a)=w$, we get that $w^*$ forces that $\dot{\sigma}'(\bu,\bd_0,\bS)=u,d_0,S$ as well.

In addition, we may insist that $\sigma'$ moves the parameters $\bar{i},\bar{j},\bar{\vec{\P}},\btheta,\bs,\bp_0,\ldots,\bp_n,t_0,\ldots,t_n$ the same way $\dot{\sigma}^{(a,n)}$ does. Note that already $a$ forced with respect to $\P_{l(a)}$ that $\bar{i},\bar{j},\bar{\vec{\P}},\btheta$ are mapped to $i,j,\vec{\P},\theta$ by $\dot{\sigma}^{(a,n)}$.

%Note that $\sigma'$ lifts to an elementary embedding $\sigma':\bN[\bG_{l(\bd_0)}]\prec N[H]$ with $\sigma'(\bG_{l(\bd_0)})=H_{l(d_0)}$. It follows that $\sigma'(\bu\rest l(\bd_0))=u\rest l(d_0)\in H$, since $\bu\in\bG_{\bar{j}}$, so that $\bu\rest l(\bd_0)\in\bG_{l(\bd_0)}$.

%Since $\bu\in\bG$, it follows that $\bu\rest l(d_0)\in\bG\rest l(\bd_0)$, and so, $w^*$ forces that $u\rest l(d_0)=\sigma'(\bu\rest l(\bd_0))\in\dot{G}_{l(d_0)}$. Thus, given any $P_{l(a),l(d_0)}$-generic filter $H$ containing $w^*$, we may find a common extension of $w^*$ and $u\rest l(d_0)$ in $H$. So we may assume that $w^*\le u\rest l(d_0)$.

Now, set $b=w^*$, $\dot{\sigma}^b=\dot{\sigma}'$. It follows that the conditions \ref{item:FirstConditionDefiningPredenseSet}-\ref{item:LastConditionDefiningPredenseSet} are satisfied, that is, $b\in D$. Most of these are straightforward to verifty; let me just remark that $b$ forces that $\dot{\sigma}^b(\bG_{\bar{i}})=\dot{G}_i$ because it forces that $\dot{\sigma}^b(\bar{i})=i$ and $\dot{\sigma}^b(\bG_{l(\bar{a})})=\dot{G}_{l(a)}$. Condition \ref{item:InitialSegmentOfbForcesStuffIntoRangeOfTheOldEmbedding} holds because $b\rest l(a)=w$. 
For the same reason, we have that $b\rest l(a)\in G_{l(a)}$, completing the proof that $D\rest l(a)$ is predense below $a$. This concludes the treatment of case 1.

\noindent\emph{Case 2:} $j$ is a successor ordinal.

Let $j=k+1$. Since we assumed $i<j$, it follows that $i\le k$. Inductively, we know that $\phi(k)$ holds. Note that $\bar{j}$ is of the form $\bk+1$, where $p$ forces with respect to $\P_i$ that $\dot{\sigma}(\bk)=k$, and if we let $\bG_{\bk}=\bG_\bj\rest\bk$, then the assumptions \ref{item:FirstAssumption}-\ref{item:LastAssumption} are satisfied by 
$p\in \P_i$, $\dot{\sigma}\in\V^{P_i}$, $q\in P_j$, $y\le q\rest i$,
$\theta$, $\tau$, $A$, $N$, $\bN$, $\bs,\bar{\vec{\P}},\bar{i},\bar{k}\in\bN$, and $k$. By $\phi(k)$, we obtain a condition $p^{**}\in\P_k$ with $p^{**}\rest i=p$ and a $\P_k$-name $\dot{\bsigma}$ such that $p^{**}$ forces
\begin{enumerate}[label=(\alph*1)]
  \item
  $\dot{\bsigma}(\check{\bs},\check{\bar{\vec{\P}}},\check{\bar{i}},\check{\bk},\check{\bar{j}},\check{\btheta},\check{\bar{q}})=\sigma(\check{\bs},\check{\bar{\vec{\P}}},\check{\bar{i}},\check{\bk},\check{\bar{j}},\check{\btheta},\check{\bar{q}})$,
  \item
  $\dot{\bsigma}^{-1}``\dot{G}_{i,k}$ is generic over $\check{\bN}[\dot{\bG}_{\bar{i}}]$
  \item
  $\dot{\bsigma}:\check{\bN}[\dot{\bG}_{\bar{i}}]\prec \check{N}[\dot{G}_i]$.
\end{enumerate}
It follows then that $p^{**}$ forces that $\dot{\bsigma}^{-1}``\dot{G}_k$ is generic over $\bN$, and hence that $\dot{\bsigma}$ lifts to an elementary embedding from $\bN[\bG_\bk]\prec N[\dot{G}_k]$ that maps $\bG_\bk$ to $\dot{G}_k$. Let $\dot{\tsigma}$ be a $\P_k$-name such that $p^{**}$ forces that $\dot{\tsigma}$ is that lifted embedding.

Temporarily fix a $\P_k$-generic filter $H$ that contains $p^{**}$. In $\V[H]$, the forcing $\P_{k,k+1}=\P_{k,j}=\P_j/H$ is $\infty$-subproper. Letting $\tsigma=\dot{\tsigma}^H$, we have that $\tsigma:\bN[\bG_\bk]\prec N[H]$, and thus, since $\bN[\bG_\bk]$ is full, there is a condition $r$ in $\P_{k,j}$ such that $r$ forces the existence of an elementary embedding $\sigma'$ with
\begin{enumerate}[label=(\alph*2)]
  \item
  $\sigma'(\check{\bs},\check{\bar{\vec{\P}}},\check{\bar{i}},\check{\bk},\check{\bar{j}},\check{\btheta},\check{\bar{q}},\dot{\bG}_{\bar{i}},)=\tsigma(\check{\bs},\check{\bar{\vec{\P}}},\check{\bar{i}},\check{\bk},\check{\bar{j}},\check{\btheta},\check{\bar{q}},\dot{\bG}_{\bar{i}},)$,
  \item
  $(\sigma')^{-1}``\dot{G}_{k,j}$ is generic over $\check{\bN}[\dot{\bG}_{\bar{k}}]$
  \item
  $\sigma':\check{\bN}[\dot{\bG}_{\bar{k}}]\prec N[H]$.
\end{enumerate}
Since this holds in $\V[H]$ whenever $p^{**}\in H$, there is a $\P_k$-name $\tau$ which is essentially a name for $r$ above - more precisely, $\tau$ is such that $p^{**}$ forces that $\tau\in \P_j$, $\tau\rest k\in\dot{G}_k$ and $\tau\rest[k,j)$ has the properties of $r$, as listed above.
Since the iteration is nice, there is a condition $p^*\in \P_j$ such that $p^*\rest k=p^{**}$ and $p^*$ forces that $\tau\rest[k,j)\equiv p^*\rest[k,j)$; see Definition \ref{def:NiceIterations}, part (1). I claim that $p^*$ is as wished.

First, note that $p^*\rest i=(p^*\rest k)\rest i=p^{**}\rest i=p$. Now, let $G_j$ be a $\P_j$-generic filter with $p^*\in G_j$. We have to show that in $\V[G_j]$, there is a $\sigma'$ such that, letting $\sigma=\dot{\sigma}^{G_i}$, the following hold:
\begin{enumerate}[label=(\alph*)]
  \item
  $\sigma'(\bs,\bar{\vec{\P}},\bar{i},\bar{j},\btheta,\bar{q})=\sigma(\bs,\bar{\vec{\P}},\bar{i},\bar{j},\btheta,\bar{q})$,
  \item
  $(\sigma ' {}^{-1})``G_{i,j}$ is generic over $\bN[\bG_{\bar{i}}]$,
  \item
  $\sigma':\bN[\bG_{\bar{i}}]\prec N[G_i]$.
\end{enumerate}
But this follows, because $\V[G_j]=\V[G_k][G_{k,j}]$, where $p^*\rest[k,j)\in G_{k,j}$, where $p^{**}\in G_k$, so writing $H$ for $G_k$, this is the situation described above. Moreover, $p^*\rest[k,j)\in G_{k,j}$ and $p^*\rest[k,j)\equiv\dot{r}^{G_j}$, where $\dot{r}$ is a name for the condition $r$ mentioned above. Thus, there is a $\sigma'$ in $\V[G_j]$ such that the conditions (a2)-(c2) listed above hold in $\V[G_j]$. Remembering that $\tsigma$ lifts $\bsigma$ and $\bsigma$ moves the required parameters as prescribed (by (a1)-(c1)), it follows that (a)-(c) are satisfied.
\end{proof}

This completes the iteration theorems for $\infty$-subcomplete and $\infty$-subproper forcing notions. Let me make one strengthening of these theorems that will be useful in applications. The key step in both proofs was the construction of the fusion sequence in the limit stage and in particular the construction of the conditions $u$ and $\bar{u}$. However, in both proofs $u$ only needed to be stronger than a certain condition and it could have been strengthened it further if needed. Thus the following theorem comes from the proofs above.
\begin{theorem}
Let $\langle \P_\alpha \; | \; \alpha \leq \nu \rangle$, $\theta$, $N$, etc be as in either Theorem \ref{sciteration} or Theorem \ref{spiteration} with $\nu$ limit and suppose for all $n < \omega$ we have that $E_n \subseteq P_\nu$ satisfies the following: for every $p \in \mathbb P_\nu$ and every $\alpha < \nu$ if in $\V^{P_\alpha}$ there is a name for an embedding $\dot{\sigma} : \bN \prec N$ with $p$ forced to be in the range of $\dot{\sigma}$ and for any $u \leq p$ in the range of $\dot{\sigma}$ with $u \hook \alpha \in \dot{G}_\alpha$ then there is an $s \leq u$ in the range of $\dot{\sigma}$ so that $s \hook \alpha \in \dot{G}_\alpha$ and $s \in E_n$. Then there is a $q \leq p$ forcing that there is a decreasing sequence $q_0 \geq q_1 \geq ... \geq q_n \geq ...$ all in $G$ so that $q_{n+1} \in E_n$. In particular, $q$ forces $G \cap E_n \neq \emptyset$ for all $n < \omega$.
\label{thmEn}
\end{theorem}

\begin{proof}
In the case of semiproper forcing this is checked in detail by Miyamoto as \cite[Lemma 4.3]{miyamoto}. Making the exact same modification he makes in that case to my proofs of Theorems \ref{sciteration} and \ref{spiteration} works here. The reader is referred to Miyamoto's paper for the details. 
\end{proof}

\section{Trees}
In this section I lift some results about preservation of properties of trees from \cite{miyamoto} to the context of $\infty$-subproper forcing.

\begin{lemma}
Let $S = (S, \leq_S)$ be a Souslin tree and $\mathbb P = \langle \P_\alpha \; | \; \alpha \leq \nu\rangle$ be a nice iteration of $\infty$-subproper forcings such that for each $i$ with $i + 1 \leq \nu$, $\forces_i \P_{i, i+1} \; {\rm preserves \; } \check{S}$ then $\P_\nu$ preserves $S$.
\label{souslinpre}
\end{lemma}

Let me stress that the proof of this theorem is similar to that of Lemma 5.0 and Theorem 5.1 of \cite{miyamoto}.

\begin{proof}
This is proved by induction on $\nu$. The successor stage is by hypothesis so I focus on the limit case and the inductive assumption is not just that $P_{i, i+1}$ preserves $S$ but in fact $P_i$ preserves $S$ for all $i < \nu$.  From now on assume $\nu$ is a limit ordinal. Let $\dot{A}$ be a $\P_\nu$ name for an antichain of $S$ and let $p \in \P_\nu$ force that $\dot{A}$ is maximal. I need to find a $q \leq p$ forcing that $\dot{A}$ is countable. Fix $\theta$ sufficiently large that $\dot{A}, \mathbb P, S \in H_\theta$ and fix $\sigma:\bN \prec N$ as in the standard setup. Denote by $\delta = \omega_1 \cap \bN$. Note that for all $\alpha < \delta$ and all $s \in S_\alpha$ we may assume that $\sigma (\bar{s}) = s$ since we may assume $S \subseteq H_{\omega_1}$. Enumerate the $\delta^{\rm th}$ level of $S$ as $\langle s_n \; | \; n < \omega\rangle$. For each $n$, define $E_n = \{r \in \P_\nu \; | \; \exists s \in S \; {\rm such \; that} \; s <_S s_n \; {\rm and} \; r \forces s \in \dot{A}\}$. We need to check that the $E_n$'s satisfy the predensity condition stipulated in Theorem \ref{thmEn}. If we can do this then it follows there is a $q$ forcing that $\dot{G} \cap E_n \neq \emptyset$ for all $n< \omega$ and hence that the maximal antichain is bounded below $\delta$ so countable.

To check the predensity condition, fix $u \leq p$, in the range of $\sigma$, $\alpha < \nu$ and a $P_\alpha$-name $\dot{\sigma}_\alpha$ which will evaluate to an embedding witnessing the $\infty$-subproperness of $P_\alpha$. Let $G_\alpha$ be $\P_\alpha$-generic over $V$, $\sigma_\alpha = (\dot{\sigma}_\alpha)^{G_\alpha}$ and $\sigma_\alpha(\bar{u}, \bar{\alpha}, \bar{S}) = u, \alpha, S$. I want to find a $\bar{r} \in \bN \cap \bar{P}_\nu$ so that $\bar{r} \leq \bar{u}$, $\sigma_\alpha (\bar{r} \hook \bar{\alpha}) = r \hook \alpha \in G_\alpha$ and $r \in E_n$. Let $D$ be the set of $s \in \bar{S}$ for which there is a condition $\bar{r} \in \bar{\P_\nu}$ which strengthens $\bar{u}$ and so that $\bar{r}\hook \bar{\alpha} \in \bar{G}_\alpha$ and $\bar{r} \forces s \in \bar{\dot{A}}$. In symbols $D = \{s \in \bar{S} \; | \; \exists \bar{r} \in \bar{P_\nu} \;  \bar{r} \leq \bar{u} \; \bar{r} \hook \bar{\alpha} \in \bG_\alpha \; {\rm and} \; \bar{r} \forces s \in \bar{\dot{A}}\}$ where $\bG_\alpha := \sigma_\alpha^{-1} ``G_\alpha$. Note that since $\sigma_\alpha$ is an $\infty$-subcompleteness embedding it lifts to an embedding $\sigma '_\alpha : \bN[\barG_\alpha] \prec N[G_\alpha]$ and this set $D$ is in $\bN[\barG_\alpha]$. Moreover $D$ is a predense subset of $\bar{S}$ in $\bN[G_\alpha]$ since, by the maximality of $\dot{A}$ for any $t \in S$ there is densely many conditions forcing some $s \in \dot{A}$ compatible with $t$. Finally since $S$ remains Souslin in $\V[G_\alpha]$ by hypothesis and thus $\bar{S}$ remains Souslin in $\bN[\barG_\alpha]$ there is an $s \in D \cap \bN[\barG_\alpha]$ below $s_n$. Letting $\bar{r} \in \bN[\barG_\alpha]$ be the witness for this $s$ completes the proof.
\end{proof}

A similar modification of Lemma 5.2 and Theorem 5.3 of \cite{miyamoto} can be used to prove the preservation of \say{not adding uncountable branches through trees}.

\begin{lemma}
Let $T$ be an $\omega_1$-tree and let  $\mathbb P = \langle \P_\alpha \; | \; \alpha \leq \nu\rangle$ be a nice iteration of $\infty$-subproper forcings such that for each $i$ with $i + 1 \leq \nu$, $\forces_i ``\P_{i, i+1}$ does  not add an uncountable branch through $\check{T}$" then $\P_\nu$ does not add an uncountable branch through $T$.
\end{lemma}

\begin{proof}
The lemma proceeds by induction on $\nu$ and is by contradiction. Since the successor case is by assumption, the inductive hypothesis is that, for all $i < j$ $\P_i$ adds no new cofinal branch through $T$.  Let $\dot{B}$ be a $\P_\nu$ name for a branch through $T$ and, towards a contradiction, let $p \in \P_\nu$ force that $\dot{b}$ is uncountable. I need to find a $q \leq p$ forcing that actually $\dot{b}$ is countable. Fix $\theta$ sufficiently large that $\dot{b}, \mathbb P, T \in H_\theta$ and fix $\sigma:\bN \prec N$ as in the standard setup. Denote by $\delta = \omega_1 \cap \bN$. Note that for all $\alpha < \delta$ and all $t \in T_\alpha$ we may assume that $\sigma (\bar{t}) = t$ since we may assume $T \subseteq H_{\omega_1}$. Enumerate the $\delta^{\rm th}$ level of $T$ as $\langle t_n \; | \; n < \omega\rangle$. For each $n$, define $E_n = \{r \in P_\nu \; | \; \exists t \in T_{<\delta} \; {\rm such \; that} \; t \nleq_T t_n \; {\rm and} \; r \forces \check{t} \in \dot{B}\}$. I need to check that the $E_n$'s satisfy the predensity condition stipulated in Theorem \ref{thmEn}. If I can do this then it follows there is a $q \leq p$ forcing that $\dot{G} \cap E_n \neq \emptyset$ for all $n< \omega$ and hence the branch is bounded below $\delta$ so it cannot be uncountable.

Fix $u \leq p$, in the range of $\sigma$, $\alpha < \nu$ and a $\P_\alpha$-name $\dot{\sigma}_\alpha$ which will evaluate to an embedding witnessing the $\infty$-subproperness of $\P_\alpha$. Let $G_\alpha$ be $\P_\alpha$-generic over $\V$, $\sigma_\alpha = (\dot{\sigma}_\alpha)_{G_\alpha}$ and $\sigma_\alpha(\bar{u}, \bar{\alpha}, \bar{T}) = u, \alpha, T$. I want to find a $\bar{r} \in \bN \cap \bar{P}_\nu$ so that $\bar{r} \leq \bar{u}$, $\sigma_\alpha (\bar{r} \hook \bar{\alpha}) = r \hook \alpha \in G_\alpha$ and $r \in E_n$. Note that since $\P_\alpha$ didn't add an uncountable branch to $T$, there are incomparable conditions $u_1, u_2$ whose restrictions to $\alpha$ are the same but force incompatible elements $t_1$ and $t_2$ respectively to be in $\dot{B}$. By elementarity, this situation is true as well in $\bN$ using $\bar{u}$ and therefore there are $\bar{u}_1, \bar{u}_2 \leq \bar{u}$ in $\bar{\P}_\nu$ so that $\bar{u}_1 \hook \alpha = \bar{u}_2 \hook \alpha \in \bG_\alpha$ but $\bar{u}_1$ and $\bar{u}_2$ force incomparable elements $\bar{t}_1$ and $\bar{t}_2$ to be in $\bar{\dot{B}}$. At least one of these elements is incomparable with $t_n$ and both of them are of level less than $\delta$ (since they're in $\bN$). Therefore at least one of $\bar{u}_1$, $\bar{u}_2$ works.
\end{proof}

Using these preservation results I can now provide new models of $\SCFA$, providing a proof to part of Theorem \ref{mainthm}.

\begin{theorem}
Assume $\kappa$ is supercompact. Then there is a $\kappa$-length nice iteration of $\infty$-subproper forcing notions forcing $\SCFA$ so that in the extension $2^{\aleph_0} = \aleph_2 = \kappa$ and there are Souslin trees.
\label{conSCFAcohen}
\end{theorem}

The idea of the proof is to use what I call {\em weaving constructions}: weaving other forcing notions into the standard iteration that forces $\SCFA$.

\begin{proof}
By forcing if necessary, assume first that in $V$ that there is a Souslin tree $S$. Let $\kappa$ be supercompact. I will define the following $\kappa$-length nice iteration, $\mathbb P_\kappa$ as follows: let $f: \kappa \to \V_\kappa$ be a Laver function. At stage $\alpha$ if $f(\alpha) = (\dot{\mathbb P}, \mathcal D)$ is a pair of $\mathbb P_\alpha$ names such that $\dot{\mathbb P}$ is a subcomplete forcing notion and $\mathcal D$ is a $\gamma$-sequence of dense subsets of $\dot{\mathbb P}$ for some $\gamma < \kappa$ then let $\dot{\mathbb Q}_\alpha = \dot{\mathbb P}$ otherwise add a Cohen real. Since every iterand is either subcomplete or proper, the entire iteration is $\infty$-subproper. Moreover, since subcomplete forcing doesn't kill Souslin trees, and neither does Cohen forcing, the entire iteration doesn't kill $S$ by Lemma \ref{souslinpre}. 

A standard $\Delta$-system argument shows that $\mathbb P_\kappa$ has the $\kappa$-c.c. However $\mathbb P_\kappa$ collapses everything in between $\omega_1$ and $\kappa$ so in the extension $\kappa = \omega_2$. Also Cohen reals are added unboundedly often there are $\kappa$ reals in the extension so $\kappa = 2^{\aleph_0}$. Finally the standard Baumgartner argument shows that $\SCFA$ must be forced as well.
\end{proof}

Observe that the only properties of Cohen forcing used in the previous proof was that it is proper, adds a real, and preserves $S$. Using this observation, the above method can be used to obtain a stronger consistency result. Recall that a forcing notion $\mathbb P$ is $\sigma$-linked if it can be written as the countable union, $\mathbb P = \bigcup_{n < \omega} \mathbb P_n$ where for each $n < \omega$ $\mathbb P_n$ consists of pairwise compatible elements. Note that $\sigma$-linked forcing notions are ccc. The following is well known.

\begin{proposition}
If $S$ is a Souslin tree and $\mathbb P$ is $\sigma$-linked then forcing with $\mathbb P$ preserves $S$.
\label{sigmalinkedsouslin}
\end{proposition}

\begin{proof}
Let $\mathbb P$ and $S$ be as in the statement and since $\mathbb P$ is $\sigma$-linked it can be written as $\bigcup_{n < \omega} \mathbb P_n$. Now suppose $\dot{A}$ names a maximal antichain in $S$ and suppose $p \in \mathbb P$ forces that $\dot{A}$ is uncountable. For each $n < \omega$ let $A_n = \{s \in S \; | \; \exists q \leq p \, q \in \mathbb P_n \, q\forces \check{s} \in \dot{A}\}$. Since each $\mathbb P_n$ consists of pairwise compatible elements, it follows that each $A_n$ is an antichain (in $\V$). Therefore it's countable. But that means that $\bigcup_{n < \omega} A_n$ is countable i.e. the set of all $s\in S$ so that there is some condition stronger than $p$ forcing $s$ to be in $\dot{A}$ is countable, which contradicts the fact that $p$ forced $\dot{A}$ to be uncountable.
\end{proof}

Using this fact and substituting arbitrary $\sigma$-linked posets for Cohen forcing in the proof of Theorem \ref{conSCFAcohen} the following can be shown as well.

\begin{theorem}
Assume $\kappa$ is supercompact. Then there is a $\kappa$-length nice iteration of $\infty$-subproper forcing notions forcing $\SCFA$ so that in the extension $2^{\aleph_0} = \aleph_2 = \kappa$, there are Souslin trees and $\MA_{\aleph_1}(\sigma{\rm -linked})$ holds.
\end{theorem}

\section{Preservation Theorems for the Reals}
In this section I prove a general preservation theorem for nice iterations of $\infty$-subproper forcing notions related to controlling the reals. The result implies the analogue of several preservation results for proper forcing and allows one to obtain the consistency of several constellations of the Cicho\'n diagram with $\SCFA$. Moreover in many cases one can arrange to either have Souslin trees or not. First I need a few notions about preservation of the reals, see Chapter 6 of \cite{BarJu95} for more details.

\begin{definition}[Definition 6.1.9 of \cite{BarJu95}]
Suppose that $\mathbb P$ is a forcing notion, and $\dot{f}$ is a name for an element of Baire space. We say that a decreasing sequence of conditions $\langle p_n \; | \; n < \omega \rangle$ {\em interprets} $\dot{f}$ if for all $n \in \omega$ $p_n$ decides $\dot{f} \hook n$.
\end{definition}

The basic setup for the preservation theorem is as follows: let $R \subseteq \omega \times \omega$ be a relation with a definition that is absolute between forcing notions (all the interesting examples are arithmetic, so this is not much to ask). Assume moreover that it is non-trivial in the sense that for each $n$ there are infinitely many $k$ so that $n\mathrel{R}k$. Let $R_n \subseteq (\omega^\omega)^2$ then be the relation defined by $x\mathrel{R_n} y$ if and only if $\forall k \geq n$  $x(k) \mathrel{R} y(k)$. Finally let $R^* = \bigcup_{n < \omega} R_n$. In other words $R^*$ is the \say{eventually $R$} relation: $x \mathrel{R^*} y$ if and only if for all but finitely many $n$, we have that $x(n) \mathrel{R} y(n)$. From now on fix some such $R$. 

\begin{definition}[Defintion 6.1.10 of \cite{BarJu95}]
Let $\mathbb P$ be a subproper forcing notion. We say $\mathbb P$ is $R^*$-{\em preserving} if given any $\sigma: \bN \prec N = L_\tau[A] \models \ZFC^-$ etc in the standard setup with $\dot{f}$ a $\mathbb P$-name for an element of Baire space in the range of $\sigma$ and every descending sequence of conditions in $\mathbb P$ $\langle p_n \; | \; n< \omega \rangle$, also in the range of $\sigma$, interpreting $\dot{f}$, say as some $f_0$, and some $g$ which is $R^*$ above every element of $\bN$, there is some $q \leq p_0$ so that, if $G \ni q$ is generic over $V$ and $\bG$ is the inverse image of $G$ under some $\sigma '$ as in the definition of subproperness then $g$ is $R^*$ above every element of $\bN[\bG]$ and for all $n < \omega$ $q \forces (\check{f}_0 \mathrel{R_n} g \to \dot{f} \mathrel{R_n} g)$.
\end{definition}

Let $\mfb (R^*)$ denote the bounding number for $R^*$ i.e. the least size of a set $A \subseteq \baire$ so that no $x \in \baire$ is $R^*$ above every $y \in A$. The following theorem is what is actually proved, though admittedly it is a little long-winded to write. In English, it essentially says that if $\mfb(R^*) > \aleph_0$ and $R^*$-preserving forcing notions are closed under two step iterations, then nice iterations of $\infty$-subproper, $R^*$-preserving forcing notions are again $R^*$-preserving.

%that nice iterations of $\omega^\omega$-bounding $\infty$-subproper forcing notions are $\omega^\omega$-bounding. Recall that if $x, y \in \omega^\omega$ and $n \in \omega$ then we write $x \leq_n y$ if for all $m \geq n$, $x(m) \leq y(m)$ and $x \leq^* y$ if there is some $n$ so that $x \leq_n y$. We say that a forcing notion $\mathbb P$ is $\omega^\omega$-bounding if for every $\tau \in V^\mathbb P$ and every $p \in \mathbb P$ if $p \forces \tau:\omega \to \omega$ then there is a $q \leq p$ and real $x$ so that $q \forces \tau \leq^* \check{x}$. Note that it's equivalent to force $\tau  x$. The following theorem is stated to reflect what is actually proved, though admittedly it is a little longwinded. In English, it states that if $\langle P_\alpha \; | \; \alpha \leq \nu \rangle$ is a nice iteration of $\infty$-subproper, $\omega^\omega$-bounding forcing notions, then for all $i \leq j \leq \nu$ the quotient algebra $P_{i, j} / G_i$ is $\omega^\omega$-bounding. In particular, letting $i = 0$ and $j = \nu$ we get that the whole iteration is again $\omega^\omega$-bounding. 

\begin{theorem}
Fix $R$ as described in the previous paragraph and assume $\mfb (R^*) > \aleph_0$. Let $\langle \P_\alpha \; | \; \alpha \leq \nu\rangle$ be a nice iteration so that $\P_0 = \{1_0\}$ and for all $i$ with $i + 1 < \nu$ $\forces_i \P_{i, i+1}$ is $\infty$-subproper and $R^*$-preserving. Assume moreover that $R^*$-preservation is preserved by two step iterations (this will be true in applications, but not in general). Then for all $j\le\nu$ the following statement $\phi(j)$ holds:

{\bf If} $i\le j$, $p\in \P_i$, $\dot{\sigma}, \dot{\bG}_{\bar{i}} \in\V^{\P_i}$, $q\in \P_j$ with $q \hook i = p$,
$\theta$ is a sufficiently large cardinal, $\tau$ is an ordinal, $\dot{x}$ is a $\mathbb\P_j$-name, $H_\theta\sub N=L_\tau[A]\models\ZFC^-$
$\bN$ is a countable, full, transitive model which elementarily embeds into $N$ so that $\bs,\bar{\vec{\P}},\bar{i},\bar{j}, \bar{\dot{x}} \in \bN$ and $p$ forces with respect to $\P_i$ that the following assumptions hold:
\begin{enumerate}[label=(A\arabic*)]
%  \item 
%  $H_\theta^{\V[\dot{G}_i]}\sub L_{\check{\tau}}[\dot{A}]$,
%  \item $\dot{N}=L_{\check{\tau}}^{\dot{A}}$,
  \item
  \label{item:FirstAssumption}
  $\dot{\sigma}(\check{\bs}, \check{\bar{\vec{\P}}},\check{\bar{i}},\check{\bar{j}},\check{\btheta}, \check{\bar{q}}, \bar{\dot{x}})=\check{s}, \check{\vec{\P}},\check{i},\check{j},\check{q}, \dot{x}$
  \item
  \label{item:Generics}
   $\dot{\bG}_{\bar{i}}$ is the pointwise preimage of the generic under $\dot{\sigma}$ and is $\bar{\P_i}$-generic over $\bN$
\item 
$q$ forces that $\dot{x}$ is a name for an element of Baire space
  \item 
  \label{item:LastAssumption}
  $\dot{\sigma}:\check{\bN}[\dot{\bG}_{\bar{i}}]\prec \check{N}[\dot{G}_i]$ is countable, transitive and full.
\end{enumerate}
{\bf then} there is a condition $p^*\in \P_j$ such that $p^*\hook [i, j)\le q \hook [i, j)$,
$p^*\rest i=p$ and there is a real $y \in \omega^\omega$ so that whenever $G_j\ni p^*$ is $\P_j$-generic, then in $\V[G_j]$, there is a $\sigma'$ such that, letting $\sigma=\dot{\sigma}^{G_i}$, and $\bG_{\bar{i}} = \dot{\bG}_{\bar{i}}^{G_i}$ the following hold:
\begin{enumerate}[label=(\alph*)]
  \item 
  \label{item:FirstConclusion}
  $\sigma'(\bs,\bar{\vec{\P}},\bar{i},\bar{j},\btheta,\bG_{\bar{i}}, \bar{q}, \bar{\dot{x}})=\sigma(\bs,\bar{\vec{\P}},\bar{i},\bar{j},\btheta,\bG_{\bar{i}}, \bar{q}, \bar{\dot{x}}),$
  \item
  \label{item:Sigma'MovesEverythingCorrectly}
  $(\sigma ' {}^{-1}) `` G_{i,j} : = \bG_{\bar{i},\bar{j}}$ is $\bar{\P_{i, j}}$-generic over $\bN[\bG_{\bar{i}}]$,
\item
$p^* \forces \dot{x} \mathrel{R_0} \check{y}$
  \item
  \label{item:LastConclusion}
  $\sigma':\bN[\bG_{\bar{i}}]\prec N[G_i]$.
\end{enumerate}
\label{RRiteration}
\end{theorem}

\begin{proof}
As always with these proofs we induct on $j$. Thus fix $j$ and assume $\phi(j')$ holds for all $j ' \leq j$. Again, we may assume $i < j$ since the case $i = j$ is trivial. Since we have already seen that the iteration is subproper I focus on proving $R^*$-preservation. Note that the case where $j$ is a successor is by assumption so the case to show is when $j$ is a limit ordinal. 

Let $\dot{x}$ be a $\P_j$-name and fix $p, q, ...$ etc as in the statement of the theorem. In particular, assume that $p \in \P_i$ and $q \in \P_j$ is such that $q \hook i = p$ and $q \forces_j \dot{x}:\omega \to \omega$. I need to find a $p^* \leq q$ as in the statement of the theorem and a real $y \in \omega^\omega \cap \V$ so that $p^* \forces_j \dot{x} \mathrel{R_0} \check{y}$. In fact I'll show something stronger: that $\P_j$ is $R^*$-preserving. Let $p \in G_i \subseteq \mathbb P_i$ be generic over $\V$ and, working briefly in $\V[G_i]$, let $\sigma: \barN \prec N = L_\tau[A] \models {\rm ZFC}^-$ with $H_\theta \subseteq N$ for $\theta$ large enough, $\barN$ countable, transitive and full etc as in the statement of the theorem (essentially letting $p$ force that these objects are in the standard set up). Back in $\V$, let $\langle t_n \; | \; n < \omega \rangle$ enumerate the elements of $\barN$, and let $W$ be the $L_\tau[A]$-least nested antichain so that $q$ is a mixture of $W$ up to $j$ with $T_0 = \{1_0\}$. As in the previous proofs, we may assume that $W$ is definable and hence in the range of any embedding we will discuss and $\bar{W}$ is its preimage in $\bN$. Also let $\{\bar{D}_n \; | \; n <\omega \}$ enumerate the names in $\bN$ for the dense open subsets of $\bar{\P}_{i, j}$. Without loss assume that $\bar{q}\hook [\bar{i}, \bar{j})$ is forced to be in $\bar{D}_0$ (for instance we can make $\bar{D}_0$ be the canonical name for the whole poset). By assumption we have that $\sigma(\bar{q}, \bar{\P}_j, \dot{\bar{x}}, \bar{W}) = q, \P_j, \dot{x}, W$. Note that by elementarity $\bar{q}$ forces $\dot{\bar{x}}$ to be a real and is a mixture of $\bar{W}$ up $\bar{j}$ (defined in $\barN$ as the length of the iteration $\bar{\P}_j$). 

Note that for any real $z \in \omega^\omega \cap \barN$ we get that $\sigma(z) = z$ by the absoluteness of $\omega$. Therefore we do not need the \say{bar} convention to describe reals. Fix some strictly decreasing sequence of conditions $\vec{s} = \langle \bar{s}_k \; | \; k < \omega \rangle $ so that for all $k < \omega$ $\bar{s}_k \in \bar{\P}_j$ and $\sigma(\bar{s}_k) = s_k \in \P_j$ and some real $x_0 \in (\omega^\omega)^{\bN}$ with $s_k \forces_j \check{x_0} \hook \check{k} = \dot{x} \hook \check{k}$ and $s_0 \leq q$. Note that by elementarity the \say{bar} versions of the $s$'s force the same to be true of the name $\bar{\dot{x}}$. Finally in $\V[G_i]$ fix some $y \in (\omega^\omega)^\V$ so that for all $z \in \omega^\omega \cap \barN[\bar{G}_i]$ $z \mathrel{R^*} y$ (such a $y$ exists by assumption) and in particular $x_0 \mathrel{R_0} y$ (by finite modifications this can be arranged, using the non-triviality of $R$) where $\bG_i$ is the pointwise preimage of $G_i$ under $\sigma$. This is generic over $\bN$ by our inductive hypothesis. An important point is that we can find $y \in V$ (as opposed to $\V[G_i]$) exactly by the inductive hypothesis that $\P_i$ is $R^*$-preservation. In fact, fixing this $y$, the inductive hypothesis on $l < j$ is that $y$ is $R^*$ above $\bN[\bG_i][\bG_{i, l}]$.

Much like in the iteration theorem the goal is to build a fusion structure and a sequence of names and show that doing so preserves the required properties. Specifically here I want a fusion structure and a collection of names $\langle q^{(a,n)}, T^{(a, n)}, \dot{x}^{(a, n)}, \langle \dot{s}_k^{(a, n)} \;  | \; k < \omega \rangle, \dot{\sigma}^{(a, n)}\; | \; n < \omega \; {\rm and} \; a \in T_n\rangle$ so that the following all hold:

{\em 1.} $q^{(1_0, 0)}= q$, $T^{(1_0, 0)} =  W$,  $\check{\sigma}^{(1_0, 0)}$, $\dot{x}^{(1_0, 0)} = \check{x}_0$ and $\langle \dot{s}_k^{(1_0, 0)} \;  | \; k < \omega \rangle = \langle \check{s}_k \;  | \; k < \omega \rangle$. 

\noindent For all $a \in T_n$ we have 

{\em 2.} $a \forces_{l(a)}$ \say{$\dot{\sigma}^{(a, n)} : \check{\bN} \prec N$ and $\dot{\sigma}^{(a, n)}(\bar{\theta}, \bar{\P_j}, \bar{\dot{x}},\bar{\langle \bar{\dot{s}}_k \; | \; k < \omega \rangle } ) = \theta, \P_j, \dot{x}, \langle \dot{s}_k \; | \; k < \omega \rangle$ and $l(a), T^{(a, n)}, q^{(a, n)}, \dot{x}^{(a, n)}, \langle \bar{\dot{s}}^{(a, n)}_k \; | \; k < \omega \rangle$ are in the range of $\dot{\sigma}^{(a, n)}$ and $(\dot{\sigma}^{(a, n)})^{-1}(q^{(a, n)}) \in \bar{D}_n$.} and $q^{(a, n)} \leq q$. 

{\em 3.} For each $k < \omega$ $\dot{s}_k^{(a, n)}$ is a $l(a)$-name for an element of $\P_j$ so that $q^{(a, n)} \hook l(a) \forces_{l(a)} q^{(a, n)} = \dot{s}_0^{(a, n)}$ and $\dot{s}_{k+1}^{(a, n)} \leq \dot{s}_{k}^{(a, n)}$ and $\dot{s}_k^{(a, n)} \hook l(a) \in \dot{G}\hook l(a)$

{\em 4.} $\dot{x}^{(a, n)}$ is an $l(a)$ name that $q^{(a, n)} \hook l(a)$ forces to be an element of Baire space and for each $k < \omega$ $q^{(a, n)} \hook l(a) \forces_{l(a)}$\say{$\dot{s}_k^{(a, n)} \forces \dot{x} \hook \check{k} = \dot{x}^{(a, n)} \hook \check{k}$}.

\noindent For every $b \in {\rm suc}^T_n(a)$ we have that 

{\em 5.} For all $k \leq n$ $b \forces_{l(b)} \dot{\sigma}^{(b, n+1)}(t_k) = \dot{\sigma}^{(a, n)}(t_k)$ and  $b \forces_{l(b)} (\dot{\sigma}^{(b, n+1)} )^{-1}(\bar{q}^{(a_m, m)}) = (\dot{\sigma}^{(a, n)})^{-1}(\check{q}^{(a_m, m)})$ for all $m \leq n$ and all $a_m \in T_m$ extended by $a$. 

{\em 6.} $q^{(b, n+1)} \forces \dot{x} \hook n \mathrel{R_0} \check{y} \hook n$. Where for $z \hook n \mathrel{R_0} y \hook n$ means that for each $i < n$ $z(i) \mathrel{R} y(i)$.

Supposing we can construct such a sequence and letting $p^*$ be the fusion of the fusion sequence it follows almost immediately that $p^* \forces \dot{x} \mathrel{R_0} \check{y}$ so we would be done. Thus it suffices to show that such a sequence can be constructed. This is done by recursion on $n< \omega$. The case $n=0$ is given by $1$ and it's routine to check that the parameters given there satisfy $2$, $3$ and $4$. For the inductive step, suppose for some $n < \omega$ we have constructed  $q^{(a,n)}, T^{(a, n)}, \dot{x}^{(a, n)}, \langle \dot{s}_k^{(a, n)} \; | \; k < \omega \rangle$ and $\dot{\sigma}^{(a, n)}$ for some $a \in T_n$ satisfying $1$ to $4$ and ${\rm suc}^T_m$ has been defined for all $m \leq n$. Assume also $q^{(a, n)}$ decides $\dot{x} \hook n -1$ (note that we get this for free at stage $0$). I have to define $\suc_T^n(a)$. To this end, let $D$ be the set of all $b$ of length longer than $l(a)$ so that $l(b) < j$ and 

{\em 7.} $b \hook l(a) \leq a$

\noindent There are $\dot{\sigma}^b$, $u$ and $\langle \dot{r}_k^b \; | \; k < \omega \rangle$, $\dot{x}^b$ and and nested antichain $S$ so that

{\em 8.} $S \angle T^{(a, n)}$

{\em 9.} $u \leq q^{(a, n)}$ is a mixture of $S$ up to $j$

{\em 10.} $b$ decides $\dot{\sigma}^b(t_{n+1})$, $b \forces_{l(b)} S, u, l(b) \in range (\dot{\sigma}^b)$ and $b \forces_{l(b)} \dot{\sigma}^b: \check{\bN} \prec N$ and $\forall m \leq n$ $\dot{\sigma}^b(t_m) = \dot{\sigma}^{(a, n)}(t_m)$ and $\dot{\sigma}^b (\bar{\theta}, \bar{\P_{\bar{j}}}, \bar{j}) = \theta, \P_j, j$ and there is a $\bar{u} \in \bar{D}_{n+1}$ such that $b \forces_{l(b)} \dot{\sigma}^b(\bar{u}) = u$.

{\em 11.} $3$, $4$, and $6$ all hold with $u$ replacing $q^{(b, n+1)}$, $\langle \dot{r}_k^b \; | \; k < \omega \rangle$, replacing $\langle \dot{s}_k^{(b, n+1)} \; | \; k < \omega \rangle$, and $\dot{x}^b$ replacing $\dot{x}^{(b, n+1)}$.

As in the proofs of the two iteration theorems, note that if $b\in D$ and $b'\le_{l(b)}b$, then $b'\in D$ as well. It follows that $D\rest l(a):=\{b\rest l(a)\st b\in D\}$ is open in $\P_{l(a)}$.
As a result, again it suffices to show that $D\rest l(a)$ is predense below $a$ in $P_{l(a)}$. For if we know this, $D\rest l(a)$ is dense below $a$, and we may choose a maximal antichain $A\sub D\rest l(a)$ (with respect to $P_{l(a)}$), which then is a maximal antichain in $P_{l(a)}$ below $a$. Thus, for every $c\in A$, we may pick a condition $b(c)\in D$ such that $b(c)\rest l(a)=c$, and define $\suc_T^n(a)=\{b(c)\st c\in A\}$ with $\dot{\sigma}^b =\dot{\sigma}^{(b(c), n+1)} , u = q^{(b(c), n+1)}, \langle \dot{r}^b_k \; | \; k < \omega \rangle = \langle \dot{s}^{(b(c), n+1)}_k \; | \; k < \omega \rangle$, and $\dot{x}^b = \dot{x}^{(b(c), n+1)}$.

To prove that $D\hook l(a)$ is dense, let $G_{i, l(a)}$ be a $\P_{i, l(a)}$ generic filter over $\V[G_i]$ and assume $a \in G_i * G_{i, l(a)}$. Work in $\V[G_i][G_{i, {l(a)}}]$. For readability I write $G_{l(a)}$ for $G_i * G_{i, l(a)}$. Let $\sigma_n$ be the evaluation of $\dot{\sigma}^{(a, n)}$ by $G_{l(a)}$. We need to argue that $(D \hook l(a)) \cap G_{l(a)}$ is non-empty. There are two steps to this. First we will find conditions that ensure $\infty$-subproperness is preserved, as in the previous iteration theorems. Then we will use them to construct conditions ensuring $R^*$-preservation is preserved. The following 3 paragraphs, which describe how to find the $u$, are essentially verbatim from the iteration theorem for $\infty$-subproper forcing and I repeat them for the reader's convenience. Once this is found, I will argue further to find the name $\dot{x}^b$ and the $\dot{r}_k$'s.

Let's first find $u$ and $\bar{u}$ again. By elementarity, $\bar{q}^{(a, n)}$ is a mixture of $\bar{T}^{(a, n)}$ up to $\bar{j}$. We have that $\sigma_n:\kla{L_{\btau}[\bA][\bG_{\bar{i}}],\in,\bA}\prec\kla{L_\tau[A][G_i],\in,A}$, so in particular, $\bsigma:=\sigma_n\rest L_\btau[\bA]:\kla{L_\btau[\bA],\in,\bA}\prec\kla{L_\tau[A],\in,A}$, and $\bsigma(\bar{q}^{(a,n)},\bT^{(a,n)})=q^{(a,n)},T^{(a,n)}$. Clearly, in $L_\tau[A]$, it is true that $q^{(a,n)}$ is a mixture of $T^{(a,n)}$ up to $j$, so it is true in $L_\btau[\bA]$ that $\bar{q}^{(a,n)}$ is a mixture of $\bT^{(a,n)}$ up to $\bar{j}$, and by absoluteness, this it true in $\V$ as well.

Let $\bT_0^{(a,n)}=\{\ba_0\}$. Since we're working in $V[G_{l(a)}]$, we have access to all $\bG_{k}$ for $k \leq l(a)$ by considering the pointwise pre image of $\sigma^{(a, n)}$. By induction these are all generic over $\bN[G_{\bar{i}}]$. Moreover, by Fact \ref{fact:CharacterizationOfMixtures}.\ref{item:Root}, $\ba_0\equiv\bar{q}^{(a,n)}\rest l(\ba_0)\in\bG_{l(\ba_0)}$, since $l(\bar{q}^{(a,n)})=\bar{j}>l(\ba_0)$. So $\ba_0\in\bG_{l(\ba_0)}$.
Let $\bar{v}\in\bT_1^{(\ba,n)}$, that is, $\bar{v}\in\suc_{\bT^{(\ba,n)}}^0(\ba_0)$, be such that $\bar{v}\rest l(a_0)\in\bar{G}_{l(\ba_0)}$. There is such a $\bar{v}$ by Definition \ref{definition:NestedAC-Mixture-beta-Nice-hooks}(5). By Fact \ref{fact:CharacterizationOfMixtures}.\ref{item:HigherUp}, again since $l(\bar{v})<\bar{j}=l(\bar{q}^{(a,n)})$, it follows that $\bar{v}\equiv\bar{v}\rest l(\ba_0)\verl\bar{q}^{(a,n)}\rest[l(\ba_0),l(\bar{j}))$.
Since $\bar{v}\rest l(\ba_0)\in\bG_{l(\ba_0)}$, this implies that $\bar{v}\rest[l(\ba_0),l(\bar{v}))\equiv\bar{q}^{(a,n)}\rest[l(\ba_0),l(\bar{v}))$ (in the partial order $\bar{\P}_{l(\ba_0),l(\bar{v})}=\bar{\P}_{l(\bar{v})}/\bG_{l(\ba_0)}$), and $\bar{q}^{(a,n)}\rest[l(\ba_0),l(\bar{v}))\in\bG\rest[l(\ba_0),l(\bar{v}))$.
So we have that $\bar{v}\rest l(\ba_0)\in\bG_{l(\ba_0)}$ and $\bar{v}\rest[l(\ba_0),l(\bar{v}))\in\bG\rest[l(\ba_0),l(\bar{v}))$. By Fact \ref{fact:FactorsAndQuotients}.\ref{item:CriterionForBeingInGeneric}, this implies that $\bar{v}\in\bG_{l(\bar{v})}$.

Now let $\bar{u}$ strengthen $\bar{v}^{\frown}\bar{q}^{(a, n)} \hook [\bar{l(v)}, j)$ so that $\bar{u} \in \bar{D}_{n+1}$ and decides $\dot{x} \hook n$. This last clause is the only thing different in this case from the proofs of the iteration theorems above.

By Lemma \ref{2.11}, applied in $\bN$, there is a nested antichain $\bar{S} \hooks \bar{T}^{(a, n)}$ such that $\bar{u}$ is a mixture of $\bS$ up to $\bar{j}$ and such that letting $\bS_0 = \{\bd_0\}$, we have that  $l(\bar{v}) \leq l(\bd_0)$ and $\bd_0 \rest l(\bar{v}) \leq \bar{v}$. Let $S,d_0,u=\sigma_n(\bS,\bd_0,\bu)$, and let $w\in G_{l(a)}$ force this. Since $a\in G_{l(a)}$, we may choose $w$ so that $w\le a$.

Note that $S,d_0,u$ are in $N$ (and hence in $\V$), since $\bS,\bd_0,\bu\in\bN$.

Let $l(d_0) = \beta$. Now we force one further step and let $G_{l(a), \beta}$ be $\P_{l(a), \beta}$-generic over $\V[G_{l(a)}]$ with $u \hook [l(a), \beta) \in G_{l(a), \beta}$. Let $G_\beta$ be the composition of the $G_{l(a)}$ and $G_{l(a), \beta}$ and let $\bG_{\bar{\beta}}$ be the pointwise preimage of $G_\beta$ under $\sigma_n$. Now, working in $\bN [\bG_{\bar{\beta}}]$, define recursively a descending sequence of conditions $\bar{r}_k \in \bar{\P}_{\bar{j}}$, so that $\bar{r}_k$ decides $\dot{x} \hook k$, $\bar{u} = \bar{r}_0$ and if $\bar{s}_k \hook\beta \in \bG_\beta$ then $\bar{r}_k = \bar{s}_k$. Let $x_1 \in \bN[\bG_{l(a)}]$ be the real so that $\bar{r}_k \forces \check{x}_1 \hook \check{k} = \dot{x} \hook \check{k}$.

Back in $\V[G_{l(a)}]$ let $\langle \bar{\dot{r}_k} \; | \; k < \omega \rangle$ name the sequence of $\bar{r}$'s, let $\dot{r}_k = \sigma_n (\bar{\dot{r}}_k)$ and let $\dot{x}_1$ be the name for $x_1$. Apply the inductive hypothesis $\phi(\beta)$, noting that $\beta<j$, to $i =l(a) \leq \beta$, the filters $\bG_{l(\bar{a})}$, $\bG_{{l(\bar{a})},l(\bar{\beta})}$, the models $\bN$, $N$, the condition $w$ (in place of $p$), the name $\dot{\sigma}^{(a,n)}$ (in place of $\dot{\sigma}$) and the parameters $\langle \bar{\dot{r}}_k \; | ; k < \omega \rangle, \bar{\dot{x}}_1 \in\bN$. The hypothesis allows us to obtain a condition $w^*\in\P_{\beta}$ with $w^*\rest l(a)=w$ and a name $\dot{\sigma}'$ such that $w^*$ forces with respect to $\P_{l(a)}$:
%such that if $H\ni w^*$ is generic for $\P_{l(d_0)}$, then in $\V[H]$ there is a $\sigma'$ such that letting $\sigma'_n=(\sigma^{(a,n)})^H$,
\begin{enumerate}[label=(\alph*)]
  \item
  $\dot{\sigma}'(\langle \bar{\dot{r}}_k \; | ; k < \omega \rangle\check{}, \check{\bar{\dot{x}}}_1,\check{\bar{\vec{\P}}}, \check{l}(\bar{a}),\check{\bar{\beta}},\check{\btheta},\dot{\bG}_{l(\bar{a})}, \check{\bu},\check{\bd}_0,\check{\bS})=\dot{\sigma}^{(a,n)}(\langle \bar{\dot{r}}_k \; | ; k < \omega \rangle\check{}, \check{\bar{\dot{x}}}_1,\check{\bar{\vec{\P}}},\check{l}(\bar{a}),\check{\bar{\beta}},\check{\btheta},\dot{\bG}_{l(\bar{a})}, \check{\bu},\check{\bd}_0,\check{\bS})$,
  \item
  $(\dot{\sigma}'{}^{-1})``\dot{G}_{l(\bar{a}),\bar{\beta}}$ is generic over $\check{\bN}[\dot{\bG}_{l(\bar{a})}$ and,
  \item
  $\dot{\sigma}':\check{\bN}[\dot{\bG}_{l(\bar{a})}]\prec \check{N}[\dot{G}_{l(a)}]$.
\end{enumerate}

Note that $w$ forced that $\dot{\sigma}^{(a,n)}(\bu,\bd_0,\bS)=u,d_0,S$ and hence, since $w^*\rest l(a)=w$, we get that $w^*$ forces that $\dot{\sigma}'(\bu,\bd_0,\bS)=u,d_0,S$ as well. In addition, we may insist that $\sigma'$ moves the parameters $\bar{i},\bar{j},\bar{\vec{\P}},\btheta,\bs,\bp_0,\ldots,\bp_n,t_0,\ldots,t_n$ the same way $\dot{\sigma}^{(a,n)}$ does. Note that already $a$ forced with respect to $\P_{l(a)}$ that $\bar{i},\bar{j},\bar{\vec{\P}},\btheta$ are mapped to $i,j,\vec{\P},\theta$ by $\dot{\sigma}^{(a,n)}$. Finally, applying the $R^*$-preservation part of the inductive hypothesis to $\beta$ and $\dot{x}_1$ we may assume that $w^*$ forces that $\dot{x}_1 \mathrel{R_0} \check{y}$.

To finish, I claim that $w^* \in D$ as witnessed by $u$, $\langle \dot{r}_k \; | \; k < \omega\rangle$, $S$ $\dot{x}_1$ and $\dot{\sigma} '$. Much of this follows from the previous iteration theorem. Indeed the only thing that requires checking is $6$: that $u \forces \dot{x}_0 \hook n \mathrel{R_0} \check{y} \hook n$. But this is now clear since we explicitly constructed $u$ to decide $\dot{x} \hook n$ and $w^*$ forced $s_n \leq u$ to force that $\dot{x}_0 \hook n = \dot{x}_1 \hook n$ and $w^*$ forced that $\dot{x}_1 \mathrel{R_0} \check{y}$. This completes the inductive step of the construction and hence the proof.
\end{proof}

There are many classic examples of such $R$ in the literature. Let me list a few.
\begin{example}
Let $R$ be the order $\leq$ on $\omega$. Then $\leq^*$ is the well-known eventual domination order. The property of $\leq^*$-preserving is equivalent to $\omega^\omega$-bounding. This follows from Lemma 6.3.4 of \cite{BarJu95}.
\end{example}

\begin{example}
Let $R$ be defined by letting $n \mathrel{R} m$ if $m$ codes a pair $(k, a)$ where $k \in \omega$ and $a \in [\omega]^{< \omega}$ and $n < k$ implies $n \in a$. Then, fixing $y \in \V \cap \baire$ and, via some coding, $R^*$ can be thought of as a relation whose domain is $\baire$ and whose range is the set of pairs $(y, s)$ where $s$ is a slalom. Then $x \mathrel{R^*} (y, s)$ if and only if $x \leq^* y$ implies $x \in^* s$. Unioning over all such $y$, the associated property of a forcing notion is the Laver property.
\end{example}

\begin{example}
Let $R$ be defined by $n \mathrel{R} m$ if $m$ codes an element $a \in [\omega]^{< \omega}$ and $n \in a$. Then, again by coding, we can think of $R^*$ as a relation between elements of Baire space and slaloms so that $x \mathrel{R^*} s$ is $x \in^* s$. In this case the forcing property of $R^*$-preserving is equivalent to the Sacks property\footnote{This is slightly weaker than the definition of the Sacks property given as Definition \ref{sacks} as that definition considered functions $f:\omega \to V$ whereas here we only look at functions $f:\omega \to \omega$.}.
\end{example}

Note that for the properties of being $\baire$-bounding, Laver or Sacks are all preserved by two step iterations, each being a special case of the general Lemma 6.1.12 of \cite{BarJu95}. Plugging these examples into the preservation theorem above gives the following.
\begin{theorem}
Let $\langle \P_j \; | \; j \leq \nu\rangle$ be a nice iteration of $\infty$-subproper forcing notions.
\begin{enumerate}
\item
If for each $i$ with $i + 1 < \nu$ $\forces_{\P_i} `` \P_{i, i+1} \; {\rm is \; }\omega^\omega$-bounding" then the whole iteration is $\omega^\omega$-bounding.
\item
If for each $i$ with $i + 1 < \nu$ $\forces_{\P_i} `` \P_{i, i+1}$ has the Laver property" then the whole iteration has the Laver property.
\item
If for each $i$ with $i + 1 < \nu$ $\forces_{\P_i} `` \P_{i, i+1}$ has the Sacks property" then the whole iteration has the Sacks property.
\end{enumerate}
\end{theorem}

Finally this allows us to build even more new models of $\SCFA$, again using weaving constructions of the type described previously in Theorem \ref{conSCFAcohen}.

\begin{theorem}
Assume there is a supercompact cardinal $\kappa$. Then there are $\kappa$-length nice iterations of subproper forcing notions $\mathbb P$, $\mathbb Q$ and $\mathbb R$ so that
\begin{enumerate}
\item
If $G \subseteq \mathbb P$ is $V$-generic then in $V[G]$ we have that $\SCFA$ holds as does $\mfd = \aleph_1$ and $cov (\Null) = \kappa = \aleph_2 = \mfc$.
\item
If $H \subseteq \mathbb Q$ is $V$-generic then in $V[H]$ we have that $\SCFA$ holds as does $cov (\Me) = \aleph_1 < \mfb = \mfc = \aleph_2 = \kappa$.
\item
If $K \subseteq \mathbb R$ is $V$-generic then in $V[K]$ we have that $\SCFA$ holds, $\mfc = \aleph_2 = \kappa$ and all cardinals in the Cicho\'n diagram are $\aleph_1$.
\end{enumerate}
\end{theorem}

\begin{proof}
In each case, the argument is similar to that of Theorem \ref{conSCFAcohen}, so I focus on what the \say{woven-in} forcing will be. In the first, use random reals. These are $\baire$-bounding, so $\baire \cap V$ will be a dominating family, hence $\mfd = \aleph_1$ but since unboundedly often random reals are added, no $\omega_1$ sized set of null sets can cover $\mathbb R$ so $cov (\Null) = \aleph_2$.

In the second case, iterate with Laver reals. By above, the iteration will have the Laver property, hence $cov (\Me) = \aleph_1$, but since dominating reals are added unboundedly often, $\mfb = \aleph_2$.

For the third case, use Sacks forcing. It follows that the whole iteration has the Sacks property so $cof(\Null)$ is still $\aleph_1$ in the extension.
\end{proof}

\section{Open Questions}
There are still several things not known about models of $\SCFA$. In particular, Sean Cox observed that there is a gap in the proof of \cite[Corollary 7.2]{Jen14} and hence the following is open.
\begin{question}
Is $\SCFA$ consistent with the continuum larger than $\aleph_2$?
\end{question}

One can also ask about what positive implications there are for the continuum and $H_{\omega_1}$ under $\SCFA$. 
\begin{question}
Does $\SCFA$ decide any positive implications about the reals or combinatorics on $\omega_1$?
\end{question}

Finally in anticipation of the next chapter I ask:
\begin{question}
Does $\SCFA$ imply there are no Kurepa trees?
\end{question}

\clearpage

\chapter{Specializing Wide Trees and the Dee-Complete Forcing Axiom}
\chaptermark{Specializing Wide Trees}
In this chapter I consider a different class of forcing notions which do not add reals: the dee-complete forcing notions. My initial interest in this class was to contrast its behavior with that of subcomplete forcing. Dee-complete forcing was first introduced by Shelah in Chapter V of \cite{PIPShelah} in the course of his alternative proof of Jensen's result that $\CH$ is consistent with every Aronszajn tree being special. In that chapter Shelah defined a poset for specializing Aronszajn trees without adding reals and proved that it had the properties of being ${<}\omega_1$-proper and dee-complete (to be defined below). He then showed that posets with these properties can be iterated with countable support without adding reals, hence one can iteratively specialize all Aronszajn trees without adding reals. The method turns out to be rather general and using a supercompact cardinal he obtains the consistency of the associated forcing axiom for this class coupled with the cardinal arithmetic $2^{\aleph_0} = \aleph_1$ and $2^{\aleph_1} = \aleph_2$. This axiom is called Axiom II in \cite[p. 377]{PIPShelah} and several consequences are given, including the failure of Kurepa's hypothesis. Some of these consequences, including the failure of Kurepa's hypothesis, actually do not need the cardinal arithmetic. Here, I will follow Jensen, \cite{JensenCH}, and refer to the forcing axiom (with no assumption on the cardinal arithmetic) as $\DCFA$ (the \say{dee-complete forcing axiom}). So $\DCFA$ is a competitor with $\SCFA$ in that both are compatible with $\CH$. 

Since the results of the previous chapter suggest that $\SCFA$ wields very little influence over combinatorics at the level of $H_{\omega_1}$, it's striking that Shelah obtained so many consequences of $\DCFA$ at this level. In this chapter I give a new example of a dee-complete and ${<}\omega_1$-proper poset which is a new variation of Shelah's aforementioned specializing poset and which, under certain circumstances allows one to specialize trees of uncountable width without adding reals. Specifically, the main theorem is the following.

\begin{theorem}
Suppose $T$ is an $\omega_1$-tree (countable levels) and $S \subseteq T$ is a (potentially wide) Aronszajn tree with the induced suborder. Then there is a forcing notion $\mathbb P = \mathbb P_{S, T}$ which specializes $S$ and is dee-complete and ${<}\omega_1$-proper.
\label{specwidetree}
\end{theorem}

A similar idea is sketched in \cite[Chapter VII]{PIPShelah}, however the poset here is slightly more general in that it applies to a larger class of trees. I also use this to give a more fleshed out proof of the fact that $\DCFA$ implies there are no Kurepa trees. The general question of when one can specialize a wide tree without adding reals turns out to be very interesting and there are many open questions still. I investigate this further as well and make some observations.

\section{Preliminaries: Dee-complete Forcing, ${<}\omega_1$-Properness and Trees}

\subsection{Strengthening Properness}

Given a sufficiently large cardinal $\theta$, a countable model $N$ so that either $N \prec H_\theta$ or else $N$ is transitive and elementarily embeds into $H_\theta$, a forcing notion $\mathbb P \in N$ and a condition $p \in \mathbb P$ let ${\rm Gen}(N, \mathbb P, p)$ be the set of $\mathbb P$-generic filters over $N$ containing $p$ i.e. the set of filters $G \subseteq \mathbb P \cap N$ which intersect every dense subset $D$ of $\mathbb P$ in $N$. The following definitions come from \cite[Chapter V]{PIPShelah} and a particularly good exposition is also given in \cite{AbrahamHB}. What I call a completeness system here is called an \say{$\aleph_1$-complete} completeness system in \cite{PIPShelah} and a ``countably complete" completeness system in \cite{AbrahamHB}. However, every completeness system considered in this chapter is countably complete so I omit the additional notation. For reference, countable completeness is the condition \ref{countcomp} in the definition below.

\begin{definition}
A {\em completeness system} is a function $\mathbb D$ defined on some set of triples $(N, \mathbb P, p)$ such that $N \prec H_\theta$ for some $\theta$, $N$ is countable, $\mathbb P \in N$ is a forcing notion and $p \in \mathbb P \cap N$ is a condition and the following hold:
\begin{enumerate}
\item
$\mathbb D(N, \mathbb P, p)$ is a family of sets, $A$, such that each $A \subseteq {\rm Gen}(N, \mathbb P, p)$.
\item
If $A_i \in \mathbb D(N, \mathbb P, p)$ for each $i < \omega$ then the intersection $\bigcap_{i < \omega} A_i$ is non-empty.
\label{countcomp}
\end{enumerate}
If for a fixed $\mathbb P$ and some cardinal $\theta$, $\mathbb D$ is defined on the set of all triples $(N, \mathbb P, p)$ with $p \in \mathbb P \in N$, $p \in N$ and $N \prec H_\theta$ then we call $\mathbb D$ a completeness system on $\theta$ for $\P$.
\end{definition}
Completeness systems in general are quite easy to construct, which leads one to question their utility. In general we will only be interested therefore in ones which are \say{nicely defined}, a notion Shelah refers to as {\em simple}.
\begin{definition}
A completeness system $\mathbb D$ is {\em simple} if there are a formula $\phi$ and a parameter $s \in H_{\omega_1}$ such that for all $(N, \mathbb P, p)$ in the domain of $\mathbb D$ we have that $\mathbb D(N, \mathbb P, p) = \{A^{N, \P, p}_u \; | \; u \in H_{\omega_1}\}$ where $A^{N, \P, p}_u$ is defined as follows: for $N\prec H_\theta$, let $\bar{N}$ be the Mostowski collapse of $N$ and $\pi_N:\bar{N} \to N$ the inverse of the Mostowski collapse. We let $\bar{A}^{N, \mathbb P, p}_u := \{\bG \in {\rm Gen}(\bar{N}, \pi^{-1}(\P), \pi_N^{-1}(p)) \; | \; H_{\omega_1} \models \phi (\bar{N}, \bG, \pi^{-1}(\P), \pi_N^{-1}(p), u, s)\}$. Finally let $A^{N, \P, p}_u$ be the set of generics of the form $\pi_N``\bG$ for $\bG \in \bar{A}^{N, \mathbb P, p}_u$.
\end{definition}
Using this, I can define dee-completeness. %%simplicity versus in the ground model
\begin{definition}
We say that $\mathbb P$ is {\em dee-complete} if for every sufficiently large $\theta$ there is a simple completeness system $\mathbb D$ on $\theta$ for $\P$ such that whenever $\mathbb P \in N \prec H_\theta$, with $N$ countable and $p \in \mathbb P \cap N$ there is an $A \in \mathbb D(N, \mathbb P, p)$ such that for all $G \in A$ there is a condition $q \in \mathbb P$ so that $G = \{r \in N \cap \mathbb P \; | \; q \leq r\}$.
\end{definition}
Given a poset $\mathbb P$ we say that a (not necessarily simple) completeness system $\mathbb D$ is a {\em completeness system for} $\mathbb P$ if it satisfies the requirements of the definition of dee-completeness (other than simplicity)\footnote{The simplicity condition is not given as part of the definition of dee-completeness in either \cite{PIPShelah} nor \cite{AbrahamHB}. However, in order to have an iteration theorem for this class, one needs to assume simplicity, as is done below in Theorem \ref{iterationofdeecomplete}, or else assume that all the completeness systems already appear in the ground model, as in \cite[Theorem 5.17]{AbrahamHB}. Since I will be only considering simple completeness systems and using the corresponding iteration theorem in this chapter, I have rolled this into the definition in order to, if you will, simplify the discussion.}. Observe that the existence of a completeness system for $\mathbb P$ implies that $\mathbb P$ is proper and adds no new reals (or indeed $\omega$ sequences of elements from $V$) since the condition $q$ as in the definition of dee-completeness is an $(N, \mathbb P)$-generic condition and if $\dot{a}: \check{\omega} \to \check{V}$ names an $\omega$-sequence, then there is a model $N \ni \dot{a}$ and a $\P$-generic $G$ over $N$ which has a lower bound $q$ so $q$ decides $\dot{a}(\check{n})$ for all $n < \omega$.

\begin{definition}[$\alpha$-Properness]
Let $\theta$ be a cardinal and $\alpha < \omega_1$. An $\alpha$-{\em tower} for $H_\theta$ is a sequence $\vec{N} = \langle N_i \; | \; i < \alpha \rangle$ of countable elementary substructures of $H_\theta$ so that for each $\beta < \alpha$, we have $\langle N_i \; | \; i \leq \beta \rangle \in N_{\beta + 1}$ and if $\lambda < \alpha$ is a limit ordinal then $N_\lambda = \bigcup_{i < \lambda} N_i$. We say that $\mathbb P$ is $\alpha$-proper if for all sufficiently large $\theta$, all $p \in \mathbb P$ and all $\alpha$-towers $\vec{N}$ in $H_\theta$ so that $p, \mathbb P \in N_0$ there is a $q \leq p$ which is simultaneously $(N_i, \mathbb P)$-generic for every $i < \alpha$. We say that $\mathbb P$ is ${<}\omega_1$-proper if it is $\alpha$-proper for all $\alpha < \omega_1$.
\end{definition}
Note that properness is $1$-properness. The point is the following iteration theorem due to Shelah, \cite[Chapter V, Theorem 7.1]{PIPShelah}.
\begin{theorem}
If $\langle (\mathbb P_\alpha, \dot{\mathbb Q}_\alpha) \; | \; \alpha < \nu\rangle$ is a countable support iteration of some length $\nu$ so that for each $\alpha < \nu$, $\forces_{\mathbb P_\alpha} ``\dot{\mathbb Q}_\alpha$ is dee-complete and ${<}\omega_1$-proper", then $\mathbb P_\nu$ is dee-complete and ${<}\omega_1$-proper. In particular such iterations do not add reals.
\label{iterationofdeecomplete}
\end{theorem}

As an immediate consequence, we obtain, relative to a supercompact, the consistency of $\DCFA$, the forcing axiom for dee-complete and ${<}\omega_1$-proper forcing notions and even its consistency with $\CH$. Of course $\DCFA$ does not imply $\CH$ since $\PFA$ implies $\DCFA$ trivially. Very little attention has gone into $\DCFA$ as an axiom in its own right outside of \cite{JensenCH}. However one notable exception is \cite{AbrahamTodPP} where it is shown that $\DCFA + \CH$ implies the P-Ideal Dichotomy.

\subsection{Trees}
The main purpose of this chapter is to look at applications of dee-complete forcing to trees. Let me review some notation and terminology related to this here for reference. Recall that a {\em tree} $T = \langle T, \leq_T\rangle$ is a partially ordered set so that for each $t \in T$ the set of $s \leq_T t$ is well ordered by $\leq_T$. A {\em branch} through a tree is a maximal linearly ordered subset.

\begin{definition}
Let $T$ be a tree, $\alpha$ an ordinal and $\kappa$ and $\lambda$ cardinals.
\begin{enumerate}
\item
The $\alpha^{\rm th}$-{\rm level} of $T$, denoted $T_\alpha$ is the set of all $t \in T$ so that $\{s \; | \; s <_T t\}$ has order type $\alpha$. Also let $T_{\leq \alpha} = \bigcup_{i \leq \alpha} T_i$ and $T_{< \alpha} = \bigcup_{i < \alpha} T_i$.
\item
The {\em height} of $T$ is the least $\alpha$ with $T_\alpha = \emptyset$. 
\item
If $\alpha < \beta$ are ordinals, $T$ is a tree of height at least $\beta + 1$ and $t \in T_\beta$ then denote by $t \hook \alpha$ the unique $s \in T_\alpha$ so that $s \leq_T t$.
\item
We say that $T$ is a $\kappa$-tree if it has height $\kappa$ and each level has size $<\kappa$.
\item
$T$ is a $\kappa$-{\em Aronszajn} tree if it is a $\kappa$-tree with no branch of size $\kappa$. If $\kappa = \aleph_1$ we just say Aronszajn tree.
\item
$T$ is a $(\kappa, {\leq}\lambda)$-Aronszajn tree if it is a tree of height $\kappa$ with each level of size ${\leq}\lambda$ and no branch of size $\kappa$. An $(\aleph_1, {\leq}\lambda)$-Aronszajn tree is called a {\em wide Aronszajn tree} if $\lambda$ is uncountable and the equality is witnessed at some level i.e. it is not a $(\omega_1, {<}\omega_1)$-tree\footnote{The use of the word ``wide" appears to come from the recent (and fascinating) paper \cite{DzSh2020}, though the concept has been in the literature for over 50 years. }. The latter we sometimes call {\em thin} to emphasize that it's not wide.
\item
A (wide) Aronszajn tree is {\em special} if it can be decomposed into countably many antichains. Equivalently if there is a {\em specializing function} $f:T \to \mathbb Q^+$, the set of positive rationals so that $f$ is strictly increasing on linearly ordered subsets of $T$.
\item
An $\omega_1$-tree is Kurepa if it has more than $\aleph_1$ many uncountable branches. It's a {\em weak} Kurepa tree if it is a tree of height and cardinality $\aleph_1$ with more than $\aleph_1$ many branches.
\end{enumerate}
\end{definition}
Throughout this chapter I will only be considering {\em normal} trees. Recall that a tree $T$ is {\em normal} if $|T_0| = 1$, every node is comparable with nodes on every level, and for each $s, t \in T$ of limit height $\alpha$, if $s\neq t$ there is a $\beta < \alpha$ so that $s \hook \beta \neq t \hook \beta$. Unless otherwise specified, in what follows ``tree" means ``normal tree".

Special trees were first investigated in connection with forcing in \cite{BMR70} where it was shown that the poset to add a specializing function with finite approximations is ccc and hence $\MA + \neg \CH$ implies that all trees of height $\aleph_1$, cardinality less than $2^{\aleph_0}$ and no uncountable branch are special. This poset obviously adds reals. Specializing without adding reals is more delicate as we will see.

\section{Specializing a Wide Tree}
In this section I work towards proving Theorem \ref{specwidetree}. The forcing notion used is very similar to the poset from \cite[Section 4]{AbrahamShelah93} which specializes a thin tree without adding reals. This is due to Abraham and Shelah, building on the original example of such a poset from \cite[Chapter V, Theorem 6.1]{PIPShelah}. Throughout, fix an $\omega_1$-tree $T$ (possibly with uncountable branches) and let $S \subseteq T$ be an $(\omega_1, {\leq}\omega_1)$-Aronszajn tree with the induced suborder. Some writers assume that a subtree must be downward closed in the larger tree, note that I am explicitly {\em not} assuming this. Also, note that without loss that we may assume that $T \subseteq H_{\omega_1}$. 

The first step will be to make a reduction in terms of the hypotheses needed. Given a $\leq_T$-increasing sequence $\vec{t} = \langle t_i \; | \; i < \omega \rangle$ of elements of $T$ the {\em limit} of $\vec{t}$, denoted $\lim \vec{t}$, is the unique minimal element $t \in T$ so that for all $i < \omega$ $t_i \leq_T t$. Note that the limit may not always exist, however if it does it is unique by normality of $T$. We say that a subtree $U \subseteq T$ with the induced suborder is {\em closed} if it contains all of its limit points. Given any subtree $U \subseteq T$ define the {\em closure} of $U$, denoted $\overline{U}$ to be the smallest closed subtree containing $U$ i.e. $\overline{U} = \bigcap \{U' \; | \; U \subseteq U' \; {\rm and} \; U' \; {\rm is \; closed}\}$. 

Observe that we can construct $\overline{U}$ more concretely by simply unioning up $U$ with its limit points. Clearly $\overline{U}$ contains all these points, so I need to just argue that this union is already closed. Let $U'$ be the union of $U$ with its limit points and let $\vec{t} = \langle t_i \; | \; i < \omega \rangle \subseteq U'$ be a $\leq_T$-strictly increasing sequence whose limit exists in $T$. I need to show that $\lim \vec{t} \in U'$. Every element of $\vec{t}$ is either a limit point of $U$ or is in $U$ so we can find an element $s_i \in U$ so that $t_i \lneq_T s_i \leq_T t_{i+1}$ for each $i < \omega$. Let $\vec{s} = \langle s_i \; | \; i < \omega \rangle$. We have that $\vec{s} \subseteq U$ and $\lim \vec{t} = \lim \vec{s}$ so the limit of this sequence is in $U'$ as needed.

The point of this detour is the following.
 
\begin{proposition}
If $U \subseteq T$ is a (potentially wide) Aronszajn tree with the induced suborder then so is $\overline{U}$.
\end{proposition}

\begin{proof}
Suppose $\overline{U}$ is not an Aronszajn tree and let $b \subseteq \overline{U}$ be an uncountable branch. Since $U$ is Aronszajn $b \cap U$ is bounded. It follows that there are at most countably many limit points of $b \cap U$. But then, by the observation preceding this paragraph $b \cap \overline{U}$ is bounded, contradiction.
\end{proof}

As a consequence of this observation we may from now on assume that $S$ is closed in $T$, since if it's not we can replace it with its closure which, if special, will imply that $S$ is special as well. Also assume without loss that the root of $S$ is the root of $T$, i.e. the unique element of $T_0$ is in $S$.

Our next goal is to define the forcing $\mathbb P$. The idea is to force with partial specializing functions $f:S \to \mathbb Q$ but use the structure of $T$ to control the forcing. 

I begin by defining the objects that will build up the conditions. Throughout there is a subtlety concerning partial functions from $T$ to $\mathbb Q$ that I want to address up front. Let $\beta < \omega_1$. Often times I will be considering some function $h$ which maps a finite subset of $T_\beta$ to $\mathbb Q$ and I would like to consider the projection of this function to some set $X \subseteq T$ of elements of rank $< \beta$ i.e. a new function $\hat{h}$ so that for each $t \in {\rm dom}(h)$ with $t \hook \alpha \in X$ for some $\alpha < \beta$ and $\hat{h}(t \hook \alpha) = h(t)$. The issue is that $\hat{h}$ as written may not be a function since several different $t$'s on level $\beta$ may have the same projection to lower levels. To avoid this, I will implicitly assume that these projections are defined, i.e. the projection into $X$ is injective on ${\rm dom}(h)$, and roll this into the definition. Thus we will need to show whenever we work with such a projection that it is well defined in this sense.

Let's start with some notation.
\begin{definition}
Let $t \in T$, $X \subseteq T$ and $\alpha  <\omega_1$.
\begin{enumerate} 
\item
Define $t \hook S$ to be the $\leq_T$ maximal $u \in S$ so that $u \leq_T t$. Note that this maximal element exists because $S$ is closed.
\item
Define $X \hook S = \{ t\hook S \; | \; t \in X\}$.
\item
Let $t \downarrow_S \alpha = (t \hook \alpha )\hook S$.
\item
Let $X \hook \alpha = \{ x \hook \alpha \; | \; x \in X\}$ and $X \downarrow_S \alpha = \{ x \downarrow_S \alpha \; | \; x \in X\}$.
\end{enumerate}
\end{definition}

Now I move on to the definitions needed to define the poset.

\begin{definition}
%Throughout, unless otherwise noted, for a node $t \in S$, I mean by $t \hook \alpha$ the projection of $t$ to level $\alpha$ {\em in the sense of} $T$ (as opposed to $S$).

\begin{enumerate}
\item
A {\em partial specializing function} of height $\alpha = last(f)$  is a function $f: T_{\leq \alpha} \cap S \to \mathbb Q$ so that for all $s, t \in T_{\leq \alpha} \cap S$, if $s <_T t$ then $f(s) < f(t)$.
\item
If $f$ is a partial specializing function, $last(f) = \alpha$, $\beta \geq \alpha$ and $h:T_\beta \hook S \to \mathbb Q$ is a finite partial function then we say that $h$ {\em bounds} $f$ is for every $t \in {\rm dom} (h)$ $f(t\downarrow_S \alpha) < h(t)$.
\item
A {\em requirement} $H$ of height $\beta = ht(H) < \omega_1$ and arity $n = n(H) \in \omega$ is a countably infinite family of finite functions $h:T_\beta \hook S \to \mathbb Q$ whose domains have size $n$ and which are {\em dispersed} in the sense that for every finite $\tau \subseteq T_\beta \hook S$ there is an $h \in H$ so that $\tau \cap {\rm dom}(h) =\emptyset$.
\item
A partial specializing function $f$ {\em fulfills} a requirement $H$ if $ht(H) = last(f) = \alpha$ and for every finite $\tau \subseteq T_\alpha \hook S$, there is an $h \in H$ whose domain is disjoint from $\tau$ and {\em bounds} $f$.
\item
A {\em promise} is a function $\Gamma$ defined on a tail set of countable ordinals, the first of which we denote $\beta = \beta (\Gamma)$ so that for each $\gamma \geq \beta$, $\Gamma (\gamma)$ is a countable set of requirements of height $\gamma$ satisfying the following {\em projection property}: 
\begin{center}
if $\beta(\Gamma) \leq \gamma \leq \gamma '$ then $\Gamma (\gamma) = \{H \downarrow_S \gamma \; | \; H \in \Gamma(\gamma ')\}$
\end{center}
where $H \downarrow_S \gamma = \{ h \downarrow_S \gamma \; | \; h \in H\}$ and $h \downarrow_S \gamma$ is the function whose domain is ${\rm dom}(h) \downarrow_S \gamma$ and for each $x \in {\rm dom}(h)$ $h\downarrow_S \gamma (x \downarrow_S \gamma)$ is the projection of $h$ to $x$. As noted in the paragraph above the previous definition, I'm implicitly assuming in this that this function is defined i.e. if $x, y \in {\rm dom}(h)$ are distinct then $x \downarrow_S \gamma \neq y \downarrow_S \gamma$. 
\item
A partial specializing function $f$ {\em keeps} a promise $\Gamma$ if $\beta (\Gamma) = last(f)$ and $f$ fulfills every $H \in \Gamma (last(f))$.
\item
The forcing notion $\mathbb P = \mathbb P_{T, S}$ consists of pairs $p = (f_p, \Gamma_p)$ where $f_p$ is a partial specializing function, $\Gamma_p$ is a promise and $f_p$ keeps $\Gamma_p$. We write $ht(p)$ for $last(f_p) = \beta (\Gamma)$. Finally let $q \leq p$ if $f_p \subseteq f_q$, $ht(p) \leq ht(q)$ and for all $\gamma \geq ht(q)$,  $\Gamma_p (\gamma) \subseteq \Gamma_q(\gamma)$. 
\end{enumerate}
\end{definition}

The proof of Theorem \ref{specwidetree} is broken up into a number of lemmas which collectively show that $\mathbb P$ has the properties advertised in the theorem. First let's show that any condition can be extended arbitrarily high up the tree. Note that this will imply that $\mathbb P$ specializes $S$.
\begin{lemma}[The Extension Lemma]
Suppose $p \in \mathbb P$ of height $\alpha$ and let $\beta \geq \alpha$ and $\varepsilon \in \mathbb Q$ be positive. Then there is a $q \leq p$ of height $\beta$, with $\Gamma_q = \Gamma_p\hook[\beta, \omega_1)$ so that for all $x \in T_\alpha \hook S$ and all $y \in T_\beta \hook S$ if $x \leq_T y$ then $f_q(y) - f_p(x) < \varepsilon$.
\label{extensionlemma}
\end{lemma}

Before I prove this lemma, let me note the significance of the condition concerning $\varepsilon$. It implies in particular that given a finite partial function $g:T_\beta \hook S  \to \mathbb Q$ which bounds $f_p$ we can extend $p$ to a stronger condition $q$ so that $g$ still bounds $f_q$. This follows by letting $\varepsilon = {\rm min}\{g(y) - f_p(y \downarrow_S \alpha) \; | \; y \in {\rm dom}(g)\}$.

\begin{proof}
The proof is by induction on $\beta$. There are two cases.

\noindent {\bf\underline{Case I}:} $\beta$ is a successor ordinal, say $\beta = \beta_0 + 1$. Fix a positive $\varepsilon \in \mathbb Q$ and, using the inductive hypothesis extend $p$ to a condition $p'$ of height $\beta_0$ with $\Gamma_{p'} = \Gamma_p \hook [\beta_0, \omega_1)$ so that for all $x \in T_{\beta_0} \hook S$ $f_{p'}(x) - f_p(x \downarrow_S \alpha) < \varepsilon/2$. We need to see how to extend $p'$ further to a $q$ of height $\beta$. This is done as follows. First note that we may assume that $T_\beta \cap S \neq \emptyset$ for if this is the case then, trivially $(f_{p'}, \Gamma_p \hook[\beta, \omega_1))$ is the required condition. Thus, from on we assume there is some $t \in T_\beta \cap S$. This set may be finite or infinite. Let's suppose its size is $k \leq \omega$. Enumerate $T_\beta \hook S \setminus T_{\beta_0} \hook S = T_{\beta} \cap S$ as $\{t_i \; | \; i < k\}$. Also, enumerate $[T_\beta \hook S]^{< \omega} \times \Gamma_{p'}(\beta)$ as $\{(\tau_i, H_i) \; | \; i < \omega\}$. If $\Gamma(\beta) = \emptyset$ then let each $H_i$ consist of ``empty functions" $h$ which we think of as bounding any function. I will define recursively finite functions $\langle g_i \; | \; i < \omega\rangle$ so that for every $i < k$ the following conditions hold:
\begin{enumerate}
\item
\label{sequence}
$g_i \subseteq g_{i+1}$, 
\item
$g_i$ is a partial finite function from $T_\beta \hook S$ to $\mathbb Q$,
\item
\label{domainistotal}
$t_i \in {\rm dom}(g_i)$,
\item
\label{epcondition}
For every $x \in {\rm dom}(g_i)$ if $x \in T_\beta \cap S$ then $0 < g_i(x) - f_{p'}(x \downarrow_S \beta_0) < \varepsilon/2$, 
\item
\label{extcondition}
For every $x \in {\rm dom}(g_i)$ if $x \in T_{\beta_0} \hook S$ then $ g_i(x) = f_{p'}(x \downarrow_S \beta_0)$ and,
\item
\label{promisethis}
There is an $h_i \in H_i$ whose domain is disjoint from $\tau_i$ and bounds $g_i$. Note that this implies in particular that ${\rm dom}(h_i) \subseteq {\rm dom}(g_i)$.
\end{enumerate}
Supposing that such a sequence can be constructed, let's see that this finishes the case. Let $f_q = f_{p'} \cup \bigcup_{i < k} g_i$ and let $q = (f_q, \Gamma_{p'} \hook[\beta, \omega_1))$. Clearly if $q$ is a condition then it's a strengthening of $p$ so we just need to see that $q \in \mathbb P$ and it satisfies the slow growth condition that for all $y \in T_\beta \hook S$ and $x \in T_\alpha \hook S$, if $x \leq_T y$ then $f_q(y) - f_p(x) < \varepsilon$. By combining \ref{sequence} - \ref{domainistotal} with \ref{extcondition} it follows that $f_q$ is a partial specializing function extending $f_p$ and \ref{epcondition} implies that it satisfies the slow growth condition. Thus it remains to see that $f_q$ fulfills the promise $\Gamma_{p'} \hook [\beta, \omega_1)$. But this is exactly what \ref{promisethis} says.

So to finish the case we need to build the sequence of $g_i$'s satisfying \ref{sequence} - \ref{promisethis}. This is done recursively as follows. Assume for some $j < \omega$ we have defined $g_0, ...g_{j-1}$ which satisfy \ref{sequence} - \ref{promisethis}. Let $\tau = \tau_j \cup {\rm dom}(g_{j-1}) \cup \{t_j\}$ and $\bar{\tau} = \tau \downarrow_S \beta_0$. By the definition of a condition in $\mathbb P$, there is an $h_j \in H_j$ so that $h_j \downarrow_S \beta_0$ bounds $f_{p'}$ and ${\rm dom}(h_j \downarrow_S \beta_0) \cap \bar{\tau} = \emptyset$. The domain of $g_j$ will be $d:=\tau \cup {\rm dom}(h_j)$. For each $x \in d$ let
\[
  g_j(x) =
  \begin{cases}
                                   g_{j-1}(x) & \text{if $x \in {\rm dom}(g_{j-1})$} \\
			   f_{p'}(x) & \text{if $x \in T_{\beta_0} \hook S$} \\
			  f_{p'}(x \downarrow_S \beta_0) + \frac{\varepsilon}{2} & \text{if $x= t_j \notin {\rm dom}(g_{j-1})$} \\
                                   f_{p'}(x \downarrow_S \beta_0) + {\rm min}(\frac{\varepsilon}{4}, \frac{h_j(x) - f_{p'}( x \downarrow_S 			             	\beta_0)}{2}) & \text{$x \in T_\beta \cap S \; \& \; x \notin \tau$}
  \end{cases}
\]

This function $g_j$ is then as wished for. 

\noindent {\bf\underline{Case II}:} $\beta$ is a limit ordinal. Fix a strictly increasing sequence $\langle \beta_n \; | \; n < \omega\rangle$ so that $\beta_0 = \alpha$ and ${\rm sup}_{n} \beta_n = \beta$. Fix $\varepsilon > 0$. The idea is to weave the procedure described in Case I to build a function on $T_\beta \cap S$ with the inductive assumption that allows us to extend $f_p$ to each $\beta_n$. As before, enumerate $T_{\beta} \cap S$ as $\{t_i \; | \; i < k\}$ where $k \leq \omega$ is the cardinality of this set (which may be $0$), and enumerate $[T_\beta \hook S]^{< \omega} \times \Gamma_{p}(\beta)$ as $\{(\tau_i, H_i) \; | \; i < \omega\}$. Similar to last time it's possible that $\Gamma(\gamma)$ is empty in which case we again treat each $H_i$ as a set of ``empty functions" which bound any condition and are disjoint from any $\tau$. I want to construct sequences $\langle f_i \; | \; i < \omega\rangle$, $\langle g_i \; | \; i < \omega\rangle$ and $\langle h_i \; | \; i <\omega \rangle$ so that the following hold for all $i < \omega$.
\begin{enumerate}
\item
$p_i = (f_i, \Gamma_p \hook [\beta_i, \omega_1) )$ is a condition of height $\beta_i$
\item
$p_{i+1} \leq p_i \leq p = p_0$
\item
For every $y \in T_{\beta_{i+1}} \hook S \;$ $f_{i + 1} (y) - f_i(y \downarrow_S \beta_i) < \frac{\varepsilon}{2^{i+2}}$
\item
\label{sequence1}
$g_i \subseteq g_{i+1}$, 
\item
$g_i$ is a finite function from $T_\beta \hook S$ to $\mathbb Q$
\item
\label{domainistotal1}
If $i < k$ then $t_i \in {\rm dom}(g_i)$
\item
\label{bound}
$g_i \downarrow_S \beta_i$ bounds $f_i$
\item
$h_i \in H_i$ and
\item
\label{promisethis1}
${\rm dom} (h_i) \cap (\tau_i \cup \bigcup_{j < i} h_j) = \emptyset$, ${\rm dom} (h_i) = {\rm dom}(g_i) \setminus \bigcup_{j < i} {\rm dom}(g_j) \cup \{t_i\}$
\item
For all $x \in {\rm dom}(h_i)$ $g_i(x) < h_i(x)$
\end{enumerate}
Suppose first that such a triple of sequences can be constructed and let $f_q = (\bigcup_{i < \omega} f_i) \cup (\bigcup_{i < \omega} g_i)$. Then $f_q$ is a partial specializing function with $last(f_q) = \beta$ and $q = (f_q, \Gamma_p \hook [\beta, \omega_1))$ is the condition needed. The verification of this last point is nearly identical to the first case.

Thus it remains to show that these sequences can be constructed. Recursively assume for some $j < \omega$ $\langle f_i \; | \; i < j\rangle$, $\langle g_i \; | \; i < j\rangle$ and $\langle h_i \; | \; i <j \rangle$ have been constructed. By assumption $(f_{j-1}, \Gamma_p \hook [\beta_{j-1}, \omega_1))$ is a condition in $\mathbb P$ so it satisfies all of the requirements in $\Gamma_p(\beta_{j-1})$. In particular there is an $h_j \in H_j$ so that $h_j \downarrow_S \beta_{j-1}$ (which is in $\Gamma_p(\beta_{j-1})$ by the projection property) is such that $({\rm dom}( h_j) \downarrow_S \beta_{j-1}) \cap ((\tau_j \cup \bigcup_{i < j} {\rm dom}(g_i) \cup \{t_j\}) \downarrow_S\beta_{j-1}) = \emptyset$ and bounds $f_{j-1}$. This is the $h_j$ we need. Let $g_{j-1} \subseteq g_{j}$ be so that ${\rm dom}(g_j) = {\rm dom}(h_j) \cup \bigcup_{i < j} {\rm dom}(g_i) \cup \{t_j\}$, $h_j$ bounds $g_j$ pointwise on their shared domain, $g_j$ bounds $f_{j-1}$ and for all $x \in {\rm dom}(g_j)$ $g_j(x \hook S) - f_p(x \downarrow_S \alpha) < \varepsilon$. This is the $g_j$ we need. Finally by our inductive assumption we can find a function $f_j:T_{\leq \beta_j} \cap S \to \mathbb Q$ so that $p_j = (f_j, \Gamma_p \hook [\beta_j, \omega_1)) \in \mathbb P$, extends $p_{j-1}$, is bounded by $g_j$ and is such that for all $y \in T_{\beta_{j}} \hook S$ $f_{j} (y) - f_{j-1}(y \downarrow_S \beta_{j-1}) < \frac{\varepsilon}{2^{j+1}}$. This $f_j$ is then as required so the construction is complete.
\end{proof}

Next I show how to add promises.  Given two promises $\Gamma$ and $\Psi$ I write $\Psi \subseteq \Gamma$ to mean that $\beta(\Gamma) \geq \beta(\Psi)$ and for all $\gamma \geq \beta (\Gamma)$ we have that $\Psi(\gamma) \subseteq \Gamma (\gamma)$. Also, I will write $\bar{\Gamma \cup \Psi}$ to mean the promise $\Delta$ so that $\beta(\Delta) = {\rm max}\{\beta(\Gamma), \beta(\Psi)\}$ and for all $\gamma \geq \beta(\Delta)$ $\Delta(\gamma) = \Gamma(\gamma) \cup \Psi(\gamma)$.

\begin{lemma}
Suppose $p \in \mathbb P$ is of height $\alpha$, $\beta \geq \alpha$ and $g:T_\alpha \hook S \to \mathbb Q$ is a finite function bounding $f_p$. Let $\Psi_g$ be a promise so that $\beta(\Psi_g) = \beta$ and for all $\gamma > \beta$ $H \in \Psi_g (\gamma)$ if $h \in H$ then $h \downarrow_S \alpha = g$. There there is a $q \leq p$ so that $ht(q) = \beta$ and $\Psi_g \subseteq \Gamma_q$.
\label{basis}
\end{lemma}

Following \cite[Lemma 4.4]{AbrahamShelah93} we refer to the $g$ in the above lemma as a {\em basis} for the promise $\Psi_g$ and say that $g$ {\em generates} $\Psi_g$.

\begin{proof}
Let $\varepsilon = {\rm min}\{g(x) - f(x)\; | \; x \in {\rm dom}(g)\}$. By Lemma \ref{extensionlemma} there is a $q' \leq p$ of height $\beta$ so that for all $y \in T_\beta \hook S$ $f_{q'} (y) - f_p (y \downarrow_S \alpha) < \varepsilon$. I claim that $f_{q'}$ fulfills the promise $\Psi_g$. To see this, let $H \in \Psi_g(\beta)$, let $\tau \in [T_\beta \hook S]^{<\omega}$ and, by the fact that $\Psi_g$ is a promise, find an $h \in H$ whose domain is disjoint from $\tau$. Observe then that $h$ bounds $f_{q'}$ since for every $x \in {\rm dom}(h)$ $h(x) = g(x \hook \alpha)$ and $f_{q'}(x) < f_p(x \downarrow_S \alpha) + \varepsilon \leq f(x \downarrow_S \alpha) +(g(x \downarrow_S \alpha) - f_p(x \downarrow \alpha)) = g(x \downarrow_S \alpha) = h(x)$. Therefore $q : = (f_{q'}, \bar{\Gamma_{q'} \cup \Psi_g})$ is a condition and satisfies the conclusion of the lemma .
\end{proof}

To prove that $\mathbb P$ is proper the following lemma is the most important.
\begin{lemma}
Let $\theta$ be sufficiently large and let $M \prec H_\theta$ be countable containing $T, S, \mathbb P, etc$. Let $p \in \mathbb P \cap M$ and let $\delta = M \cap \omega_1$. Note that $M \cap T = T_\delta$. Let $D \in M$ be a dense open subset of $\mathbb P$ and let $h:T_\delta \hook S \to \mathbb Q$ be a finite function bounding $f_p$. Then there is an extension $q \in D \cap M$ so that $f_q$ is also bounded by $h$.
\label{submodel}
\end{lemma}

\begin{proof}
Suppose the statement of the lemma is false and let $M$, $T$, $S$, $p$, $D$, $h$ etc be a counter example where $h$ is chosen to be of minimal possible cardinality and $p$ is chosen to witness this. Let me fix that $ht(p) = \alpha$. Note that the assumption implies that if $q \leq p$ and $q \in D \cap M$, then $q$ is not bounded by $h$. Let us enumerate the domain of $h$ by $t^h_0, ..., t^h_{n-1}$. First let me make a reduction in the hypothesis needed. I claim that for each $i < n$ $t^h_i \in S$ and in fact, is a limit point of $S$. Indeed suppose not and note that by the fact that $S$ is closed in $T$, if any $t^h_i$ is not a limit point then its set of predecessors in $S$ is bounded in $T$. In this case, let $\beta > \alpha$ be any ordinal less than $\delta$ so that for all $t^h_i$ which are not limit points of $S$ we have that $\beta$ is greater than the level of $t_i^h \hook S$. Now $h \downarrow_S \beta \in M$ and bounds $p$ so by Lemma \ref{extensionlemma}, applied in $M$ there is a $p' \leq p$ of height $\beta$ bounded by $h \downarrow_S \beta$ and hence $h$. But now, let $h '$ be $h$ restricted to its limit points. Since $|{\rm dom}(h')| < |{\rm dom}(h)|$ by minimality we know that that there is a $q \leq p'$ so that $q \in D \cap M$ and $q$ is bounded by $h'$. But then $f_q$ is actually bounded by $h$, which is a contradiction.

From now on, we assume that all the elements of the domain of $h$ are limit points of $S$, and hence in $S$. Moreover, by strengthening if necessary (and using the normality of $T$) we may assume without loss of generality that $\alpha$ is large enough that for all $i, j < n$ with $i \neq j$ $t^h_i \downarrow_S \alpha \neq t^h_j \downarrow_S \alpha$. Note $h \downarrow_S \alpha \in M$. Work in $M$. Observe that by our initial assumption $h \downarrow_S \alpha$ has the property that if $\gamma \in [\alpha, \omega_1)$ and $g:T_\gamma \hook S \to \mathbb Q$ is a finite function with domain of size $n$ so that $g \downarrow_S \alpha = h \downarrow_S \alpha$ then for any $q \leq p$ with $ht(q) \leq \gamma$ bounded by $g$ $q \notin D$. Let us denote any such $g$ with this property as being {\em bad}. Note that for any $\gamma \in [\alpha, \omega_1)$ $h \downarrow_S\gamma$ is bad, however many other functions may also be bad. It follows in particular though that $M$ thinks there are functions which are bad for whom the minimal level of an element in the domain is arbitrarily high. By elementarity this is also true in $V$. Finally note that if $g: T_\beta \to \mathbb Q$ is bad, and $\alpha \leq \gamma \leq \beta$ then $g \downarrow_S \gamma$ is bad as well. In other words, if a function is bad, then so are its $\downarrow_S$ projections to any level above $\alpha$ as well. Let $B = \{ {\rm dom }(g) \; | \; g \; {\rm is \; bad}\}$. By what we have just argued, $B$ consists of $n$-tuples of arbitrarily large minimal ranks in $T$ and is downward closed above $\alpha$ under projections by $\downarrow_S$.

% In fact more is true, since the domain of $h$ consists of limit points, in $V$ it's true for each $\beta \in [\alpha , \delta)$ there is a bad function whose domain consists of limit points of $T$, all of the same level which is higher than $\beta$. Since this statement only uses parameters from $M$, it's true in $M$ as well and hence $M$ actually thinks that for each $\beta \in [\alpha , \omega_1)$ there is a bad function whose domain is some finite subset of $T_\gamma \cap S$ for some $\gamma > \beta$.

Given any $j \in [\alpha, \omega_1)$ let $B(j)$ be the set of tuples $\vec{s} \in[T_j \hook S]^n$ which are the domain of a bad function. Given tuples $\vec{s}, \vec{t} \in B_\infty$ write $\vec{s} \leq_T \vec{t}$ if for each $i < n$ $s_i \leq_T t_i$ where $s_i$ is the unique element of $\vec{s}$ so that $s_i \downarrow_S \alpha = t^h_i \downarrow_S \alpha$ and idem for $t_i$.

Define recursively $B_0 = B$, $B_{i+1} = \{\vec{s} \in B_i \; | \; {\rm for \; uncountably \; many \; levels \; } j \, \exists \vec{u} \in B_i(j) \, \vec{s} \leq_T \vec{u}\}$, and $B_\lambda = \bigcap_{i < \lambda} B_i$ for $\lambda$ limit. Observe by construction that if $i \leq j$ then $B_i \supseteq B_j$. Let $B_\infty = B_\rho$ where $\rho$ is the least so that $B_\rho = B_{\rho + 1}$. Note that ${\rm dom}(h \downarrow_S \gamma) \in B_i$ for every $i$ so in particular $B_\infty$ is not empty. By construction every element of $B_\infty$ has extensions on uncountably many levels and is closed downwards above $\alpha$.

\begin{claim}
Every $\vec{s} \in B_\infty$ has two extensions, $\vec{s}_0$ and $\vec{s}_1$ in $B_\infty$ so that no element of $\vec{s}_0$ is $\leq_T$-comparable with any element from $\vec{s}_1$.
\end{claim}

This is essentially a consequence of the more general \cite[Lemma 16.18]{JechST}, which states that if $S$ is an Aronszajn tree and $W$ is an uncountable collection of pairwise disjoint subsets of $S$ then there are $\vec{s}_0, \vec{s}_1$ so that no $x \in \vec{s}_0$ is comparable with any $y \in \vec{s}_1$. Letting $W$ be the set of extensions of $\vec{s}$ in $B_\infty$ and applying this lemma gives what we want with one minor caveat. Since in our case $S$ is potentially wide, we need to replace ``uncountable" with ``unbounded" however this is just a cosmetic change, see \cite[Claim 4.8]{DzSh2020}. For completeness here is the proof in our case.

\begin{proof}[Proof of Claim]
Suppose not and let $\vec{s} \in B_\infty$ be a counter example. I will use $\vec{s}$ to define a branch through $S$ contradicting the fact that $S$ is Aronszajn. Let $W \subseteq B_\infty$ be the collection of all $\vec{u}$ extending $\vec{s}$ and for each $i < \omega_1$ let $W(i)$ denote the set of tuples $\vec{z} \in W \cap B(i)$. Since every element of $B_\infty$ has extensions on cofinally many levels, there are uncountably many $\gamma$ so that $W(\gamma) \neq \emptyset$. Also, let's denote the height of $\vec{s}$ by $\gamma_s$. Finally note that since we're assuming that $\vec{s}$ does not have disjoint extensions, given any $\gamma_s < i < j < \omega_1$ we have that if $\vec{z}^i \in W(i)$ and $\vec{z}^j \in W(j)$ then it must be that there is a $k, k' < n$ for which $z^j_k \hook i = z^i_{k'}$.

Let $U$ be an ultrafilter on $W$ all of whose elements contains tuples unboundedly high up in $S$. For any $x \in S$ and $k < n$ let $Y_{x, k}$ be the collection of all elements $\vec{z} \in W$ so that $x$ is comparable with the $k^{\rm th}$ element of $\vec{z}$. Notice by the above assumption, for every $i \in (\gamma_s, \omega_1)$ and any $\vec{z}\in W(i)$ we get that $W = \bigcup_{l < n}\bigcup_{k < n} Y_{z_l, k}$. Since $U$ is an ultrafilter, for any such $i$ we must have that there is an $l_i < n$ and a $k_i < n$ so that $Y_{z^i_{l_i}, k_i} \in U$. But then for some $k$ the set $I = \{ i \in (\gamma_s, \omega_1) \; | \; k_i = k\}$ is uncountable. Let $i < j \in I$ and let $\vec{z} \in W(i)$ and $\vec{y} \in W(j)$. I claim that $z_{l_i} \leq_S y_{l_j}$. To see this, note that since $Y_{z_{l_i}, k} \cap Y_{y_{l_j}, k} \in U$ so there is a $\vec{x} \in W$ of height $\lambda$ in this intersection for some $\lambda > j$ and hence $z_{l_i}, y_{l_j} \leq_S x_k$ so $z_{l_i}, y_{l_j}$ are comparable. But now the set $\{z_{l_i} \; | \; i \in I \; {\rm and}\; \vec{z} \in W(i)\}$ must generate a cofinal branch in $S$, contradiction.
\end{proof}

Let us say that two tuples $\vec{s}_0, \vec{s}_1$ as found in the claim are {\em pairwise disjoint}. By bootstrapping the above argument, there is a level $\beta$ so that $B_\infty (\beta)$ contains an infinite family of pairwise disjoint bad tuples. This is because, given any $\vec{s} \in B_\infty$, by the claim it has two pairwise disjoint extensions $\vec{s}^{ \; 0, 0}$ and $\vec{s}^{\; 0, 1}$ in some $B_\infty(i_0)$ and recursively if for any $n < \omega$, we're given $\vec{s}^{\; n,0}, \vec{s}^{\; n, 1} \in B_\infty(i_n)$ with $i_n > i_{n-1}$ we can find two pairwise disjoint extensions of $\vec{s}^{\; n, 0}$, call them $\vec{s}^{\; n+1, 0}$ and $\vec{s}^{\; n+1, 1}$ in $B_\infty(i_{n+1})$ for some $i_{n+1} > i_n$. Let $\beta$ be the supremum of $\{i_n \; | \; n < \omega\}$ and let $\vec{s}_n$ be an extension of $\vec{s}^{\; n, 1}$ to this level. Then the set $\{ \vec{s}_n \in B_\infty(\beta) \; | \; n < \omega\}$ is an infinite set of pairwise disjoint tuples. 

Now, for each $\gamma \in [\beta, \omega_1)$ let $H_\gamma = \{g \; |  {\rm dom}(g) \in  B_\infty(\gamma) \;{\rm and \;} g {\rm \;is\; bad}\}$. Define in $M$ the promise $\Psi$ as follows: $\beta(\Psi) = \beta$ and $\Psi(\gamma) =\{ H_\gamma\}$. Observe that $\Psi$ is a promise since each $H_\gamma$ is dispersed by the argument in the previous paragraph. By construction $\Psi$ has a basis, namely $h \downarrow_S \alpha$ and, since $h \downarrow_S \alpha$ bounds $f_p$ we can strengthen $p$ so as to include $\Psi$ in its promise as in Lemma \ref{basis}. This is the desired contradiction though since if, $r \leq q$ is a strengthening so that $r \in D \cap M$ then there must be some bad function bounding $r$ contradicting the defining property of being bad.

\end{proof}

\begin{lemma}
$\mathbb P$ is proper. In fact, $\mathbb P$ is dee-complete for some simple completeness system $\mathbb D$.
\end{lemma}

\begin{proof}
Work in the setting of Lemma \ref{submodel}, in particular letting $M \prec H_\theta$ be as before with $\omega_1 \cap M = \delta$. I want to prove the existence of a master condition for $M$. Let $\langle D_n \; | \; n < \omega \rangle$ be an enumeration of the dense open subsets of $\mathbb P$ in $M$. Let $p \in \mathbb P \cap M$ and let $\langle t_i \; | \; i < \omega \rangle$ enumerate the elements of $T_\delta \hook S$. If this set is finite then allow for repetitions. Let $\langle \tau_k \; | \; k < \omega\rangle$ enumerate all the finite subsets of $T_\delta$. I want to define a sequence $p \geq p_0 \geq p_1 \geq ... \geq p_{n} \geq ...$ so that $p_i \in D_i \cap M$ for all $i < \omega$ and there is a condition $q$ extending the union of the $p_i$'s. Such a $q$ defines a generic over $M$. The idea is to use Lemma \ref{submodel} $\omega$-many times to make sure that the union of an $M$ generic filter is bounded and hence can be extended into a further condition. I will then extract from the proof a definition of the generics bounded by such a $q$ and this will be used to define a simple completeness system as needed. 

Fix an enumeration in order type $\omega$ of all triples $e_l = (m_l, n_l, k_l)$ so that $m_l, n_l, k_l \in \omega$ and the first occurrence of $m$ in the first coordinate is after the $m^{\rm th}$ element of the enumeration and each such triple appears infinitely often. For each condition $p' \in M \cap \mathbb P$ let us fix ahead of time an enumeration of $\Gamma_{p'}(\delta)$ in order type $\omega$. Now, using Lemma \ref{submodel}, recursively define conditions $p_{i+1}$ and functions $g_i$, $h_i$ satisfying the following conditions:

\begin{enumerate}
\item
For all $i < \omega$ $p_{i+1} \leq p_i$ and $p_{i+1} \in D_{i+1}$.

\item
For all $i < \omega$, $g_{i} \subseteq g_{i+1}$ is a finite function from $T_\delta \hook S \to \mathbb Q$ bounding $f_{p_{i+1}}$. This uses Lemma \ref{submodel}.

\item
If $H$ is the $n_i^{\rm th}$ requirement of $\Gamma_{p_{m_i}}(\delta)$ in our prefixed enumeration then $h_i \in H$ and has domain disjoint from $\tau_{k_i} \cup \bigcup_{j < i} {\rm dom}(g_j) \cup \{t_{i}\}$ and bounds $p_i$. 

\item
$g_{i+1}$ has domain ${\rm dom}(h_i) \cup \tau_{k_i} \cup \bigcup_{j < i} {\rm dom}(g_j) \cup \{t_{i}\}$ and is bounded by $h_i$ on their shared domain.

\end{enumerate}

Such a sequence can be constructed in much the same way as in the limit case of Lemma \ref{extensionlemma} applying Lemma \ref{submodel} to ensure that each successive $g_{i+1}$ bounds the condition $p_{i+1} \in D_{i+1}$. Moreover this sequence generates a generic filter on $M$. I need to show that there is a lower bound, $q$. Note that $\bigcup_{n < \omega} f_{p_n}$ is a partial specializing function defined on $T_{ < \delta} \hook S$. I claim that we can extend it to a function defined on $T_\delta \hook S$ which keeps the promises $\bigcup_{n < \omega} \Gamma_{p_n}$. Indeed, let $f_q(t_i) = g_{i+1}(t_i)$. This is defined, since we insisted that $t_i \in {\rm dom}(g_i)$. Also, since $g_i$ bounded all $p_i$, $f_q(t_i)$ is at least the supremum of the values of $f_n (t_i \downarrow_S \beta)$ for all $\beta < \delta$. What needs to be checked is that $f_q$ actually keeps all the promises in the $p_i$'s. This is what was planned for though. If $H \in \Gamma_q(\delta)$ then $H \in \Gamma_{p_i} (\delta)$ for some $i$ and for any $\tau \subseteq T_\delta$ finite, there was a stage where we ensured that $f_q$ was bounded by some $h_n$ which included being bounded by some $H$ on a node disjoint from $\tau$. Then, from that stage on, since all $p_j$'s were bounded by this $h_n$ since they're bounded by the $g_n$ we constructed at that stage which itself was boudned by $h_n$, we get that $f_q$ keeps that instance of the promise.

Thus we have shown that $q$ is an $(M, \mathbb P)$-master condition so $\mathbb P$ is proper. It remains to show that it is in fact dee-complete. The issue is that the proof above required knowledge of $T_\delta$ and $T_\delta \hook S$, which we do not have from $M$ alone. However, given a countable set of ``potential nodes" for $T_\delta$ and $T_\delta \hook S$ we can run the argument above, and, if we add in on top of that the true $T_\delta$ and $T_\delta \hook S$ then this won't change anything. This motivates the idea of a {\em possible continuation} defined below.

Given a sufficiently large $\theta$ and a countable transitive $\sigma:\bar{M} \prec H_\theta$ so that $\sigma (\bar{\P}, \bar{T}, \bar{S}) = \P, T, S$, let us say that an element $u \in H_{\omega_1}$ is a {\em possible continuation} for $\bar{M}$ and $\bar{\P}$ if $u$ codes a triple $\langle T^*, S^*, c\rangle$ so that:
\begin{enumerate}
\item
$T^*$ is a normal tree of height $\delta +1$ so that $T^*_{<\delta} = T_{<\delta}$ and $T^*_{\delta}$ is countable for $\delta = M \cap \omega_1$. Here we associate an element $t \in T^*_{\delta}$ with its set of predecessors so we can think of $t$ as a subset of $\bar{M}$.
\item
$S^* \subseteq T^*$ is a closed subtree with the induced suborder so that $S^* \cap T_{<\delta} = S \cap T_{< \delta}$.
\item
$c: \bar{\P} \to \mathcal P(\bar{M})$ is a function so that $c(\bar{p})$ is a countable set of requirements of height $\delta$ so that for each $\alpha < \delta$ and each $\bar{H} \in \Gamma_{\bar{p}}(\alpha)$ there is a $H \in c(\bar{p})$ so that $\bar{H} = \{h \downarrow_S \alpha \; | \; h \in H\}$.
\end{enumerate}
Note that being a possible continuation is definable in $H_{\omega_1}$ from $\bar{M}$ and $\bar{\P}$. Now, if $\bar{M}$ is the transitive collapse of some $M \prec H_\theta$ and $u = \langle T^*, S^*, c\rangle$ is a possible continuation for $\bar{M}$ and $\bar{\P}$ let us say that a $\bar{\mathbb P}$-generic $\bar{G}$ over $\bar{M}$ {\em respects} $u$ if there is a function $f:T^{**} \to \mathbb Q$ for some tree $T^{**}$ of height $\delta+1$ so that 
\begin{enumerate}
\item
$T^{**}_\delta \supseteq T^{*}_\delta$ and $T^{**}_{<\delta} = T_{<\delta}$
\item
$f$ is a partial specializing function.
\item
$f \supseteq \bigcup_{\bar{q} \in \bar{G}} f_{\bar{q}}$
\item
$f$ fulfills all of the requirements in $c(\bar{q})$ for each $\bar{q} \in \bar{G}$.
\end{enumerate}
Finally we say that a $\mathbb P$-generic $G$ over $M$ {\em respects} $u$ if $G$ is generated by $\sigma `` \bar{G}$ for a $\bar{G}$ which respects $u$. Note that if $u = \langle T^*, S^*, c\rangle$ is such that $T^*$ contains the set of branches with upper bounds in $T_\delta \cap S$ and $G$ respects $u$ as witnessed by $f$ then $G$ has a lower bound: the pair consisting of the partial specializing function $f$ and the promise generated by $\sigma``\bigcup_{\bar{q} \in \bar{G}}c(\bar{q})$. Also, given any possible continuation $u$, generics which respect $u$ exist for every condition and model by the argument for properness in the first half of this proof.

Finally we can define our completeness system by letting $\mathbb D(\bar{M}, \bar{\mathbb P}, \bar{p})$ be the set of $A_u$ for $u \in H_{\omega_1}$ where if $u$ is a possible continuation for $\bar{M}$ and $\bar{\mathbb P}$ then $A_u$ is the set of generics which respect $u$ and if $u$ is not a possible continuation then $A_u$ is all generics. This is definable and satisfies the conditions of a completeness system. The only thing that is not immediately clear is the countable closure. This is why promises consist of countable sets of requirements: suppose that $\{\langle T^*_i, S^*_i, c_i\rangle \; | \; i < \omega\}$ is a countable set of possible continuations and let $T^* = \bigcup_{i < \omega} T^*_i$, $S^* = \bigcup_{i < \omega} S^*_i$ and $c$ be the function sending $\bar{p} \mapsto \bigcup_{i < \omega} c_i(\bar{p})$. Then $u = \langle T^*, S^*, c \rangle$ is a possible continuation and any generic that respects $u$ respects all $\langle T^*_i, S^*_i, c_i\rangle$ hence $\bigcap_{i < \omega} A_{\langle T_i^*, S^*_i, c_i\rangle}$ is nonempty.

\end{proof}

Finally I prove that $\mathbb P$ is $\alpha$-proper for all $\alpha < \omega_1$. 
\begin{lemma}
Let $\alpha < \omega_1$ and let $\vec{N} = \langle N_i \; | \; i \leq \alpha\rangle$ be a tower of length $\alpha$ for $N_i \prec H_\theta$, $\theta$ sufficiently large with $\mathbb P \in N_0$. Then for any $p \in N_0 \cap \mathbb P$ there is a $q \leq p$ which is $(N_i, \mathbb P)$-generic simultaneously for every $i \leq \alpha$.
\end{lemma}

\begin{proof}
If $\alpha$ is a successor ordinal, this is just the proof of properness given above so assume that $\alpha$ is a limit ordinal. Pick an increasing sequence $\langle \alpha_n \; | \; n < \omega \rangle$ with ${\rm sup}_{n < \omega} \alpha_n = \alpha$. Let $\delta = \omega_1 \cap N_\alpha$. One can perform the same proof as when it was proved that $\mathbb P$ was proper, except now we insist (via the inductive assumption) that $p_i$ be $(N_j , \mathbb P)$-generic for all $j < i$ and $p_i \in N_i$ as opposed to $p_i$ being in some specified dense open. Since, by the definition of a tower $\langle N_j \; | \; j < i \rangle \in N_i$ this is possible (given the sequence, by elementarity, $N_i$ can find a master condition). Moreover, since, again by definition of a tower, the sequence of models is continuous and in particular, $N_\alpha = \bigcup_{n < \omega} N_{\alpha_n}$ the set $\{r \in N_\alpha \cap \mathbb P \; |\; \exists i \, p_i \leq r\}$ is $(N_\alpha, \mathbb P)$-generic. The only thing to be careful about is that the union of the $p_i$'s can be extended to some $q$ of height $\delta$. However, by iteratively applying Lemma \ref{submodel} as in the previous proof this is easily accounted for.
\end{proof}

Therefore $\mathbb P$ is dee-complete and ${<}\omega_1$-proper, thus proving Theorem \ref{specwidetree}. We get as an immediate corollary the following.
\begin{corollary}
Assume $\DCFA$. Every wide Aronszajn tree which embeds into an $\omega_1$-tree is special.
\end{corollary}

We can also iterate this forcing with countable support to obtain the following (with no consistency strength). Note that under $\CH$ the forcing notion $\mathbb P$ has the $\aleph_2$-c.c. since any two conditions with the same partial specializing function are compatible. 

\begin{corollary}
It's consistent with $\CH$ that all wide Aronszajn trees which embed into an $\omega_1$-tree are special.
\end{corollary}

In contrast to the next section, note that by the $\aleph_2$-c.c. all cofinalities, and hence cardinals, are preserved which implies that whatever Kurepa trees existed in the ground model are still Kurepa in the extension above. 

I will give a concrete application of such a wide Aronszajn tree in the next section. Let me note first that the condition is not trivial: there are wide Aronszajn trees in $\ZFC$ which cannot be embedded into $\omega_1$ trees.

\begin{lemma}(Essentially Todor\v{c}evi\'c, see \cite[Definition 3.2]{Todorcevic1981})
There is an $(\omega_1, {\leq}2^{\aleph_0})$-Aronszajn tree which is $\ZFC$-provably non-special, and cannot be specialized by any forcing not adding reals.
\label{lemmaT(S)}
\end{lemma}

\begin{proof}
Let $E \subseteq \omega_1$ be stationary, co-stationary and let $T(E)$ be the tree of attempts to shoot a club through $E$. In other words, elements of $T$ are closed, bounded, countable initial segments of $E$ ordered by end extension. This poset is well known to be $\sigma$-distributive, hence the tree has height $\aleph_1$. Also, every element is a countable set of ordinals hence it can be coded by a real and therefore the tree has width $2^{\aleph_0}$. So we conclude that $T(E)$ is an $(\omega_1, {\le}2^{\aleph_0})$-Aronszajn tree. However, it can't be special, since, as mentioned before, forcing with this tree does not add reals, so in particular, $\omega_1$ is preserved. To see that it remains non-special in every forcing extension not adding reals, note that, if $\mathbb P$ does not add reals then the reinterpretation of $T(E)$ in $V^\mathbb P$ is just $\check{T(E)}$ so it's still $\sigma$-distributive and hence it must still not be special.
\end{proof}

Putting together this lemma and Theorem \ref{specwidetree} we conclude the following odd result which may be of independent interest. Note that the theorem below is provable in $\ZFC$.

\begin{theorem}
For any stationary, co-stationary $E \subseteq \omega_1$ the tree $T(E)$ cannot be embedded into any $\omega_1$-tree.
\label{embedthm}
\end{theorem}

\begin{proof}
Suppose $T(E)$ could be embedded into an $\omega_1$-tree. Then, by forcing with the forcing from Theorem \ref{specwidetree} we could make $T(E)$ special without adding reals. But this contradicts Lemma \ref{lemmaT(S)}.
\end{proof}

\begin{corollary}
$\DCFA$ is consistent with the existence of non-special wide Aronszajn trees.
\label{cornonspec}
\end{corollary}

\begin{proof}
If $\CH$ holds, which it does in the natural model of $\DCFA$, then the tree $T(E)$ witnesses the corollary.
\end{proof}

\section{$\DCFA$ implies there are no Kurepa Trees}

In this section I use the forcing from the previous section to prove that $\DCFA$ implies there are no Kurepa trees. This fleshes out an idea sketched in \cite[Chapter VII, Application G]{PIPShelah} and was the motivation of proving Theorem \ref{specwidetree}.
\begin{theorem}
Under $\DCFA$ there are no Kurepa Trees.
\label{DCFAimpliesnotKH}
\end{theorem}

My interest in this result stems from that fact that, in contrast to the situation with $\SCFA$ from the previous section, there are positive implications of $\DCFA$ on the level of $\omega_1$, regardless of the size of the continuum. The proof of Theorem \ref{DCFAimpliesnotKH} follows Baumgartner's original proof from $\PFA$, however using the poset from Theorem \ref{specwidetree}. I include a detailed proof for completeness, but to be clear, there is nothing new here beyond the use of the forcing from Theorem \ref{specwidetree}. Again, the idea of using dee-complete specializing forcing in this way was already present in \cite{PIPShelah}. Note that there the additional hypotheses of $\CH$ and $2^{\aleph_1} = \aleph_2$ are stated, but they do not actually appear in the proof sketch Shelah gives.

\begin{proof}
Assume $\DCFA$ and suppose towards a contradiction that $T$ is a Kurepa tree. Let $\lambda \geq \aleph_2$ be the number of branches through $T$. First, force with $Col (\lambda, \aleph_1)$, the $\sigma$-closed forcing to collapse $\lambda$ to $\aleph_1$. Note that, being $\sigma$-closed, this is dee-complete and ${<}\omega_1$-proper. Work in the collapse extension. As noted in Lemma 7.11 of \cite{BaumPFA} $\sigma$-closed forcing won't add uncountable branches to a tree of width ${<}2^{\aleph_0}$ hence, in particular, there are no new branches added to $T$ in the extension.

I use the following claim, due to Baumgartner, see \cite[Lemma 7.7]{BaumPFA}.
\begin{claim}
There is an uncountable subtree $S \subseteq T$ with no uncountable branches.
\end{claim}

\begin{proof}
Let $\mathcal B$ denote the set of uncountable branches through $T$. By the remark preceding the claim we have that $|\mathcal B| = \aleph_1$. By Lemma 7.6 of \cite{BaumPFA} there is an injection $g:\mathcal B \to T$ so that for each $b \in \mathcal B$ $g(b) \in b$ and whenever $g(b_1) <_T g(b_2)$ then $g(b_2) \notin b_1$. Now let $S = \{t \in T \; | \; \forall b \in \mathcal B \, {\rm if} \; t \in b \, {\rm then} \; t \leq_T g(b)\}$. This is a tree with the order inherited from $T$. Moreover, it's uncountable since it countains the range of $g$. It remains to see that this tree has no uncountable branches.

Towards a contradiction, suppose that $b$ were an uncountable branch through $S$. Let $\bar{b} = \{t \in T \; | \; \exists s \in b \; t <_T s\}$ i.e. the downward closure of $b$ in $T$. This must be an uncountable branch through $T$. But then since $g(\bar{b}) \in \bar{b}$ we get an $s \in b$ with $g(\bar{b}) \leq_T s$ contradicting the definition of $S$.
\end{proof} 

Now, by applying the specializing forcing $\mathbb P_{T, S}$ from Theorem \ref{specwidetree} to $T$ and $S$ and working in that extension we have that $S$ is special. Let $f:S \to \mathbb Q$ be such a specializing function. Let $t\in T \setminus S$. We extend $f$ to include $t$ as follows. Since $t \notin S$ there is a branch $b$ so that $t \in b$ but $g(b) \leq_T t$. This branch is unique: If $g(b_1) <_T g(b_2) <_T t$ with $t\in b_1 \cap b_2$ then in particular $g(b_2) \in b_1$ which contradicts the choice of $g$. Now let $f(t) = f(g(b))$ for this unique branch. 

\begin{claim}
$f:T \to \mathbb Q$ has the property that if $f(s) = f(t) = f(u)$ and $s \leq_T t, u$ then $t$ and $u$ are comparable.
\end{claim}

\begin{proof}
Let $s \leq_T t, u$ be as in the claim. Since $f(t) = f(s)$ at least one of $t$ and $s$ is not in $S$ since $f$ is injective on chains in $S$. In fact neither $s$ nor $t$ are in $S$ unless $s = g(b)$ for some $b$. To see this, first note that if $s \in S$ then, since $t \notin S$ we would have that there is some $b$ so that $b$ is the unique branch with $t \in b$ and $g(b) \leq_T t$ and, since $s \in b$ as well and $s \in S$ we have that $s \leq_T g(b)$ and so either $s = g(b)$ or $f(s) \neq f(g(b)) = f(t)$ which is a contradiction. Similarly if $t \in S$ then since $s \notin S$ there is some branch $c$ so that $s \in c$ but $g(c) \leq_T s$ and since $g(c), t \in S$ and $g(c) <_T t$ we have that $f(g(c)) \neq t$ but this is a contradiction since $f(g(c)) = f(s) = f(t)$.

Now, let $b$ be the unique branch so that $t \in b$ and $g(b) \leq_T t$. As noted before, $s \in b$ as well. If $s <_T g(b)$ then there is a branch $c \neq b$ so that $s \in c$ and $g(c) \leq_T s$ (since either $s = g(c)$ or is above it, by the argument in the previous paragraph). But now $g(c), g(b) \in S$ and $g(c) <_T g(b)$ so $f(g(c)) \neq f(g(b))$ but this is a contradiction since $f(s) = f(g(c))$ and $f(t) = f(g(b))$. Therefore $g(b) \leq_T s$, $b = c$ and hence $s \in b$. A symmetric argument allows one to conclude the same for $u$ so $t, s, u \in b$ and hence are comparable.
\end{proof}

Now, finally applying $\DCFA$ we can pull back to $V$ and find an $f:T \to \mathbb Q$ as in the last claim (note that such an $f$ required meeting only $\aleph_1$ many dense sets since $|T| = \aleph_1$). Therefore, the proof of Theorem \ref{DCFAimpliesnotKH} will be done once we show that the existence of such a function $f$ implies there are at most $\aleph_1$ many cofinal branches through $T$. This is Theorem 7.4 of \cite{BaumPFA} coupled with the remarks preceeding its statement on page 949 of the same article. From such an $f:T \to \mathbb Q$ we can define a function $g:\mathcal B \to T$ as follows. For each $b \in \mathcal B$ by pigeonhole there is some $r \in \mathbb Q$ so that $\{t \in b \; |\; f(t) = r\}$ is cofinal in $b$. Pick such an $r$ and let $g(b)$ be the least $t \in b$ with $f(t) = r$ (or indeed any such $t$). By the definition of $f$, if $b_1 \neq b_2$ then $g(b_1) \neq g(b_2)$. To see this, suppose that $g(b_1) = g(b_2) = s$, let $f(s) = r$ and let $t \in b_1 \setminus b_2$ with $f(t) = r$ and $u \in b_2 \setminus b_1$ with $f(u) = r$. Such $t$ and $u$ exist by the assumption on $r$. But this is a contradiction since we have that $s \leq_T t, u$ with $f(s) = f(t) = f(u)$ and $t$ and $u$ are incomparable. Therefore $g$ is an injection from $\mathcal B$ into $T$ so $|\mathcal B| \leq \aleph_1$.
\end{proof} 

An $\omega_1$-tree $T$ is called {\em essentially special} if there is an $f:T \to \mathbb Q$ which is (weakly) increasing on chains and for all $s, t, u \in T$ if $s \leq_T t, u$ and $f(s) = f(t) = f(u)$ then $t$ and $u$ are comparable. The above proof actually shows the following.

\begin{theorem}
Under $\DCFA$ all $\omega_1$-trees are essentially special. 
\end{theorem}

In contrast with the case of $\PFA$, note that by Corollary \ref{cornonspec} this result cannot be improved to trees of width $\omega_1$. Note also that $\SCFA$ does not imply this since consistently there may be Souslin trees in a model of $\SCFA$. As mentioned in the previous chapter it is open whether or not $\SCFA$ implies that there are no Kurepa trees.

\section{Looking Forward: Some Cardinal Characteristics and Many Open Questions}

The previous sections suggest some new directions for studying wide trees, particularly in connection with cardinal characteristics. While I leave an in depth investigation of these ideas for future research I want to finish this chapter by recording some easy observations and connecting them back to what has been shown.

The main observation is that the behavior of trees is as much connected to their width and cardinality as to their height. This is obscured by the fact that the ccc forcing to specialize a tree works equally well regardless of the width of the tree. However, the trees of the form $T(E)$ suggest that there is something more subtle going on with regards to specializing wider trees. The following cardinals attempt to measure this.

\begin{definition}
\begin{enumerate}
\item
$\mathfrak{st}$, the $\mathfrak{s}$pecial $\mathfrak{t}$ree number, is the least cardinal $\lambda$ such that there is a non-special $(\omega_1, {\leq}\lambda)$-Aronszajn tree of cardinality $\lambda$.
\item
$\mathfrak{no}$, the $\mathfrak{no}$ new reals number, is the least cardinal $\lambda$ of an $(\omega_1, 
{\leq}\lambda)$-Aronszajn tree of cardinality $\lambda$ which can be forced to be special without adding reals.
\end{enumerate}
\end{definition}

Let's make some easy observations.
\begin{observation}
$\aleph_1 \leq \mathfrak{st} \leq \mathfrak{no}\leq \mfc$
\end{observation}

\begin{proof}
That $\mathfrak{st}$ is uncountable is essentially by definition. To see that $\mathfrak{st}\leq \mathfrak{no}$ it suffices to note that any special tree is obviously specializable without adding reals (by trivial forcing). Finally Todorcevic's tree $T(E)$ defined above witnesses that there is always a tree of size continuum that cannot be specialized without adding reals.
\end{proof}

I do not know exactly what these cardinals can be. It's clear that $\mathfrak{st}$ can remain $\aleph_1$ in models where many other cardinal characteristics are big since nearly all known cardinal characteristics can be made to have size continuum ($\neq \aleph_1$) while preserving the existence of a Souslin tree since we can make all cardinals (except $\mathfrak{m}$) in the Cicho\'n and van Douwen diagrams large using $\sigma$-linked forcing. The following however is less clear.
\begin{question}
What provable bounds exist between known cardinal invariants and $\mathfrak{st}$? For instance, is it provable that $\mathfrak{st} \leq \mfd$? 
\end{question}
The number $\mathfrak{no}$ seems even more mysterious. I do not even know if it can consistently be less than the continuum.
\begin{question}
Is it consistent that $\mathfrak{no} < \mfc$? Is it consistent that $\mathfrak{st} < \mathfrak{no}$?
\end{question}
A potentially easier question, for which I conjecture the answer is ``yes" is the following:
\begin{question}
Does $\DCFA$ imply that $\mathfrak{no} = \mfc$?
\end{question}

Finally let me ask about extensions of the main theorem of this chapter.
\begin{question}
Are there (in $\ZFC$) trees which can be specialized without adding reals but are not embeddable into $\omega_1$-trees?
\end{question}

The use of forcing notions which specialize wide trees is key in several important applications of $\PFA$ including failure of various square principles, and the tree property on $\omega_2$. Therefore a natural question is whether the forcing $\mathbb P_{T,S}$ can be substituted in in these arguments.
\begin{question}
What other consequences of $\DCFA$ (possibly with some additional cardinal arithmetic assumption) can be obtained using $\mathbb P_{T,S}$? Does $\DCFA + \neg \CH$ imply the tree property on $\omega_2$? Does it imply the failure of weak square on $\omega_1$?
\end{question}
\clearpage

% end matter
\backmatter

\singlespacing

%\printbibliography[heading=bibintoc,title={Bibliography}]

\bibliographystyle{plain}
\bibliography{PhDBib}

\end{document}